\documentclass[a4paper,11pt]{scrbook}

\usepackage[francais]{babel}

\usepackage[latin1]{inputenc}
\usepackage[T1]{fontenc}
\usepackage{pifont}

\usepackage{amsmath}
\usepackage{amssymb}
\usepackage{amsthm}
\usepackage{amscd}
\usepackage{dsfont}

  \usepackage[all]{xy}
\usepackage{mathtools}

\usepackage{chngcntr}
\usepackage{array}

\usepackage{youngtab}

\def\qed {{
\parfillskip=0pt        
\widowpenalty=10000     
\displaywidowpenalty=10000  
\finalhyphendemerits=0  
\leavevmode             
\unskip                 
\nobreak                
\hfil                   
\penalty50              
\hskip.2em              
\null                   
\hfill                  
$\square$
\par}}                  

\usepackage{graphicx}

\usepackage{pinlabel}

\DeclareFontFamily{U}{matha}{\hyphenchar\font45}
\DeclareFontShape{U}{matha}{m}{n}{
      <5> <6> <7> <8> <9> <10> gen * matha
      <10.95> matha10 <12> <14.4> <17.28> <20.74> <24.88> matha12
      }{}
\DeclareSymbolFont{matha}{U}{matha}{m}{n}

\DeclareMathSymbol{\varleftarrow}{3}{matha}{"D0}
\DeclareMathSymbol{\varrightarrow}{3}{matha}{"D1}

\usepackage{float}

\usepackage[toc,page]{appendix}

\usepackage[alphabetic,abbrev,backrefs,lite]{amsrefs}

\usepackage{geometry}
\numberwithin{equation}{chapter}

\newtheorem{theoreme}			     {Théorème}	[chapter]
\newtheorem{proposition}[theoreme]	 {Proposition}	
\newtheorem{corollaire}	  [theoreme]	 {Corollaire}	
\newtheorem{lemme}	      [theoreme]  {Lemme}
			
\newtheorem{definition}	        [theoreme]  {Définition}

\newtheorem{theorem} {Theorem}

\theoremstyle{definition}
\newtheorem{exemple}			 [theoreme]     {Exemple}

\newtheorem{remarque}			     [theoreme] {Remarque}	
\newtheorem{example} {Example}

\newcommand{\R}{\mathbb{R}}
\newcommand{\Q}{\mathbb{Q}}
\newcommand{\Z}{\mathbb{Z}}
\newcommand{\N}{\mathbb{N}}

\newcommand{\C}{\mathbb{C}}

\newcommand{\Hno}{OH^n}
\newcommand{\HSt}{H(\widetilde S)}

\newcommand\minidiag[1]{\raisebox{-0.25\height}{\includegraphics[height=0.5cm]{#1}}}
\newcommand\pitidiag[1]{\raisebox{-0.25\height}{\includegraphics[height=1cm]{#1}}}

\newcommand\deuxdiag[1]{\raisebox{-0.5\height}{\includegraphics[height=1.33cm]{#1}}}

\newcommand\smalldiag[1]{\raisebox{-0.5\height}{\includegraphics[height=1.5cm]{#1}}}

\newcommand\diagg[1]{\raisebox{-0.5\height}{\includegraphics[height=2cm]{#1}}}
\newcommand\bigdiag[1]{\raisebox{-0.5\height}{\includegraphics[height=4cm]{#1}}}
\newcommand\bbigdiag[1]{\raisebox{-0.5\height}{\includegraphics[height=6cm]{#1}}}

\newcommand\middiag[1]{\raisebox{-0.5\height}{\includegraphics[height=3cm]{#1}}}
\newcommand\diagsc[1]{\raisebox{-0.5\height}{\includegraphics[height=1.75cm]{#1}}}

\DeclareMathOperator{\rank}{rk}
\DeclareMathOperator{\Hom}{Hom}
\DeclareMathOperator{\Eq}{Eq}
\DeclareMathOperator{\img}{im}
\DeclareMathOperator{\ext}{Ext}
\DeclareMathOperator{\Inv}{Inv}
\DeclareMathOperator{\Id}{Id}

\newcommand{\Ext}{\mathchoice{{\textstyle\bigwedge}}%
    {{\bigwedge}}%
    {{\textstyle\wedge}}%
    {{\scriptstyle\wedge}}}

\makeatletter
\newcommand{\X}[1]{\ifmeasuring@ #1\else\fgnu@X{#1}\fi}
\newcommand{\Y}[1]{\ifmeasuring@ {}:#1\else\fgnu@Y{#1}\fi}
\def\fgnu@X#1{\hbox to \ifcase1\maxcolumn@widths\fi{\hfil$\displaystyle#1$\hfil}}
\def\fgnu@Y#1{\hbox to \ifcase2\maxcolumn@widths\fi{\qquad$\displaystyle#1$\hfil}}
\makeatother

\newdir{|>}{-<2pt,0pt>{\blacktriangleright}}

\usepackage{chemarrow}

\begin{document}


\begin{titlepage}

\Large

\quad
\vspace{\stretch{1}}

\begin{center}
\hrule
\vspace{10mm}
{\huge 
Construction impaire et \'etude de l'anneau des arcs de Khovanov
}
\vspace{10mm}
\hrule

\vspace*{\stretch{1}}

{\LARGE
Grégoire Naisse
}

\vspace*{\stretch{1}}

Promoteur: Pedro Vaz

\vspace*{\stretch{2}}

Université catholique de Louvain\\
Faculté des Sciences\\
École de Mathématique

\vspace{1cm}

{\LARGE
2014--2015}
\end{center}
\end{titlepage}


\chapter*{English summary}
\addcontentsline{toc}{chapter}{English summary}

In this master thesis, we give an oddification of the Khovanov's arc rings $H^n$ from \cite{Khovanov}. Our construction is based on the odd Khovanov homology from P. Ozsvath, J. Rasmussen and Z. Szabo (see \cite{OddKhovanov}) and thus depends on some choices of signs. More precisely we have to choose an order and an orientation for the saddle points. Set $\mathcal C^n$ be the set of all such choices.

\begin{theorem}
For each $n \in \N$ we have a family of rings $\{OH^n_C\}_{C \in \mathcal C^n}$.
\end{theorem}


By explicit computations, we show that for all $n \ge 2$ and all $C$, the ring $OH^n_C$ is non-associative. Of course, it is associative up to sign and by tensoring with $\Z/2\Z$ we get the same ring as the mod $2$ reduction of Khovanov's rings.

\begin{example}
As an example of non-associative elements in $OH^2_C$ for arbitrary $C \in \mathcal C^2$ consider $a,b \in B^2$ (see \cite{Khovanov} for a definition of $B^n$) such that
\begin{align*}
a &= \smalldiag{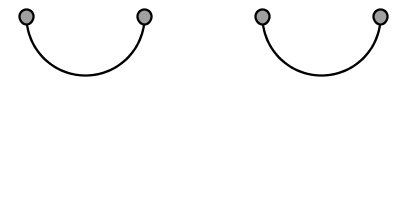}, & b &= \smalldiag{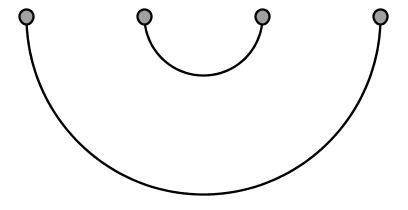}.
\end{align*}
If we take $x = a_1 \wedge 1 \in a(OH^n_C)a, y = 1 \in a(OH^n_C)b$ and $z = 1 \in b(OH^n_C)a$, where $a_1$ is the element in the exterior algebra generated by one of the circles from the diagram $W(a)a$ and $W$ is the involution which flips the diagram vertically, we get
$$(xy)z = -x(yz).$$
There is also an example of a choice $C \in \mathcal C^3$ with $a, b, c,d \in B^3$ such that if we take $x = 1 \in a(OH^3_C)b$, $y = 1 \in b(OH^3_C)c$ and $z = 1 \in c(OH^3_C)d$ we get $x(yz) = -(xy)z$.\\
\end{example}

Like in the even case (see \cite{HnCenter}), there is a link between these rings and the cohomology of the $(n,n)$ Springer varieties. 
In the context of odd theory we consider an extended version of the center which includes the anti-commutative elements and we call it the "odd center" or "supercenter". We write it $OZ(OH^n_C)$. We show that there is a link with the oddification of the cohomology of the $(n,n)$ Springer variety constructed by A. Lauda and H. Russell in \cite{OddSpringer}, denoted $OH(\mathfrak{B}_{n,n})$ in our work. 

\begin{theorem}
The odd center of $OH^n_C$ is an associative ring and does not depends on the choice of $C \in \mathcal C^n$. Furthermore, there is an isomorphism of graded rings
$$OZ(OH^n_C) \simeq OH(\mathfrak{B}_{n,n}).$$
\end{theorem}
The proof of this theorem is split in three main steps :
\begin{enumerate}
\item We construct a graded morphism $s_0$ from the ring of odd polynomials in $2n$ variables (see $\text{OPol}_{2n}$  in \cite{OddSpringer}) to $OZ(OH^n_C)$ which is similar to the isomorphism from \cite{HnCenter}. Then, by showing that the ideal used to define $OH(\mathfrak{B}_{n,n})$ is in the kernel of $s_0$, we get an induced graded morphism
$$s : OH(\mathfrak{B}_{n,n}) \rightarrow OZ(OH^n_C).$$
\item Using the equivalence modulo $2$ between the odd and the even case, the existence of a basis for $OH(\mathfrak{B}_{n,n})$ and the isomorphism $H(\mathfrak{B}_{n,n}) \simeq Z(H^n)$, we prove that $s$ is injective.
\item Finally, we show that the rank of $OH(\mathfrak{B}_{n,n})$ is the same as that of $OZ(OH^n_C)$ and consequently that $s$ is bijective. To do so, we construct a variety $\widetilde T$, in a way similar to the $\widetilde S$ from \cite{HnCenter} but using circles instead of spheres. By employing similar arguments as Khovanov, we get an isomorphism of non-graded rings between the cohomology $H(\widetilde T)$ and $OZ(OH^n_C)$ and we prove that the rank of $H(\widetilde T)$ is what we are looking for.
\end{enumerate}
Also, we get a non-graded isomorphism between $H(\widetilde T)$ and $OH(\mathfrak{B}_{n,n})$ and thus this proves that the construction of A. Lauda and H. Russell gives a presentation for the cohomology ring of $\widetilde T$ if we replace the degrees of the generators by 1 instead of 2.\\


In the last section of the thesis, thanks to Krzysztof Putyra, we construct an associator for $OH^n_C$ which is based on the degrees and the diagrams of the elements. Then, we find a sufficient condition for having isomorphic rings with two different choices of order and orientations for the saddle points.

\begin{theorem}
Let $C$ and $C'$ be in $\mathcal C^n$. If the associator of $OH^n_C$ and $OH^n_{C'}$ are the same, then there is an isomorphism of graded rings
$$OH^n_C \simeq OH^n_{C'}.$$
\end{theorem}

\paragraph{Acknowledgment} 
None of this would have been possible without the help of my thesis supervisor Pedro Vaz.

\tableofcontents


%
\chapter*{Remerciements}
\addcontentsline{toc}{chapter}{Remerciements} 

Je remercie tout d'abord mon promoteur, Pedro Vaz, pour ses qualités de mentor, son soutien durant l'écriture de ce mémoire et sa grande disponibilité ainsi que pour m'avoir ouvert la voie à la théorie des n\oe uds et à toutes les mathématiques qu'elle implique.\\ 

Je remercie aussi Krzysztof Putyra pour ses explications et ses idées qui ont permis d'enrichir fortement le Chapitre \ref{chap:allerplusloin}.\\

Je souhaite remercier tous mes professeurs de l'Université catholique de Louvain pour avoir partagé avec moi leur savoir et leur amour des mathématiques. \\

Je remercie également mes parents Anne et Frédéric pour leur soutien et pour m'avoir permis de faire ces études. De même je remercie mon frère Corentin pour avoir toujours été un modèle exemplaire pour moi. \\

Enfin je voudrais remercier Isaure pour son affection et son soutien, ainsi que tous mes amis pour être géniaux.

\chapter*{Introduction}
\addcontentsline{toc}{chapter}{Introduction} 

La théorie des n\oe uds étudie principalement le plongement de cercles dans l'espace, qu'on appelle des entrelacs. La grande question est de pouvoir déterminer si deux entrelacs sont isotopes. Actuellement, on ne peut répondre que partiellement à cette question, et ce, en utilisant des invariants, c'est-à-dire des objets algébriques associés à chaque entrelacs : deux entrelacs ayant des invariants différents étant assurément non-isotopes.

Le monde de la théorie des n\oe uds fut révolutionné en 1985 par V. F. R. Jones \cite{Jones} qui a construit un invariant polynomial très simple à calculer puisqu'il s'obtient à partir de la relation d'écheveau locale :
$$(q^{1/2} - q^{-1/2})V\left(\smalldiag{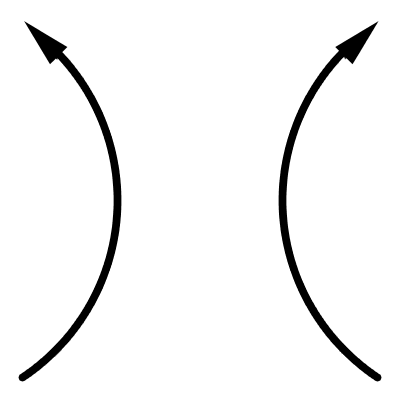} \right)= q^{-1} V\left(\smalldiag{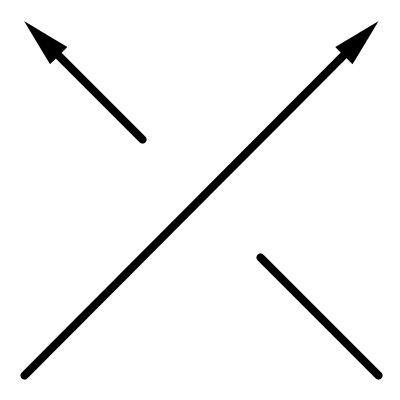}\right) - q V\left(\smalldiag{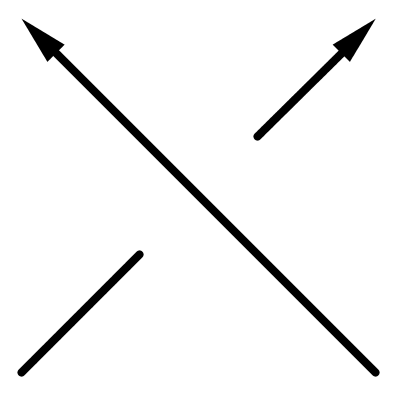}\right).$$
Cela fut révolutionnaire puisqu'avant cette date, les seuls invariants connus étaient soient très peu puissants (comme par exemple le nombre d'entrelacements), soit difficilement calculables (comme le groupe fondamental du complément d'un n\oe ud). Par après, on a vu apparaitre toute une série d'autres invariants du même type (par exemple le polynôme de HOMFLY ou encore celui de Kauffman) jusqu'à une seconde révolution dans le milieu due à M. Khovanov \cite{KhovanovHomology} qu'il qualifia de "catégorification du polynôme de Jones". 

La catégorification imaginée par M. Khovanov consiste en un invariant homologique, c'est-à-dire qu'on calcule l'homologie d'un certain complexe de chaines associé à un diagramme d'entrelacs et cette homologie permet de distinguer certains entrelacs non-isotopes. Pour construire ce complexe, on transforme un diagramme d'entrelacs en une famille de collections de cercles et ces collections sont reliées par des cobordismes. On définit alors un foncteur qui envoie une collection de $n$ cercles vers le produit tensoriel de $n$ copies d'un certain groupe abélien gradué $A$ et un cobordisme entre deux telles collections vers un homomorphismes de modules gradués sur $\Z$.
 \`A un diagramme d'entrelacs $L$, on associe donc des groupes de cohomologie doublement gradués, noté $H^{i,j}(L)$, $i$ étant l'indice du groupe de cohomologie et $j$ le degré induit par $A$. M. Khovanov a montré que ces groupes de cohomologie constituent des invariants d'entrelacs. De plus, l'homologie de Khovanov forme un foncteur de la catégorie des entrelacs (avec les flèches données par des cobordismes plongés dans $\R^4$) vers les groupes abéliens. On peut définir une caractéristique d'Euler graduée pour ces groupes de cohomologie par :
$$\sum_{i,j} (-1)^i q^j \dim_\Q (H^{i,j}(L) \otimes \Q) \in \Z[q,q^{-1}]$$
et par la construction de l'homologie de Khovanov, cela livre le polynôme de Jones de $L$ (à renormalisation près). Cette situation est comparable au cas des CW-complexes, où la caractéristique d'Euler n'est que l'ombre d'une structure beaucoup plus riche que sont les groupes d'homologie cellulaire du complexe.

En plus de cette structure additionnelle, l'intérêt de l'invariant de M. Khovanov réside dans le fait qu'il est strictement plus fort que celui de Jones \cite{khstrongerJ}, qu'il permet de détecter le n\oe ud trivial \cite{unknot} (ce qui est une conjecture pour le polynôme de Jones) et qu'il ouvre la porte à toute une série de nouveaux invariants homologiques (on cite par exemple la généralisation de l'homologie de Khovanov pour des polynômes coloriés). 

De façon générale, L. Crane et I. Frenkel définissent une catégorification comme "une procédure informelle qui transforme les entiers en groupes abéliens, les espaces vectoriels en catégories abéliennes ou triangulées et les opérateurs en foncteurs entre ces catégories", dans \cite{categorification}. On obtient alors une structure additionnelle donnée par les transformations naturelles de foncteurs qu'on ne trouve pas avant la catégorification. L'objectif de cette procédure était de transformer des invariants quantiques de variétés de dimension 3 en invariants de variétés de dimension 4. L. Crane propose aussi dans son article \cite{clockcate} d'utiliser le concept de catégorification pour construire une théorie quantique de la gravitation de dimension 3+1 à partir de théories qui fonctionnent pour la dimension 3. La construction de M. Khovanov est donc bien une catégorification puisque les entiers du polynôme de Jones sont transformés en groupes abéliens gradués (les groupes de cohomologie avec la graduation de $A$), la somme directe de ces groupes formant donc une catégorie bigraduée dont une des deux graduations correspond aux puissances de $q$ du polynôme de Jones.

Par ailleurs, on peut étendre le polynôme de Jones pour en faire un foncteur $J$ de la catégorie des enchevêtrements (qui consistent en le plongement de cercles et d'intervalles dans l'espace, reliant $n$ points fixés sur une droite à $m$ points fixés sur une autre) vers la catégorie des espaces vectoriels. Ce foncteur envoie les $n$ points vers $V^{\otimes n}$, où $V$ est la représentation irréductible de dimension 2 du groupe quantique $U_q(\mathfrak{sl}_2)$, et les $m$ points vers $V^{\otimes m}$. \`A un enchevêtrement orienté $T$ reliant $n$ points à $m$ points, ce foncteur associe un opérateur $J(T) :V^{\otimes n} \rightarrow V^{\otimes m}$ qui entrelace l'action de $U_q(\mathfrak{sl}_2)$ (voir sources \cite{quantumj},\cite{quantumsl}). Ce foncteur est tel que si on prend $n=m=0$, on obtient le polynôme de Jones.
De plus, $J$ se restreint à un foncteur $J'$ sur les $(m,n)-$enchevêtrements, c'est-à-dire les enchevêtrements reliant $2n$ points à $2m$ points, envoyant les $2n$ points sur l'espace $\Inv(n) := \Inv(V^{\otimes 2n})$ des applications multilinéaires sur $V^{\otimes 2n}$ invariantes par rapport à l'action de $U_q(\mathfrak{sl}_2)$ et envoyant un $(m,n)-$enchevêtrement $L$ sur un entrelacement $J'(L) : \Inv(n) \rightarrow \Inv(m)$ en prenant la restriction de $J(L)$ sur ces espaces (voir \cite{quantumj},\cite{spider2}).

M. Khovanov a étendu son invariant aux enchevêtrements dans \cite{Khovanov} en catégorifiant $J'$. Cette catégorification transforme $\Inv(n)$ en la catégorie triangulée $\mathcal K^n$ des complexes de modules gradués (à homotopie de chaine près) sur un certain anneau $H^n$ et $J'(L) $ en le foncteur $\mathcal K^n \rightarrow \mathcal K^m$ donné par la tensorisation par un certain complexe de bimodules sur ($H^m$,$H^n$). On appelle ces anneaux $H^n$ les \emph{anneaux des arcs de Khovanov}. On associe donc à un enchevêtrement un complexe de bimodules sur ($H^m$,$H^n$), dont la classe d'équivalence à homotopie de chaines près donne un invariant d'enchevêtrements. Cette categorification étend l'homologie de Khovanov puisque si on prend $n=m=0$, le complexe obtenu est celui utilisé pour calculer les groupes d'homologies $H^{i,j}(L)$, où $L$ est un enchevêtrement sans point d'extrémités, c'est-à-dire un entrelacs.




En 2013, P. Ozsvath, J. Rasmussen et Z. Szabo ont construit dans \cite{OddKhovanov} une "oddification" de l'homologie de Khovanov, le terme "odd", qu'on traduit par impair, signifiant qu'on retrouve une certaine antisymétrie dans les objets. Cette construction impaire utilise un foncteur projectif différent de l'homologie de Khovanov, le terme projectif signifiant qu'il n'est bien défini qu'à signe près. Ce foncteur envoie les collections de cercles, non pas sur des produits tensoriels, mais sur des produits extérieurs. On peut montrer qu'on obtient alors un nouvel invariant d'entrelacs qui ne dépend pas des signes et qui forme une autre catégorification du polynôme de Jones. Cette homologie impaire permet de distinguer des n\oe uds que la version paire ne distingue pas, comme l'a montré A. Shumakovitch dans \cite{shumakovich}.  J. Bloom a montré dans \cite{bloom} que l'homologie impaire est invariante sous opération de mutation alors qu'il existe des exemples d'entrelacs mutants qui ont des homologies paires différentes (voir \cite{wehrli}). De plus, l'homologie de Khovanov est équivalente à sa version impaire si on les considère toutes deux en modulo 2. Récemment, K. Putyra a construit dans \cite{Covering} un cadre qui permet de retrouver l'homologie de Khovanov et sa version impaire, avec un paramètre permettant d'obtenir l'une ou l'autre. 

On se demande alors ce que donnerait une construction similaire à celle des anneaux des arcs de Khovanov, mais en utilisant le foncteur de  P. Ozsvath, J. Rasmussen et Z. Szabo.

Cela livre une "oddification" des anneaux $H^n$ et constitue le premier objectif de ce travail. Puisqu'il y a un choix de signes à faire, on construit des familles d'anneaux $OH^n_C$, chacun étant caractérisé par des choix de signes notés $C$. De plus, pour que cela soit bien défini et que le foncteur projectif de \cite{OddKhovanov} devienne un foncteur au sens usuel, on a besoin d'une catégorie plus structurée que celle des cobordismes et on utilise donc la catégorie des cobordismes avec chronologies, définie par K. Putyra dans \cite{Putyra} (il utilise aussi cette catégorie pour construire son cadre dans \cite{Covering}). 

Par ailleurs, dans \cite{HnCenter}, M. Khovanov a relié ses anneaux $H^n$ à la cohomologie de la variété de Springer pour une partition $(n,n)$, montrant que le centre de l'anneau $H^n$ est isomorphe en tant qu'anneau gradué à l'anneau de cohomologie de cette variété. En 2014, A. Lauda et H. Russell ont proposé de leur côté, dans \cite{OddSpringer}, une construction impaire de la cohomologie de la variété de Springer pour une partition quelconque, construction basée sur les polynômes impairs et donnant par conséquent une antisymétrie aux éléments.

On s'interroge donc légitiment quant à l'existence d'un lien entre la construction impaire de $H^n$ introduite dans ce travail et la construction de A. Lauda et H. Russell.

On répond par l'affirmative et, en étendant le centre de $OH^n_C$ aux éléments anticommutatifs, ce qu'on appelle le centre impair, on obtient un anneau qui est isomorphe à la construction impaire de la cohomologie de Springer pour une partition $(n,n)$.

\paragraph{Plan général} Ce travail est séparé en 4 chapitres et est muni d'un ensemble d'annexes. Le premier chapitre vise à définir les anneaux des arcs de Khovanov $H^n$, en définissant d'abord la catégorie de Temperley-Lieb et la catégorie des cobordismes de dimension 2 qui sont des bases nécessaires à la construction de ces anneaux. On définit ensuite le foncteur $F$ de la catégorie des cobordismes vers la catégorie des espaces vectoriels, utilisé par M. Khovanov pour construire son homologie, et enfin on définit les anneaux $H^n$.

Une construction impaire des anneaux $H^n$ étant l'objectif principal du Chapitre $2$, on définit d'abord le foncteur de P. Ozsvath, J. Rasmussen et Z. Szabo en utilisant le travail de K. Putyra sur les cobordismes avec chronologies. On établit ensuite la construction impaire des anneaux des arcs, notée $OH^n_C$, avec des choix de signes $C$. On propose un système de calculs basé sur des diagrammes coloriés afin de faciliter les calculs dans $OH^n_C$. Finalement, dans la dernière section du chapitre, on montre quelques résultats sur les anneaux $OH^n_C$, notamment qu'ils sont non-associatifs pour $n\ge 2$ et qu'ils sont équivalents aux anneaux $H^n$ quand on les tensorise tous deux par $\Z/2\Z$.

Dans le Chapitre $3$, on rappelle dans une première section quelques propriétés du centre de $H^n$ et des variétés de Springer pour une partition $(n,n)$. On étudie ensuite dans la deuxième section les propriétés du centre de $OH^n_C$ et on introduit la notion de centre impair. On définit dans la section suivante la construction impaire de la cohomologie de la variété de Springer due à A. Lauda et H. Russell et, dans la section finale, on démontre le résultat principal de ce travail qui consiste à construire un isomorphisme entre $OZ(OH^n_C)$ et cette construction impaire pour une partition $(n,n)$.

Dans le dernier chapitre, on propose une série de questions avec pistes de réflexions afin de construire de nouveaux objets à partir des anneaux $OH^n_C$ et d'approfondir la compréhension de ceux-ci, notamment en étudiant les classes d'isomorphismes de $OH^n_C$ pour des choix de signes $C$ différents ou encore en transformant $OH^n_C$ en anneau associatif. Ensuite, on propose une idée de construction imaginée par K. Putyra afin de définir une notion de modules et bimodules qui aurait du sens sur $OH^n_C$ et qui pourrait mener à une catégorification impaire de $J'$ et donc potentiellement à un nouvel invariant d'enchevêtrements.

Enfin, les annexes sont destinées à rappeler des définitions et des résultats qui facilitent la compréhension de ce travail.

%



\chapter{Anneaux des arcs de Khovanov}

L'objectif de ce chapitre est de définir les anneaux des arcs de Khovanov $H^n$, comme introduits par M. Khovanov dans \cite{Khovanov}. \`A cette fin, on rappelle d'abord les définitions et quelques propriétés de la catégorie de Temperley-Lieb, notée $\mathcal{TL}$, et de la catégorie $Cob$ des cobordismes de dimension $2$. On décrit ensuite le foncteur $F : Cob \rightarrow \Z-Mod$ utilisé par M. Khovanov pour construire son homologie et finalement on construit les anneaux $H^n$ grâce à ce foncteur.

\section{Catégorie de Temperley-Lieb}\label{sec:TL}

L'objectif de cette section est de définir la catégorie $\mathcal{TL}$ de Temperley-Lieb dont les objets sont des collections de paires de points et les morphismes $2n \rightarrow 2m$ sont des enchevêtrements plats à isotopie près. 



\begin{definition}
Un $(m,n)$-enchevêtrement plat entre deux collections de $2n$ et $2m$ points consiste en le plongement dans le plan d'une collection de $n+m$ intervalles qui ne se croisent pas, $p_i : [0,1] \rightarrow \R \times [0,1]$, appelés brins, et d'un nombre fini de composantes de cercles libres. On demande en plus que les bords des brins soient envoyés sur tous les points de $ \{1, \dots, 2n\} \times  \{0\}$ et de $\{1, \dots, 2m\} \times  \{1\}$, qu'on appelle points de base. On demande aussi que les brins soient perpendiculaires autour de ces points.\\
On note $\widehat B^m_n$ comme étant l'ensemble des classes de $(m,n)-$enchevêtrements plats, à isotopie préservant les points de base près, et $B^m_n$ comme étant ceux ne possédant pas de composante cercle. On note aussi $B^m := B_0^m$ et $B_n := B_n^0$.
\end{definition}

Le terme plat vient du fait qu'on ne permet pas aux brins de se croiser contrairement à l'appellation d'enchevêtrement au sens classique. Par ailleurs, il existe aussi une définition pour un enchevêtrement plat avec des collections impaires de $2n+1$ et $2m+1$ points, mais on n'en a pas besoin dans ce travail. 

On dessine toujours un $(m,n)-$enchevêtrement comme allant de bas en haut, c'est-à-dire qu'on place les $2m$ points en haut et les $2n$ en bas. On donne un exemple d'élément de $\widehat B^3_2$, donc de $(3,2)$-enchevêtrement plat, en Figure \ref{fig:32tangle}.

\begin{figure}[h]
    \center
    \includegraphics[width=6cm]{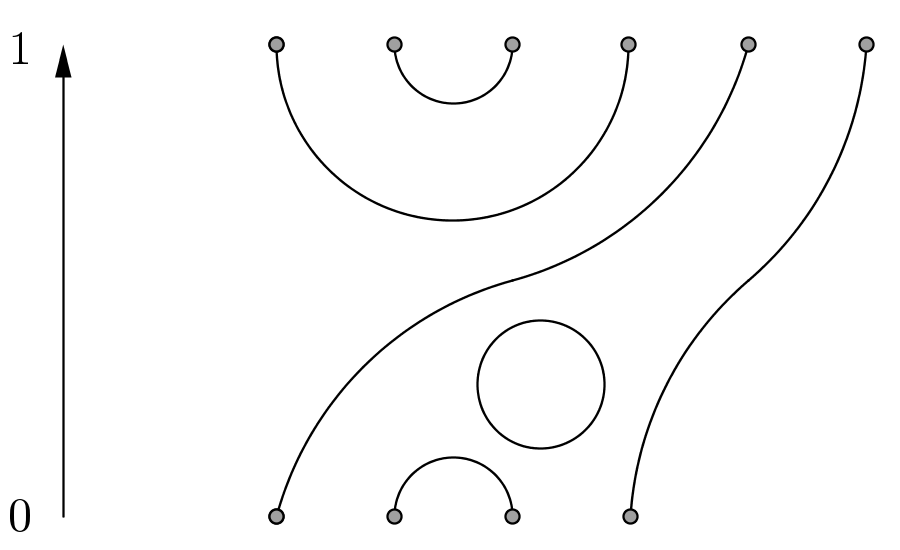}
    \caption{\label{fig:32tangle}$(3,2)$-enchevêtrement plat, la flèche à gauche représente le sens dans lequel on parcourt l'enchevêtrement et les chiffres représentent les abscisses des points.}
\end{figure}

\begin{exemple}\label{ex:B2}
Etant donné qu'on utilise principalement $B^n$ dans ce travail, on observe que $B^2$ est composé des deux classes d'enchevêtrements plats données par
\begin{align*}
 \deuxdiag{Images_arxiv/B2_1.png}, &&&  \deuxdiag{Images_arxiv/B2_2.png}.
\end{align*}
\end{exemple}
\begin{exemple}\label{ex:B3}
De même $B^3$ est composé des cinq classes d'enchevêtrements plats suivant
\begin{align*}
 \diagg{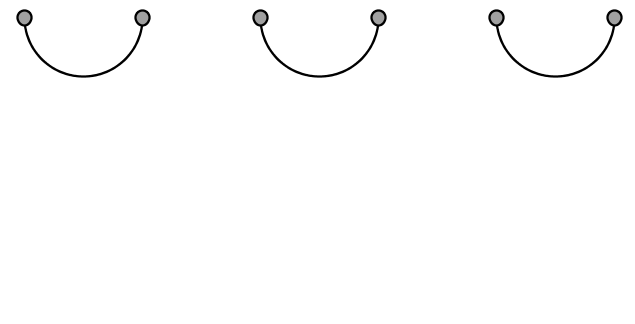}, &&&  \diagg{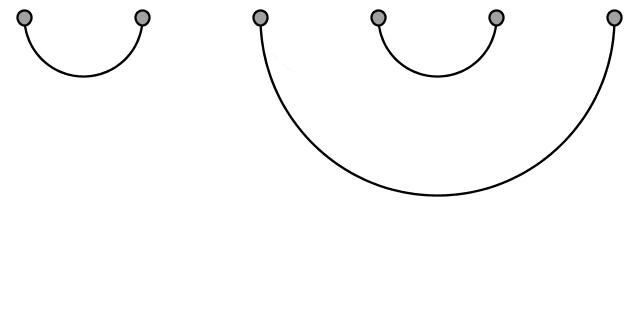}, \\  \diagg{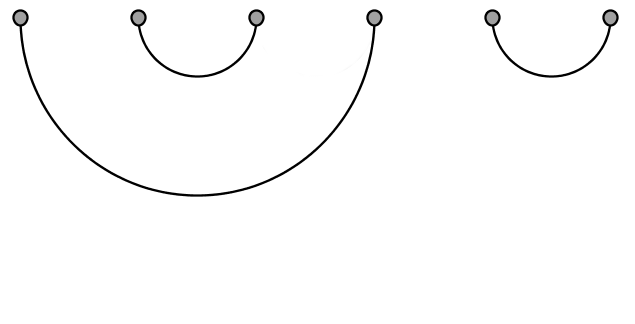}, &&&  \diagg{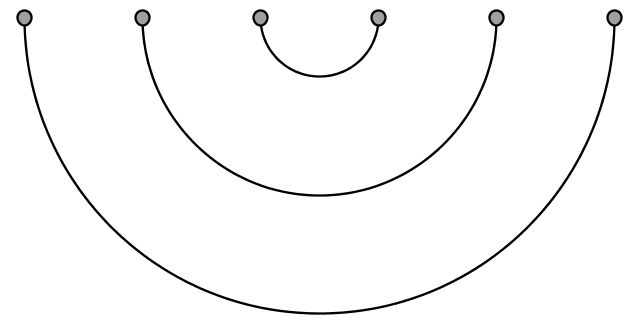}, \\
\diagg{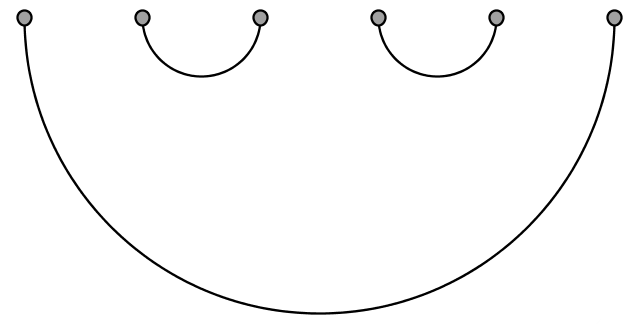}.
\end{align*}
\end{exemple}

\begin{definition}
On définit la composition d'un $(m,n)$ et d'un $(n,p)$-enchevêtrements plats en plaçant le premier au dessus du second et en faisant une homothétie divisant la hauteur totale par deux, c'est-à-dire que pour $a \in \widehat B^{m}_n$ et $b \in \widehat B^{n}_p$ on a 
$$\raisebox{-0.25\height}{\young(a)} \circ \raisebox{-0.25\height}{\young(b)} = \raisebox{-0.4\height}{\young(a,b)}  \in \widehat B^{m}_p.$$
\end{definition}

Il est clair que cette composition donne un $(m,p)$-enchevêtrement plat puisque, par la condition de perpendicularité, ils se recollent de façon lisse, comme illustré en Figure \ref{fig:tanglescomp}. Il faut noter qu'en général, une composition d'un élément de $B_n^m$ avec un de $B_p^n$ n'est pas dans $B_m^p$ mais dans $\widehat B_m^p$ puisqu'on peut faire apparaitre des composantes de cercles libres.

\begin{figure}[h]
    \center
    \includegraphics[width=12cm]{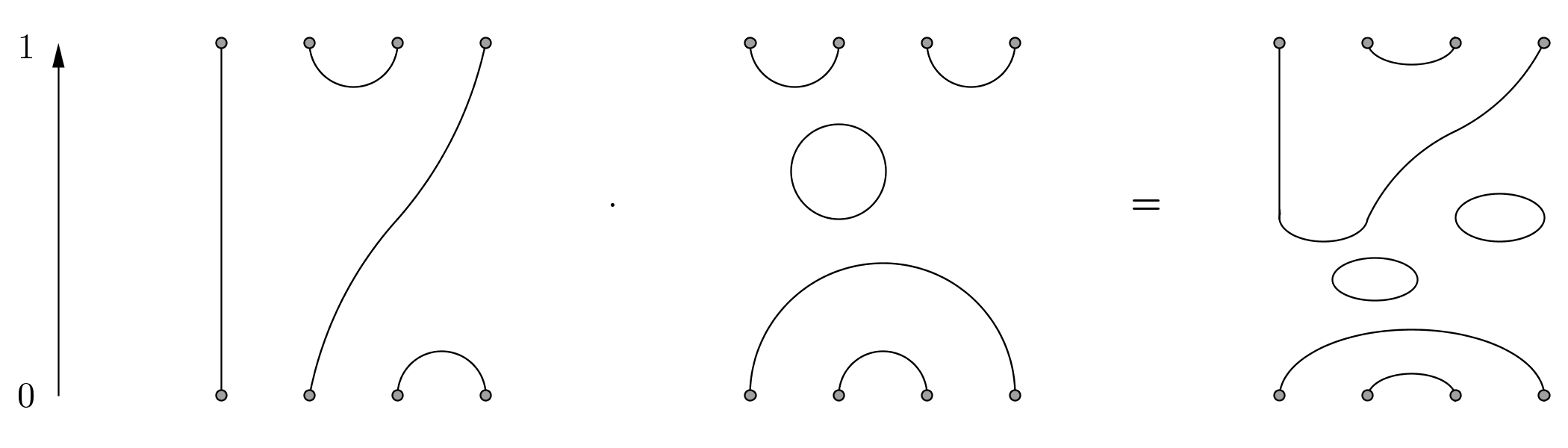}
    \caption{\label{fig:tanglescomp}Composition de deux $(2,2)-$enchevêtrements plats.}
\end{figure}

On pose $Vert_{2n}$ comme étant le $(n,n)$-enchevêtrement plat donné les $2n$ segments de droites $\{i\}\times[0,1]$ pour $1 \le i \le 2n$. Par exemple on a
$$Vert_{6} := \middiag{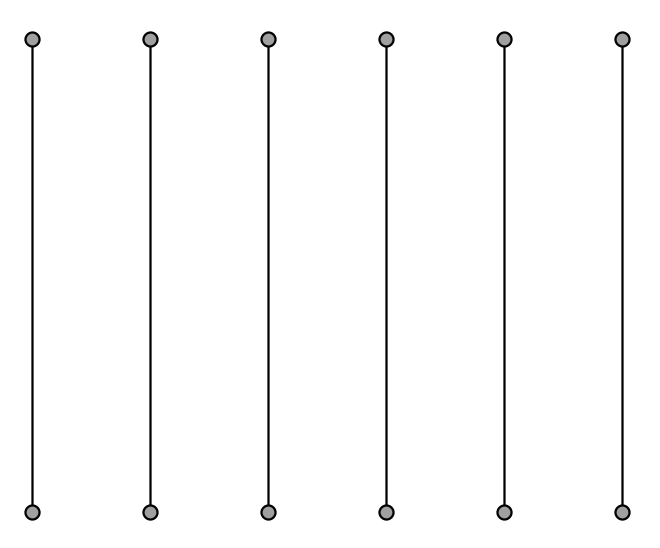}.$$
Il est facile de voir que pour $x \in B_n^m$ et $y \in B_m^n$ on a $xVert_{2n} = x$ et $Vert_{2n}y = y$. 

\begin{definition}
On définit  la catégorie de Temperley-Lieb, notée $\mathcal{TL}$, comme étant la catégorie ayant pour objets les collections de $2n$ points $\{1, \dots, 2n\}$,  pour tout $n \ge 0$, et pour flèches $2n \rightarrow 2m$ les  $(m,n)$-enchevêtrements plats, à isotopie préservant les points de base près, et munis de la composition définie au dessus. L'identité $2n \rightarrow 2n$ est donnée par $Vert_{2n}$.


\end{definition}

Par ailleurs, on remarque que $B^n$ est caractérisé par la façon de relier $2n$ points ensemble par $n$ demi-cercles vers le bas qui ne peuvent pas se croiser, livrant la propriété suivante : 

\begin{proposition}
La cardinalité de $B^n$ est donnée par le $n$-ème nombre de Catalan $C_n$
$$|B^n| = C_n := \frac{1}{n+1} \begin{pmatrix}2n \\ n\end{pmatrix}.$$
\end{proposition}

\begin{proof}
Comme expliqué dans \cite[Chapitre 6]{catalan}, $C_n$ calcule le nombre de chaines de caractères de taille $2n$ qu'on peut engendrer avec les caractères $'('$ et $')'$ telles qu'aucun segment initial ne contient plus de $')'$ que de $'('$, c'est-à-dire les chaines bien parenthésées. On peut voir un élément de $B^n$ comme une telle chaine en notant un $'('$ quand un arc part d'un point et un $')'$ quand il y arrive (dans l'exemple précédent, pour $B^2$, on obtient $``()()"$ et $``(())"$) puisqu'on ne peut pas fermer plus d'arcs que ce qu'on en a ouvert. De même, toute telle chaine de caractères peut être attribuée à un élément de $B^n$ en parcourant cette chaine et en reliant par un demi-cercle inférieur tout $')'$ rencontré au précédent $'('$ qui n'a pas encore été relié à un $')'$. 
\end{proof}

\begin{definition}
On définit $W$ l'application d'involution sur un enchevêtrement plat qui retourne l'enchevêtrement en envoyant $(x,y)\in \R\times[0,1] \mapsto (x,1-y) \in  \R\times[0,1]$. On a alors
\begin{align*}
W(\widehat B^m_n) &= \widehat B^n_m, & &\text{et} & W(B^m_n) &= B^n_m.
\end{align*}
\end{definition}

On observe que $W(B^n)B^n \subset \widehat{B_0^0}$ et donne donc une collection de cercles. Tous ces cercles intersectent l'axe $(-, 1/2)$, chacun en au moins deux points de $\{(1,1/2), \dots, (2n,1/2)\}$, et donc on peut induire un ordre total sur ces cercles en considérant les abscisses minimales  de ces points (un cercle est plus petit qu'un autre si son point d'intersection d'abscisse minimale est plus petit que celui de l'autre), comme le montre par exemple la Figure \ref{fig:ordrebn}.

\begin{figure}[h]
    \center
    \includegraphics[width=10cm]{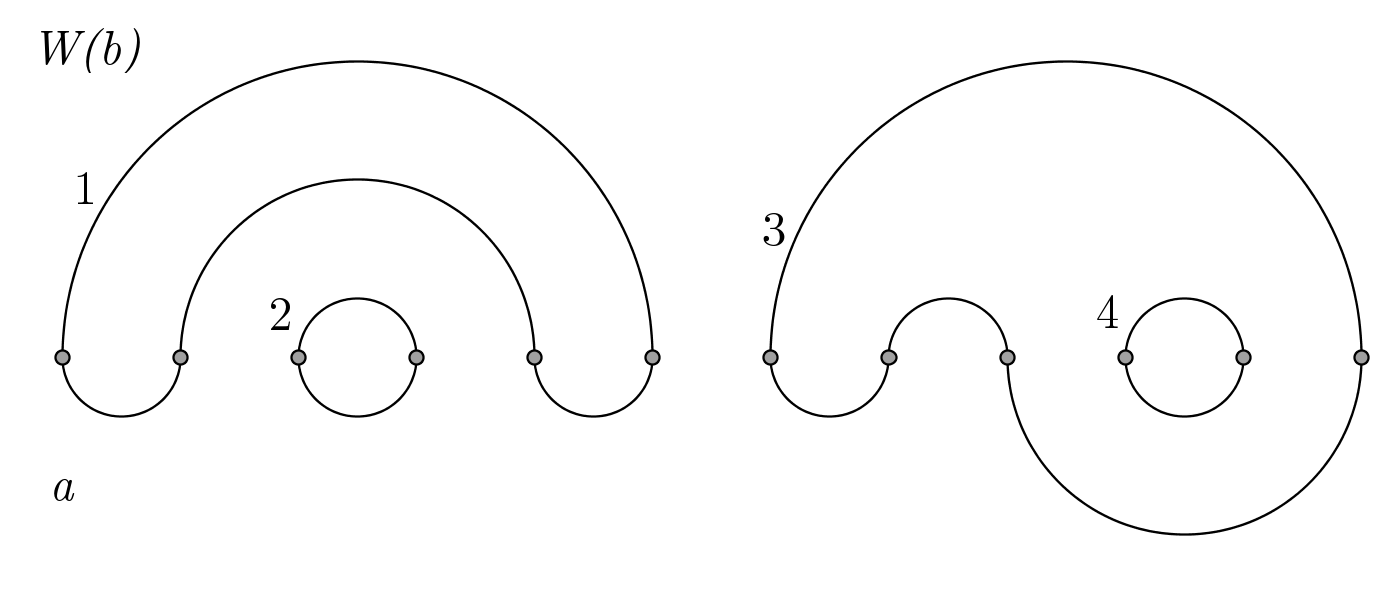}
    \caption{\label{fig:ordrebn}Ordre des composantes de $W(b)a$ induit par l'ordre des points qui intersectent l'axe $(-,1/2)$.}
\end{figure}

\section{Catégorie des cobordismes}

Cette section vise à définir la catégorie des cobordismes de dimension $2$. Toutes les variétés considérées dans cette section sont lisses.

\begin{definition}
Un \emph{cobordisme} $C$ entre deux variétés $V_1$ et $V_2$ de même dimension est une variété à bord notée $C : V_1 \rightarrow V_2$ telle que son bord soit difféomorphe à l'union disjointe de $V_1$ et $V_2$
\begin{align*}
\partial C &\simeq V_1 \sqcup V_2.
\end{align*} 
\end{definition}

On peut définir une notion de composition pour ces cobordismes.

\begin{definition}On définit le recollement de deux cobordismes $C,C'$ ayant des bords $V_1, V_2$ et $V_1', V_2'$ tels que $V_2 \simeq V_1'$ par le cobordisme obtenu en prenant
$$\frac{C \sqcup C'}{V_2 \sim V_1'}.$$
\end{definition}

On a une notion d'équivalence de cobordismes donnée par des difféomorphismes qui préservent les bords.

\begin{definition} On dit que deux cobordismes $C_1$ et $C_2$ ayant comme bords les même variétés $V_1$ et $V_2$ sont équivalents s'il existe un difféomorphisme $\phi : C_1 \rightarrow C_2$ tel que le diagramme 
$$ \xymatrix{
&C_1 \ar[dd]^\phi& \\
V_1 \ar[ru] \ar[rd] & & V_2 \ar[lu] \ar[ld] \\
&C_2 &
}$$
commute et où les morphismes de gauche et de droite sont les restrictions des difféomorphismes $\partial C_1 \simeq  V_1 \sqcup V_2$ et  $\partial C_2 \simeq  V_1 \sqcup V_2$. 
\end{definition}

On est maintenant en mesure de définir la catégorie des cobordismes.

\begin{definition} On note $Cob$ la catégorie monoïdale dont les objets sont des variétés de dimension un, compactes et sans-bord (c'est-à-dire des unions de cercles disjointes et finies) et les flèches sont des cobordismes orientables compacts entre ces variétés, à équivalence près, munis de la composition donnée par recollement de cobordismes et de la multiplication monoïdale donnée par l'union disjointe. \end{definition}

Par ailleurs, la catégorie $Cob$ possède une structure algébrique intéressante puisque par la théorie de Morse et le théorème de classification des surfaces on a le Théorème \ref{thm:cobelem}, dont on peut trouver une preuve dans \cite[Section 1.3]{Cobordismes} par exemple. Ce théorème permet de représenter les cobordismes sous forme de diagrammes qu'on lit de bas en haut, qu'on compose en les plaçant les uns au-dessus des autres et qu'on multiplie en les juxtaposant. 

\begin{theoreme}\label{thm:cobelem}
La catégorie des cobordismes $Cob$ est générée par les multiplications et compositions de l'identité, donnée par un cylindre, et des 5 cobordismes élémentaires : la naissance de cercle, la fusion, la scission, la mort de cercle et la permutation
\begin{align*}
\smalldiag{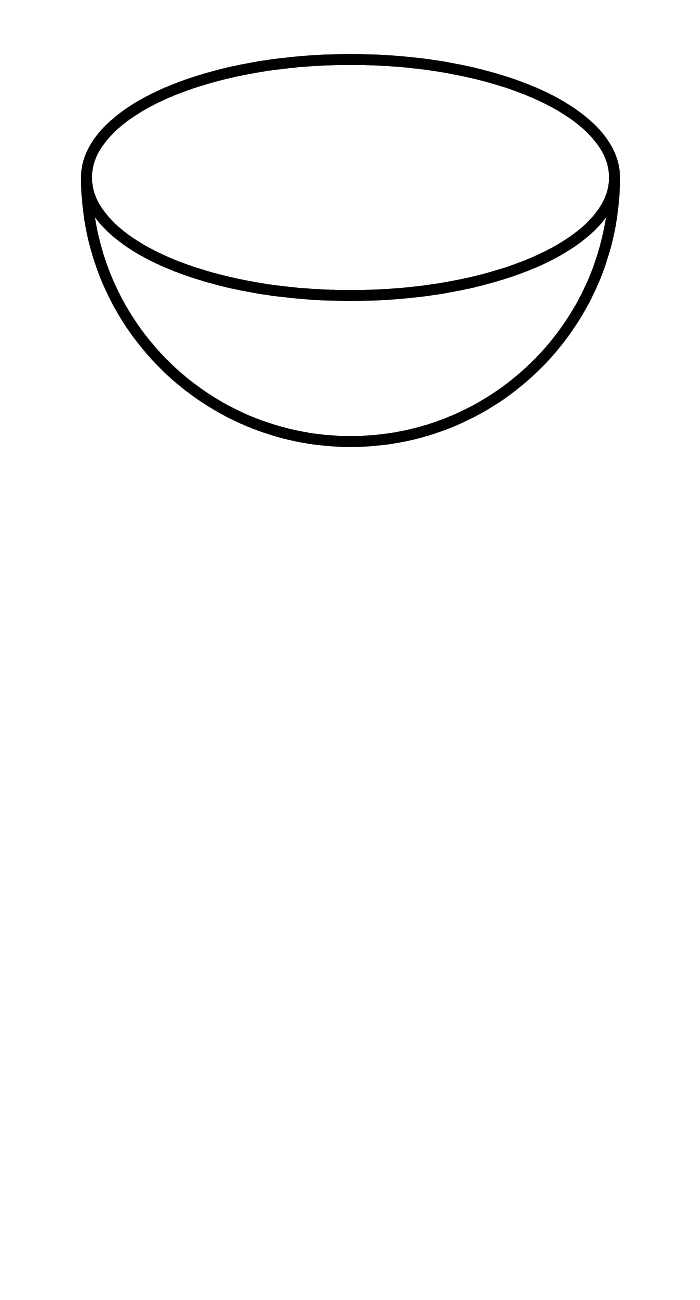} && \smalldiag{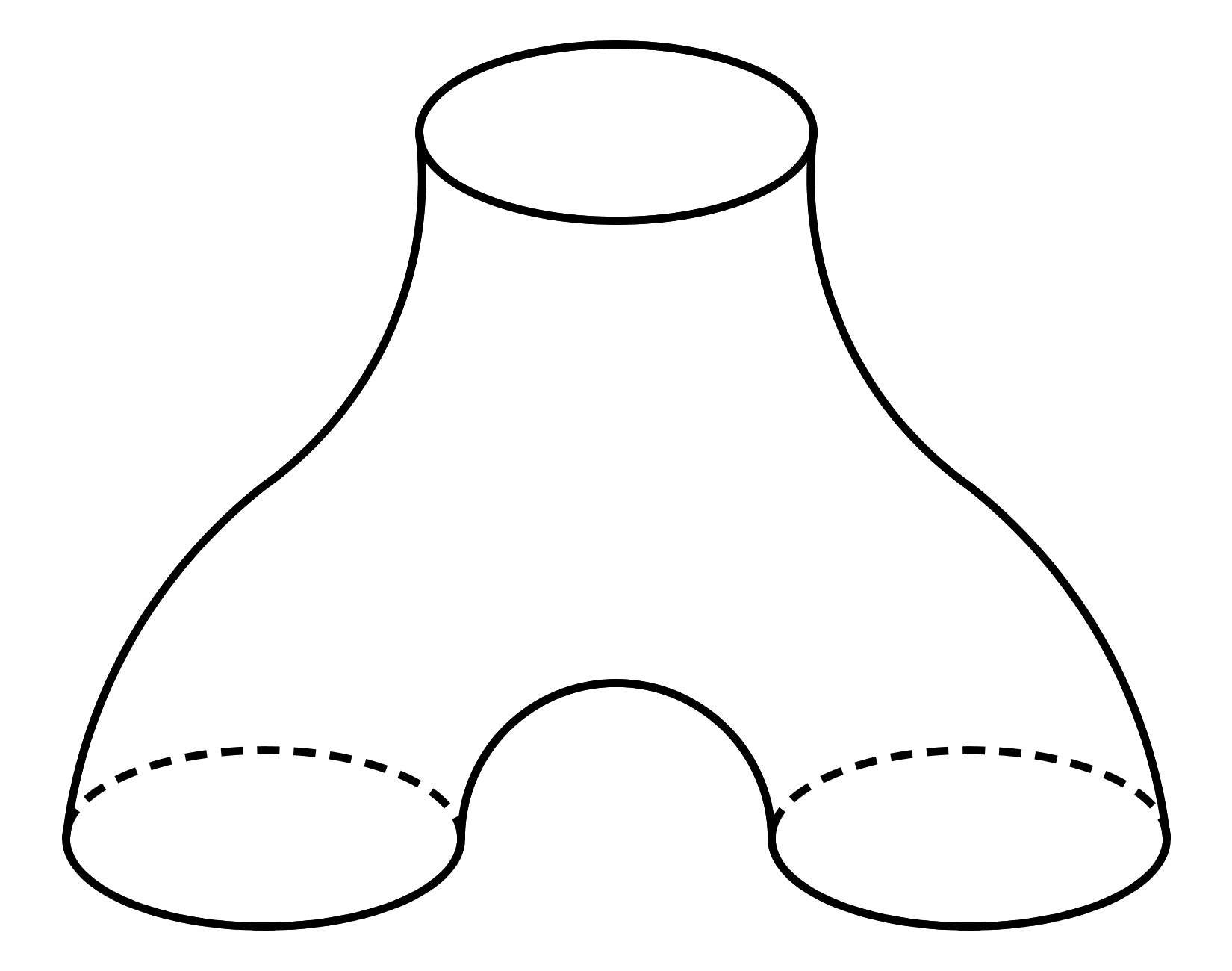} & & \smalldiag{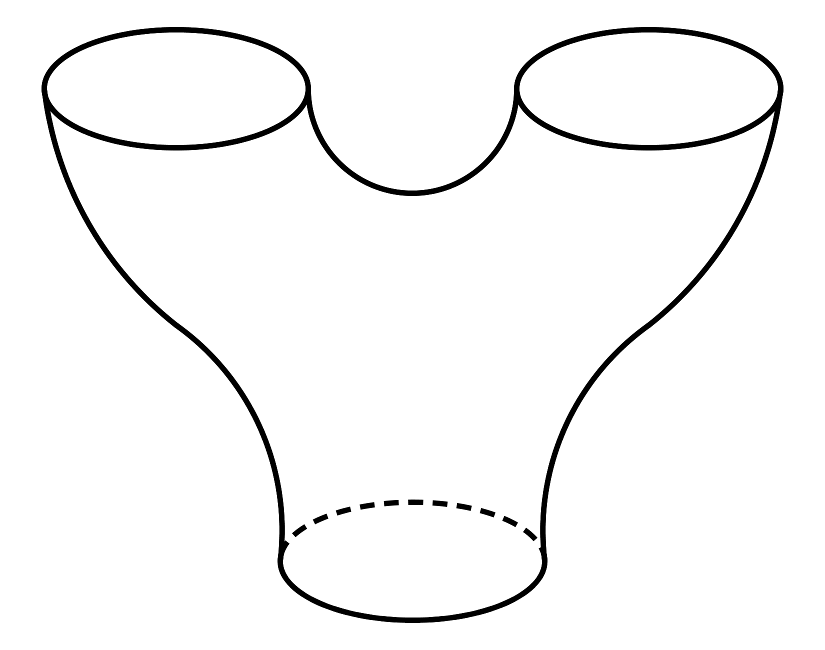} &&
\smalldiag{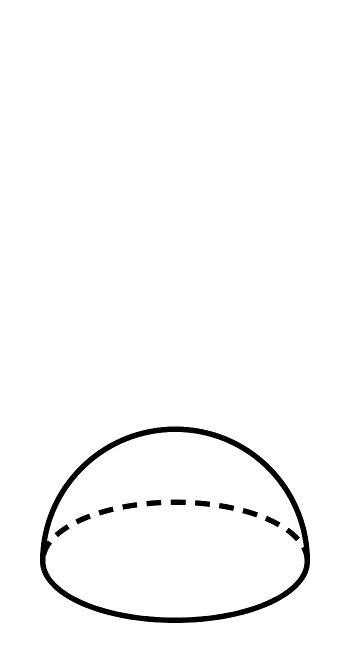} && \smalldiag{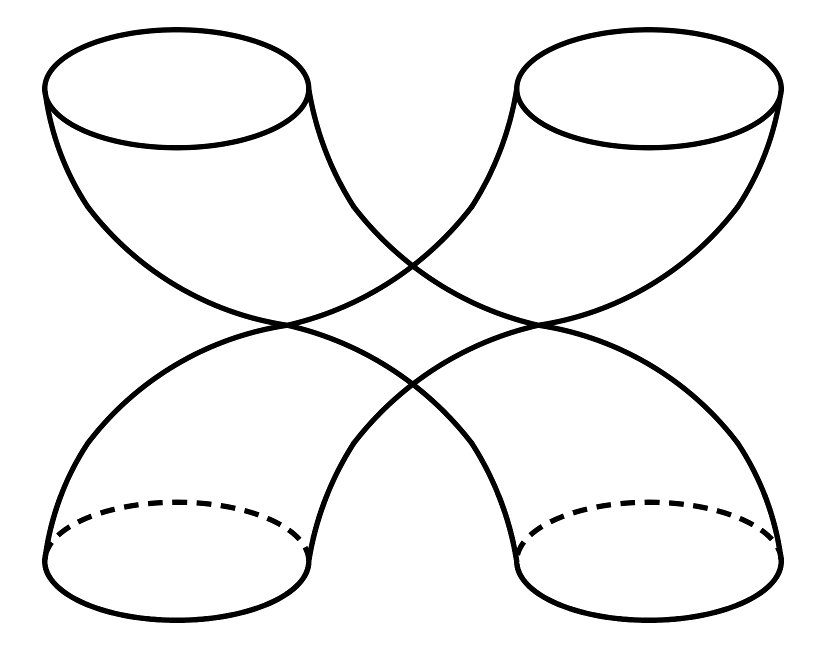}  
\end{align*}
sous les relations suivantes :
\begin{enumerate}
\item Commutativité et co-commutativité
\begin{align*}\middiag{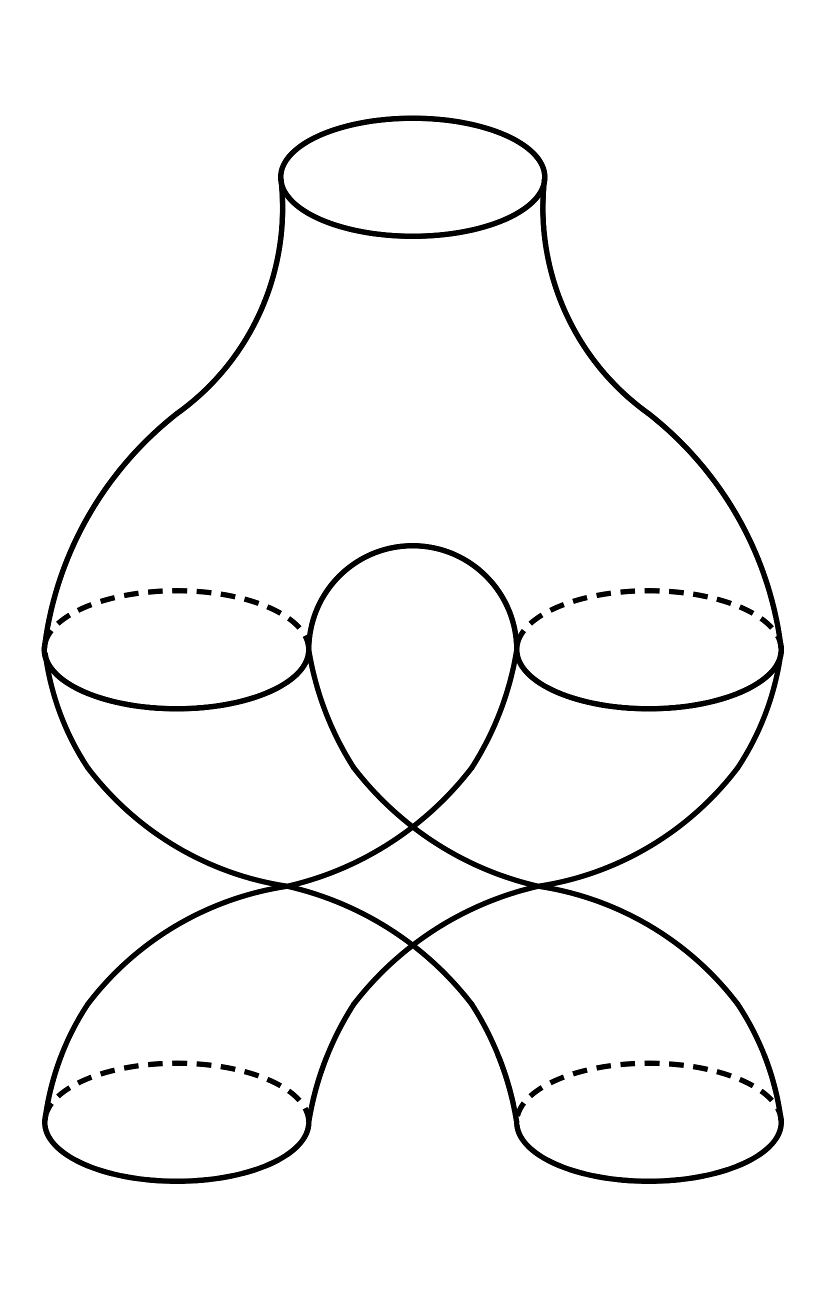} &= \smalldiag{Images_arxiv/Cob_merge.png}
&&,& \smalldiag{Images_arxiv/Cob_split.png} &= \middiag{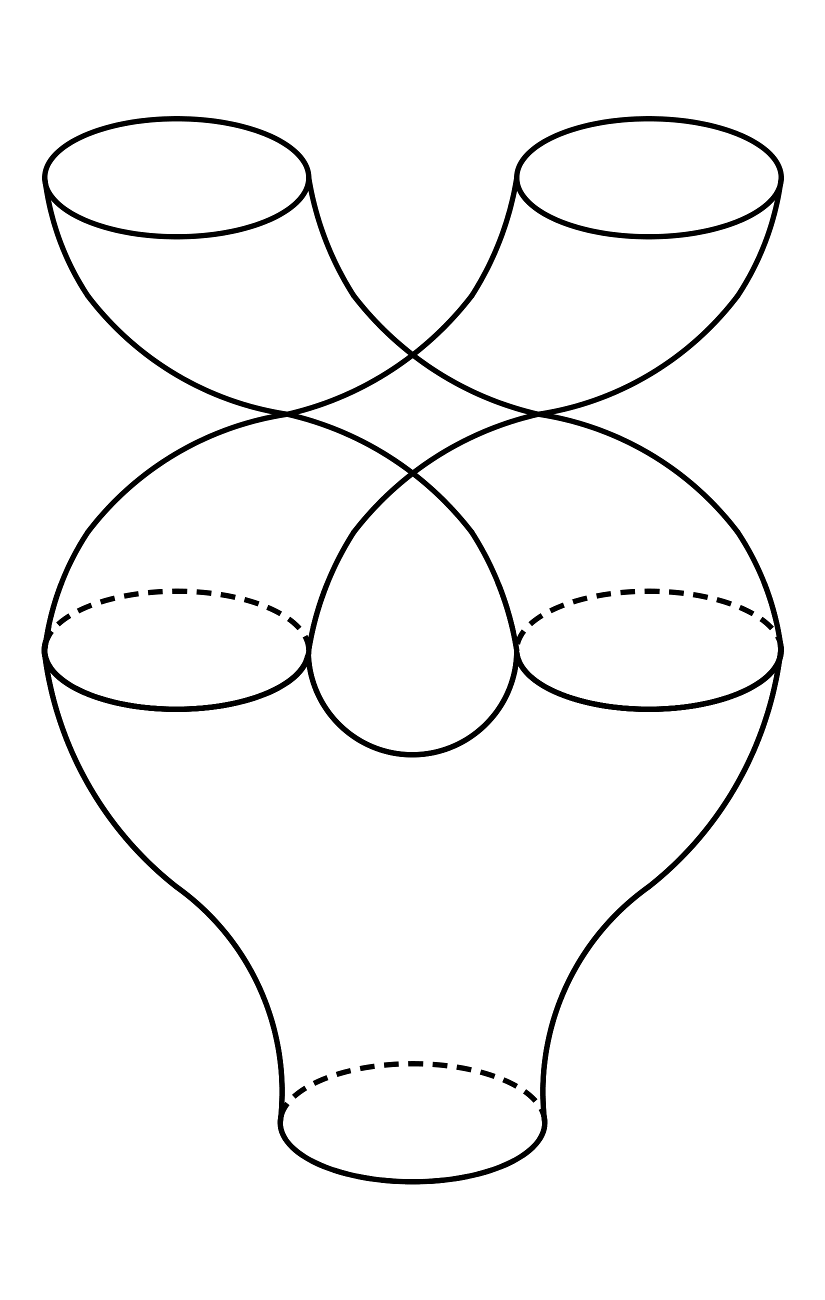}\end{align*}
\item Associativité et coassociativité
\begin{align*}\middiag{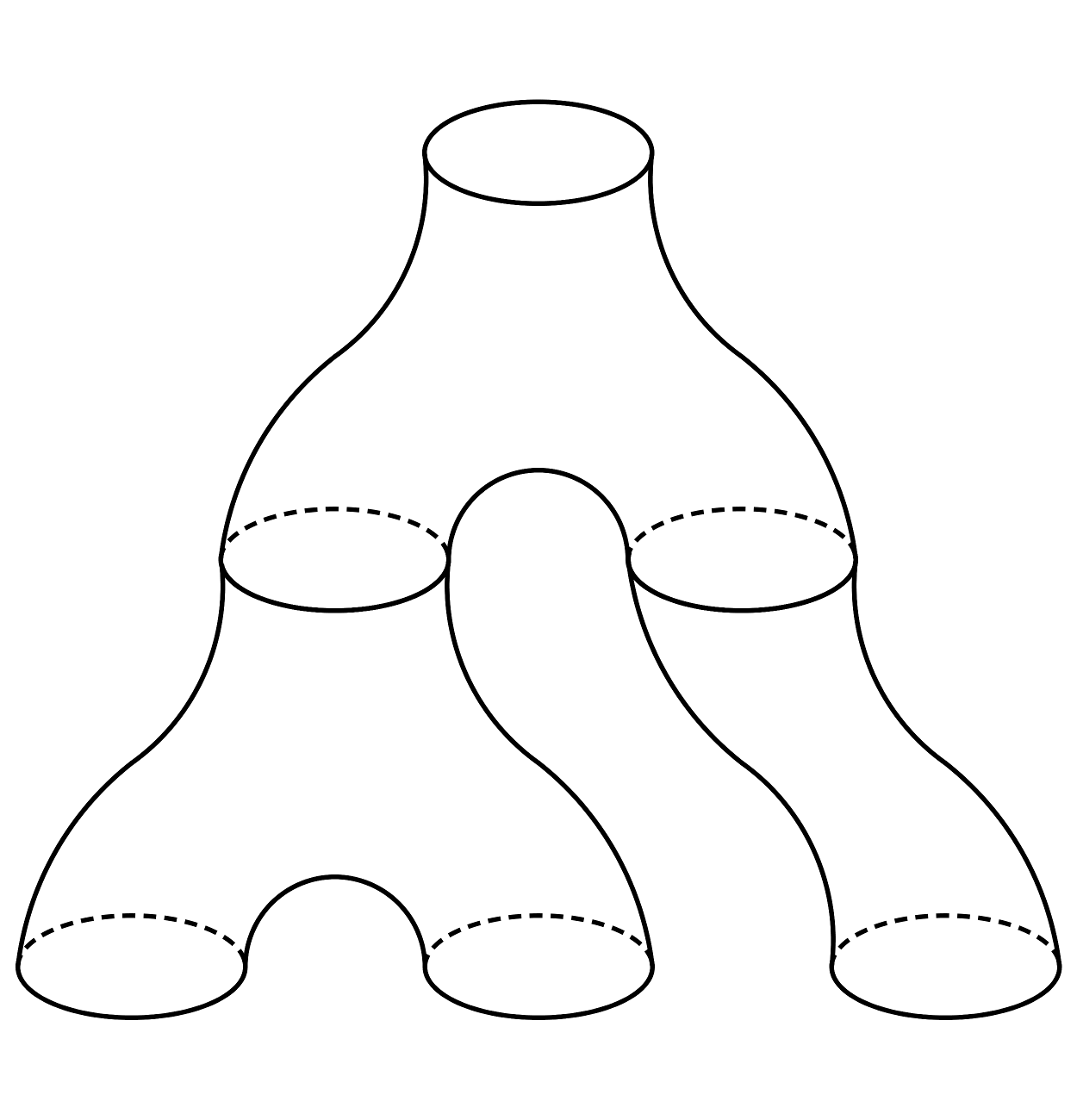} &= \middiag{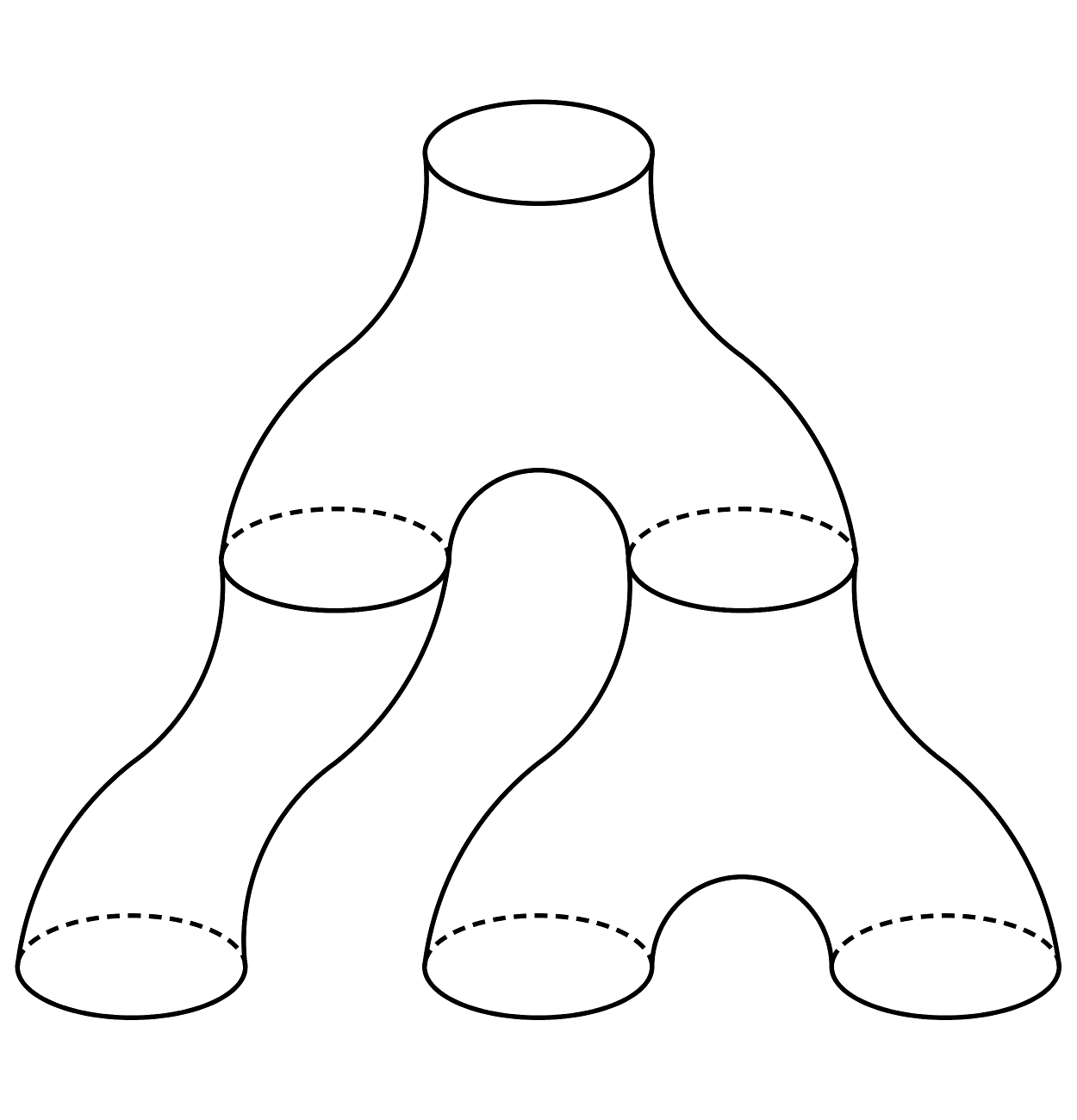}
&&,& \middiag{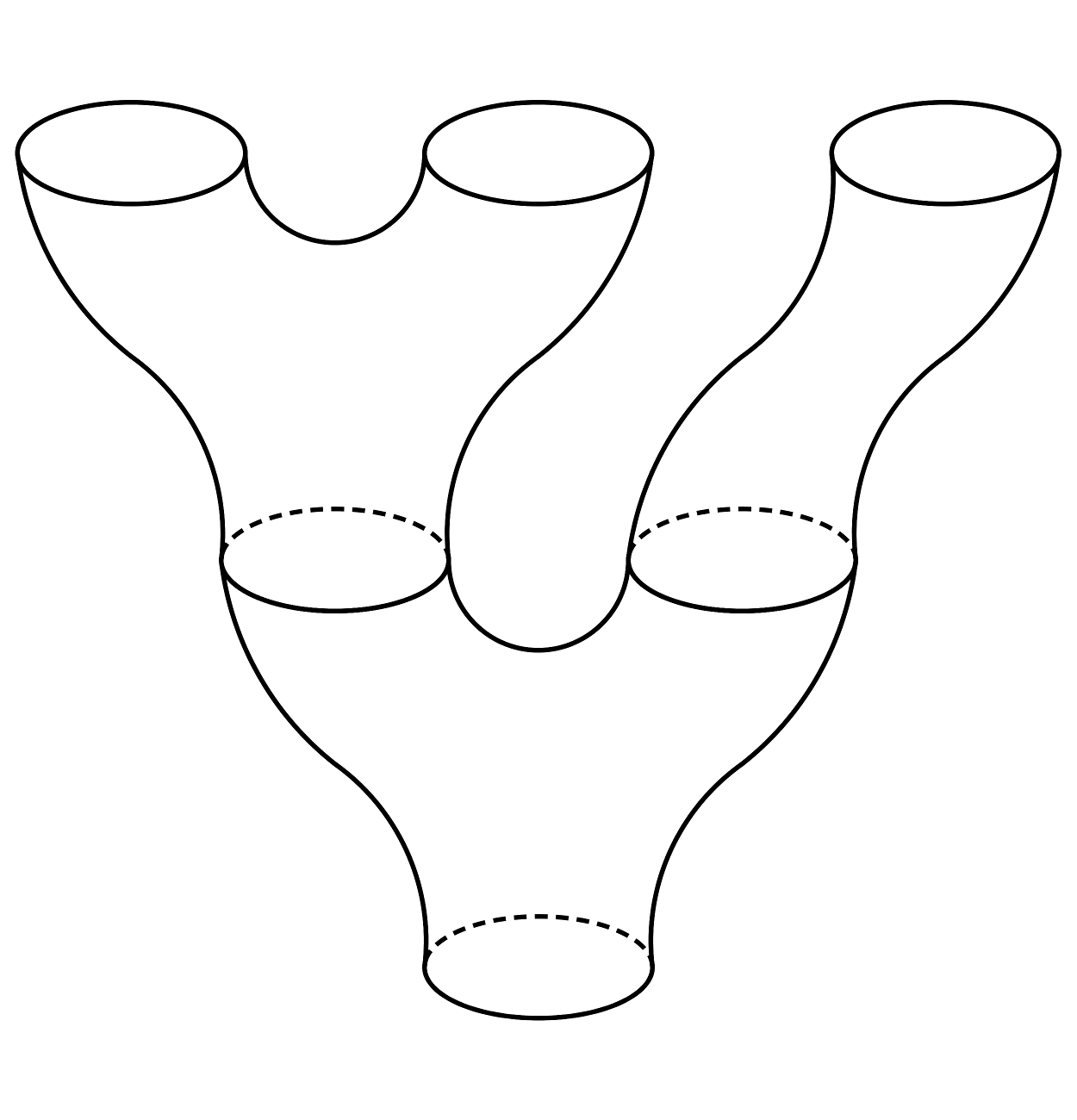} &= \middiag{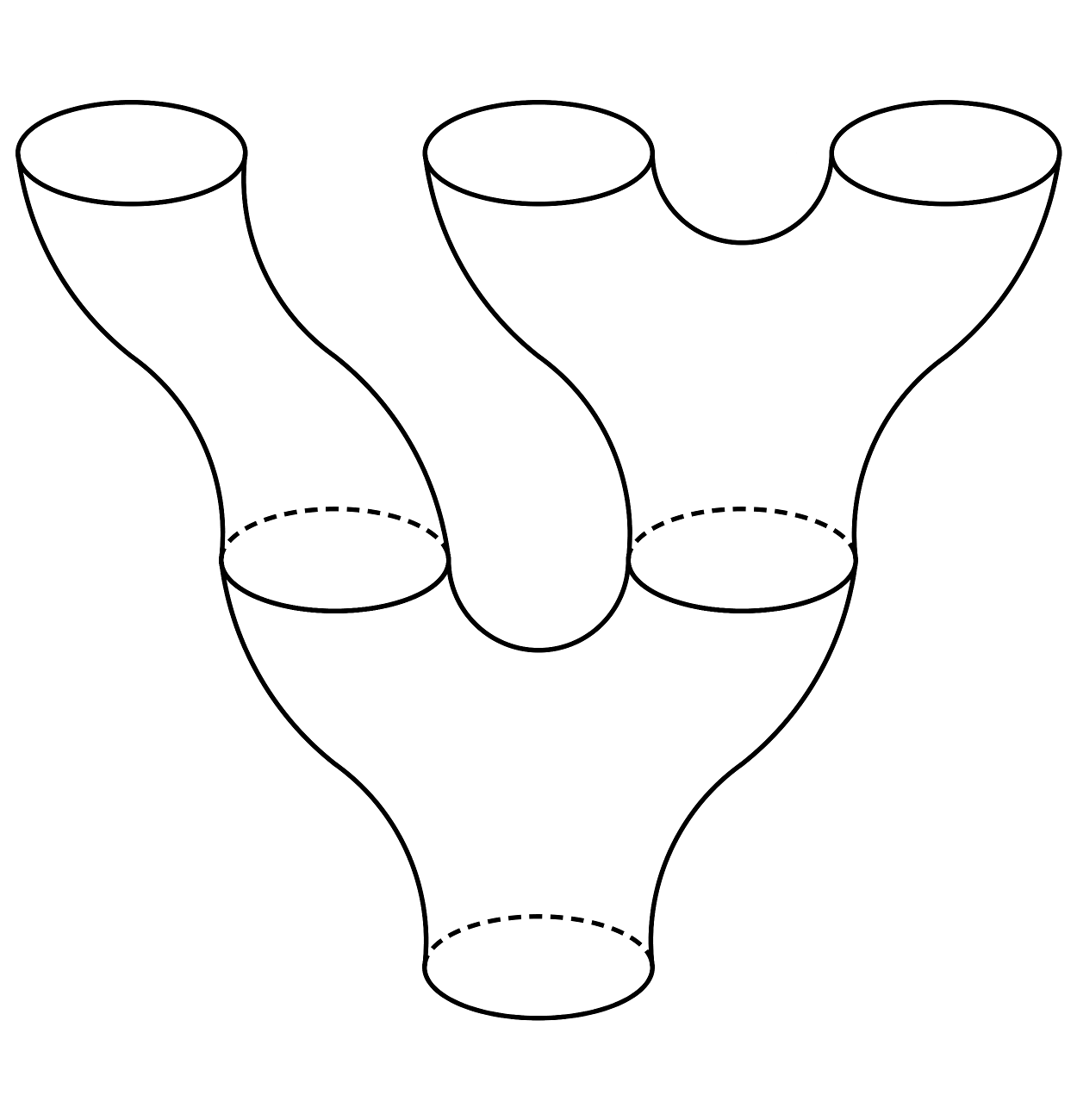}\end{align*}
\item Relations de Frobenius
$$ \middiag{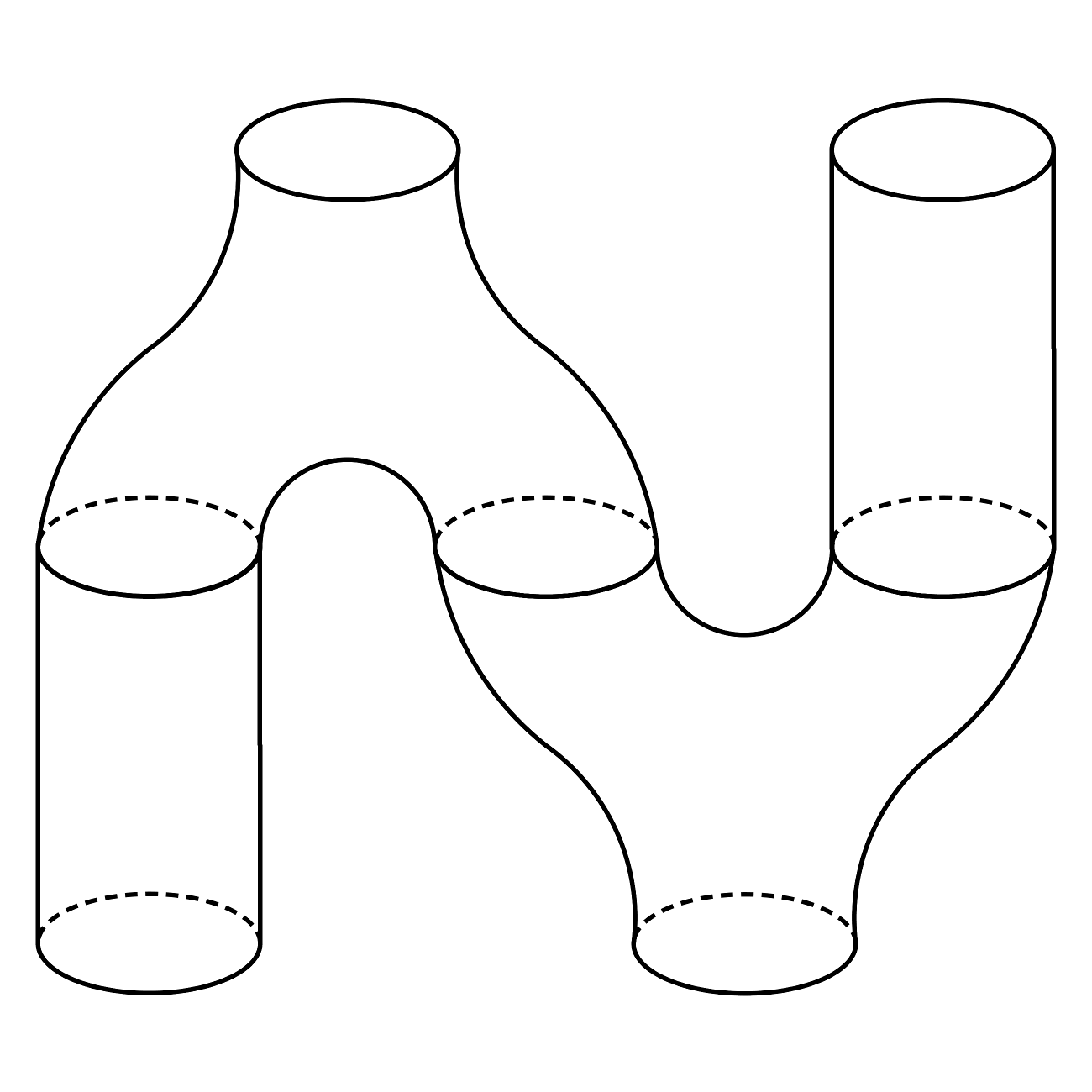} = \middiag{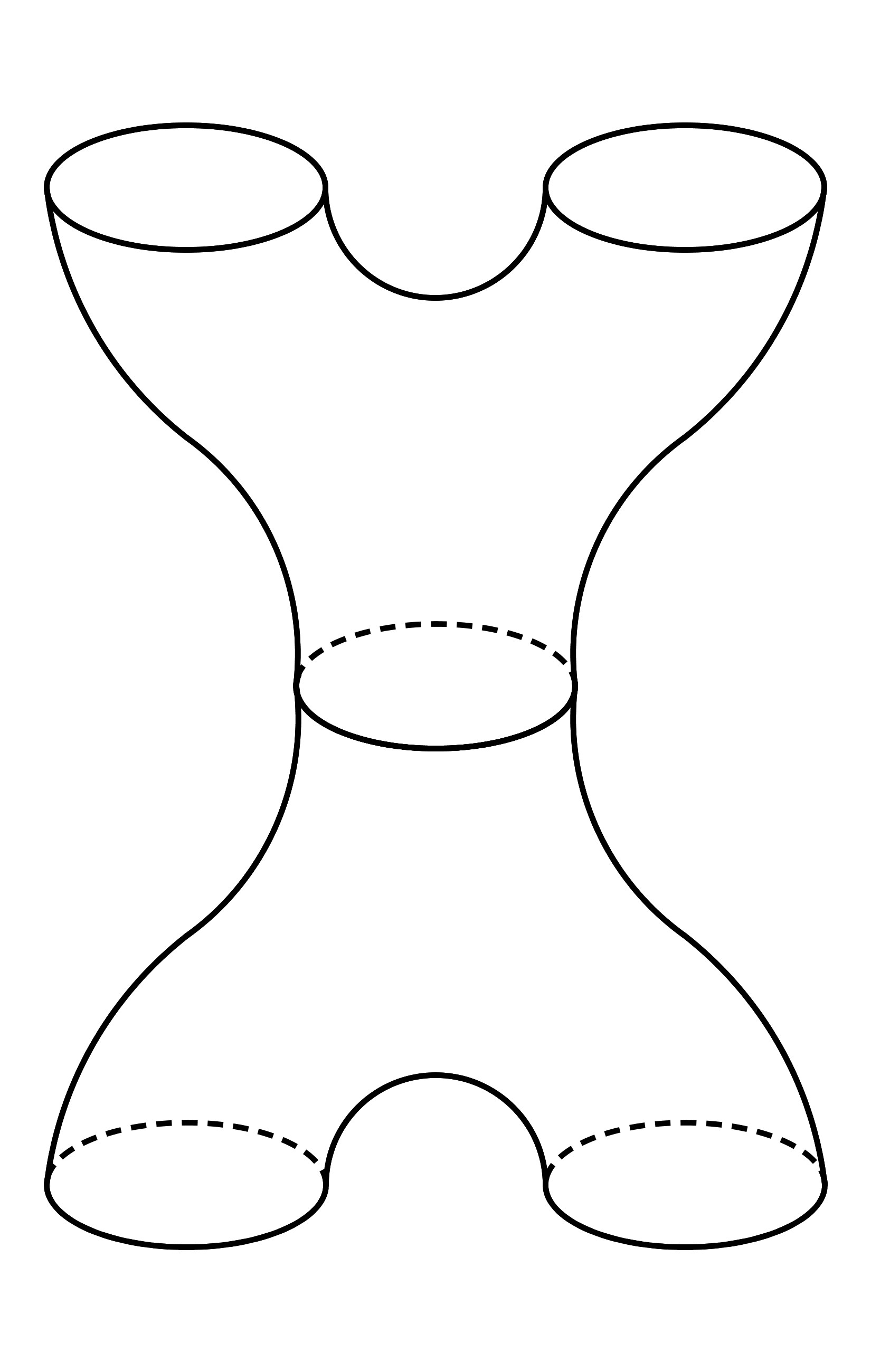} = \middiag{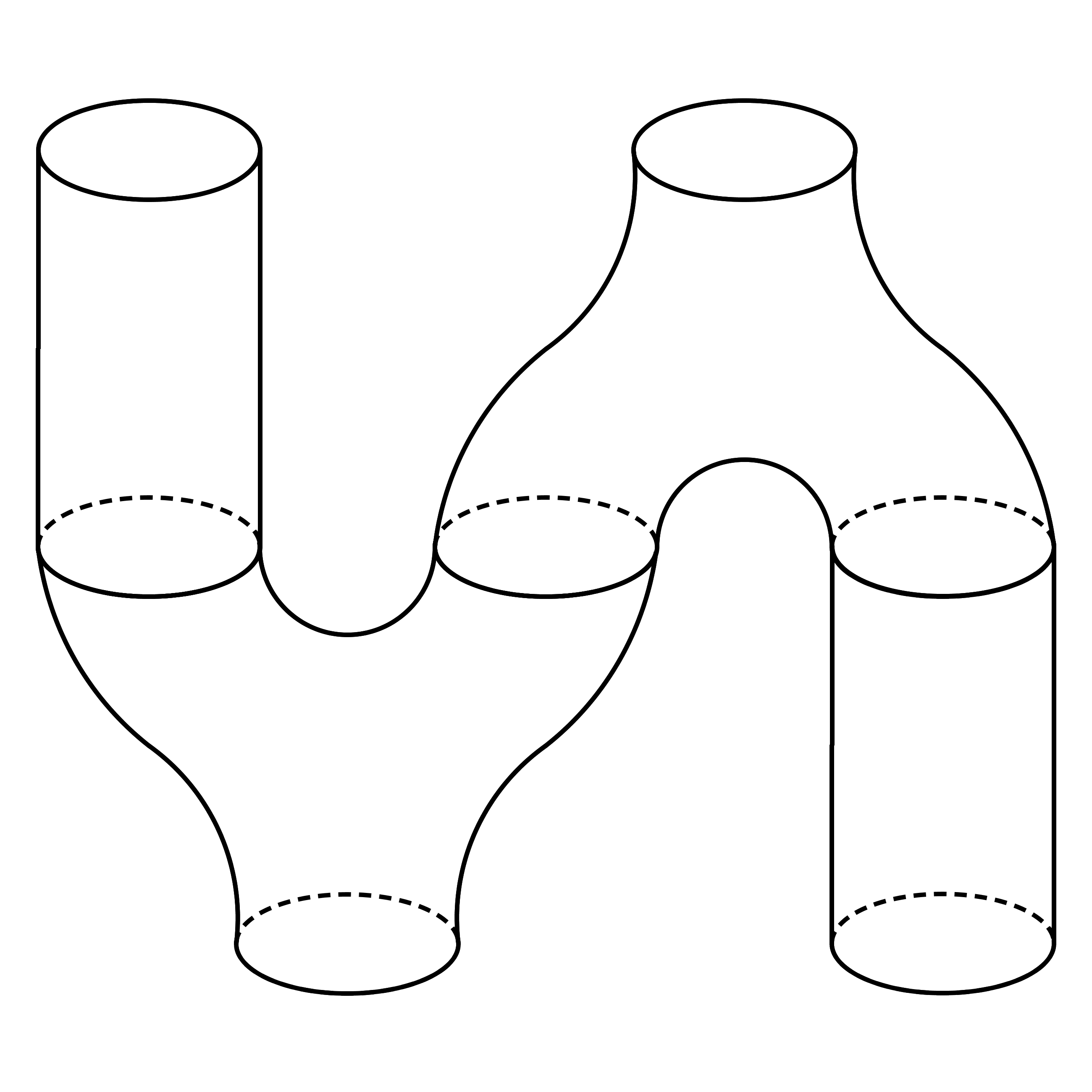}  $$
\item Unité et counité
\begin{align*}\middiag{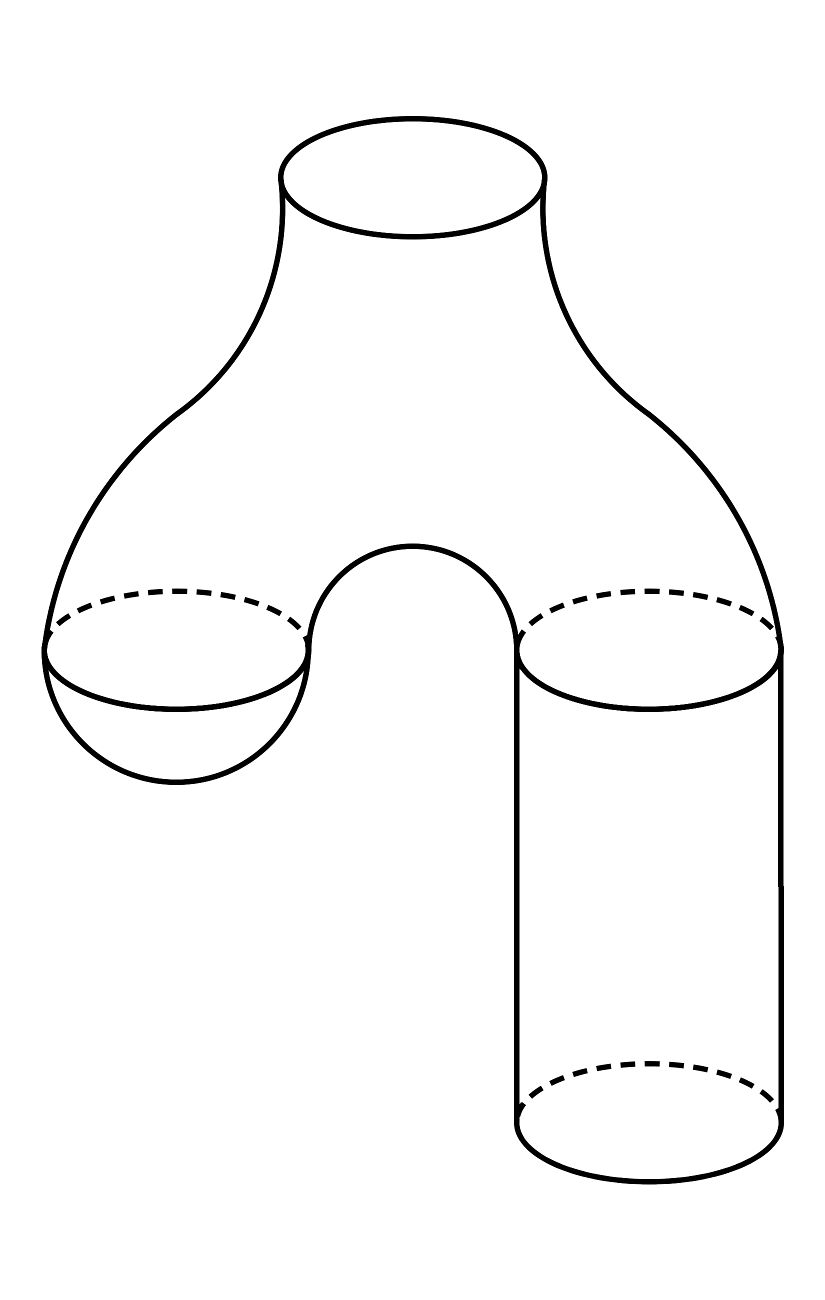} &= \smalldiag{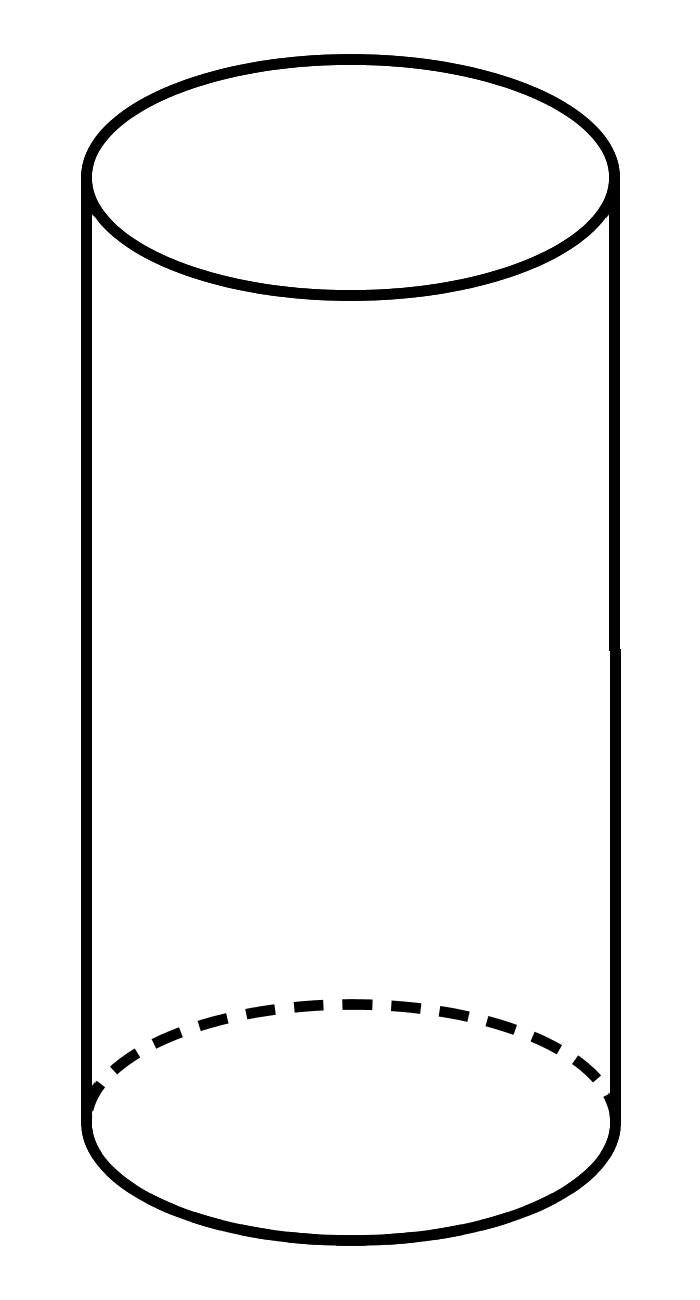}
&&,& \smalldiag{Images_arxiv/Cob_identity.png} &= \middiag{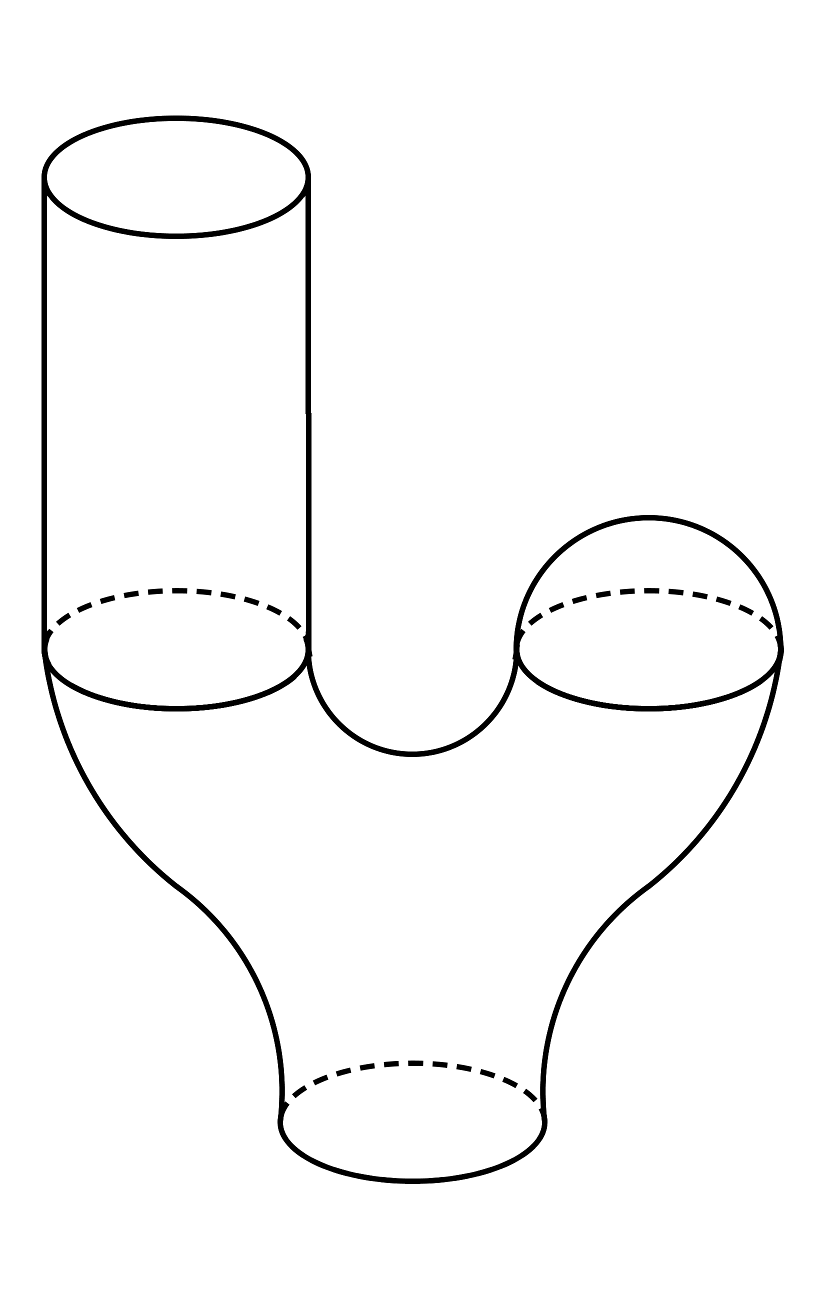}\end{align*}
\item Relations de permutations
\begin{align*}
\middiag{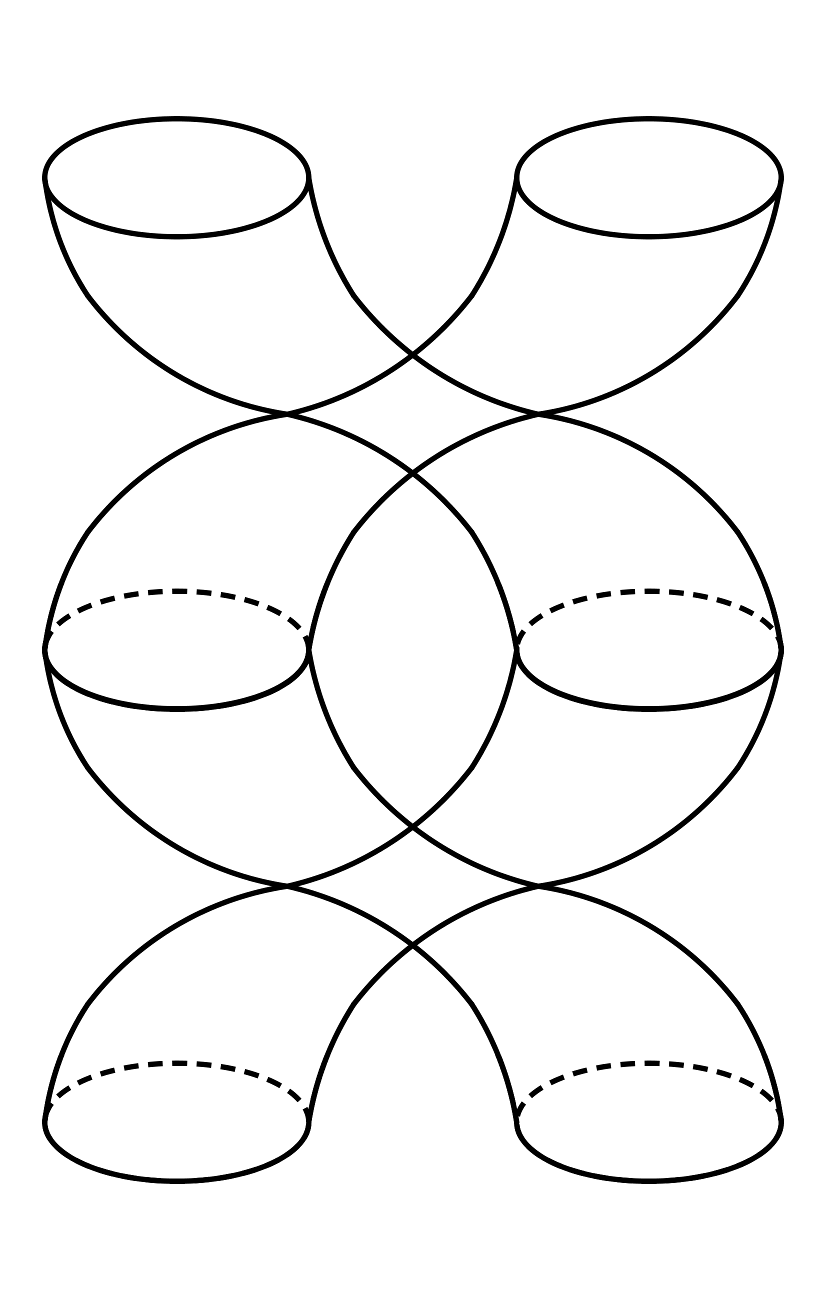} &= \smalldiag{Images_arxiv/Cob_identity.png}\smalldiag{Images_arxiv/Cob_identity.png}
&&,& \bigdiag{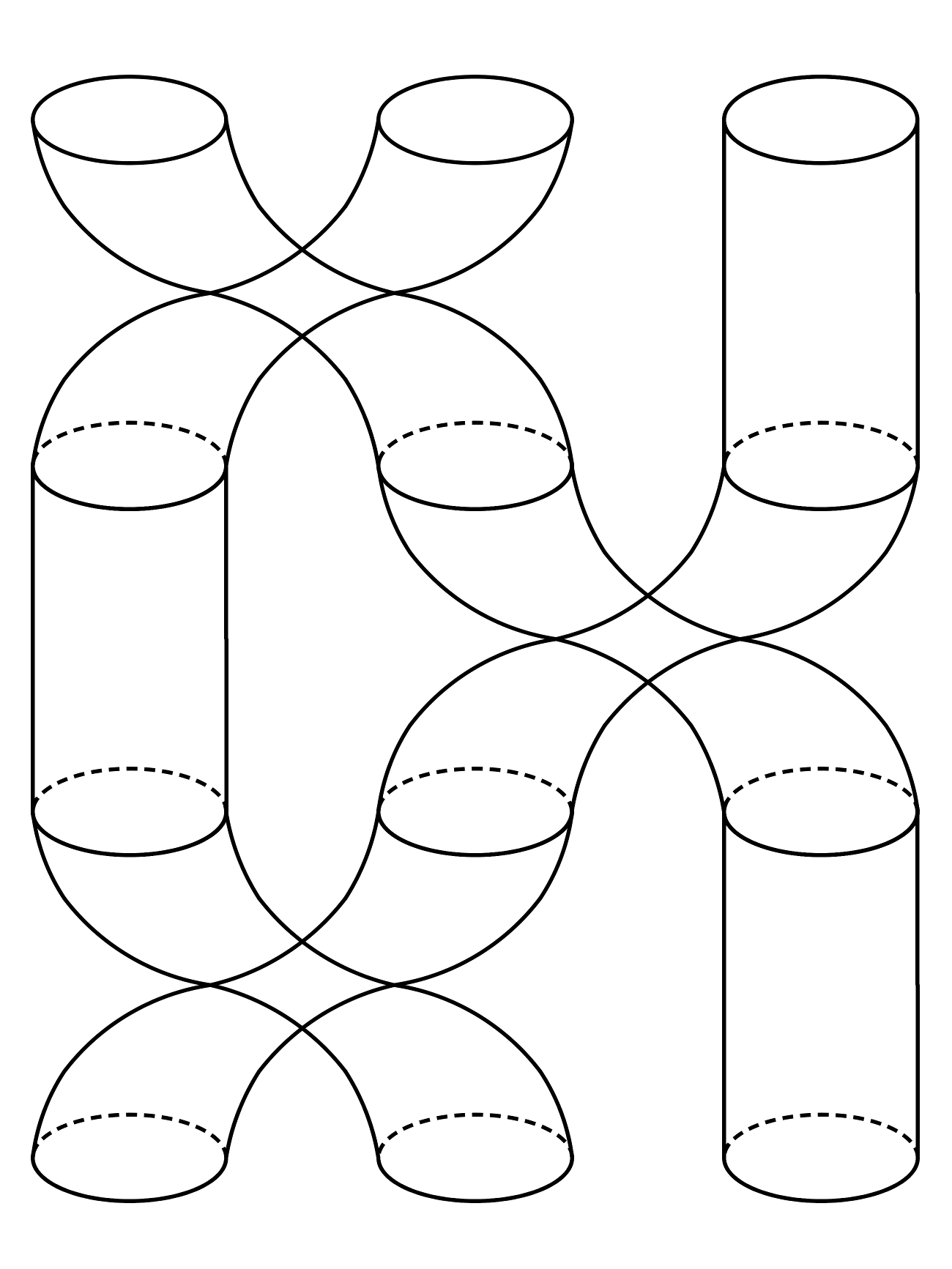}  &= \bigdiag{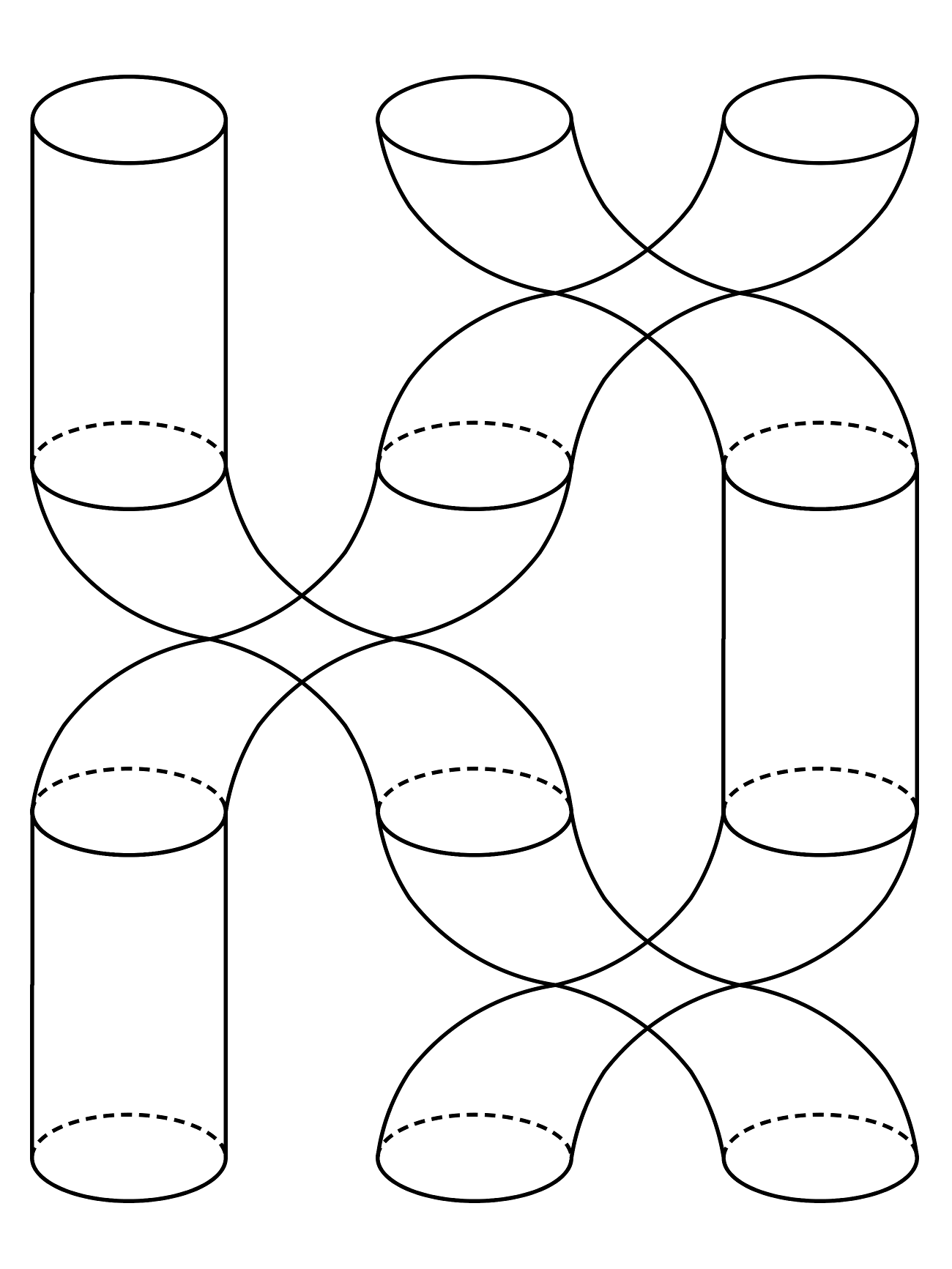} 
\end{align*}
\item Permutations de l'unité et de la counité
\begin{align*}\middiag{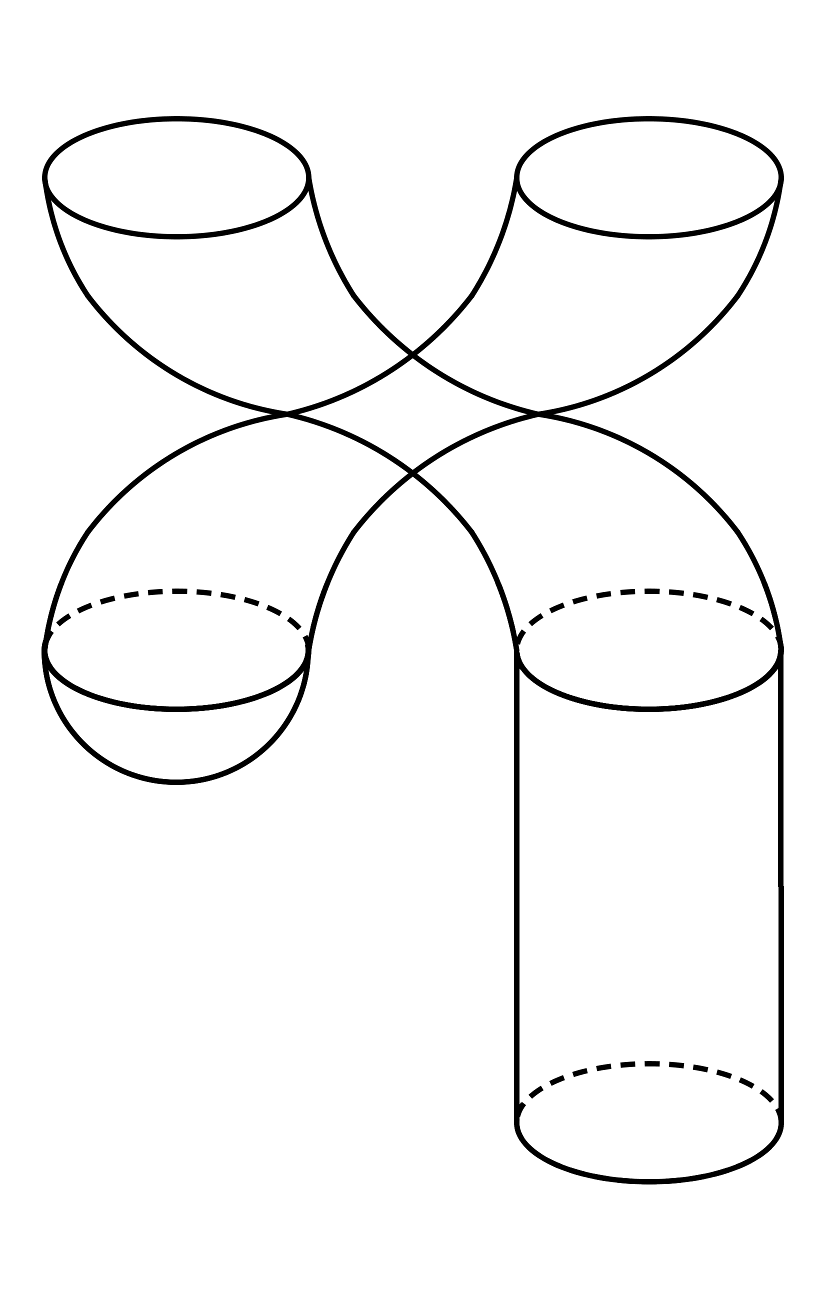} &= \smalldiag{Images_arxiv/Cob_identity.png} \smalldiag{Images_arxiv/Cob_birth.png}
&&,& \smalldiag{Images_arxiv/Cob_death.png} \smalldiag{Images_arxiv/Cob_identity.png} &= \middiag{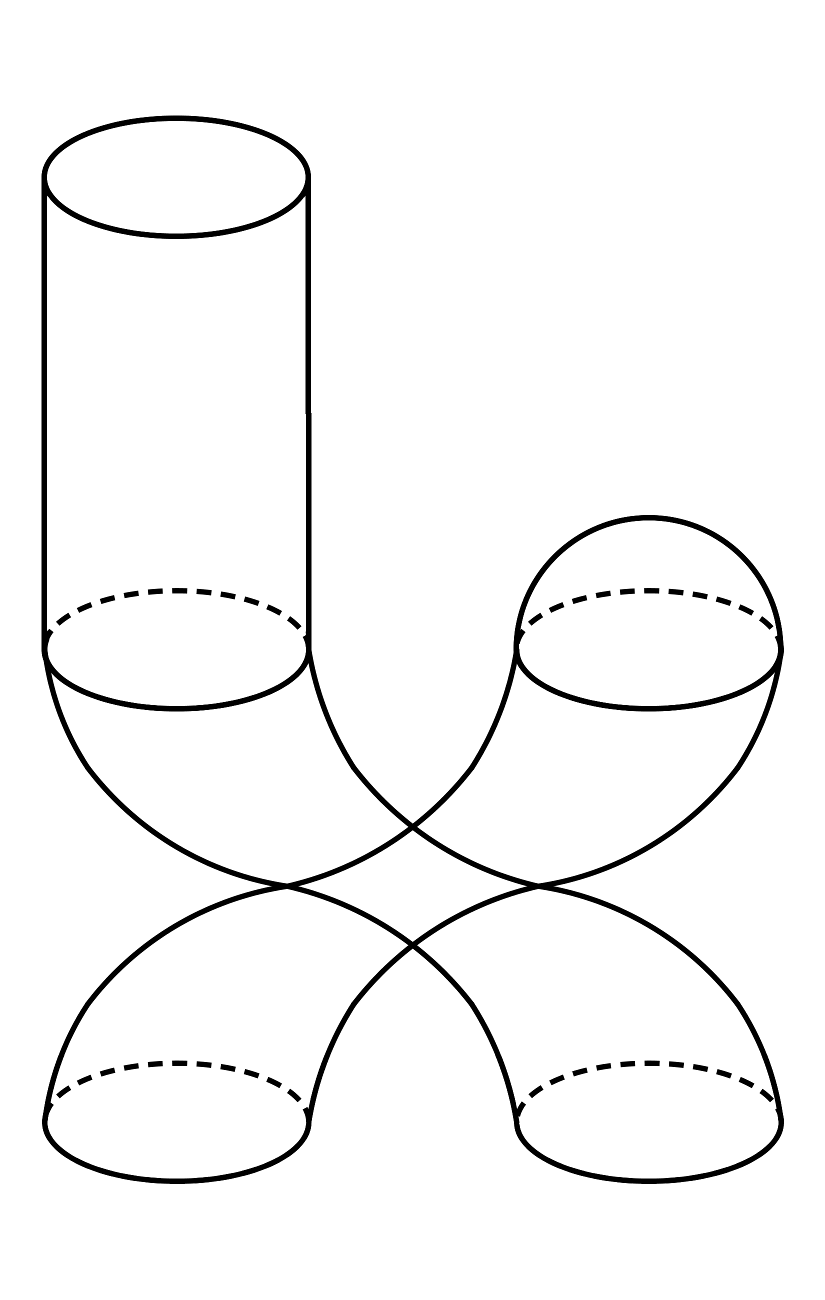}\end{align*}
\item Permutations de la fusion et de la scission
\begin{align*}
\middiag{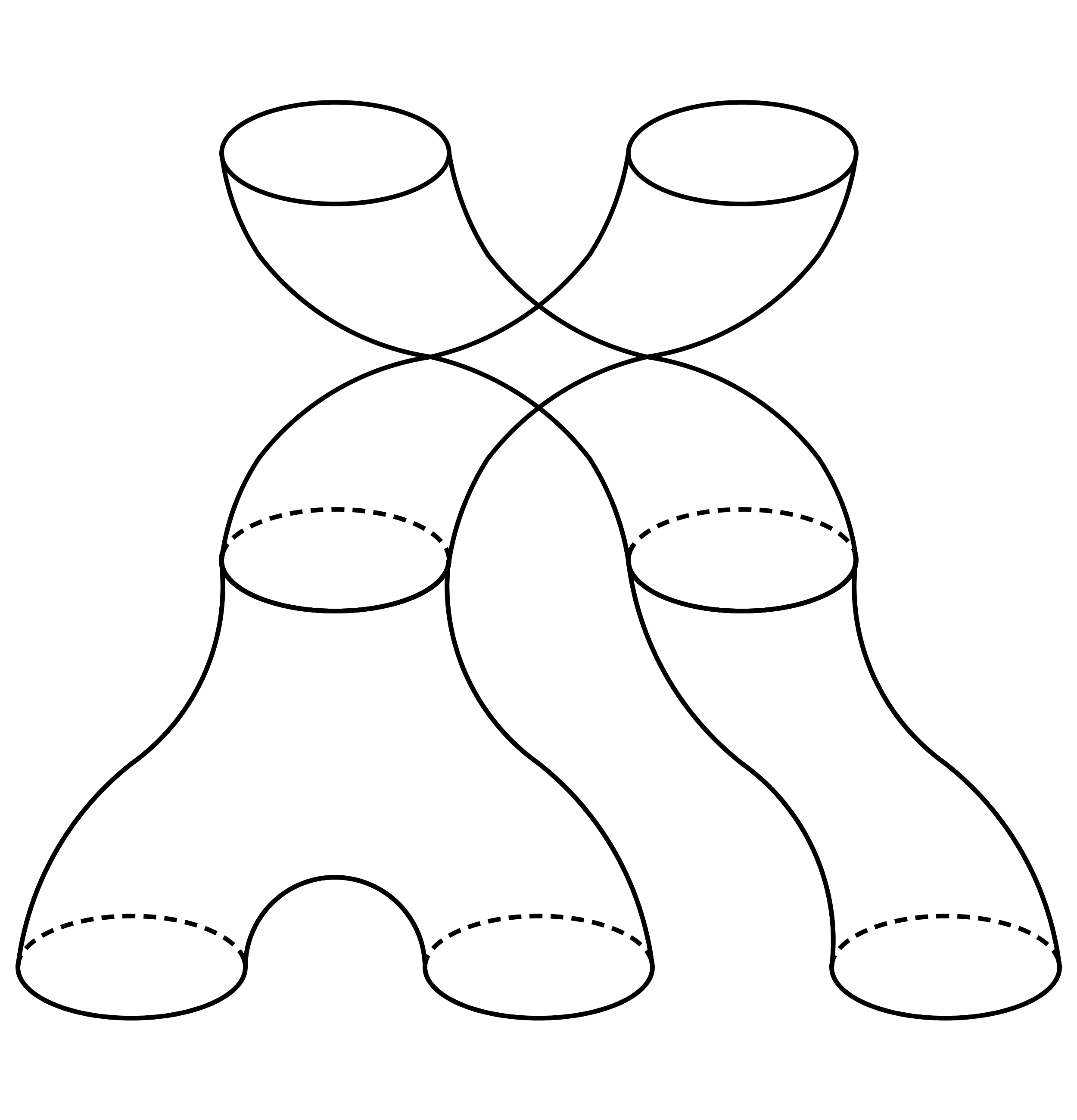} &= \bigdiag{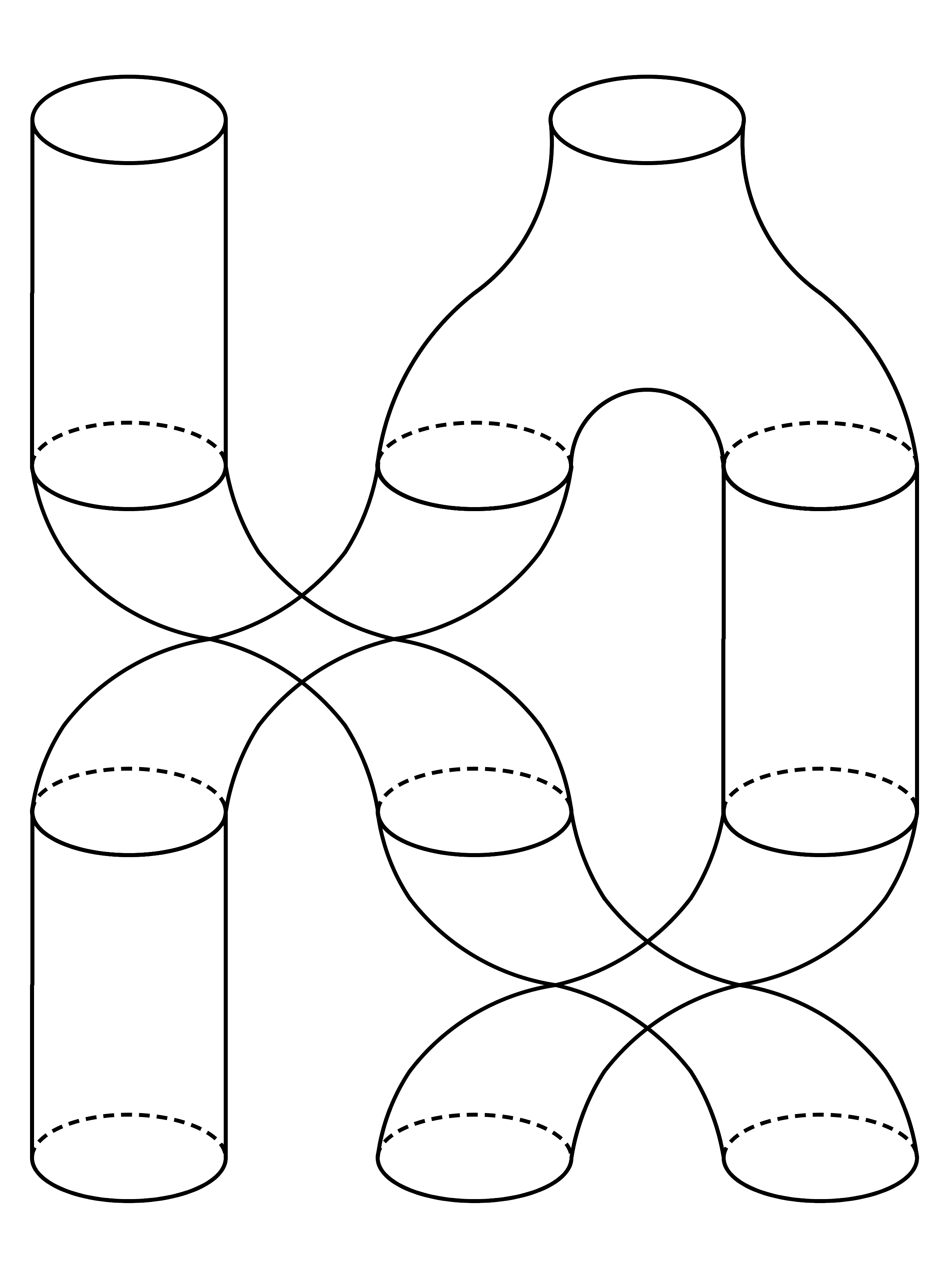}
&&,& \bigdiag{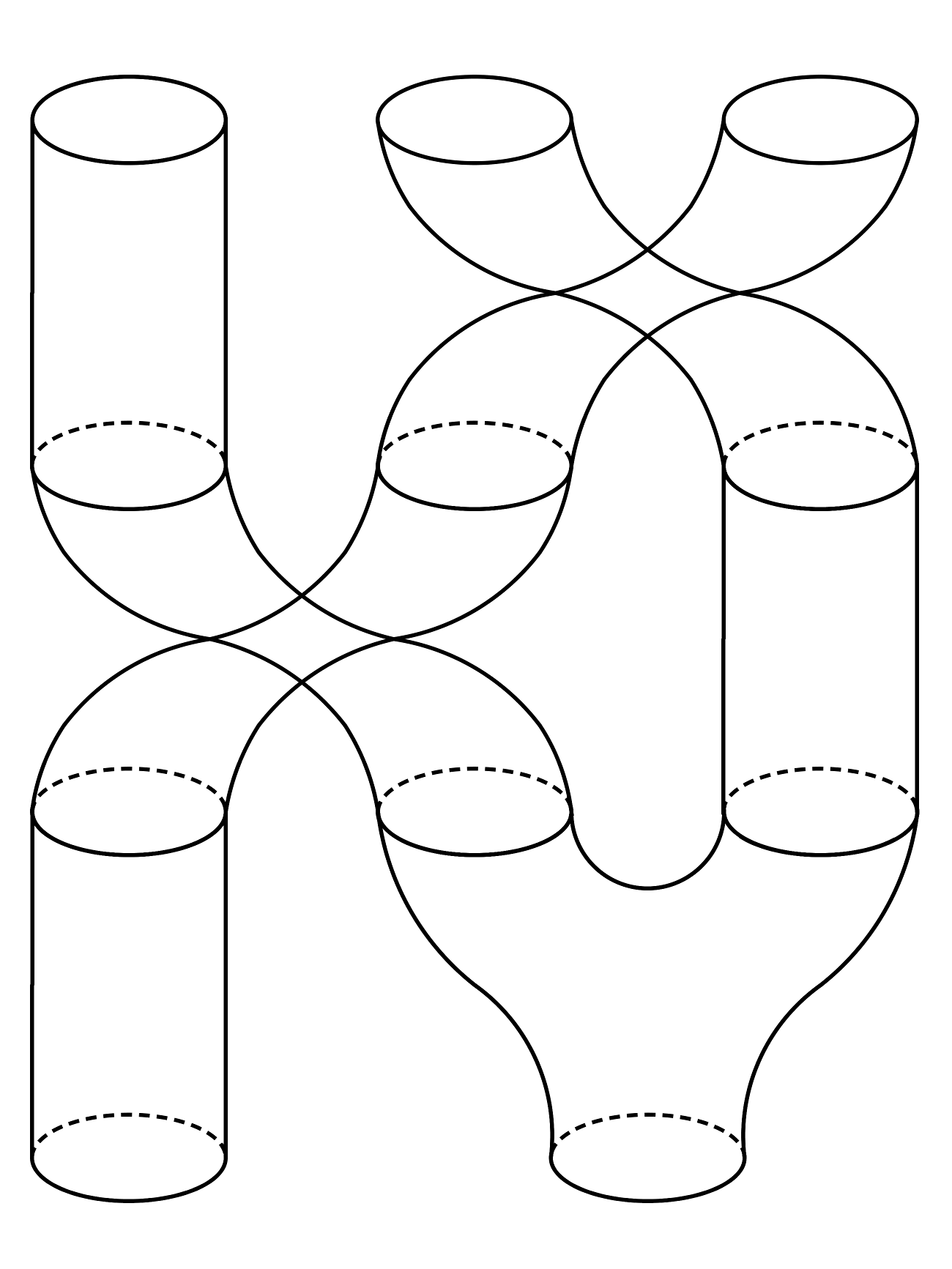}  &= \middiag{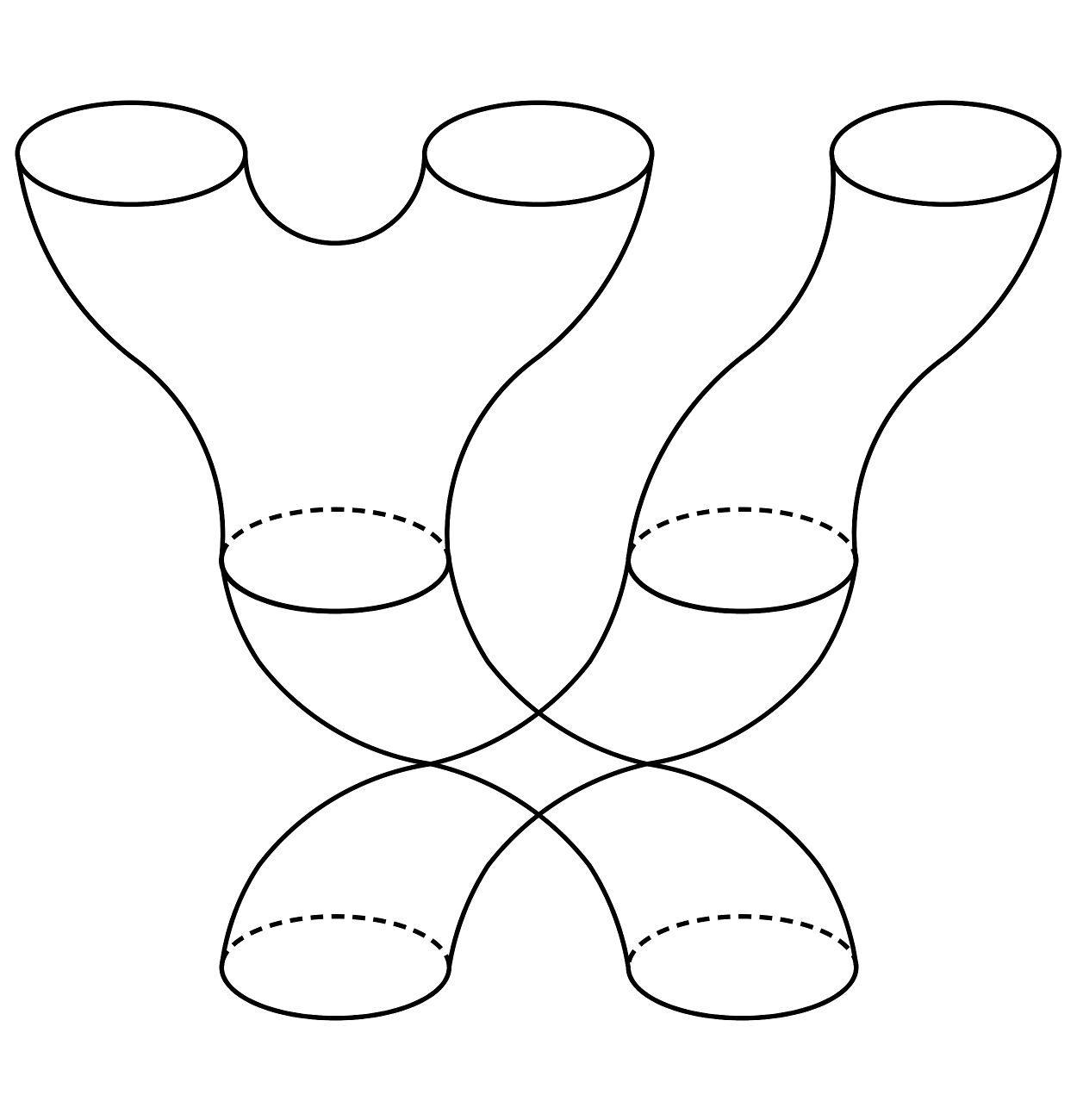} 
\end{align*}
\end{enumerate}
\end{theoreme}


\section{Le foncteur $F$}

M. Khovanov définit un foncteur $F : Cob \rightarrow \Z-Mod$ dans \cite{Khovanov} (qui se trouve aussi dans  \cite{KhovanovHomology} pour construire l'homologie de Khovanov) de la catégorie des cobordismes vers la catégorie des modules gradués sur $\Z$. Ce foncteur associe à une collection de $n$ cercles le module donné par le produit tensoriel sur $\Z$ de $n$ fois un certain module $A$ défini comme suit :
 \begin{align*}
A &:= \frac{ \Z[t]}{t^2}\{-1\}, & F\left(\underset{n}{\underbrace{\minidiag{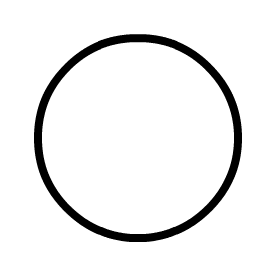} \dots \minidiag{Images_arxiv/Cercle.png}}}\right) &:= A^{\otimes n} = \underset{n}{\underbrace{A \otimes_\Z \dots \otimes_\Z A}},
\end{align*}
où $\deg(t) = 2$ et $\{-1\}$ est le décalage du degré par $-1$ comme expliqué dans les Annexes à la Section \ref{sec:gradue}. Autrement dit, $A$ est le groupe abélien libre gradué engendré par $1$ et $t$ avec $\deg(1) = -1$ et $\deg(t) = 1$.

Par le Théorème \ref{thm:cobelem}, il suffit de définir ce foncteur sur chacun des cobordismes élémentaires pour qu'il soit défini sur tous. La permutation et l'identité sont bien évidemment envoyés respectivement sur la permutation des $A$ dans le produit tensoriel et sur l'identité. \`A la naissance de cercle, on associe l'application unité 
$$F\left(\pitidiag{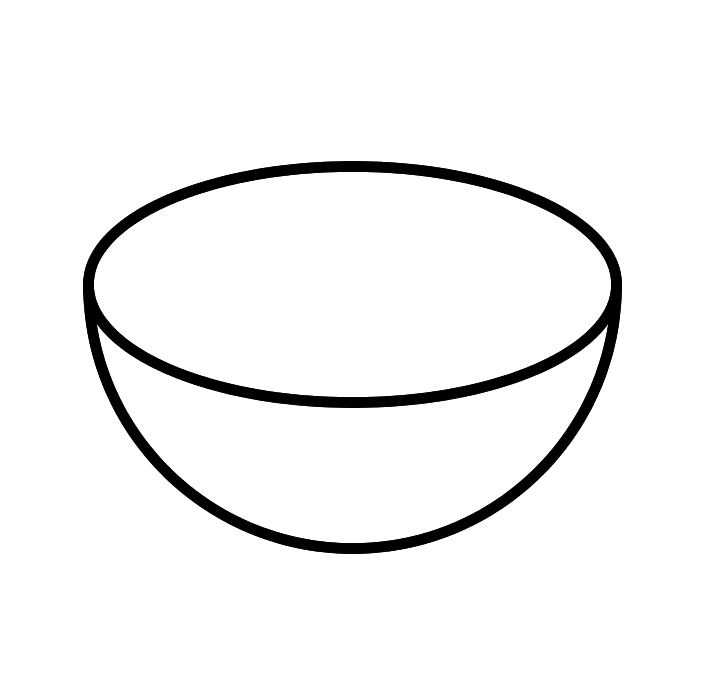}\right) :=  \imath : \Z \rightarrow A : 1 \mapsto 1.$$
\`A la fusion on associe la multiplication dans $A$ en remplaçant le produit tensoriel par le produit de polynômes
$$F\left(\diagg{Images_arxiv/Cob_merge.png}\right) := m : A \otimes A \rightarrow A : \begin{cases}
1 \otimes 1 &\mapsto 1, \\
t \otimes 1, 1\otimes t &\mapsto t, \\
t \otimes t &\mapsto t^2 = 0. \end{cases}$$
La scission est envoyée sur l'application de comultiplication
$$F\left(\diagg{Images_arxiv/Cob_split.png}\right) := \Delta : A \rightarrow A \otimes A : \begin{cases}1 &\mapsto 1\otimes t + t\otimes 1,\\
t &\mapsto t \otimes t. \end{cases}$$
Et enfin à la mort de cercle on associe la trace 
$$F\left(\pitidiag{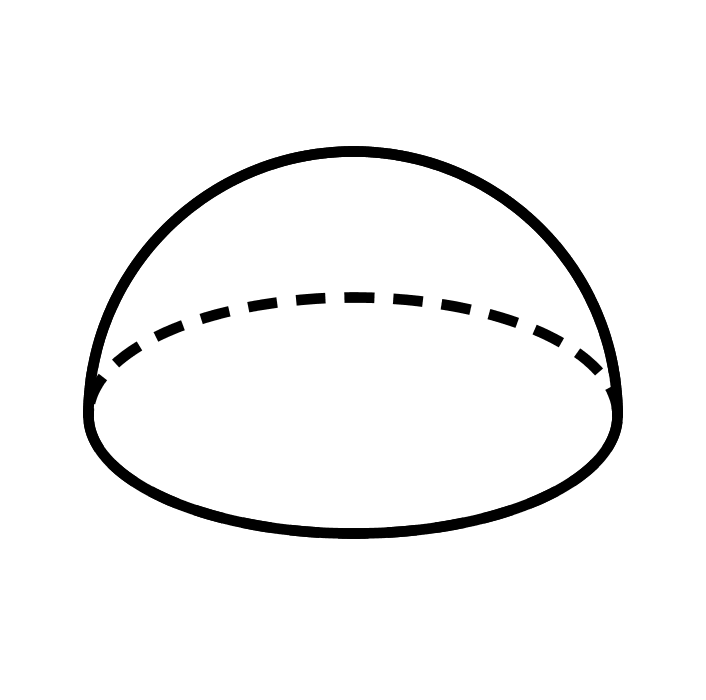}\right) := \epsilon : A \rightarrow  \Z : \begin{cases} 1 &\mapsto 0, \\ t &\mapsto 1. \end{cases}$$

On note que $A$ muni de $m$ comme multiplication, $\epsilon$ comme forme de Frobenius et $\Delta$ comme comultiplication forme une algèbre de Frobenius. On obtient donc une "two dimensional topological quantum field theory" (2D-TQFT) qui donne la fonctorialité de $F$. On renvoie vers \cite{TQFT} pour plus de détails mais, grossièrement, une 2D-TQFT est un foncteur monoïdal de la catégorie des cobordismes $Cob$ vers une catégorie algébrique. Il est aussi possible de simplement montrer que $F$ est bien défini pour les relations du Théorème \ref{thm:cobelem} si on ne veut pas parler de TQFT. De plus on obtient aisément le résultat suivant :

\begin{proposition} \label{eq:degchi}
$F$ associe à un cobordisme $C$ un homomorphisme de degré égal à moins la caractéristique d'Euler du cobordisme
$$\deg F(C) = -\chi(C). $$
\end{proposition}

\begin{proof}
On vérifie aisément avec les définitions de $\imath, m, \Delta$ et $\epsilon$ que l'équation tient puisque $\deg(\imath) = \deg(\epsilon) = -1$ et $\deg(m) = \deg(\Delta) = 1$.
\end{proof}

\begin{exemple}\label{ex:cobS2}
On calcule $F(S^2) : \Z \rightarrow \Z$ avec
 $$S^2 \simeq \smalldiag{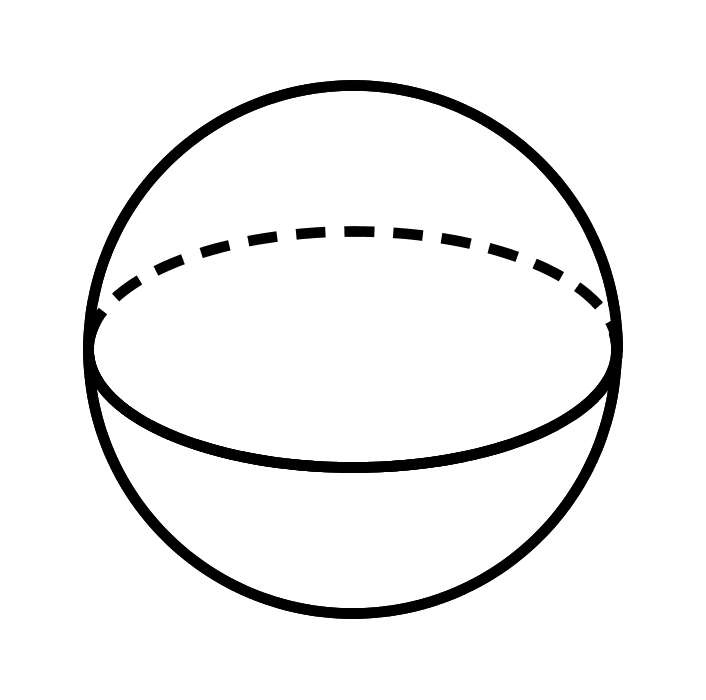}$$
qui est donc une naissance suivie d'une mort. On obtient alors
$$F(S^2) = \epsilon \circ \imath : \Z \rightarrow A \rightarrow \Z : 1 \mapsto 1 \mapsto 0$$
et donc $F(S^2) = 0$.
\end{exemple}

\begin{exemple}\label{ex:cobT2}
On calcule $F(T^2) : \Z \rightarrow \Z$ avec
 $$T^2 \simeq \bigdiag{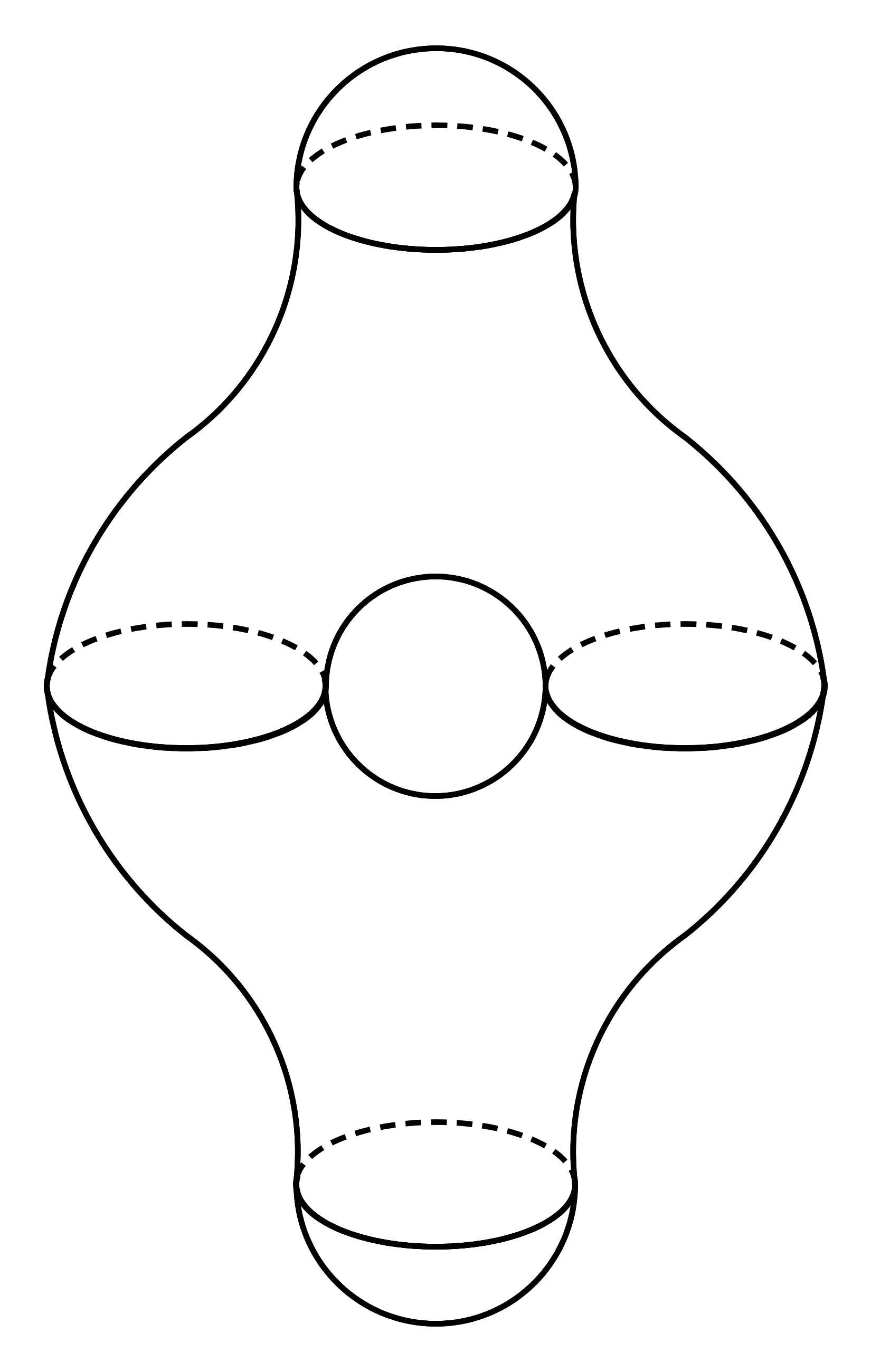}$$
qui est donc une naissance suivie d'une scission, d'une fusion et enfin d'une mort. On obtient alors
\begin{align*}
F(T^2) = \epsilon \circ m \circ \Delta \circ \imath : &\Z \rightarrow A \rightarrow A \otimes A \rightarrow A \rightarrow \Z :\\
&1 \mapsto 1 \mapsto 1 \otimes t + t \otimes 1  \mapsto 2t \mapsto 2
\end{align*}
et donc $F(T^2)$ est donné par la multiplication par $2$ dans $\Z$.
\end{exemple}

\section{Définition des anneaux des arcs}

Pour $n\ge 0$, on définit l'anneau des arcs de Khovanov d'indice $n$ vu comme groupe abélien gradué par la somme directe
\begin{align}
H^n &:= \bigoplus_{a,b\in B^n} b(H^n)a, & &\text{où}&
b(H^n)a &:= F(W(b)a)\{n\}. \label{eq:defhn}\end{align}
Comme remarqué dans la Section \ref{sec:TL}, $W(b)a \in W(B^n)B^n \subset \widehat{B}_0^0$ et $W(b)a$ est donc une union disjointe de cercles, d'où le fait qu'on puisse lui appliquer le foncteur $F$.

\begin{remarque} On note que décaler le degré des éléments par $n$ permet d'obtenir un groupe positivement gradué. En effet, l'élément $W(b)a$ qui contient le plus grand nombre de composantes de cercles est donné par un élément de la forme $W(a)a$ et contient donc $n$ composantes. Dès lors, le plus grand produit tensoriel contient $n$ éléments et l'élément de degré minimal de $F(W(a)a)$ est $1 \otimes \dots \otimes 1$ de degré $-n$.
\end{remarque}

Pour faire de $H^n$ un anneau, il reste à définir une multiplication et une unité. On pose $xy = 0$ pour $x \in d(H^n)c$ et $y \in b(H^n)a$ si $c \ne b$. Il reste encore à définir la multiplication
$$m_{cba} : c(H^n)b \otimes b(H^n) a \rightarrow c(H^n)a.$$
L'idée est de construire un cobordisme de $W(c)bW(b)a$ vers $W(c)a$ afin d'obtenir une multiplication donnée par l'image de ce cobordisme par $F$. On observe que $bW(b)$ donne tous des demi cercles, chacun possédant sa symétrie horizontale en face, comme par exemple :
\begin{align*}
b &= \diagg{Images_arxiv/B3_3}, & bW(b) &= \middiag{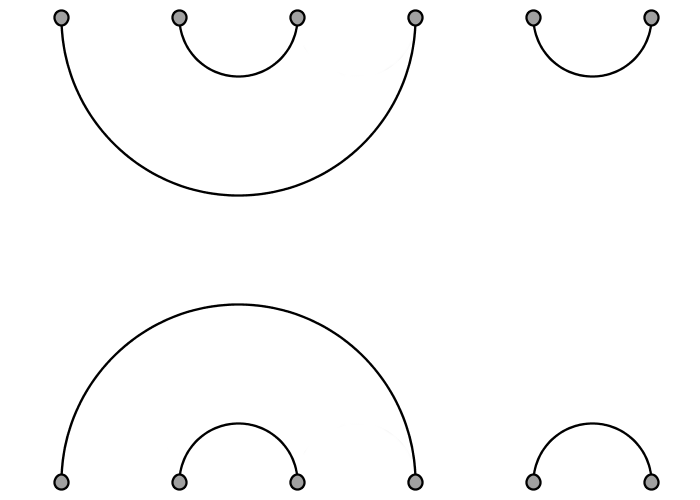}.
\end{align*}
On peut donc construire des ponts donnant des "cobordismes avec bords" possédant un unique point de selle et qui envoient à chaque fois une paire de demi cercles vers deux segments de droite, comme représenté en Figure \ref{fig:bridge}. 

\begin{figure}[h]
    \center
    \includegraphics[width=12cm]{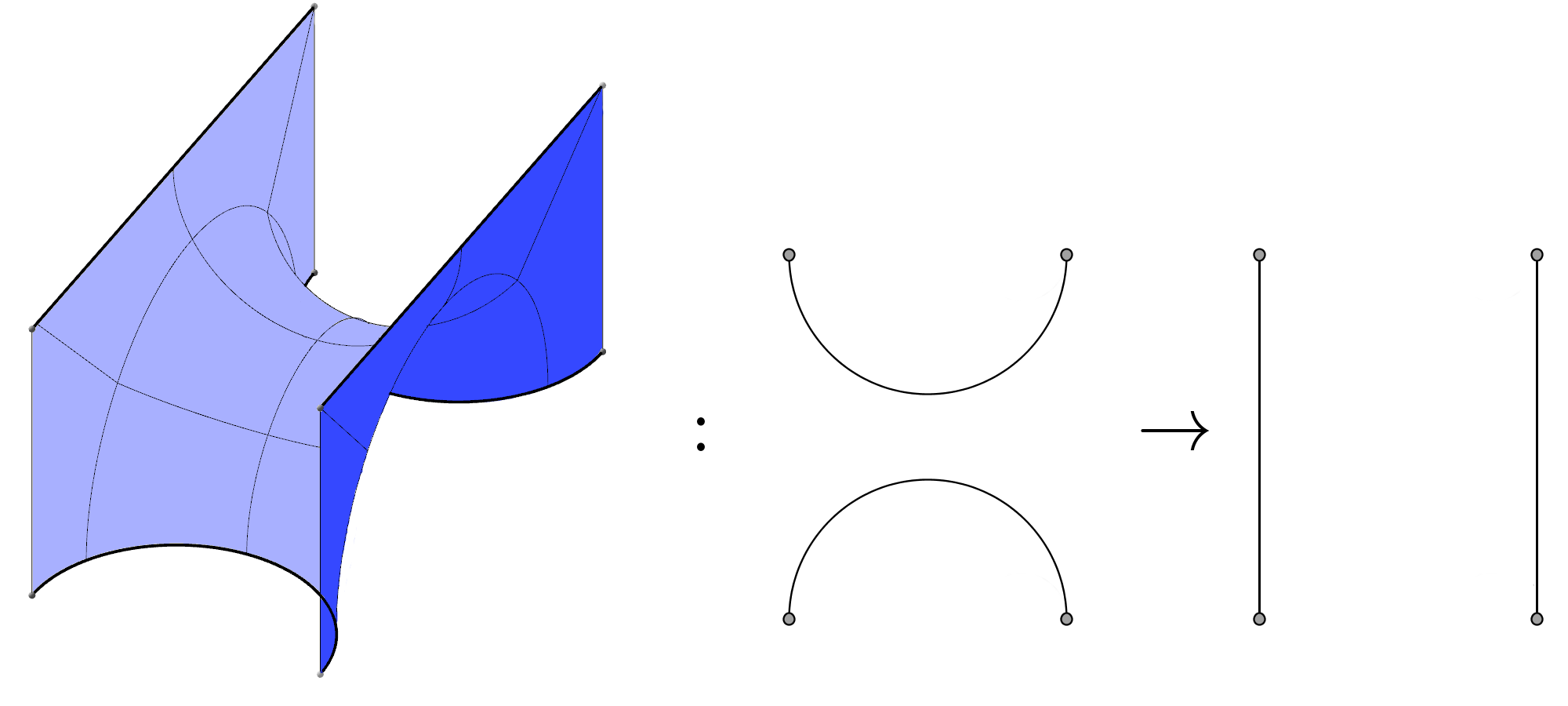}
    \caption{\label{fig:bridge}Pont : cobordisme entre deux arcs miroirs et deux lignes possédant un unique point de selle.}
\end{figure}

En composant ces ponts, on obtient un cobordisme de $bW(b)$ vers $Vert_{2n}$ possédant $n$ points de selle (un pour chaque pont), qu'on nomme $C(b)$. On donne un exemple de film en Figure \ref{fig:decoupecob} où chaque vignette représente la coupe horizontale du corbordisme en différentes hauteurs.

\begin{figure}[h]
    \center
    \includegraphics[width=14cm]{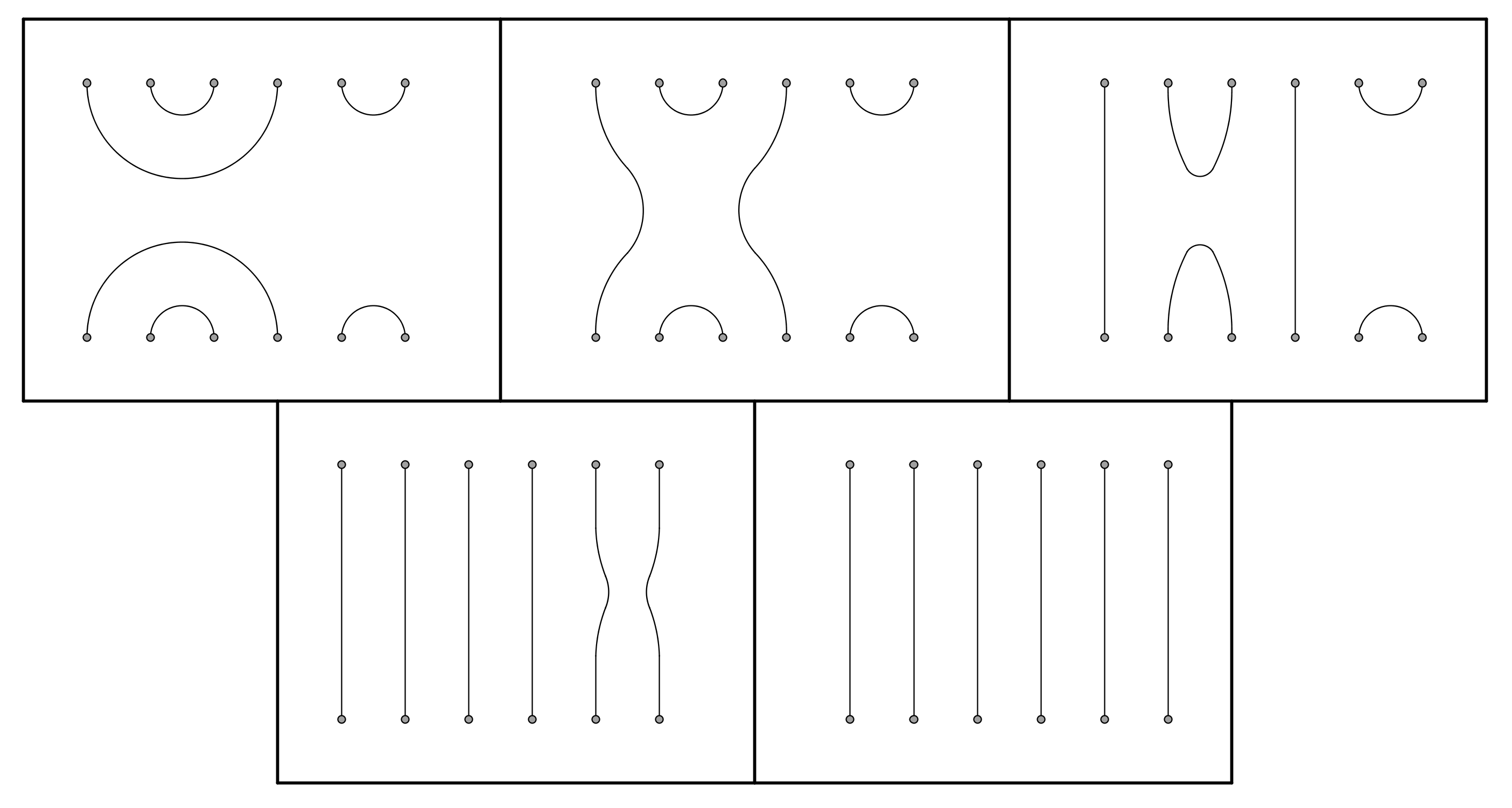}
    \caption{\label{fig:decoupecob}Film du cobordisme $C(b)$ entre $bW(b)$ et $Vert_{2n}$ : chaque vignette représente une hauteur différente en commençant par celle en haut à gauche de $bW(b)$ et en terminant en bas à droite par $Vert_{2n}$ (illustration venant de \cite{Khovanov}).}
\end{figure}

En prolongeant ce "cobordisme avec bord" $C(b)$ avec l'identité sur $a$ et sur $W(c)$, on obtient un cobordisme  de $Cob$
$$\Id_{W(c)}C(b)\Id_a : W(c)bW(b)a \rightarrow W(c)a$$
qui donne un homomorphisme de modules en lui appliquant $F$ 
\begin{equation} F\bigl(\Id_{W(c)}C(b)\Id_a\bigr) : F\bigl(W(c)bW(b)a\bigr) \rightarrow F\bigl(W(c)a\bigr). \label{eq:hommulti} \end{equation}
Puisque  $W(c)b$ et $W(b)a$ sont tous deux des unions disjointes de cercles, on a un isomorphisme canonique
$$F\bigl(W(c)b\bigr) \otimes F\bigl(W(b)a\bigr)  \simeq F\bigl(W(c)bW(b)a\bigr).$$
On obtient en composant cet isomorphisme avec le morphisme (\ref{eq:hommulti}) un homomorphisme
$$F\bigl(W(c)b\bigr)  \otimes F\bigl(W(b)a\bigr) \rightarrow F\bigl(W(c)a\bigr).$$ 
Puisque le cobordisme a $n$ points de selles, il a une caractéristique d'Euler valant $-n$ et par la Proposition \ref{eq:degchi} le morphisme obtenu a un degré $n$. On obtient alors un homomorphisme
\begin{equation}F\bigl(W(c)b\bigr)\{n\} \otimes F\bigl(W(b)a\bigr)\{n\} \rightarrow F\bigl(W(c)a\bigr)\{n\} \label{eq:hommulti2} \end{equation}
qui préserve le degré. On définit la multiplication en utilisant le diagramme commutatif suivant où les isomorphismes viennent de (\ref{eq:defhn}) :
$$\xymatrix{
c(H^n)b \otimes b(H^n)a \ar[r]^-{m_{cba}} \ar[d]_{\simeq} & c(H^n)a \\
F\bigl(W(c)b\bigr)\{n\} \otimes F\bigl(W(b)a\bigr)\{n\} \ar[r]_-{(\ref{eq:hommulti2})}  & F\bigl(W(c)a\bigr)\{n\} \ar[u]_{\simeq}.
}$$
On définit pour $a\in B^n$, 
$$1_a = 1^{\otimes n}\{n\} \in A^{\otimes n}\{n\} \simeq a(H^n)a$$
où $ 1^{\otimes n}\{n\}$ est un abus de notation pour dire qu'on prend l'unité dans $A^{\otimes n}$ dont on a décalé le degré par $n$. On vérifie que pour $x \in b(H^n)a$ on a $x.1_a = x$ et $1_a.x = 0$ si $a\ne b$ et pour $y \in a(H^n)b$, on a $y.1_a = 0$ et $1_a.y = y$. On définit alors l'unité $1\in H^n$ comme la somme
$$1 = \sum_{a\in B^n} 1_a.$$
On obtient l'associativité par fonctorialité de $F$ et associativité/co-associativité des cobordismes. Finalement, la distributivité est claire par définition.\\

On termine le chapitre par quelques exemples de calculs dans $H^n$ afin de familiariser le lecteur avec ces anneaux.

\begin{exemple}\label{ex:H1}
On peut montrer que $H^1 \simeq A\{1\}$ puisque $B^1$ ne contient qu'un seul élément $a$ donné par un demi-cercle qui relie les deux points. On a alors un isomorphisme de groupes gradués $H^1 \simeq_{ab} F(W(a)a)\{1\}$ et, $W(a)a$ n'étant composé que d'un seul cercle, on a $F(W(a)a) = A$. Par ailleurs, la multiplication 
$$a(H^1)a \times a(H^1)a \rightarrow a(H^1)a$$
est composée uniquement d'une fusion et est donc équivalente à la multiplication polynomiale, montrant le résultat voulu.
\end{exemple}

\begin{exemple}
On prend $H^2$ et on pose
\begin{align*}
a &=\deuxdiag{Images_arxiv/B2_2.png},& 
b &= \deuxdiag{Images_arxiv/B2_1.png}.
\end{align*}
On observe
\begin{align*}
W(a)a &= \diagg{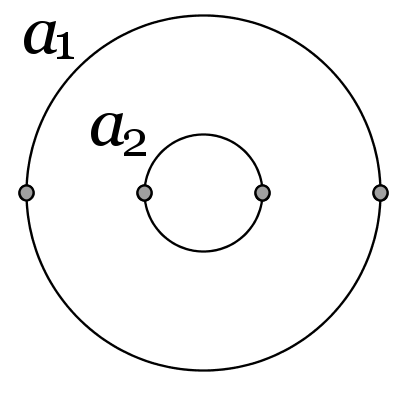},&
W(a)b &= \diagg{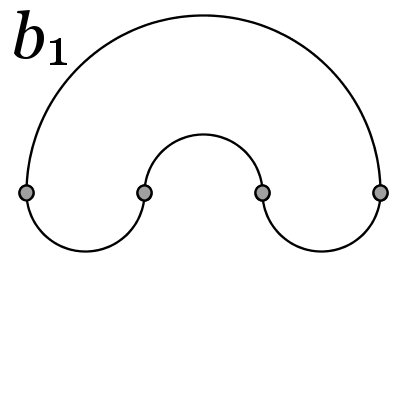},&
W(b)a &= \diagg{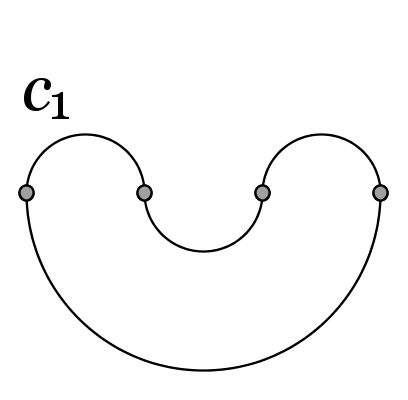}.
\end{align*}
où on note $a_1$ et $a_2$ les cercles de $W(a)a$ avec $a_1$ le cercle extérieur et $a_2$ le cercle intérieur. De même on note $b_1$ le cercle de $W(a)b$ et $c_1\in W(b)a$. Pour éviter toute confusion, on note $t_{x}$ la variable $t$ du $A$ associé par le foncteur $F$ à la composante de cercle d'étiquette $x$ et de même pour $1_{x}$. On considère la multiplication
$$a(H^2)a \times a(H^2)b \rightarrow b(H^2)a$$
et on calcule
$$(t_{a_1} \otimes 1_{a_2}).{1_{b_1}} = t_{b_1}.$$
En effet, on a
\begin{gather*}
t_{a_1} \otimes 1_{a_2} \otimes 1_{b_1}  \xmapsto{m} t_{x_1} \otimes 1_{a_2} \xmapsto{m}  t_{b_1}, \\
\bigdiag{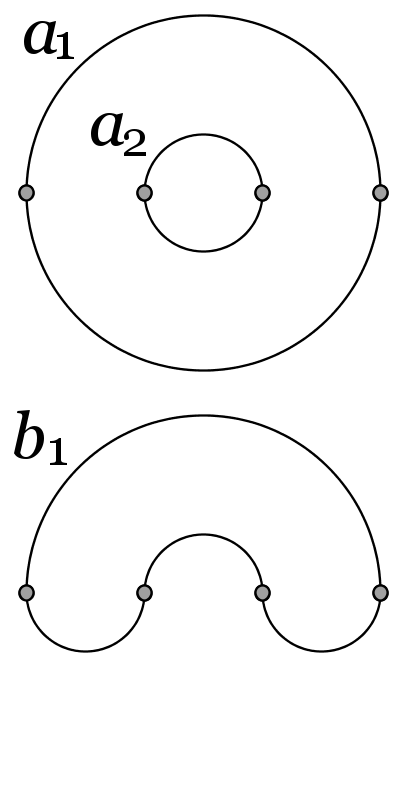} 
\xrightarrow{(a_1,b_1) \mapsto x_1} \bigdiag{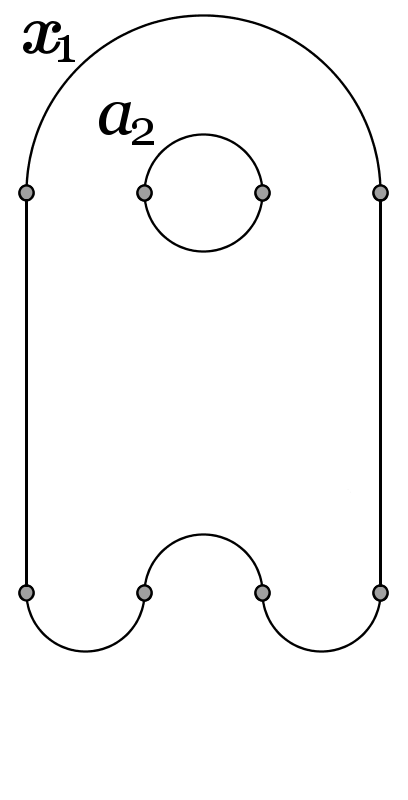}
\xrightarrow{(a_2,x_1) \mapsto b_1} \bigdiag{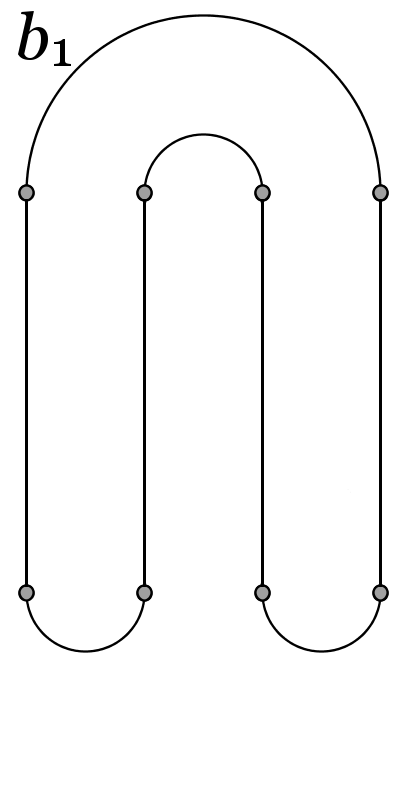}
\end{gather*}
puisqu'on a juste deux fusions.
\end{exemple}

\begin{exemple}En prenant les même notations que l'exemple précédent, on calcule maintenant 
$$1_{b_1}.t_{c_1} = t_{a_1} \otimes t_{a_2}$$
dans
$$a(H^2)b \times b(H^2)a \rightarrow a(H^2)a.$$
En effet, on a
$$
1_{b_1} \otimes t_{c_1}  \xmapsto{m} t_{x_1} \xmapsto{\Delta} t_{a_1} \otimes t_{a_2}, $$
$$\bigdiag{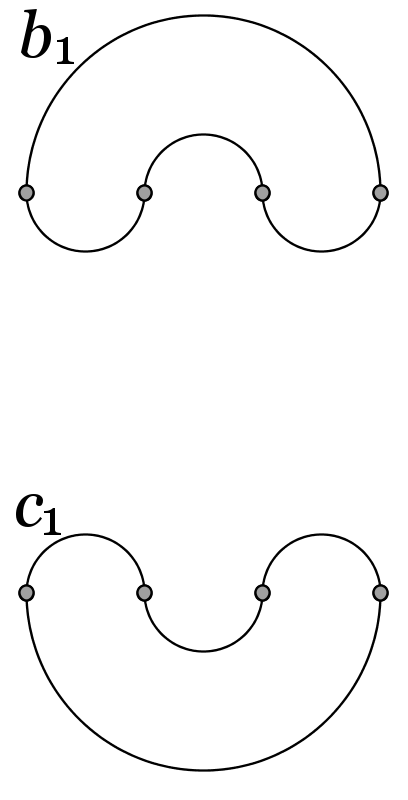} 
\xrightarrow{(b_1,c_1) \mapsto x_1} \bigdiag{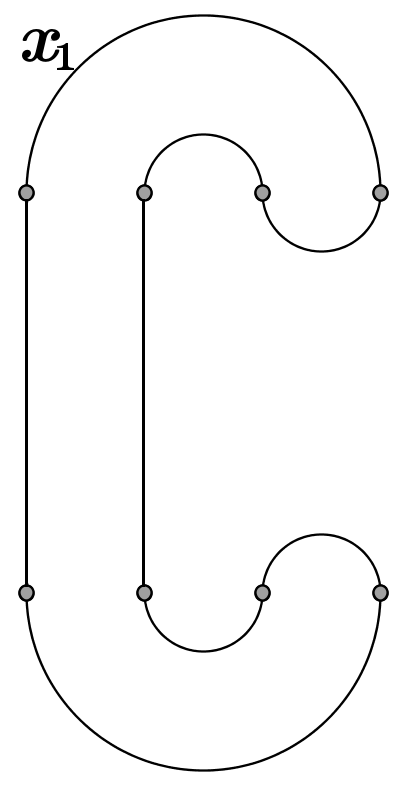}
\xrightarrow{x_1 \mapsto (a_1,a_2)} \bigdiag{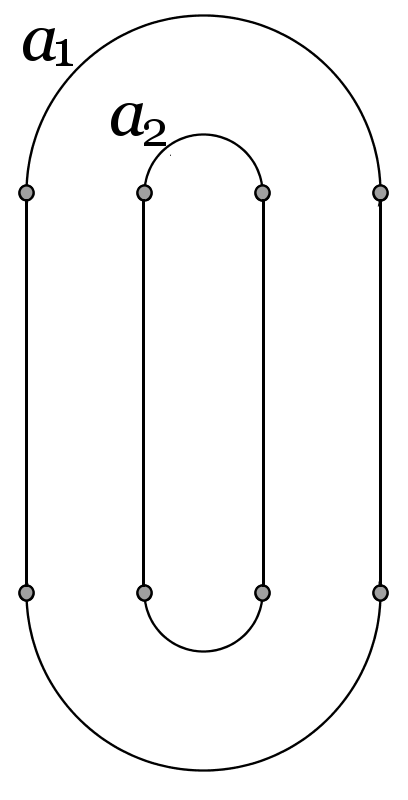}
$$
puisqu'on a une fusion suivie d'une scission.
\end{exemple}

\chapter{Constructions impaires des anneaux des arcs de Khovanov}

Dans \cite{OddKhovanov},  P. Ozsvath, J. Rasmussen et Z. Szabo construisent une version impaire de l'homologie de Khovanov. Pour ce faire, ils définissent un foncteur projectif de la catégorie des cobordismes vers la catégorie des modules sur $\Z$. Ce foncteur est projectif dans le sens où deux cobordismes équivalents peuvent livrer des morphismes de signes différents et qu'il y a un choix d'orientations à faire pour les composantes de cercles des scissions. Ils montrent ensuite que ces signes n'influencent pas l'homologie qu'on obtient. On propose dans ce chapitre de construire des anneaux similaires à $H^n$ en utilisant le foncteur de \cite{OddKhovanov} à la place de celui de M. Khovanov. Le choix de signes devenant important et pouvant potentiellement livrer des anneaux différents, on redéfinit ce foncteur sur une catégorie un peu plus structurée que celle des cobordismes : la catégorie des cobordismes avec chronologies, introduite par K. Putyra dans \cite{Putyra}.

\section{Cobordismes avec chronologies}

Les cobordismes avec chronologies (ou cobordismes chronologiques) sont des cobordismes où on donne un ordre sur les points critiques (donc un ordre sur une fonction de Morse possédant un seul point critique par niveau), ce qu'on appelle une chronologie, et un choix d'orientation pour ceux-ci (c'est-à-dire pour les fusions et les scissions, les naissances et morts de cercles ayant des orientations constantes). Plus précisément, l'orientation des points critiques d'indice 1 est donnée par le choix d'une base pour l'espace propre de la sous-matrice définie négative de la matrice hessienne de la chronologie et ceux d'indice 2 ont l'orientation induite par l'orientation du cobordisme. On donne un exemple de cobordisme avec chronologie en Figure \ref{fig:cob_ch}. Deux cobordismes chronologiques sont dits équivalents s'il existe un difféomorphisme préservant les bords entre les deux et qui conserve l'ordre de la chronologie ainsi que l'orientation des points critiques. Tout cela se définit de façon formelle et on renvoie à \cite{Putyra} pour plus de détails.

\begin{figure}[h]
    \center
    \includegraphics[width=4cm]{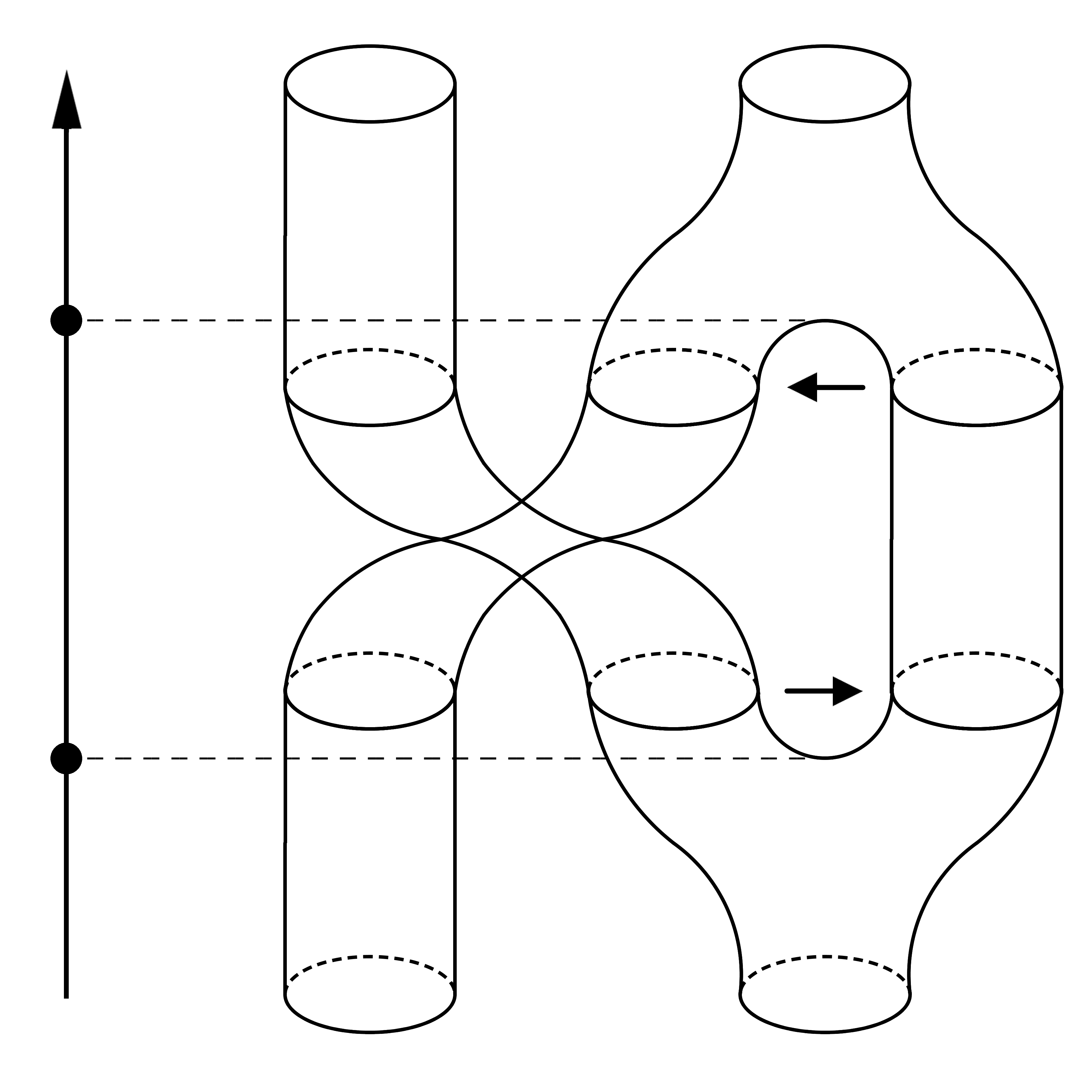}
    \caption{\label{fig:cob_ch}Exemple de cobordisme avec chronologie, les flèches sur les scissions et fusions représentent l'orientation des points critiques d'indice 1 (c'est-à-dire le "sens" de la base choisie pour la sous-matrice hessienne) et la flèche de gauche représente la chronologie, les gros points noir étant les points critiques.}
\end{figure}

Il y a une autre différence par rapport aux cobordismes sans chronologie qui se situe dans la multiplication monoïdale. En effet, on ne peut pas simplement prendre l'union disjointe des deux cobordismes en les mettant "côte à côte" car on pourrait obtenir plusieurs points critiques au même niveau de leurs chronologies. On définit alors la multiplication à gauche de $M_1 : S_1 \rightarrow S_1'$ et $M_2 : S_2 \rightarrow S_2'$ comme $M_1$ union disjointe l'identité sur $S_2$ composé avec l'identité sur $S_1'$ union disjointe $M_2$, c'est à dire qu'on décale $M_2$ pour avoir tous ses points critiques après ceux de $M_1$, comme représenté en Figure \ref{fig:multich}. De même on définit une multiplication à droite où on décale $M_1$ cette fois. 

\begin{figure}[h]
    \center
    \includegraphics[width=14cm]{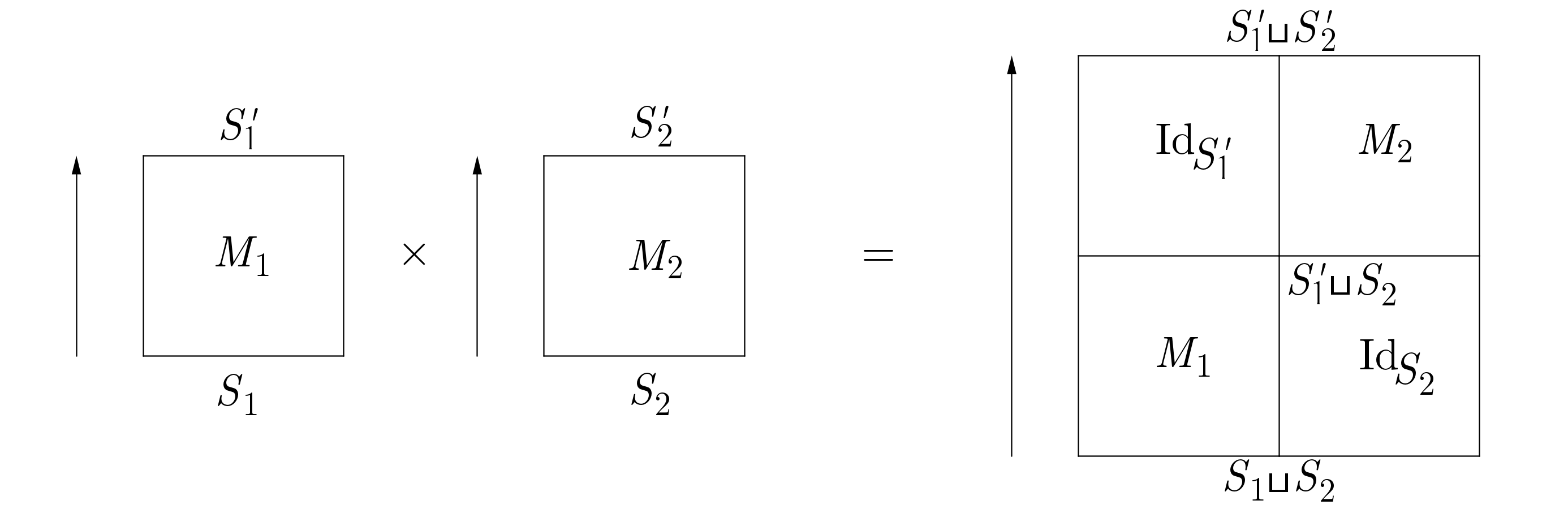}
    \caption{\label{fig:multich}Multiplication à gauche de deux cobordismes avec chronologies, les flèches indique le "sens" de la chronologie.}
\end{figure}

On note alors $ChCob$ la catégorie dont les objets sont des collections de cercles et les flèches des classes d'équivalence de cobordismes avec chronologies. Il n'est pas nécessaire pour la discussion d'entrer plus dans les détails de la construction des cobordismes avec chronologies et donc on utilise seulement la présentation en générateurs et relations de ces objets, calculée aussi par K. Putyra.

\begin{remarque}
La catégorie $ChCob$ n'est pas monoïdale puisqu'on a pas de multiplication au sens usuel. Dans \cite{Putyra}, K. Putyra propose une définition de catégorie chronologiquement monoïdale où la multiplication chronologique est un demi-foncteur et qui correspond au comportement de $ChCob$ pour la multiplication à gauche.
\end{remarque}

\begin{theoreme}\emph{(K. Putyra, \cite[Théorème 4.1]{Putyra})} \label{thm:preschcob}
La catégorie $ChCob$ des cobordismes avec chronologies est engendrée par les compositions et multiplications à gauche et à droite des huit générateurs suivants :
\begin{align*}
\smalldiag{Images_arxiv/Cob_birth.png} && \smalldiag{Images_arxiv/Cob_identity.png} && \smalldiag{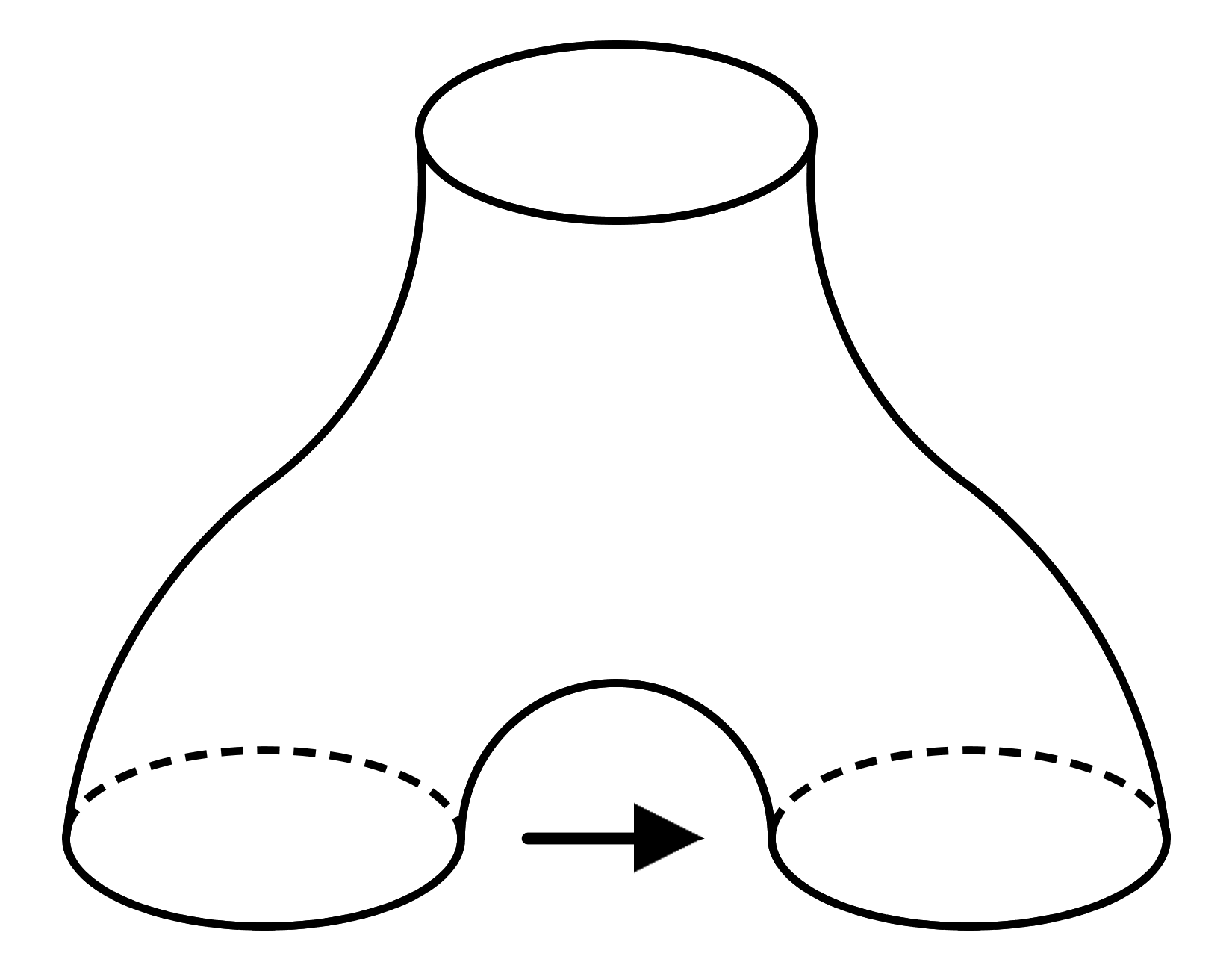} & & \smalldiag{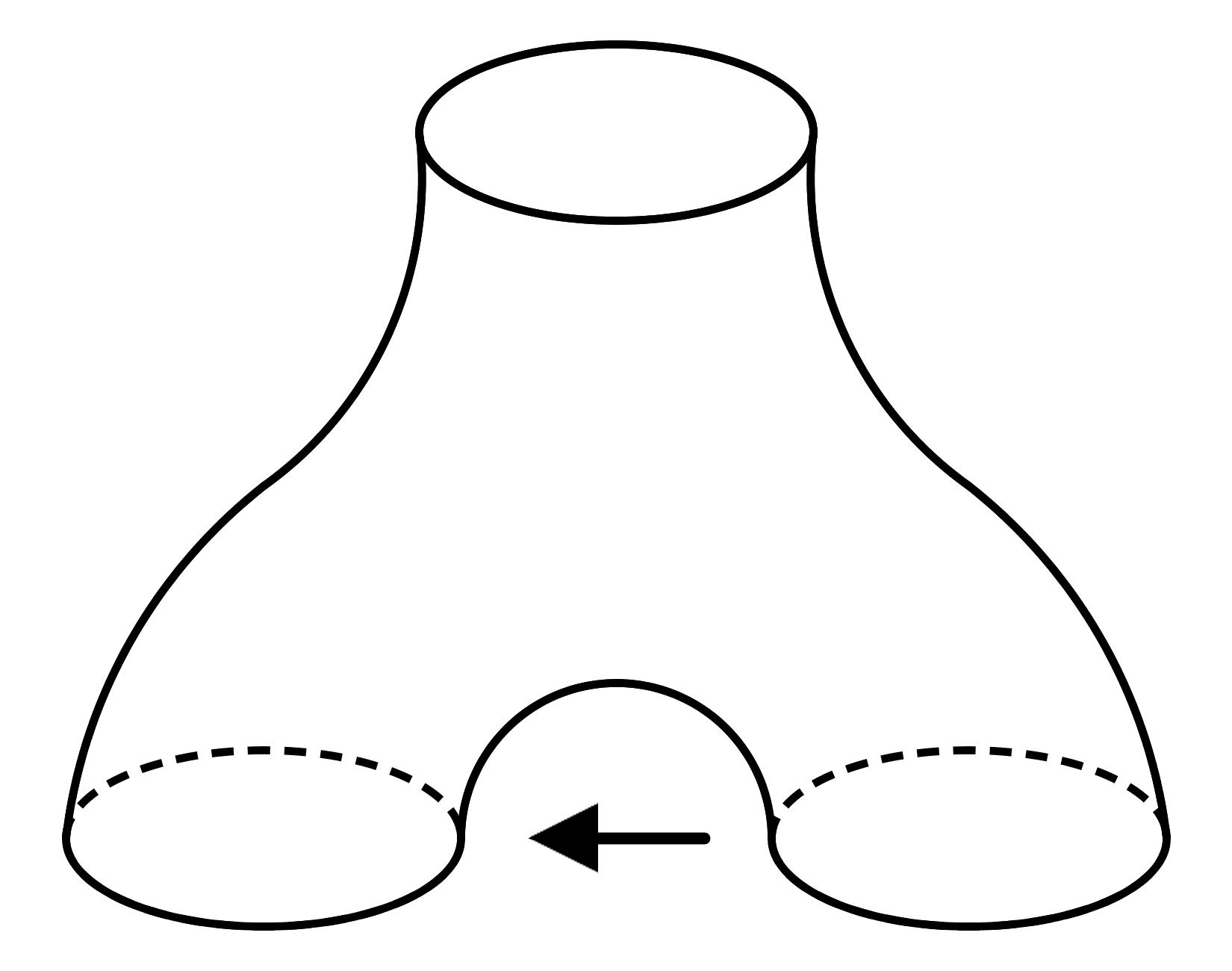}\\
\smalldiag{Images_arxiv/Cob_death.png} && \smalldiag{Images_arxiv/Cob_permutation.png}  &&  \smalldiag{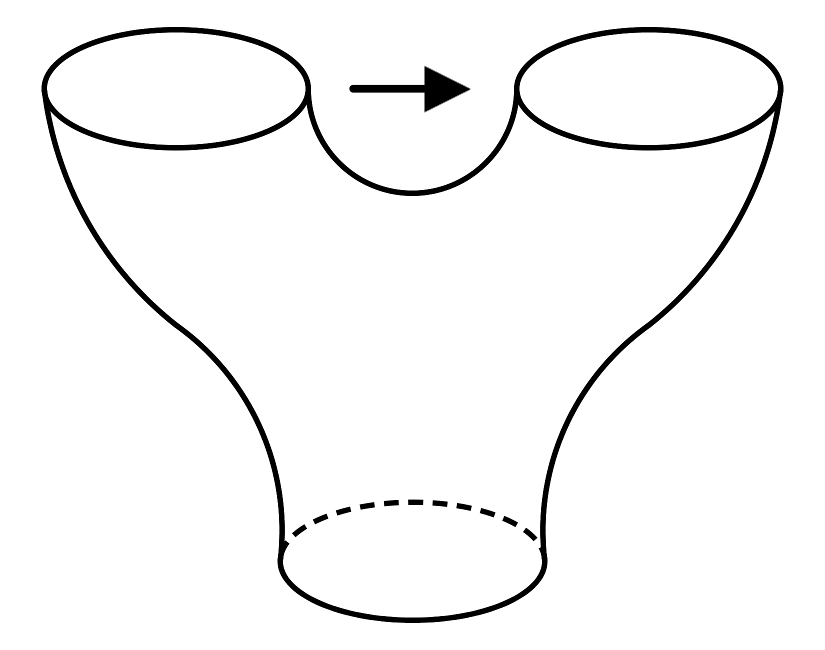}  & & \smalldiag{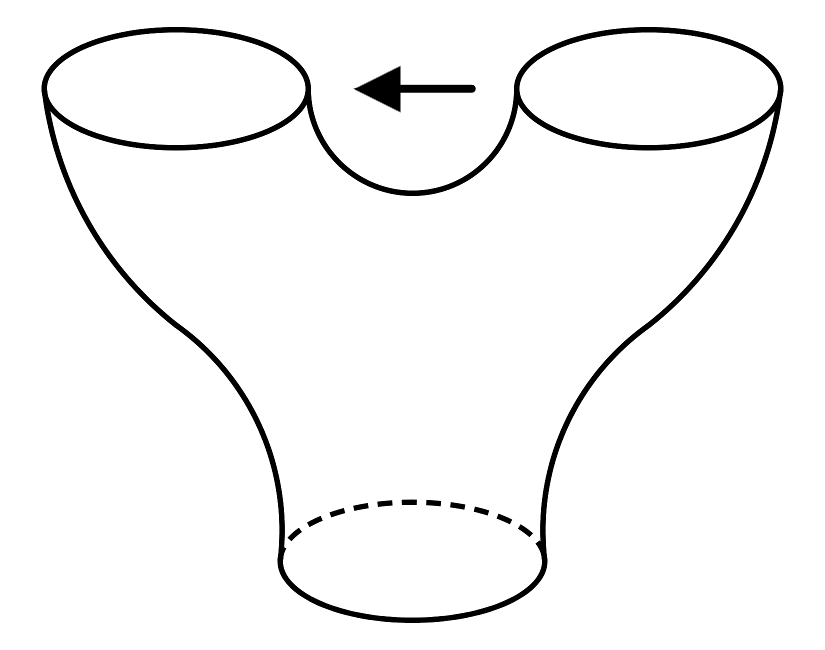} 
\end{align*}
quotientés par les relations suivantes :
\begin{enumerate}
\item Relations de permutations
\begin{align*}
\middiag{Images_arxiv/Cob_doublepermutation.png} &= \smalldiag{Images_arxiv/Cob_identity.png}\smalldiag{Images_arxiv/Cob_identity.png}
&&,& \bigdiag{Images_arxiv/Cob_triplepermutationl.png}  &= \bigdiag{Images_arxiv/Cob_triplepermutationr.png} 
\end{align*}
\item Permutations de l'unité et de la counité
\begin{align*}\middiag{Images_arxiv/Cob_unitypermutation.png} &= \smalldiag{Images_arxiv/Cob_identity.png} \smalldiag{Images_arxiv/Cob_birth.png}
&&,& \smalldiag{Images_arxiv/Cob_death.png} \smalldiag{Images_arxiv/Cob_identity.png} &= \middiag{Images_arxiv/Cob_counitypermutation.png}\end{align*}
\item Permutations de la fusion et de la scission
\begin{align*}
\middiag{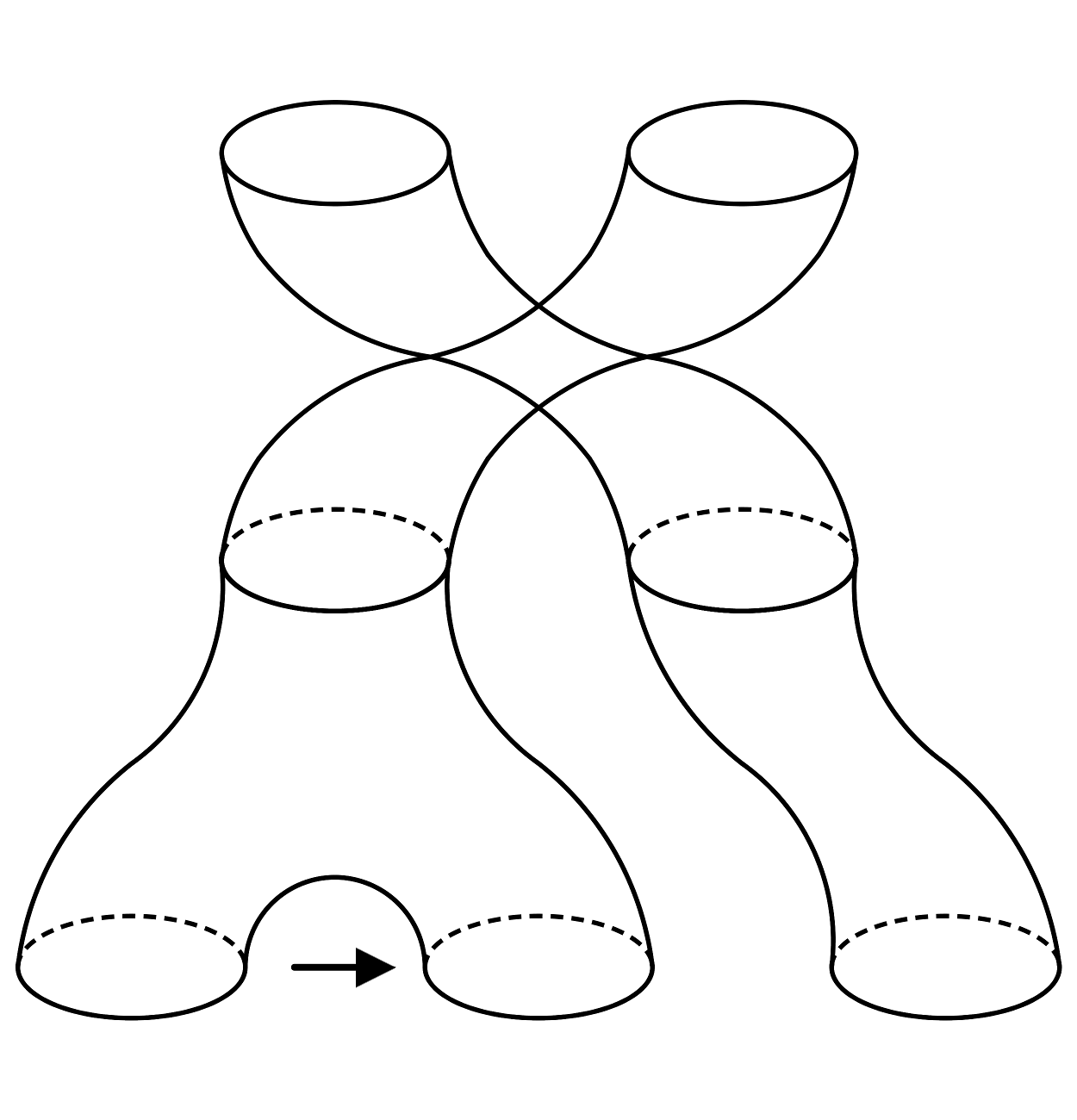} &= \bigdiag{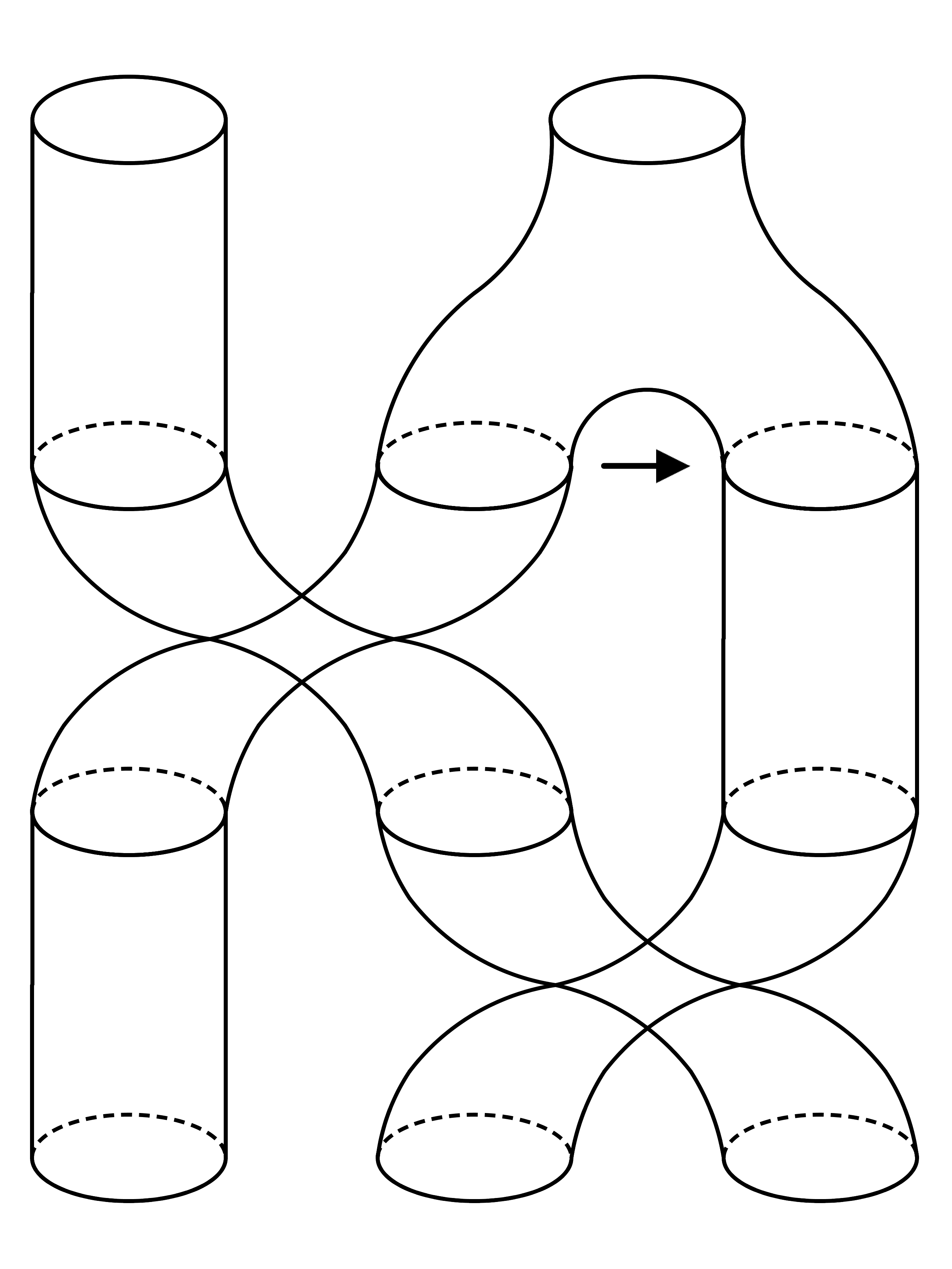}
&&,& \bigdiag{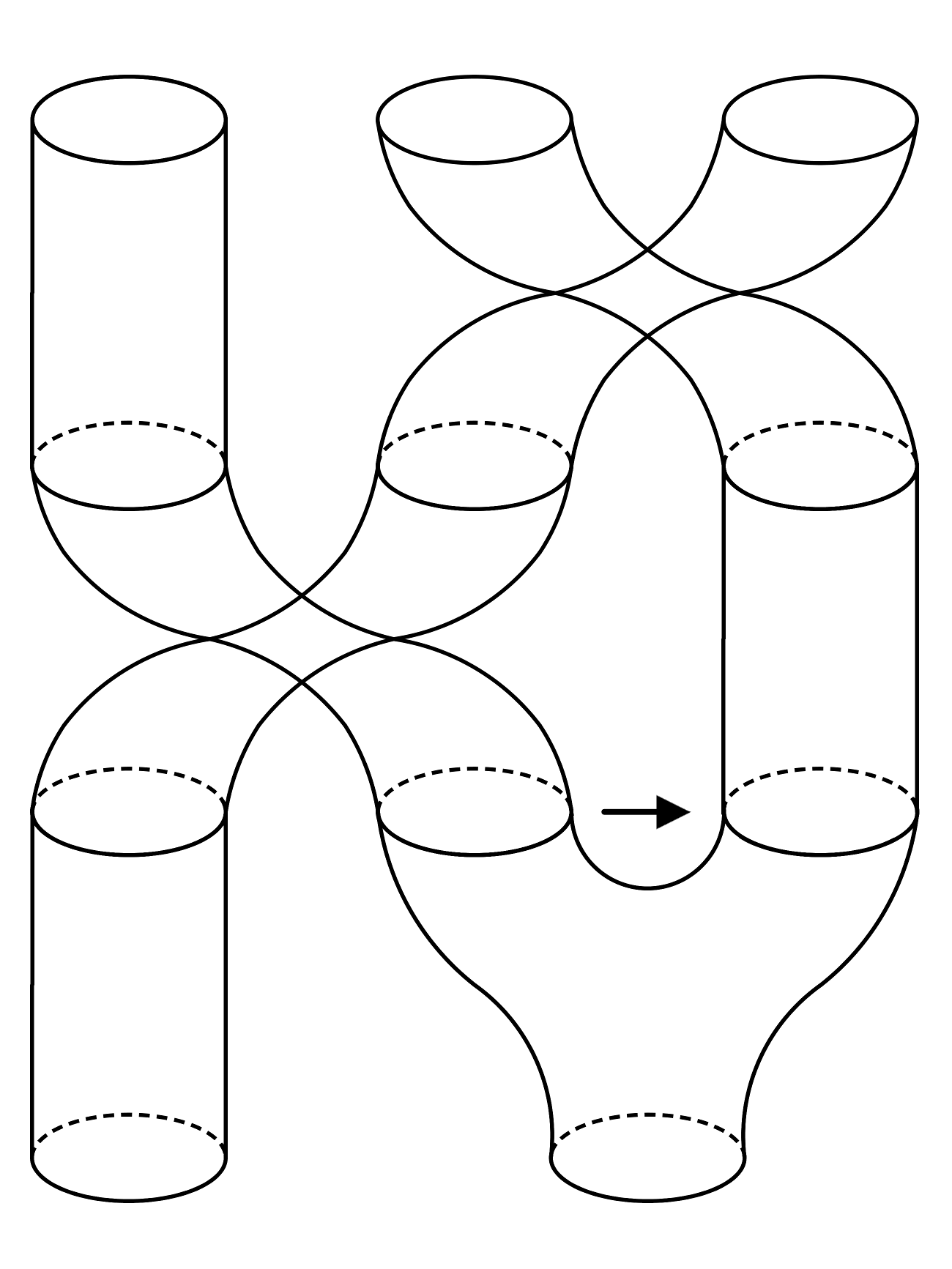}  &= \middiag{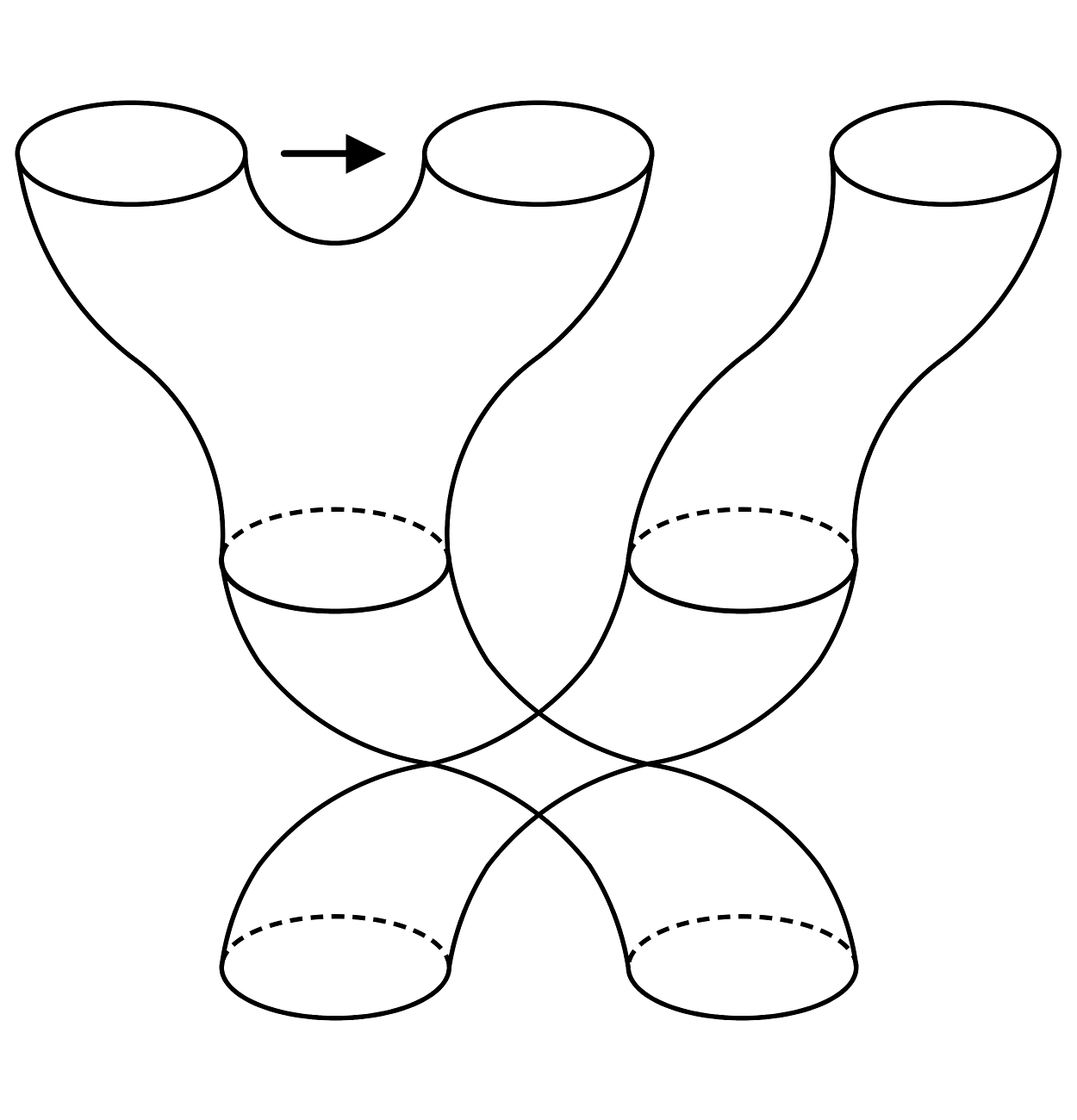} 
\end{align*}
\item Anti-commutativité et anti-co-commutativité
\begin{align*}\middiag{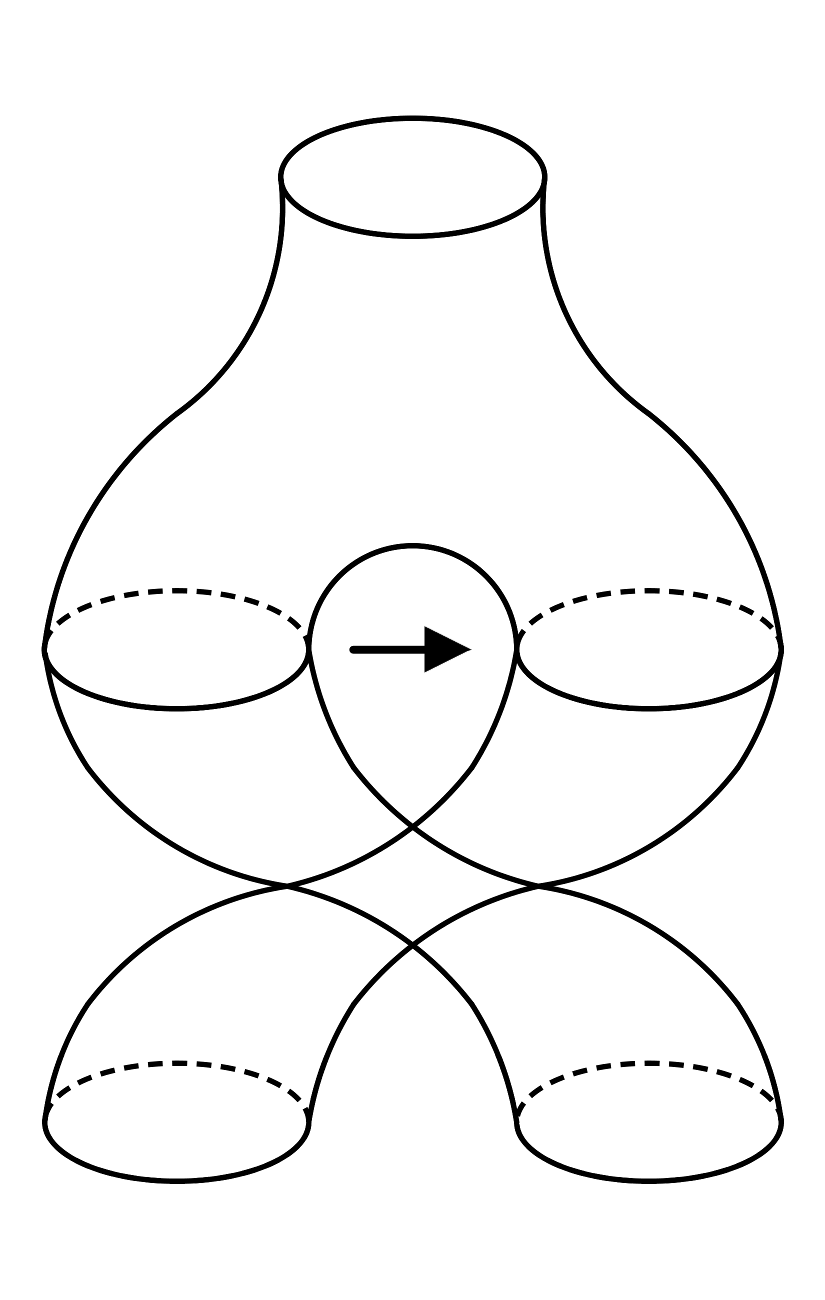} &= \smalldiag{Images_arxiv/Cob_merge_neg.png}
&&,& \smalldiag{Images_arxiv/Cob_split_neg.png} &= \middiag{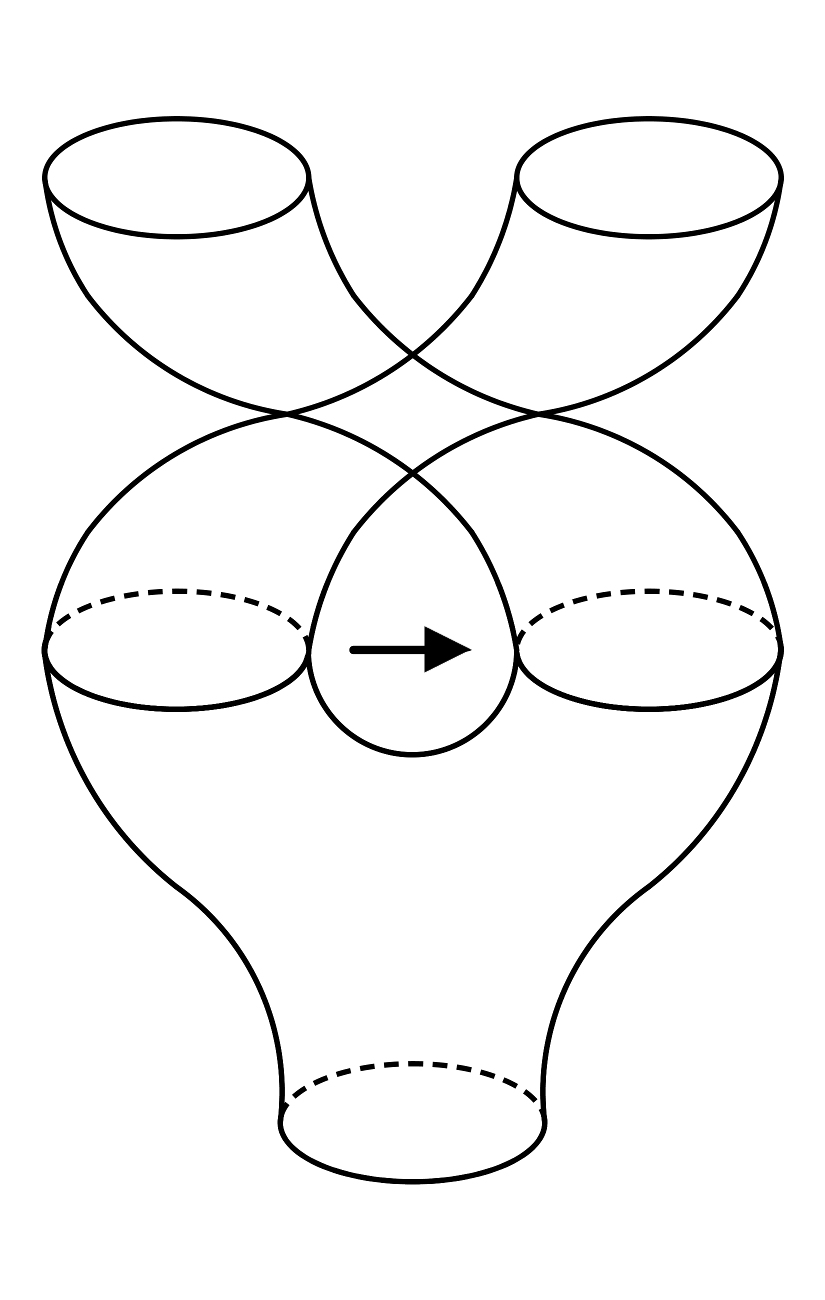}\end{align*}
\end{enumerate}
Les flèches sur les scissions et fusions représentent l'orientation des points critiques d'indice 1 (c'est-à-dire le "sens" de la base choisie pour la sous-matrice hessienne), les points d'indice 0 et 2 ayant une orientation induite par la surface.
\end{theoreme}

\begin{remarque}\label{rem:preschcob}
Par l'anti-commutativité et l'anti-co-commutativité, on peut retirer deux générateurs puisque si on a une fusion avec une orientation on peut obtenir celle avec l'autre et pareillement pour la scission.
\end{remarque}

\section{Le foncteur impair}

Le but de cette section est d'expliquer la construction du foncteur impair défini par P. Ozsvath, J. Rasmussen et Z. Szabo dans \cite{OddKhovanov}, mais dans le cadre des cobordismes chronologiques. On souhaite donc définir un foncteur $OF : ChCob \rightarrow \Z-Mod$ allant de la catégorie des cobordismes avec chronologies vers la catégorie des modules gradués sur $\Z$.\\

Tout d'abord, à un objet $S$ de $ChCob$, c'est-à-dire une collection de cercles, on associe $V(S)$ le groupe abélien libre gradué généré par les composantes de $S$ où on attribue à chaque générateur un degré $2$. On définit ensuite l'image de $S$ par $OF$ comme l'algèbre extérieure de modules (voir la Section \ref{sec:prodext} des Annexes) 
\begin{equation}\label{eq:OFobjects}
OF(S) :=  \Ext^*V(S)\{-|S|\}
\end{equation}
où $|S|$ est le nombre de composantes de cercles de $S$. Par le Théorème \ref{thm:preschcob} et la Remarque  \ref{rem:preschcob}, il suffit de définir le foncteur sur les cobordismes d'identité, de permutation, de naissance et de mort de cercle, ainsi que de fusion et scission avec une orientation donnée. Attention que la catégorie n'est pas monoïdale au sens stricte du terme et donc $OF$ non plus. Dès lors, on doit considérer ces cobordismes avec l'identité partout ailleurs sur les autres composantes de cercles. 

Bien évidemment l'identité est donnée par l'homomorphisme d'identité. Pour une permutation qui permute les deux cercles $a$ et $b$, on permute  les deux éléments dans l'algèbre extérieure,
$$OF\left(\smalldiag{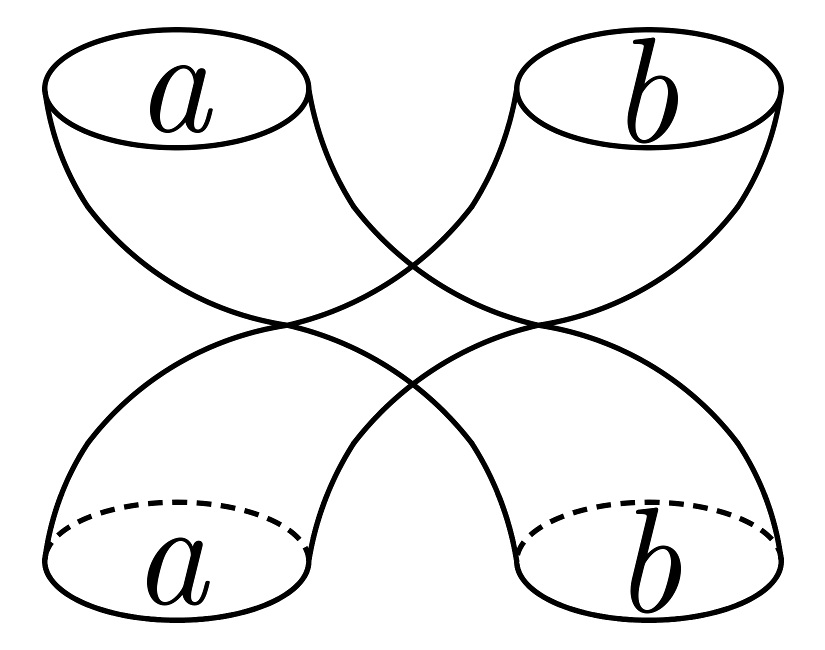}\right) : \begin{cases}
x \wedge a \wedge y &\mapsto x \wedge b \wedge y, \\
x \wedge b \wedge y &\mapsto x \wedge a \wedge y, \\
x \wedge a \wedge b \wedge y &\mapsto x \wedge b \wedge a \wedge y,
\end{cases}$$
avec $x$ et $y$ des produits extérieurs sans facteur $a$ ni $b$. On associe à la naissance de cercle $S \rightarrow S \sqcup \{a\}$ l'application d'inclusion d'algèbres extérieures
$$OF\left(\pitidiag{Images_arxiv/Birth.png}\right) : \Ext^* V(S) \rightarrow \Ext^* V(S \sqcup \{a\}) :  1 \mapsto 1$$
induite par l'inclusion de groupes $V(S) \hookrightarrow V(S) \oplus a\Z = V(S \sqcup \{a\})$. La mort $S \sqcup \{a\} \rightarrow S$ d'un cercle $a$  est donnée par la contraction avec le dual de $a$ (voir Définition \ref{def:contractdual})
$$OF\left(\pitidiag{Images_arxiv/Death.png}\right) : \Ext^* V(S \sqcup \{a\}) \rightarrow \Ext^* V(S) : x \mapsto a^*(x).$$
Pour une fusion $S_1 \rightarrow S_2$ qui joint les cercles $a_1\chemarrow a_2 \in S_1$, la flèche représentant l'orientation de la fusion, en un cercle $b \in S_2$, il y a une identification naturelle 
$$V(S_2) \simeq \frac{V(S_1)}{\{a_1 - a_2\}}$$
et on définit
$$OF\left(\smalldiag{Images_arxiv/Cob_merge_pos.png}\right) : \Ext^*V(S_1) \rightarrow \Ext^*V(S_2)$$
comme étant l'application induite par la projection 
$$\Ext^*V(S_1) \rightarrow \frac{ \Ext^* V(S_1)}{\langle a_1 - a_2 \rangle}.$$

\begin{remarque}\label{rem:assocetnonorient}
On voit aisément que l'image de la fusion ne dépend pas de l'orientation choisie, donnant 
$$OF\left(\smalldiag{Images_arxiv/Cob_merge_neg.png}\right) = OF\left(\smalldiag{Images_arxiv/Cob_merge_pos.png}\right).$$
De même, on peut observer que les fusions donnent les relations d'associativité et d'unité du Théorème \ref{thm:cobelem}. \`A partir de maintenant, on ne note donc plus l'orientation des fusions dans ce travail.
\end{remarque}
Finalement, on considère une scission $S_1 \rightarrow S_2$ qui scinde un cercle $a \in S_1$ en deux cercles $a_1$ et $a_2$ dans $S_2$ avec $a_1 \chemarrow a_2$, la flèche indiquant l'orientation de la scission. On a alors une identification naturelle
$$V(S_1) \simeq \frac{V(S_2)}{\{a_1-a_2\}}$$
qui induit un isomorphisme
$$ \Ext^* V(S_1) \simeq \Ext^*\left(\frac{V(S_2)}{\{a_1-a_2\}}\right) \simeq (a_1-a_2)\wedge \Ext^*V(S_2)$$
puisqu'on a l'égalité des équations suivantes :
\begin{align*}
(a_1 - a_2) \wedge x \wedge a_1 \wedge y &= -a_1 \wedge x \wedge a_2 \wedge y = (-1)^{|x|+1} a_1 \wedge a_2 \wedge x \wedge y\\
(a_1 - a_2) \wedge x \wedge a_2 \wedge y &= a_2 \wedge x \wedge a_1 \wedge y = (-1)^{|x|} a_2 \wedge a_1 \wedge x \wedge y.
\end{align*}
On définit alors
$$OF\left(\smalldiag{Images_arxiv/Cob_split_pos.png}\right) :  \Ext^*V(S_1) \rightarrow \Ext^*V(S_2)$$
 comme la composition
$$\Ext^*V(S_1) \overset{\simeq}{\longrightarrow} \Ext^*\left(\frac{V(S_2)}{(a_1-a_2)}\right) \overset{\simeq}{\longrightarrow}(a_1-a_2)\wedge\Ext^*V(S_2) \overset{\subset}{\longrightarrow} \Ext^*V(S_2).$$ 
On remarque donc que $1$ est envoyé sur $(a_1-a_2)$.

\begin{remarque}
On note que c'est ici que le choix d'orientation a son importance puisque l'isomorphisme
$$ \Ext^* V(S_1) \simeq  (a_1-a_2)\wedge \Ext^*V(S_2)$$
est de signe opposé à celui où on choisit $a_2 \chemarrow a_1$ qui donne
$$ \Ext^* V(S_1) \simeq  (a_2-a_1)\wedge \Ext^*V(S_2).$$
\end{remarque}

\begin{exemple}On donne d'abord deux exemples du foncteur appliqué sur des objets de $ChCob$, c'est-à-dire sur deux collections de cercles :
$$OF\left(\minidiag{Images_arxiv/Cercle.png}\right) = \Ext^*(a\Z)\{-1\} = \langle 1,a\rangle$$
 avec $\deg(1) = 0 - 1 = -1$ et $\deg(a) = 2 - 1 = 1$, ensuite
$$OF\left(\minidiag{Images_arxiv/Cercle.png}\minidiag{Images_arxiv/Cercle.png}\right)=\Ext^*(a\Z + b\Z)\{-2\} = \langle1,a,b,a\wedge b \rangle$$
 avec $\deg(1) = 0 - 2 = -2$, $\deg(a) = \deg(b) = 2-2 = 0$ et $\deg(a\wedge b) = 4-2 = 2$. 
\end{exemple}

 \begin{exemple}
On pose $S_1 = \left\{ b\minidiag{Images_arxiv/Cercle.png}, c\minidiag{Images_arxiv/Cercle.png}\right\}$ et $S_2 =  \left\{a\minidiag{Images_arxiv/Cercle.png}, b\minidiag{Images_arxiv/Cercle.png}, c\minidiag{Images_arxiv/Cercle.png}\right\}$ ainsi que le cobordisme $ M : S_1 \rightarrow S_2$ comme étant la naissance de $a$ et on calcule :
$$OF(M)(c + 3 b \wedge c) = c + 3b \wedge c.$$
\end{exemple}

 \begin{exemple}
On pose $S_1 = \left\{a\minidiag{Images_arxiv/Cercle.png}, b\minidiag{Images_arxiv/Cercle.png}, c\minidiag{Images_arxiv/Cercle.png}\right\}$ et $S_2 = \left\{ b\minidiag{Images_arxiv/Cercle.png}, c\minidiag{Images_arxiv/Cercle.png}\right\}$ ainsi que le cobordisme $ M : S_1 \rightarrow S_2$ comme étant la mort de $a$ et on calcule :
$$OF(M)(b \wedge c + b \wedge 2a \wedge c) = -2 (b \wedge c).$$
\end{exemple}

\begin{exemple}
On pose $S_1 = \left\{a_1\minidiag{Images_arxiv/Cercle.png}, a_2\minidiag{Images_arxiv/Cercle.png}, c\minidiag{Images_arxiv/Cercle.png}\right\}$ et $S_2 =  \left\{b\minidiag{Images_arxiv/Cercle.png}, c\minidiag{Images_arxiv/Cercle.png}\right\}$ ainsi que le cobordisme $M : S_1 \rightarrow S_2$ comme étant la fusion de $a_1$ et $a_2$ en $b$ et on calcule :
$$OF(M)(a_1 \wedge a_2 + a_2 \wedge c) = b \wedge b + b\wedge c = b\wedge c.$$
\end{exemple}

\begin{exemple}
On pose $S_1 = \left\{a\minidiag{Images_arxiv/Cercle.png}, c\minidiag{Images_arxiv/Cercle.png}\right\}$ et $S_2 =  \left\{a_1\minidiag{Images_arxiv/Cercle.png}, a_2\minidiag{Images_arxiv/Cercle.png}, c\minidiag{Images_arxiv/Cercle.png}\right\}$ ainsi que le cobordisme $M : S_1 \rightarrow S_2$ comme étant la scission de $a$ en $a_1 \chemarrow a_2$ et on calcule :
$$OF(M)(a\wedge c + c) = (a_1 - a_2) \wedge(a_1 \wedge c + c) = a_1 \wedge c - a_2 \wedge a_1 \wedge c - a_2 \wedge c.$$
\end{exemple}

Puisque $OF$ n'est pas bien définir sur $Cob$, il n'est pas une TQFT au sens strict du terme et il nous faut vérifier qu'il est bien défini sur $ChCob$.

\begin{proposition}
$OF$ est un foncteur de $ChCob$ vers la catégorie des modules sur $\Z$.
\end{proposition}

\emph{Démonstration.} Il suffit de vérifier que les morphismes de modules qu'on a défini respectent chacune des relations du Théorème \ref{thm:preschcob} pour que le foncteur soit bien défini puisqu'il respecte déjà la composition par définition. On donne les détails de deux des relations, toutes les autres étant des calculs similaires.
\begin{enumerate}
\item On considère la permutation de la counité
\begin{align*}
 \smalldiag{Images_arxiv/Cob_death.png} \smalldiag{Images_arxiv/Cob_identity.png} &= \middiag{Images_arxiv/Cob_counitypermutation.png}
\end{align*}
qui donne donc en termes de morphismes de modules
\begin{align*}
OF\left( \smalldiag{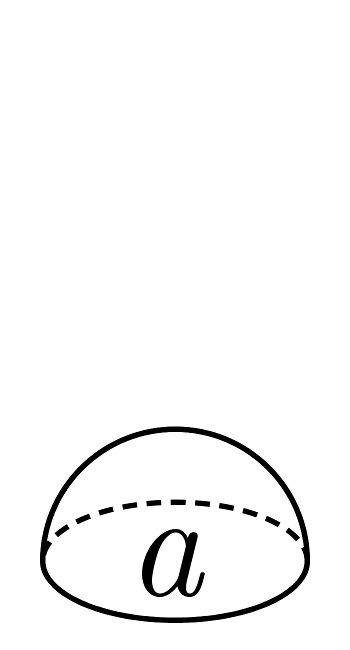} \smalldiag{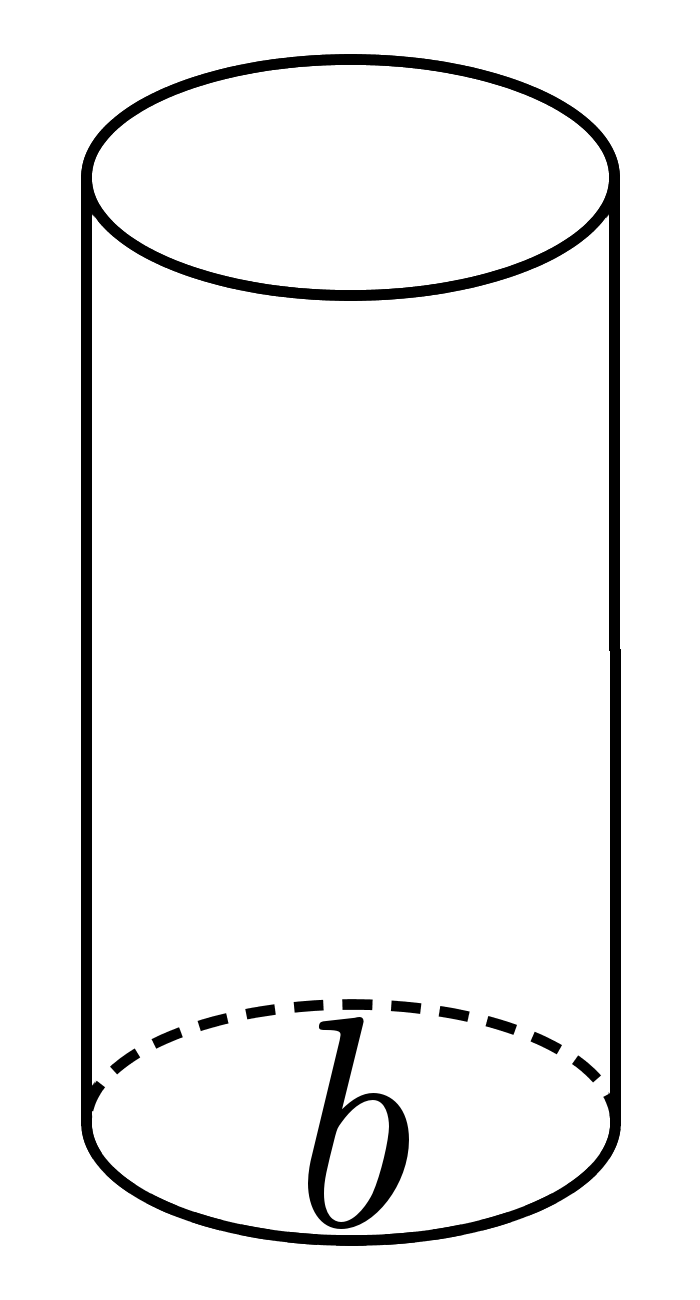} \right)& : \begin{cases}
a &\mapsto a^*(a) = 1, \\
b &\mapsto a^*(b) = 0, \\
a \wedge b &\mapsto a^*(a \wedge b) = b,
\end{cases}\\
OF\left(  \smalldiag{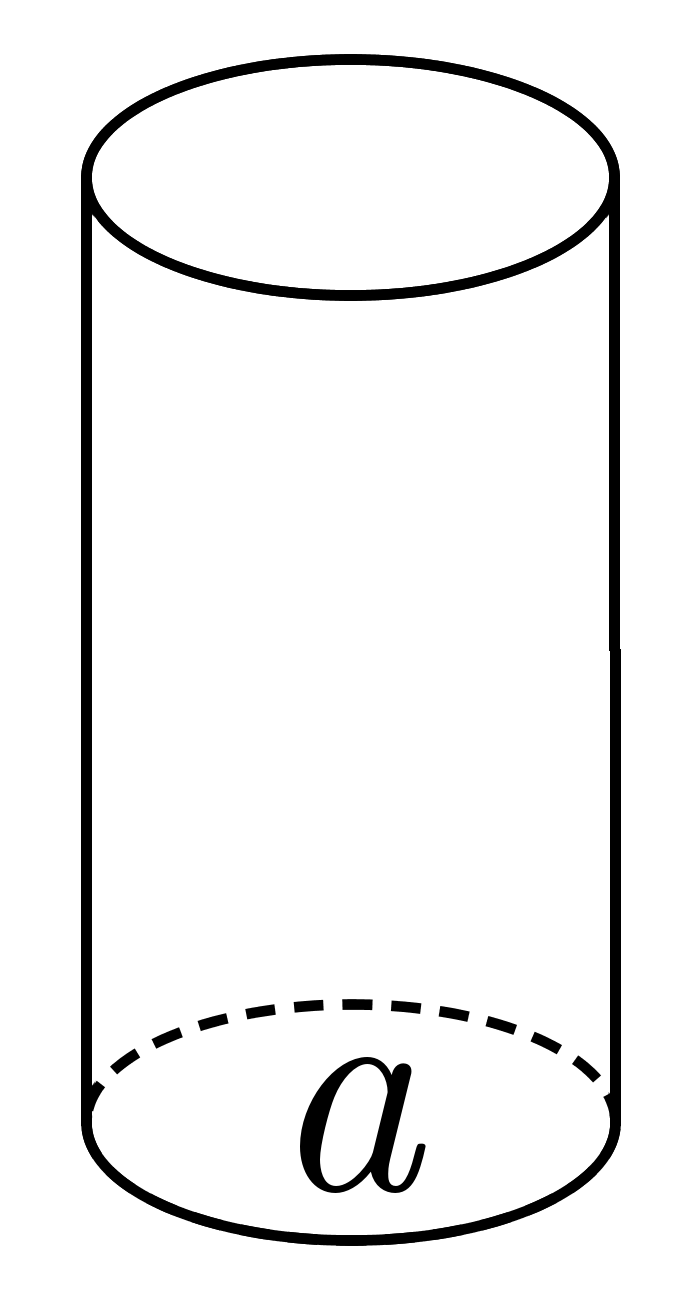} \smalldiag{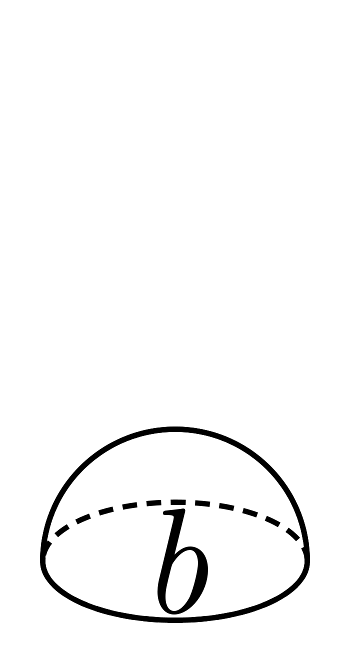} \right)\circ OF\left(\smalldiag{Images_arxiv/Cob_permutation_note.png} \right)& : \begin{cases}
a &\mapsto b \mapsto b^*(b) = 1,\\
b &\mapsto a \mapsto b^*(a) = 0, \\
a \wedge b &\mapsto (b \wedge a) \mapsto b^*(b \wedge a) = a,
\end{cases}
\end{align*}
et qui sont donc égaux car on identifie $a$ et $b$ comme éléments de sortie.

\item On considère l'anti-co-commutativité 
\begin{align*}\smalldiag{Images_arxiv/Cob_split_neg.png} &= \middiag{Images_arxiv/Cob_cocommutativity_ch.png}\end{align*}
qui donne les morphismes
\begin{align*}
OF\left(\smalldiag{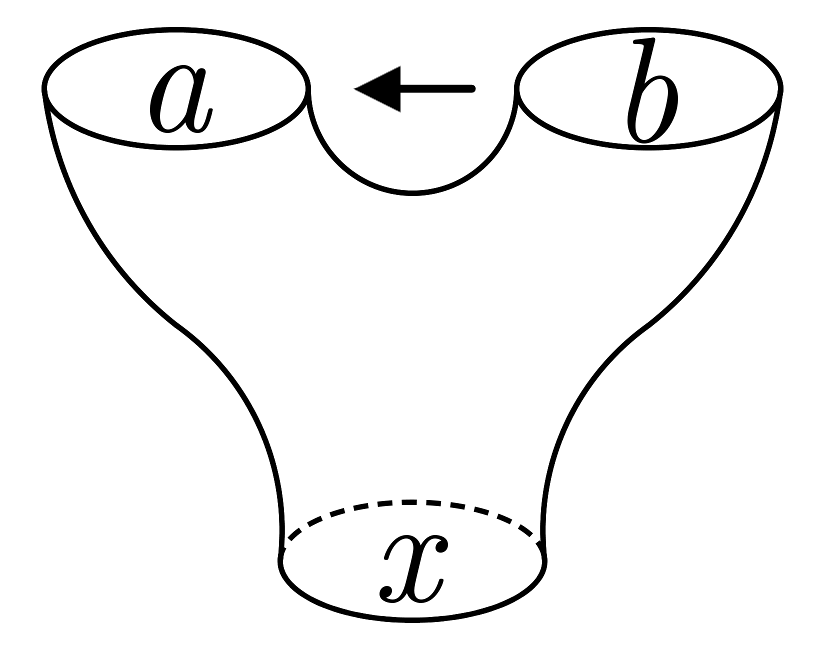}\right)& :  \begin{cases}
1 \mapsto (b - a), \\
x \mapsto (b - a) \wedge a,
\end{cases}\\
OF\left(\smalldiag{Images_arxiv/Cob_permutation_note.png}\right) \circ OF\left(\smalldiag{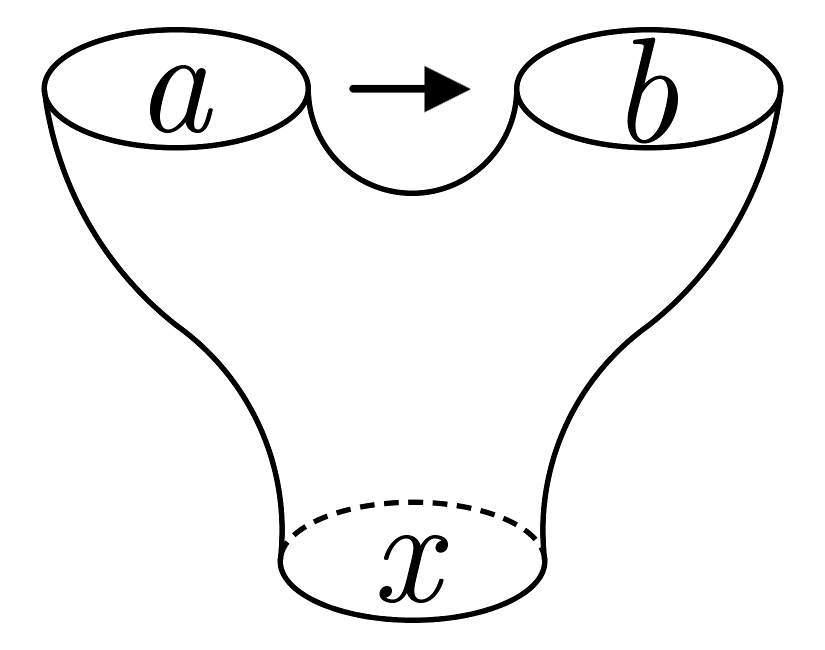}\right)& : 
\begin{cases}
 1  \mapsto (a - b) \mapsto (b - a),\\
 x  \mapsto (a - b)\wedge b \mapsto (b - a) \wedge a,
\end{cases}
\end{align*}
qui sont bien égaux. \qed
%
 \end{enumerate}

Tout comme pour le foncteur $F$, on retrouve la propriété suivante.
\begin{proposition}
Pour tout cobordisme chronologique $C$, on a
$$\deg OF(C) = -\chi(C). $$
\end{proposition}

\begin{proof}
La naissance de cercle envoie un produit extérieur vers le même, l'espace d'arrivée ayant juste une composante en plus et donc un degré décalé par $-1$ puisque dans (\ref{eq:OFobjects}) on définit l'image avec un degré décalé par le nombre de composantes.
La fusion ne change pas la longueur du produit extérieur mais diminue le nombre de composantes de $1$, donc on obtient un décalage de $1$. 
La scission augmente la longueur du produit extérieur de $1$, augmentant le degré de $2$, et augmente le nombre de composantes de $1$, donnant au final un décalage du degré par $1$.
Enfin, la mort de cercle diminue la longueur du produit extérieur de $2$ et le nombre de composantes de $1$, donnant un décalage du degré par $-1$.
%
%
%
Par ailleurs, on calcule aisément que
\begin{align*}
 \chi\left(\pitidiag{Images_arxiv/Birth.png}\right) & =  1, &&&
 \chi\left(\smalldiag{Images_arxiv/Cob_merge.png}\right) & =  -1, \\
 \chi\left(\pitidiag{Images_arxiv/Death.png}\right) &=   1, &&&
 \chi\left(\smalldiag{Images_arxiv/Cob_split.png}\right) &=  -1,
\end{align*}
puisqu'on a des surfaces avec $0$ trous et $1$ ou $3$ bord(s). Pour terminer, le résultat est clair pour la permutation et l'identité.
\end{proof}
Par ailleurs, on vérifie bien que si on regarde $OF$ sur les cobordismes sans chronologie, il n'est bien défini qu'à signe près. De plus, soit les deux homomorphismes sont les mêmes, soit ils sont de signes opposés, c'est-à-dire que le signe ne dépend pas des éléments sur lesquels on applique les homomorphismes.
\begin{proposition}\label{prop:isotopesignepres}
Pour toute paire de cobordismes (sans chronologie) $C_1,C_2 : S_1 \rightarrow S_2$ équivalents on a
$$OF(C_1) = \pm OF(C_2).$$
\end{proposition}

\begin{proof}
Il suffit d'observer le comportement du foncteur $OF$ sous les relations du Théorème \ref{thm:cobelem} en ajoutant toutes les possibilités d'orientations. On donne comme exemple le cas de la co-commutativité, les autres étant similaires. Si on oublie l'orientation et la chronologie, par le Théorème \ref{thm:cobelem}, on a l'équivalence des cobordismes suivants :
\begin{align*} \smalldiag{Images_arxiv/Cob_split_pos.png} &\sim \middiag{Images_arxiv/Cob_cocommutativity_ch.png}\end{align*}
On suppose qu'à gauche on scinde $a_1 \in S_1$ en $a_2 \chemarrow b_2 \in S_2$ et que à droite on scinde $a_1 \in S_1$ en $a_2\chemarrow b_2 \in S_3$ puis qu'on les permute. On obtient alors respectivement pour le cobordisme de gauche puis celui de droite
\begin{align*}
x &\rightarrow (a_2 - b_2) \wedge \bar x_2, \\
x & \rightarrow (a_2 - b_2) \wedge \bar x_2 \rightarrow (b_2 - a_2) \wedge \bar x_2,
\end{align*}
avec $\bar x_2$ qui est $x$ où on change tous les $a_1$ en $a_2$ et donc cela signifie que les morphismes sont de signes opposés.
\end{proof}
Par contre, on remarque que seules les scissions et morts de cercles peuvent engendrer un changement de signe.
\begin{proposition}\label{prop:isofusions}
Pour toute paire de cobordismes chronologiques  $C_1,C_2 : S_1 \rightarrow S_2$ équivalents et se décomposant en seulement des fusions, permutations et identités on a
$$OF(C_1) = OF(C_2).$$
\end{proposition}

\begin{proof}
Il suffit de faire les observations de la Remarque \ref{rem:assocetnonorient} et de voir qu'aucune des relations du Théorème \ref{thm:cobelem} ne peut engendrer de scission sauf celle de la counité mais qu'on peut oublier puisque ni $C_1$ ni $C_2$ ne possède de mort de cercle dans sa décomposition.
\end{proof}

On termine cette section par des exemples de calculs du foncteur appliqué à des cobordismes semblables à ceux des Exemples \ref{ex:cobS2} et \ref{ex:cobT2} ainsi qu'un autre qui montre la non-coassociativité venant d'un changement de chronologie, c'est-à-dire une équivalence de cobordismes qui change l'ordre des points critiques.

\begin{exemple}
On calcule $OF(S^2) : \Z \rightarrow \Z$ avec
 $$S^2 \simeq \smalldiag{Images_arxiv/S2.png}$$
qui est donc une naissance suivi d'une mort. On obtient alors
$$OF(S^2) : \Z \rightarrow \Ext^* \Z\{-1\} \rightarrow \Z : 1 \mapsto 1 \mapsto 0$$
et donc $OF(S^2) = 0 = F(S^2)$.
\end{exemple}

\begin{exemple}
On calcule $OF(T^2) : \Z \rightarrow \Z$ avec
 $$T^2 \simeq \bigdiag{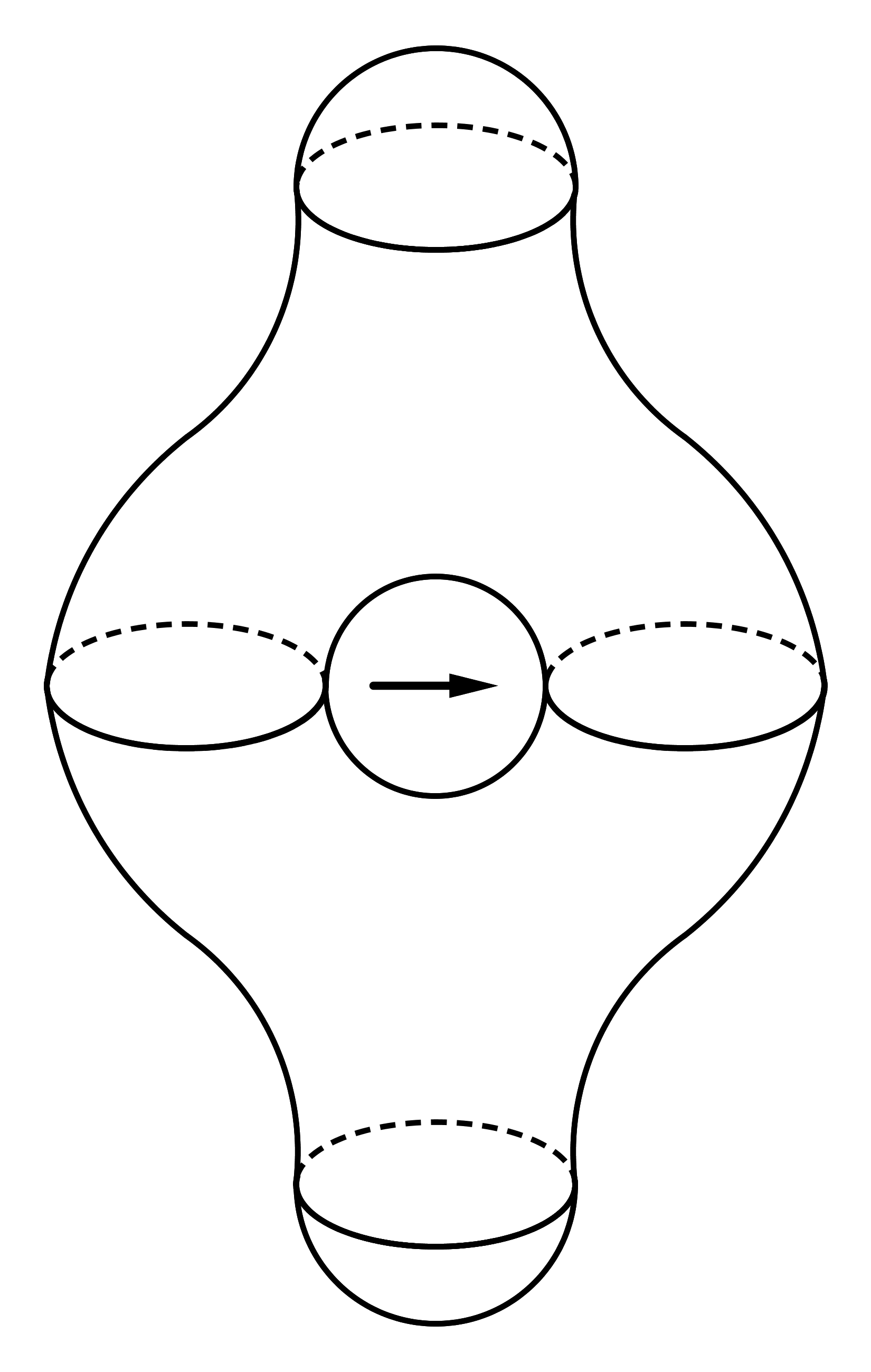}$$
qui est donc une naissance suivi d'une scission avec orientation arbitraire, d'une fusion et enfin d'une mort. On obtient alors

\begin{align*}
OF(T^2) : &\Z \rightarrow \Ext^* \bigl(a\Z\bigr)\{-1\} \rightarrow \Ext^* \bigl(a_1\Z \oplus a_2\Z\bigr) \{-2\} \rightarrow \Ext^* \bigl(b\Z\bigr)\{-1\} \rightarrow \Z :\\
&1 \mapsto 1 \mapsto (a_1 - a_2) \mapsto b - b = 0 \mapsto 0
\end{align*}
et donc $OF(T^2) = 0$, ce qui est différent du cas paire puisqu'on avait une multiplication par $2$. On obtient le même résultat si on choisit une autre orientation.
\end{exemple}

En utilisant la Proposition \ref{prop:isotopesignepres} et le théorème de classification des surfaces, on peut même montrer que tout cobordisme sans bord est envoyé par le foncteur $OF$ sur l'application nulle. 

L'exemple suivant montre la nécessité d'avoir une chronologie sur les points critiques puisque deux cobordismes équivalent dans $Cob$ peuvent donner deux morphismes différents par $OF$.

\begin{exemple}
On considère 3 collections de cercles : $\{a_1\}, \{b_1, b_2\}$ et $\{c_1, c_2, c_3\}$ et deux compositions de scissions : $M_1$ qui scindent $a_1$ en $b_1 \chemarrow b_2$ puis $b_1$ en $c_1 \chemarrow c_2$ et $M_2$ qui scinde $a_1$ en $b_1 \chemarrow b_2$ et $b_2$ en $c_2 \chemarrow c_3$, comme illustré en Figure \ref{fig:nonCommutative}.  Il est clair que ces compositions sont équivalentes et on calcule leurs images par $OF$ :
\begin{align*}
OF(M_1) &: a_1 \mapsto (b_1 - b_2) \wedge b_1 = b_1\wedge b_2 \mapsto (c_1 - c_2)\wedge(c_1 \wedge c_3) = c_1\wedge c_2\wedge c_3, \\
OF(M_2) &: a_1 \mapsto (b_1 - b_2) \wedge b_1 = b_1\wedge b_2 \mapsto (c_2 - c_3)\wedge(c_1 \wedge c_2) = -c_1\wedge c_2\wedge c_3.
\end{align*}

\begin{figure}[h]
    \center
    \includegraphics[width=10cm]{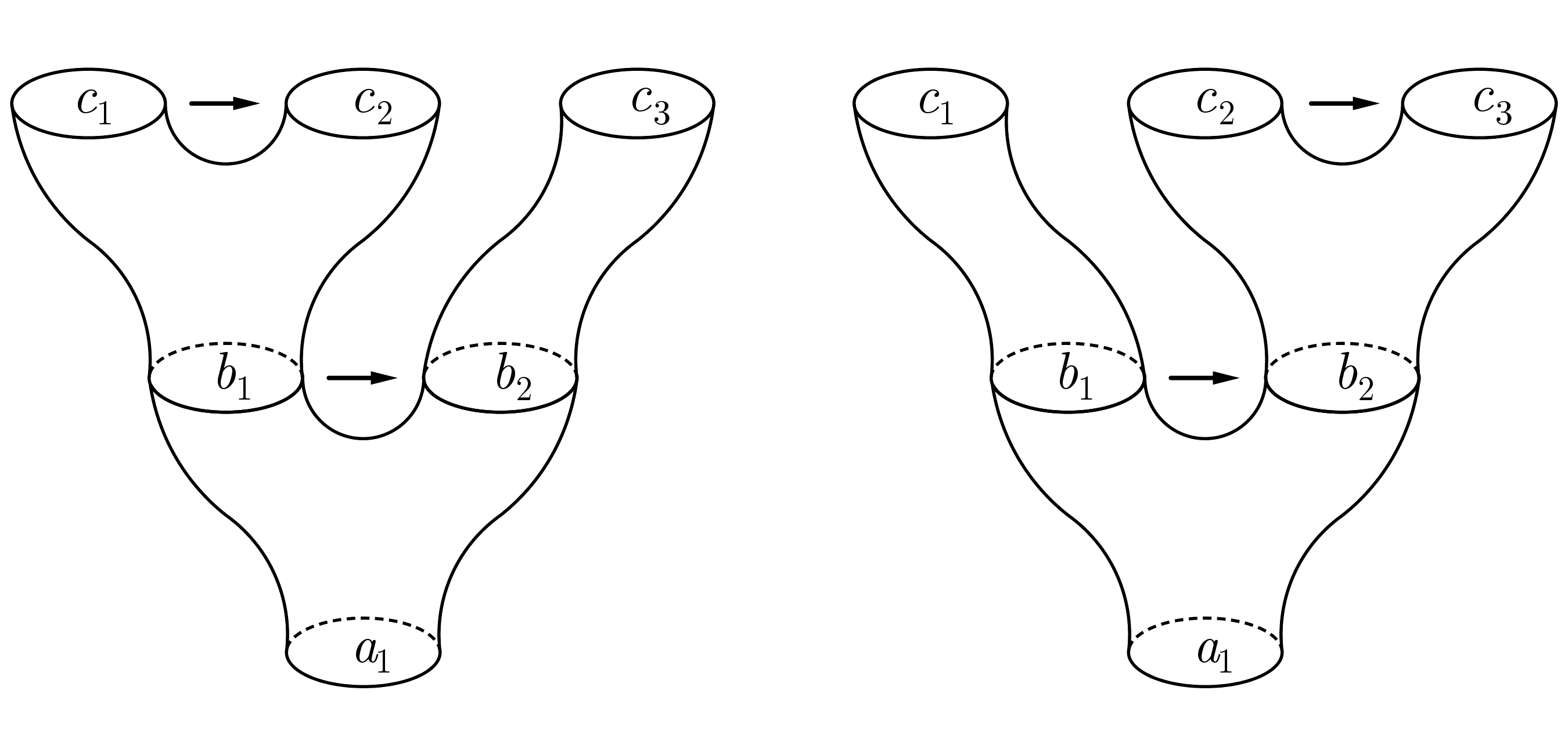}
    \caption{Si on oublie la chronologie, ces deux cobordismes sont équivalents mais n'ont pas la même image par $OF$.}
    \label{fig:nonCommutative}
\end{figure}
\end{exemple}

\section{Construction impaire des anneaux des arcs}

L'objectif de cette section est de définir une famille d'anneaux unitaires par une construction similaire à celle de l'anneau $H^n$ de Khovanov mais en utilisant le foncteur $OF$ de la section précédente. Cela donne alors une construction impaire, ou "oddification" en anglais, de $H^n$.\\

Pour $n\ge 0$, on définit le groupe impair des arcs de Khovanov d'indice $n$ comme le groupe abélien gradué donné par la somme directe
\begin{align}
\Hno &:= \bigoplus_{a,b\in B^n} b(\Hno)a,  & &\text{où}& b(\Hno)a &:= OF\bigl(W(b)a\bigr)\{n\}. \label{eq:degohn}
\end{align}

%

\begin{remarque}\label{rem:hnopos}
Comme dans $H^n$, le fait de définir $b(\Hno)a$ avec un décalage du degré par $n$ rend $\Hno$ positivement gradué. En effet, le degré minimal d'un élément de $OF(W(b)a)$ est obtenu en regardant $1$ dedans qui a un degré de au minimum $-n$ puisqu'on a un maximum de $n$ composantes.
\end{remarque}
%

Pour faire de $\Hno$ un anneau, il faut définir une multiplication dessus. On veut la définir de façon similaire à celle de $H^n$ mais en utilisant des cobordismes avec chronologies. Comme pour $H^n$, on observe que $bW(b)$ donne tous des demi cercles, chacun possédant sa symétrie horizontale en face. On peut donc à nouveau construire des ponts qui s'emboitent, donnant des cobordismes qui envoient à chaque fois une paire de demi cercle vers l'identité, mais cette fois on doit se fixer un ordre dans lequel on construit ces ponts pour avoir une chronologie et on doit orienter les scissions. Il n'y a a priori aucune raison de choisir un ordre plutôt qu'un autre ou une orientation particulière et donc on propose de tous les considérer. 

\begin{definition}
On définit une \emph{règle de multiplication} $C$ pour $OH^n$ comme les données de, pour chaque triplet  $(c,b,a) \in (B^n)^3$ :
\begin{itemize}
\item un ordre $x_1 < \dots < x_{2n}$ sur les points de base $\{1, \dots, 2n\}$, donc sur les extrémités des arcs de $b$, tel que si $x_i$ est relié à $x_j$ dans $b$ et que $x_k \in ]x_i, x_j[$ pour l'ordre usuel alors $x_i < x_k$ ou $x_j < x_k$ dans l'ordre de la règle,
\item une orientation $x_i \chemarrow x_j$ ou $x_j \chemarrow x_i$ pour tout $i,j \in [1, 2n]$ tels que $x_i$ est relié à $x_j$ par un arc de $b$.
\end{itemize}
\end{definition}
Autrement dit, on donne un ordre (qui dépend de $a$ et de $c$) sur les arcs de $b$ ainsi qu'une orientation pour chacun de ceux-ci et on ne permet pas qu'un arc $b_1$ arrive avant un autre $b_2$ dans l'ordre si $b_1$ est imbriqué dans $b_2$, c'est-à-dire si les points d'extrémités de $b_1$ sont entre les points d'extrémités de $b_2$. Cette condition est imposée afin qu'on ne construise pas un pont qui traverserait le restant de la surface si on la plongeait dans $\R^3$. On demande cela puisque M. Khovanov construit ses cobordismes comme des surfaces plongées dans l'espace pour définir $H^n$, même si cela n'a pas d'influence dans le cas pair.

On construit alors une famille de cobordismes chronologiques pour une règle de multiplication $C$
$$M(C) :=\{ C_{cba} : W(c)bW(b)a \rightarrow W(c)a | a,b,c \in B^n\} $$
par la procédure suivante, illustrée en Figure \ref{fig:multiex}, appliquée pour tout $a,b,c \in B^n$ :
\begin{enumerate}
\item On pose $i := 1$ et $D_0 := W(c)bW(b)a$. Quitte à redimensionner, on peut supposer que les points de base de $W(c)b$ sont alignés sur $(1,1), \dots, (2n,1)$ et ceux de $W(b)a$ sur $(1,0), \dots, (2n,0)$. On considère $x_1 < \dots < x_{2n}$ l'ordre sur les points pour $(c,b,a)$ de la règle de multiplication $C$.
\item Si $(x_i,0)$ est relié par un segment de droite $\{x_i\} \times [0,1]$ à $(x_i,1)$ dans $D_{i-1}$, alors on ne fait rien et on pose $D_{i} := D_{i-1}$. Sinon on considère $(x_j,0)$ l'autre extrémité de l'arc de $b$ passant par $(x_i,0)$. On construit $D_{i}$ comme étant $D_{i-1}$ où on supprime l'arc de $b$ passant par $(x_i,1)$ et l'arc de $W(b)$ passant par $(x_i,0)$ et on relie $(x_i,0)$ à $(x_i,1)$ et $(x_j,0)$ à $(x_j,1)$ par des segments de droites verticales $\{x_i\}\times [0,1]$ et $\{x_j\}\times [0,1]$. On construit ensuite un cobordisme de $D_{i-1}$ à $D_{i}$ comme étant l'identité partout sauf  pour un pont envoyant les arcs opposés de $W(b)$ et $b$ passant par le $i$-ème point vers deux segments de droites verticales et l'identité partout ailleurs sur $D_i$. Si ce pont engendre une scission, en notant $X_i$ et $X_j$ les composantes de $D_{i}$ passant respectivement par $(x_i,0)$ et $(x_j,0)$, on l'oriente $X_i \chemarrow X_j$ si $x_i \chemarrow x_j$ et $X_j \chemarrow X_i$ sinon.
\item On pose $i := i+1$. Si $ i > 2n$, on passe à la prochaine étape, sinon on revient à l'étape 2.
\item On a pose $C_{cba} :=D_0 \rightarrow D_1 \rightarrow \dots \rightarrow D_{2n}$ comme étant la composition des cobordismes $D_i$ et puisque $D_{2n} \simeq W(c)a$ cela construit un cobordisme avec chronologie $C_{cba} : W(c)bW(b)a \rightarrow W(c)a$.
\end{enumerate}

\begin{figure}[h]
    \center
    \includegraphics[width=15cm]{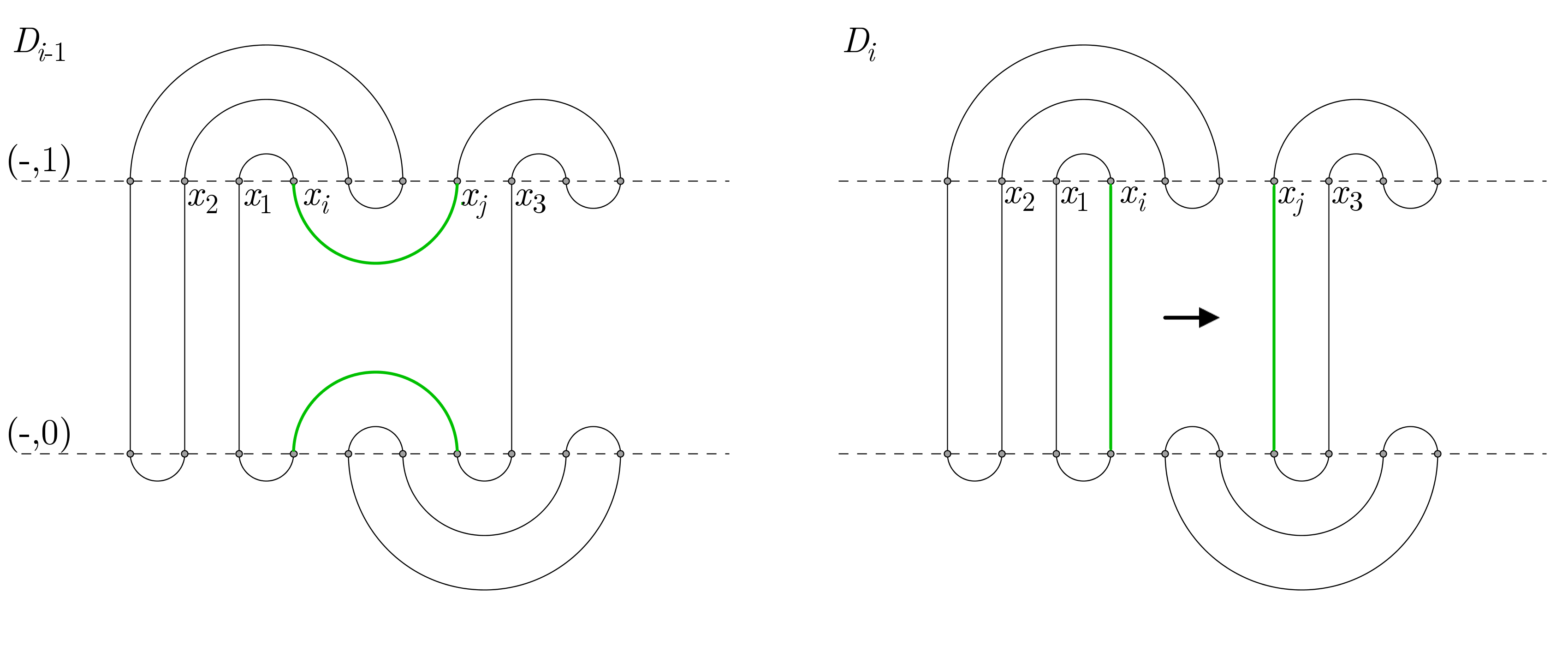}
    \caption{\label{fig:multiex} Exemple d'étape de la procédure de construction de $C_{cba}$ avec $i = 4$ et $x_i \chemarrow x_j$. On ne note que les $x_k$ pour $k \le i$ à l'exception de $x_j$.}
\end{figure}

En d'autres termes, on parcourt les arcs de $b$ dans l'ordre donné par la règle de multiplication en construisant le pont associé et s'il sépare un cercle en deux on lui donne l'orientation donnée par l'orientation de l'arc. De plus, on note que $C_{cba}$ possède $n$ points de selles et on a donc 
$$\chi(C_{cba}) = -n.$$

\begin{remarque}
Puisque la fusion ne dépend pas de l'orientation dans le cadre de $OF$, on ne s'occupe pas de l'orienter. Cependant, on pourrait le faire de sorte à être totalement rigoureux en orientant la composante passant par $(x_i,0)$ vers celle passant par $(x_i,1)$ ou alors tout simplement en prenant un choix arbitraire d'orientations pour les fusions.
\end{remarque}

A partir de cette collection de cobordismes $M(C) = \{C_{cba} : W(c)bW(b)a \rightarrow W(c)a\}$ on définit une multiplication similaire à celle de $H^n$ et on note $OH^n_C$ l'anneau obtenu en munissant $OH^n$ de cette multiplication. Si $b\ne c$, on pose $d(OH^n_C)c \times b(OH^n_C)a = 0$, sinon on remarque que $W(c)b$ et $W(b)a$ sont tous deux des unions disjointes de cercles et donc on a un isomorphisme canonique
$$V\bigl(W(c)bW(b)a\bigr) \simeq V\bigl(W(c)b\bigr) \oplus V\bigl(W(b)a\bigr)$$
avec $V(S)$ le groupe abélien libre engendré par les éléments de l'ensemble $S$. Cela induit des inclusions
\begin{align*}
\Ext^*  V\bigl(W(c)b\bigr) &\subset \Ext^* V\bigl(W(c)bW(b)a\bigr), & \Ext^*  V\bigl(W(b)a\bigr) &\subset \Ext^* V\bigl(W(c)bW(b)a\bigr),
\end{align*}
 et donc en prenant le produit extérieur on obtient un homomorphisme
\begin{equation}
OF\bigl(W(c)b\bigr) \times OF\bigl(W(b)a\bigr) \rightarrow OF\bigl(W(c)b\bigr) \wedge OF\bigl(W(b)a\bigr) \subset OF\bigl(W(c)bW(b)a\bigr). \label{eq:homohn}
\end{equation}
On obtient alors en composant (\ref{eq:homohn}) avec l'image du cobordisme $C_{cba}$ par $OF$ un homomorphisme
$$OF\bigl(W(c)b\bigr) \times OF\bigl(W(b)a\bigr) \xrightarrow{(\ref{eq:homohn})} OF\bigl(W(c)bW(b)a\bigr) \xrightarrow{OF\bigl(C_{cba}\bigr)} OF\bigl(W(c)a\bigr).$$ 
Puisque le cobordisme a une caractéristique d'Euler valant $-n$, le morphisme obtenu a un degré $n$ et donc on obtient un homomorphisme
\begin{equation} OF\bigl(W(c)b\bigr)\{n\} \times OF\bigl(W(b)a\bigr)\{n\} \rightarrow OF\bigl(W(c)a\bigr)\{n\} \label{eq:homohn2}\end{equation}
qui préserve le degré. On définit alors la multiplication en utilisant le diagramme commutatif suivant, où les isomorphismes viennent de (\ref{eq:degohn}) :
$$\xymatrix{
c(OH^n)b \times b(OH^n)a \ar[r] \ar[d]_{\simeq} & c(OH^n)a \\
F\bigl(W(c)b\bigr)\{n\} \times F\bigl(W(b)a\bigr)\{n\} \ar[r]_-{(\ref{eq:homohn2})}  & F\bigl(W(c)a\bigr)\{n\} \ar[u]_{\simeq}.
}$$

On note pour $a\in B^n$, 
$$1_a = 1\{n\} \in \Ext^*V(W(a)a)\{n\}$$
et on vérifie par la proposition suivante que pour $x \in b(\Hno_C)a$ on a $x.1_a = x$ et $1_a.x = 0$ si $a\ne b$ et pour $y \in a(\Hno_C)b$ on a $y.1_a = 0$ et $1_a.y = y$. On définit alors l'unité de $OH^n_C$ comme la somme
$$1 = \sum_{a\in B^n} 1_a.$$

\begin{proposition}\label{prop:aaabfusions}
Les multiplications
\begin{align*}
a(OH^n_C)a \times a(OH^n_C)b &\rightarrow a(OH^n_C)b, \\
a(OH^n_C)b \times b(OH^n_C)b &\rightarrow a(OH^n_C)b,
\end{align*}
sont calculées en utilisant uniquement des fusions de sorte que le produit est obtenu en prenant juste le produit extérieur après avoir renommé les éléments suivant les composantes de $W(a)b$. Autrement dit $C_{aab}$ et $C_{abb}$ sont des cobordismes qui se décomposent en $n$ fusions.
\end{proposition}

\begin{proof}
On sait que la multiplication
$$a(OH^n_C)a \times a(OH^n_C)b \rightarrow a(OH^n_C)b$$
est obtenue en prenant un cobordisme de $W(a)aW(a)b$ vers $W(a)b$ composé de $n$ ponts, donc $n$ fusions ou scissions.  Par ailleurs, $W(a)aW(a)b$ est constitué de $n+|W(a)b|$ composantes de cercle et $W(a)b$ en est constitué de $|W(a)b|$, donc on doit fusionner au moins $n$ composantes. On en conclut qu'on a exactement $n$ fusions.
\end{proof}

On propose une règle de multiplication $\widetilde C$ parmi d'autres pour définir une collection de cobordismes qui nous sert à illustrer ce travail par des exemples. Pour tout $a,b,c \in B^n$, on ordonne $\{1,\dots, 2n\}$ selon l'ordre usuel des entiers naturels et on oriente $i \chemarrow j$ si $i < j$ pour l'ordre usuel aussi. Un exemple de construction de cobordisme par cette règle de multiplication est illustré en Figure \ref{fig:ordreponts}. Afin d'alléger la notation, on note $\widetilde {OH}^n := OH^n_{\widetilde C}$.

\begin{figure}[h]
    \center
    \includegraphics[width=15cm]{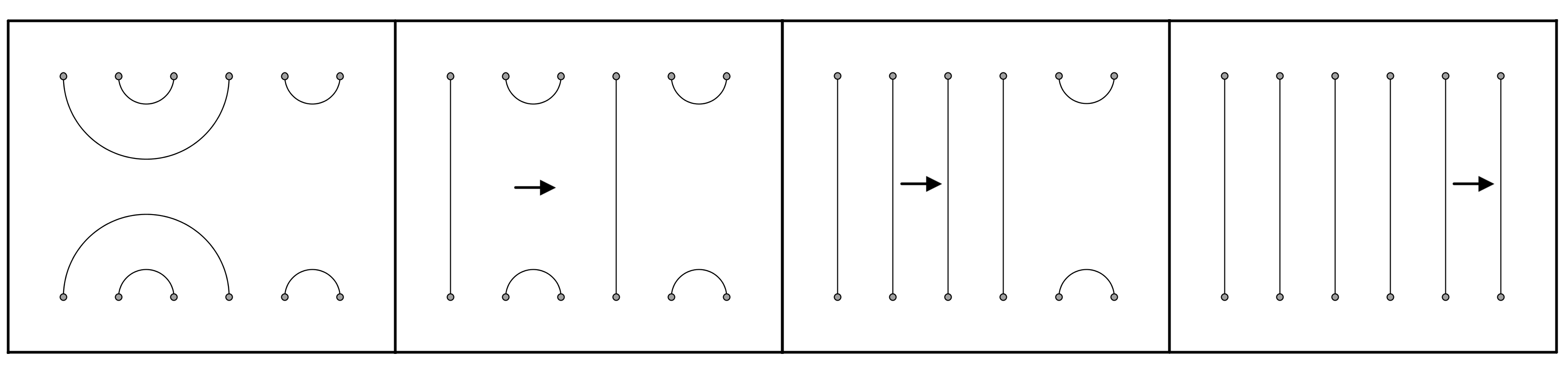}
    \caption{\label{fig:ordreponts} Exemple d'ordre de construction des ponts d'un corbordisme avec les flèches représentant l'orientation d'éventuelles scissions.}
\end{figure}

\begin{exemple}\label{ex:OH1}
On peut montrer que $OH^1_C \simeq \Ext^* \Z$ puisque $B^1$ ne contient qu'un seul élément $a$ donné par un demi-cercle qui relie les deux points. On a alors un isomorphisme de groupes gradués $OH_C^1 \simeq_{ab} OF(W(a)a)\{1\}$ et, $W(a)a$ n'étant composé que d'un seul cercle, on a $OF(W(a)a) =  \Ext^* \Z \{-1\}$. Par ailleurs, la multiplication 
$$a(OH^1_C)a \times a(OH^1_C)a \rightarrow a(OH^1_C)a$$
est composée uniquement d'une fusion et est donc équivalente au produit extérieur, montrant le résultat voulu.
\end{exemple}

\begin{exemple}\label{ex:OH2_1}
On considère $\widetilde {OH}^2$ et on pose
\begin{align*}
a &=\deuxdiag{Images_arxiv/B2_2.png},& 
b &= \deuxdiag{Images_arxiv/B2_1.png}.
\end{align*}
donnant donc
\begin{align*}
W(a)a &= \diagg{Images_arxiv/B2_bb_note.png},&
W(a)b &= \diagg{Images_arxiv/B2_ba_note.png},&
W(b)a &= \diagg{Images_arxiv/B2_ab_note.png}.
\end{align*}
On note $a_1$ et $a_2$ les cercles de $W(a)a$ avec $a_1$ le cercle extérieur et $a_2$ le cercle intérieur. De même on note $b_1$ le cercle de $W(b)a$ et $c_1\in W(a)b$. 

On considère la multiplication dans
$$a\left(\widetilde {OH}^2\right)b \times b\left(\widetilde {OH}^2\right)a \rightarrow a\left(\widetilde {OH}^2\right)a$$
et on calcule par exemple
$${_a1_b}.c_1= -a_1 \wedge a_2$$
où ${_a1_b} \in a\left(\widetilde {OH}^2\right)b$ est l'unité pour le produit extérieur.

 En effet, on a une fusion
$$\bigdiag{Images_arxiv/B2_baab_note.png} 
\xrightarrow{(c_1 \chemarrow b_1) \mapsto x_1} \bigdiag{Images_arxiv/B2_baab_1_note.png} :$$
$${_a1_b} \wedge c_1 = c_1  \longmapsto x_1,$$
suivie d'une scission 
$$\bigdiag{Images_arxiv/B2_baab_1_note.png} \xrightarrow{x_1 \mapsto (a_2 \chemarrow a_1)} \bigdiag{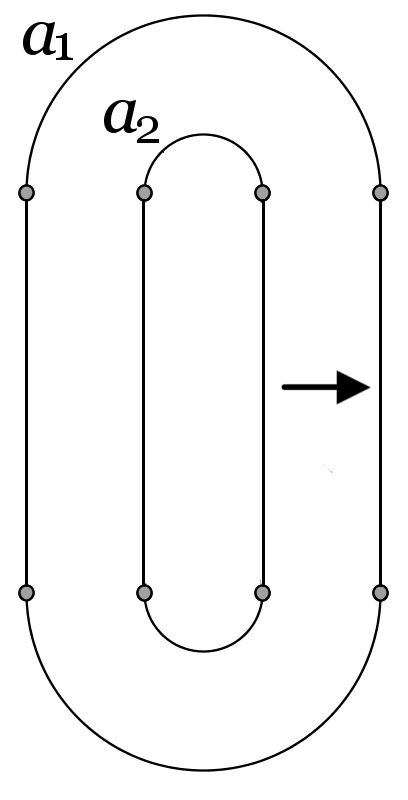} :$$
$$x_1 \longmapsto (a_2 - a_1)\wedge a_1 = -a_1 \wedge a_2,$$
 la flèche sur le diagramme représentant l'orientation de cette scission.
\end{exemple}

\begin{exemple}\label{ex:OH2_2}
On considère $\widetilde {OH}^3$ et on pose
\begin{align*}
a &=\diagg{Images_arxiv/B3_4.png},& 
b &=\diagg{Images_arxiv/B3_2.png},& 
c &=\diagg{Images_arxiv/B3_5.png},
\end{align*}
avec
\begin{align*}
W(c)b &= \middiag{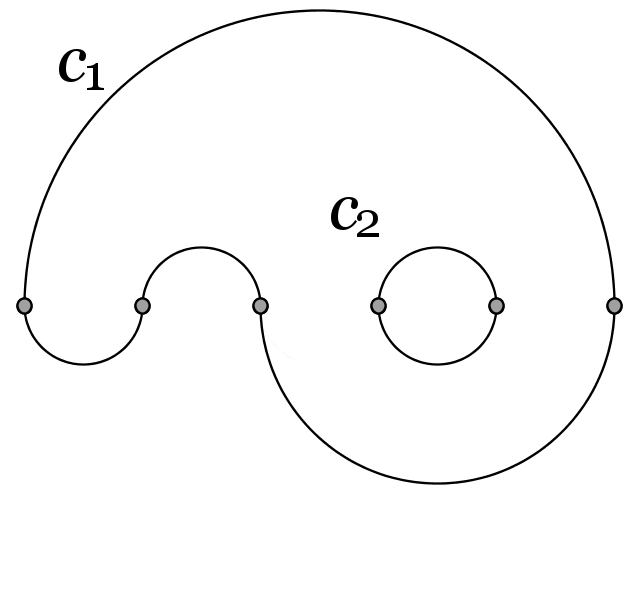},&
W(b)a &= \middiag{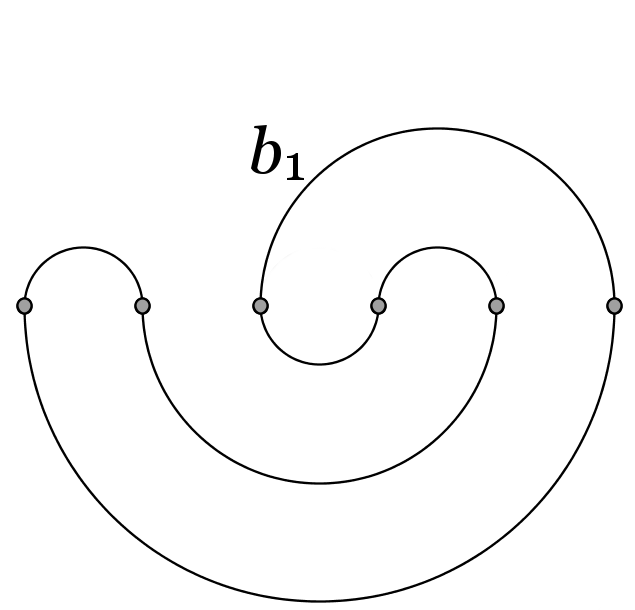},&
W(c)a &= \middiag{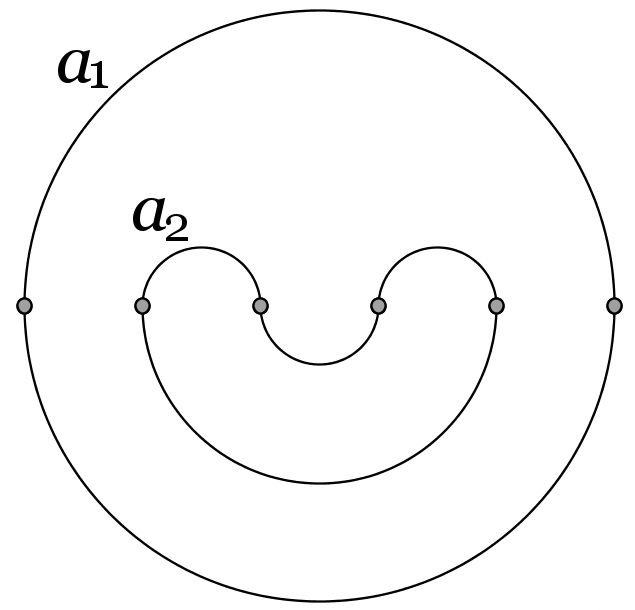}.
\end{align*}
On calcule par exemple
$$c_2.{_b1_a} = -a_1 \wedge a_2$$
dans
$$c\left(\widetilde {OH}^3\right)b \times b\left(\widetilde {OH}^3\right)a \rightarrow c\left(\widetilde {OH}^3\right)a.$$
En effet, on a une fusion
$$\bbigdiag{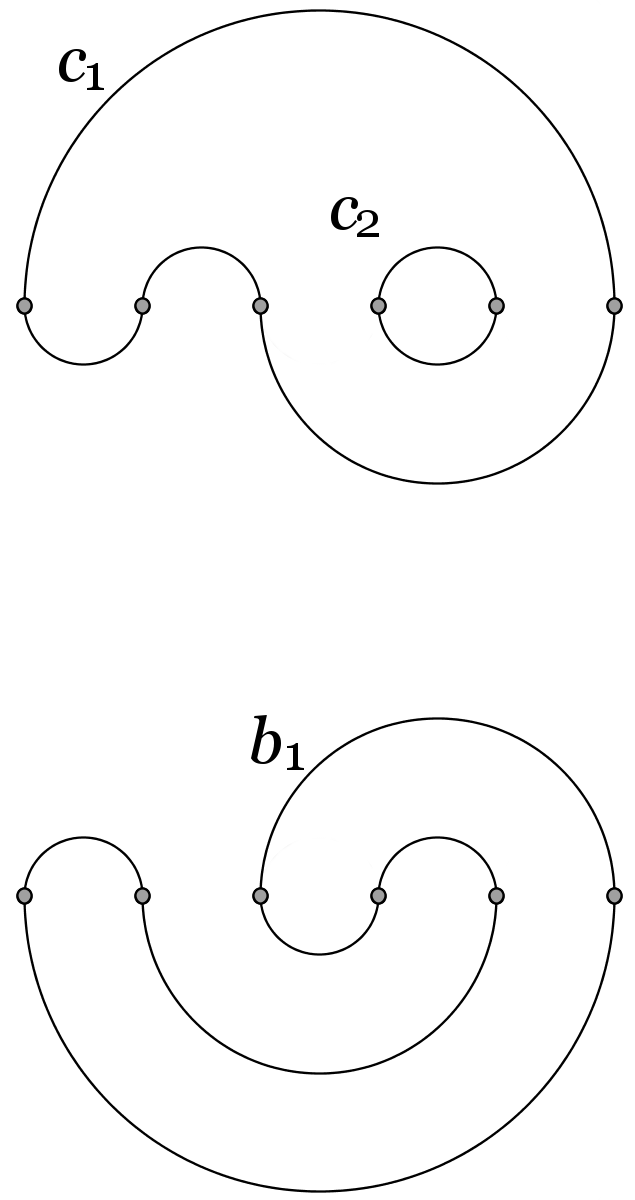}\xrightarrow{(b_1 \chemarrow c_1) \mapsto x_1} \bbigdiag{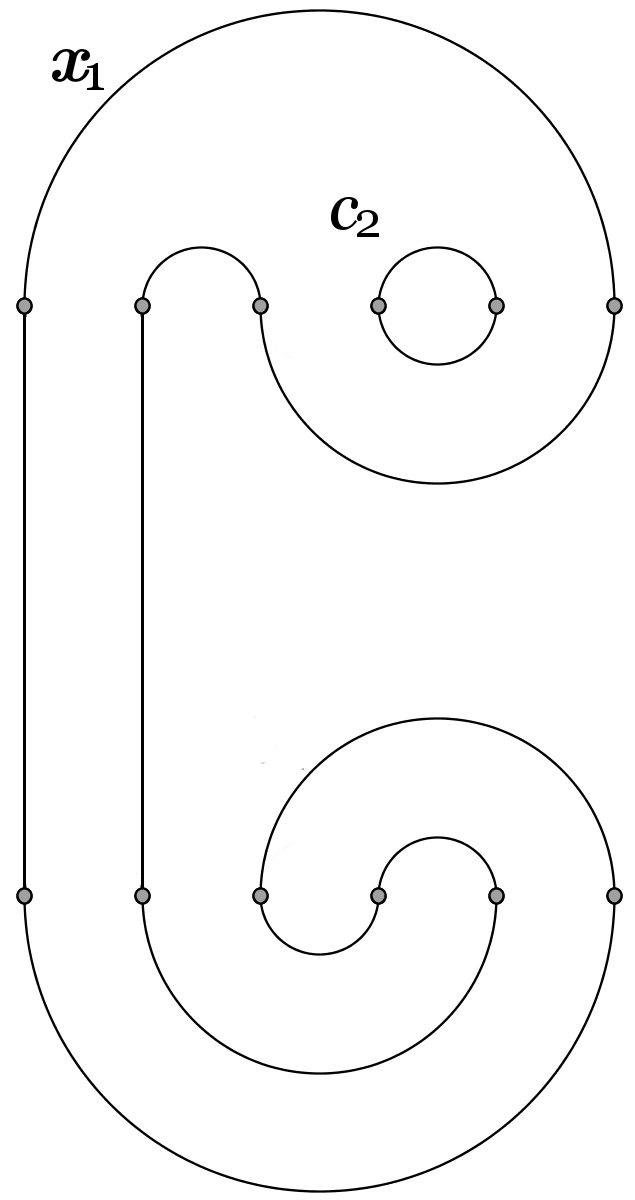} :$$
$$c_2 \wedge {_a1_b} = c_2  \longmapsto c_2,$$
ensuite une scission
$$\bbigdiag{Images_arxiv/B3_cbba_1.png} \xrightarrow{{x_1 \mapsto (y_2\chemarrow y_1)}} \bbigdiag{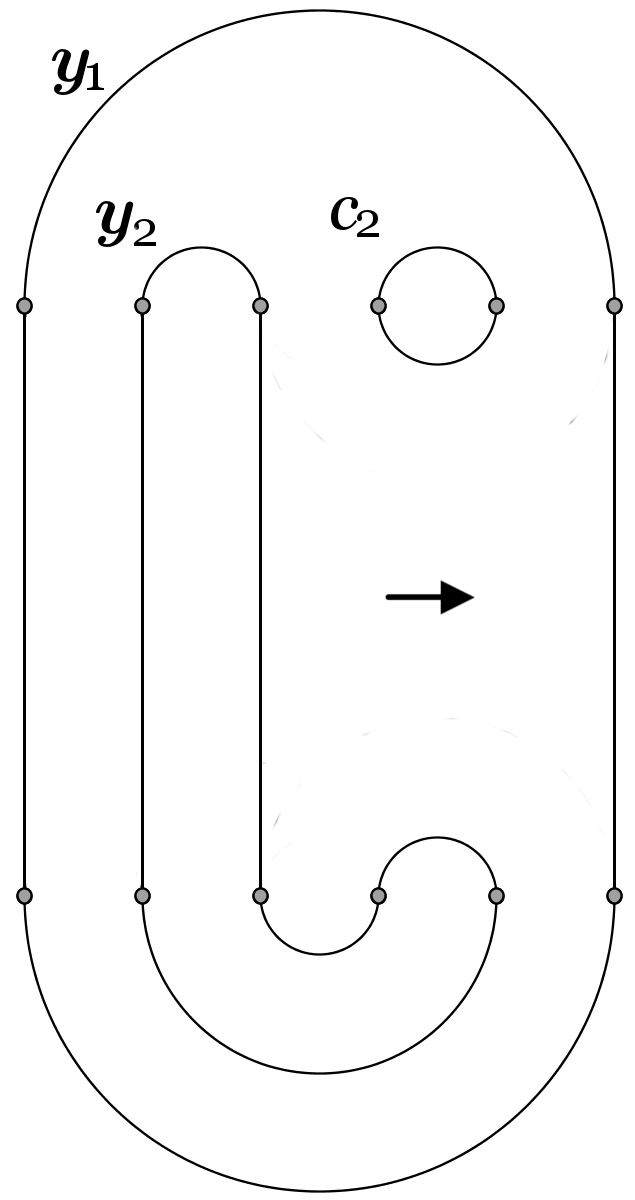}  :$$
$$c_2  \longmapsto (y_2 - y_1) \wedge c_2 = y_2 \wedge c_2 - y_1 \wedge c_2,$$
et enfin une dernière fusion
$$ \bbigdiag{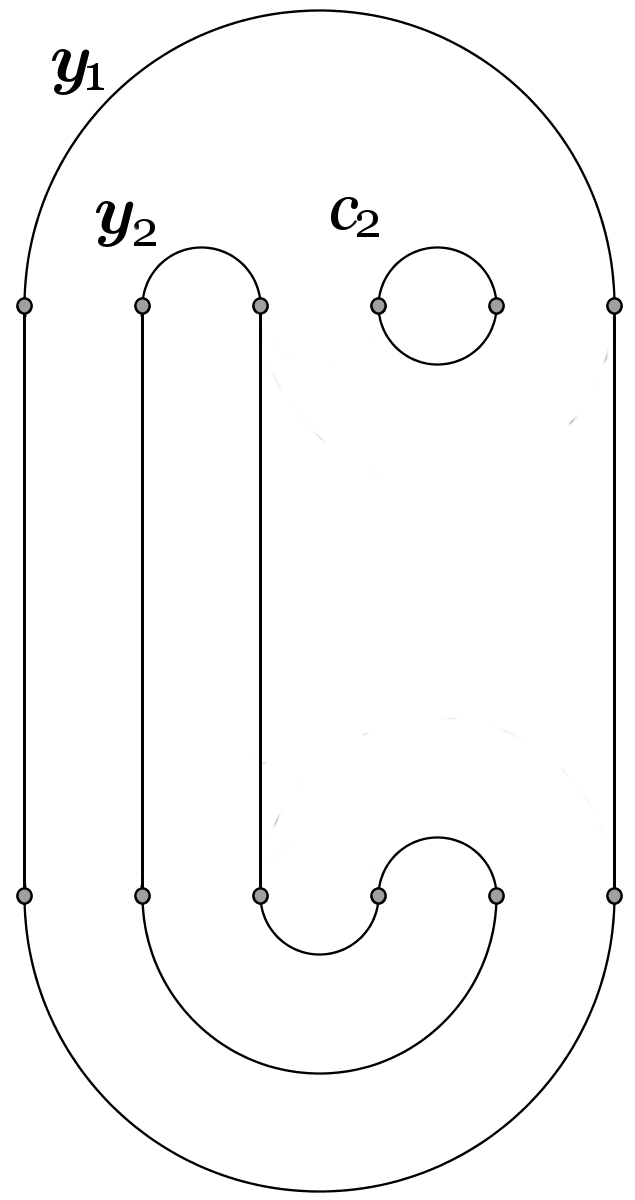} \xrightarrow{{(y_2 \chemarrow c_2) \mapsto a_2}} \bbigdiag{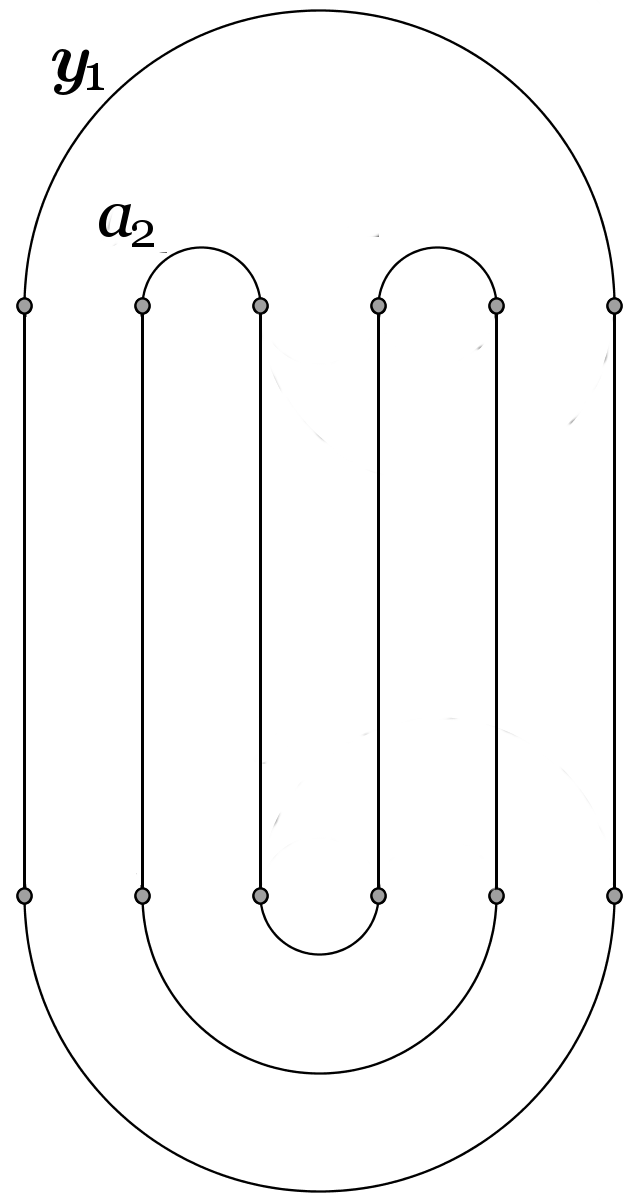} : $$
$$y_2 \wedge c_2 - y_1 \wedge c_2 \longmapsto  -y_1  \wedge a_2.$$
On obtient le résultat puisque $y_1$ correspond avec $a_1$ dans $W(c)a$.


\end{exemple}


\section{Calcul diagrammatique dans $OH^n_C$}\label{sec:calculdiag}

Afin de faciliter le calcul de produits d'éléments dans $OH^n_C$, on propose une technique de calcul diagrammatique utilisant des couleurs.

\begin{definition}
On définit un \emph{$n$-diagramme colorié} comme le plongement dans le plan d'une union disjointe de cercles pouvant être coloriés ou non (avec une seule couleur disponible, on dessine en rouge gras les cercles coloriés et en noir pointillés les autres) tel qu'ils recouvrent les \emph{points de base} $(1,0), \dots, (2n, 0)$ et qui se restreint  à des segments de droite verticales $\{i\} \times [-\epsilon, \epsilon]$ autour de chacun de ceux-ci.
\end{definition}
On munit les points de base d'un ordre induit par l'ordre sur les abscisses (de gauche à droite). Cet ordre induit un ordre sur les composantes du diagramme passant par les points de base en attribuant à chaque composante l'ordre de son point de base minimal. On étend cet ordre à un ordre partiel en disant que toutes les autres composantes, qu'on appelle \emph{libres}, sont plus grandes, comme illustré en Figure \ref{fig:ndiagramme}. On dit que deux $n$-diagrammes sont équivalents s'il existe une isotopie ambiante préservant les points $(1,0), \dots, (2n,0)$ et qui transforme l'un en l'autre en conservant les couleurs. On note $\widehat{D}^n$ l'ensemble des $n-$diagrammes à équivalence près et ${D^n}$ les classes d'équivalences sans composante libre.

\begin{figure}[h]
    \center
    \includegraphics[width=10cm]{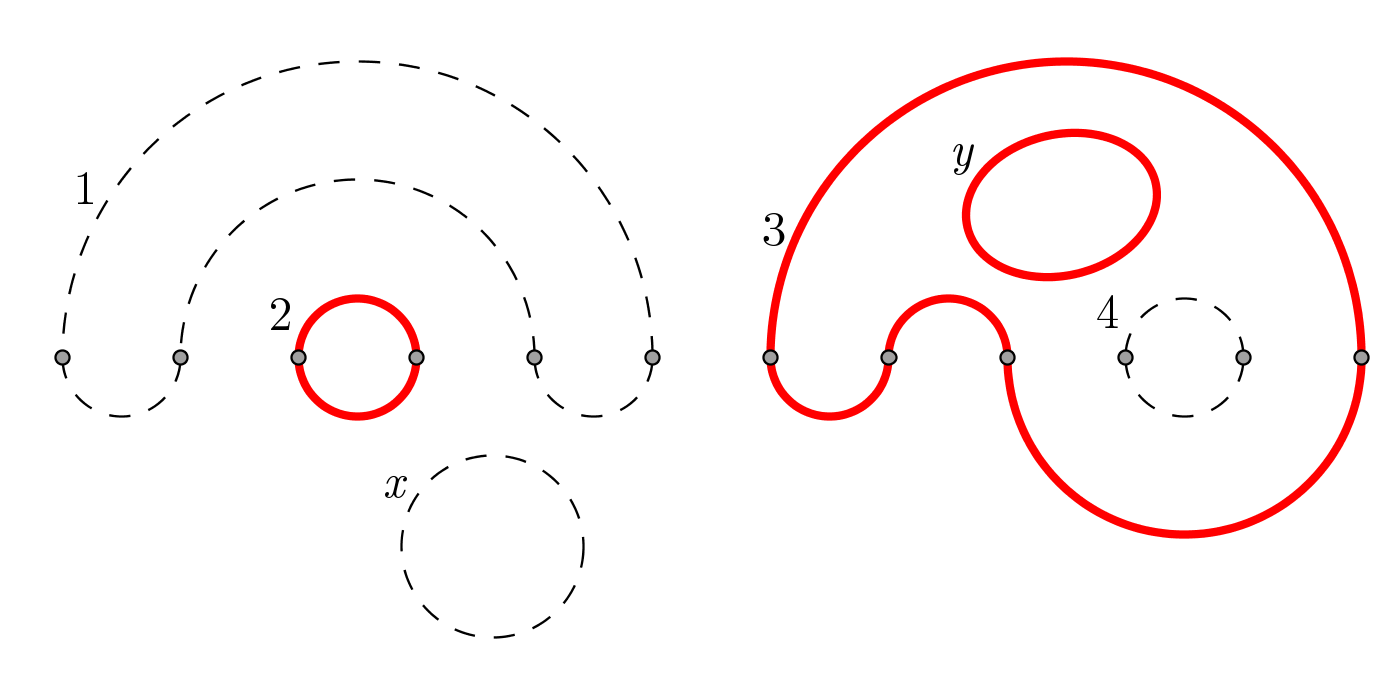}
    \caption{\label{fig:ndiagramme}Exemple de $6$-diagramme colorié avec les points de base en gris, 1,2,3 et 4 les cercles passant par les points de base, $x$ et $y$ les composantes libres et l'ordre donné par $1 < 2 < 3 < 4 < x,y$.}
\end{figure}

De plus, on dit que deux points $(i,0),(j,0) \in \{(1,0), \dots, (2n,0)\}$ sont reliés par un \emph{demi-cercle supérieur} (resp. inférieur) d'un $n$-diagramme s'il existe une composante de cercle qui passe par ces deux points telle qu'elle se restreint à un arc ne passant par aucun autre des points de base $\{(1,0), \dots, (2n,0)\}$ et que pour un voisinage suffisamment proche de $(i,0)$ et $(j,0)$ tous les points de cette restriction ont une ordonnée positive (resp. négative). On dit alors qu'un $n$-diagramme $D_1$ est \emph{compatible au dessus} d'un autre $n$-diagramme $D_2$ si toute paire de points $(i,0),(j,0) \in \{(1,0), \dots, (2n,0)\}$ reliés par un demi-cercle supérieur de $D_2$ est reliée par un demi-cercle inférieur de $D_1$ (et vice-versa) et on demande en plus qu'il n'y ait pas de composante libre entourant des points de base ni dans $D_1$ ni dans $D_2$.

\begin{remarque}
On remarque que, par l'hypothèse de verticalité du diagramme autour des points de base, chacun de ceux-ci a un seul demi-cercle supérieur et un seul inférieur.
\end{remarque}

\begin{proposition}\label{prop:isodiaghn}
Il y a une isomorphisme de groupes abéliens entre le groupe impair $OH^n$ et les combinaisons linéaires sur $\Z$ de classes d'équivalences de $n$-diagrammes coloriés ne possédant pas de composante libre
$$\Z[ {D^n}] \simeq_{ab} OH^n.$$
\end{proposition}

\begin{proof}
Soient $a,b \in B^n$ avec $W(b)a$ ayant $m$ composantes $x_1 < x_2 \dots < x_m$, l'ordre étant induit par les points de base. On construit une fonction
$$\phi : b(OH^n)a \simeq \Ext^* V(W(b)a) \rightarrow \Z[{D^n}].$$
Pour tout $\{i_{1} < i_{2} < \dots < i_{r}\} \subset \{1, 2, \dots, m\}$, $\phi$ associe à $x_{i_1} \wedge \dots \wedge x_{i_r} \in b(OH^n)a$ le $n$-diagramme $W(b)a$  (il est évident par définition d'enchevêtrement que cela donne un $n$-diagramme en translatant par $(0,-1/2)$) avec les composantes associées à $x_{i_1}, \dots, x_{i_r}$ qui sont coloriées. On étend ensuite $\phi$ par linéarité pour en faire un morphisme injectif de $OH^n$ vers $\Z[{D^n}]$. Ce morphisme est bien défini car on demande que les composantes du produit extérieur soient mis dans un certain ordre, fixant un signe. L'injectivité est obtenue par le fait que les $n$-diagrammes sont pris à isotopie ambiante fixant les points de base près. On obtient la surjectivité en observant que
$$\rank \left(\bigoplus_{a,b \in B^n} \bigl(b(OH^n)a\bigr)\right) = \sum_{a,b \in B^n} 2^{|W(b)a|} = \rank \bigl( \Z[D^n]\bigr)$$
puisque par la remarque faite au dessus, chaque demi-cercle supérieur (resp. inférieur) relie $2$ points et puisque les demi-cercles ne peuvent pas se croiser et donc les demi-cercles supérieurs (resp. inférieurs) donnent une façon de relier $2n$ points par $n$ demi-cercles vers le haut (ou bas), c'est-à-dire un élément de $B^n$. 
\end{proof}

\begin{exemple}
On sait que $B^2$ est composé des éléments
\begin{align*}
a &=\deuxdiag{Images_arxiv/B2_2.png},& 
b &= \deuxdiag{Images_arxiv/B2_1.png},
\end{align*}
donnant les diagrammes
\begin{align*}
W(a)a &= \diagg{Images_arxiv/B2_bb_note.png},&
W(a)b &= \diagg{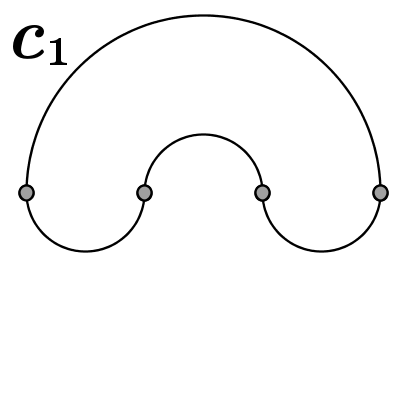},\\
W(b)a &= \diagg{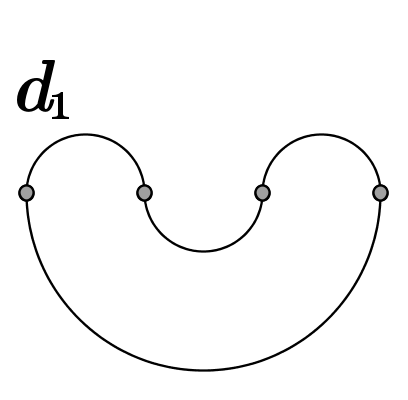},&
 W(b)b &= \diagg{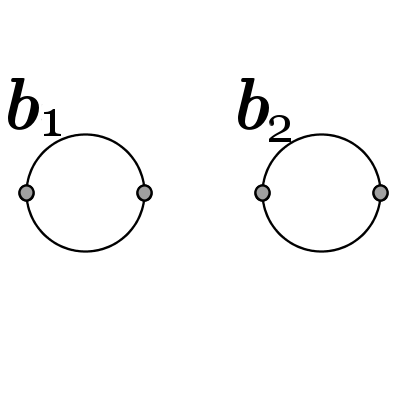}.
\end{align*}
Dès lors, cela signifie que $OH^2$ est engendré par 
$$OH^2 = \langle 1_a, a_1, a_2, a_1 \wedge a_2, 1_b, b_1, b_2, {_a1_b}, c_1, {_b1_a}, d_1 \rangle$$
avec l'isomorphisme de la proposition donnant par exemple
\begin{align*}
\phi(1_a) &= \diagg{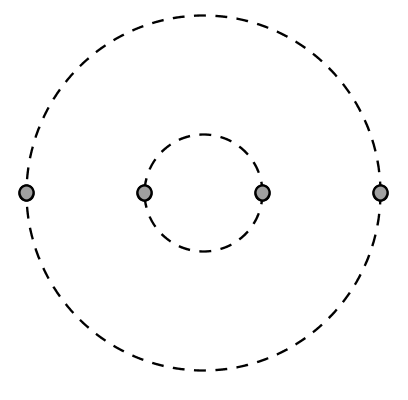},&
\phi(a_2) &= \diagg{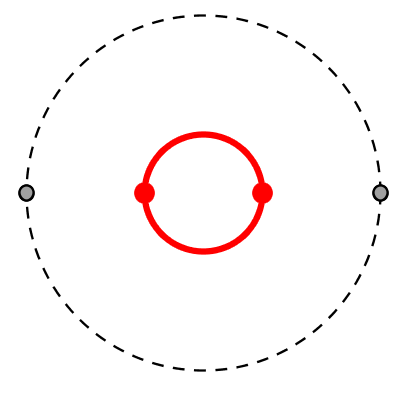},&
\phi(a_2 \wedge a_1) &= -\diagg{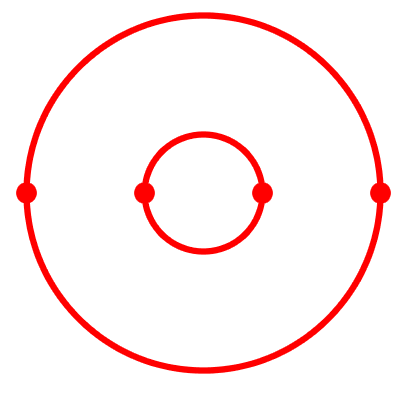},
\end{align*}\begin{align*}
\phi(c_1) &= \diagg{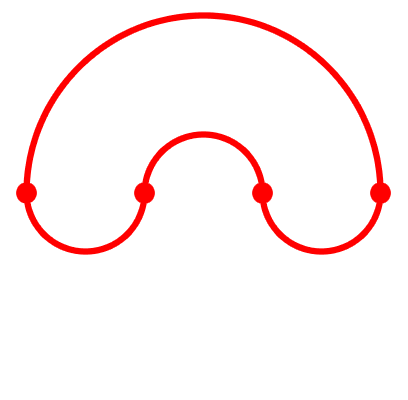},&
\phi(3{_b1_a} - 2b_1) &= 3\diagg{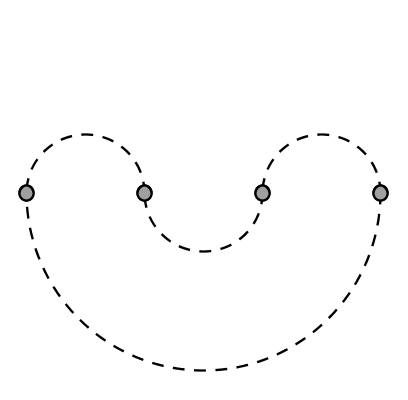} - 2 \diagg{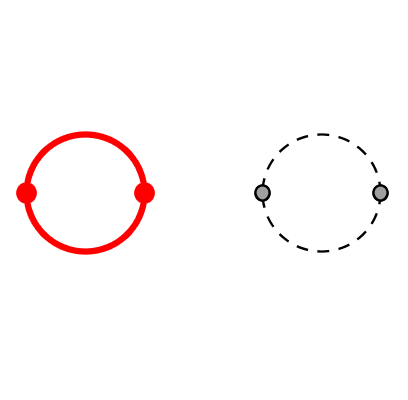}.
\end{align*}

\end{exemple}

Afin d'obtenir un anneau isomorphe à $OH^n_C$, on définit la \emph{résolution de ponts} pour une règle de multiplication $C$ d'un $n$-diagramme $D_t$ compatible au dessus d'un autre $n$-diagramme $D_b$ comme la procédure suivante :
\begin{enumerate}
\item On considère $a,b,c \in B^n$ tels que $W(c)b \simeq D_t$ et $W(b)a \simeq D_b$ par la Proposition \ref{prop:isodiaghn} et on prend $x_1 < \dots < x_{2n}$ l'ordre des points donné par $C$ pour le triplet $(c,b,a)$.
\item On place $D_b$ au dessous de $D_t$ à une certaine distance $d$ de sorte que le résultat soit une union disjointe de cercles dans le plan qu'on note $D_0$ et on pose $i := 1$. 
\item Si $(x_i,-d)$ appartient à la même composante de cercle que $(x_i,0)$ dans $D_{i-1}$, alors on ne fait rien. Sinon, on a un unique demi-cercle supérieur de $D_b$, noté $y$, qui relie $(x_i,-d)$ à un certain $(x_j,-d)$ et un unique demi-cercle inférieur de $D_t$, noté $y$, qui relie $(x_i,0)$ à $(x_j,0)$. On définit $D_{i}$ comme $D_{i-1}$ où on relie $(x_i,-d)$ à $(x_i,0)$ et $(x_j,-d)$ à $(x_j,0)$ puis on supprime $x$ et $y$. On obtient alors une union disjointe de cercles avec une ou deux composante(s) partiellement colorée(s) passant par $(x_i,0)$ et $(x_j,0)$ qu'on recolorie selon la règle suivante :
\begin{itemize}
\item Si $x$ appartient à une composante différente de celle de $y$ dans le diagramme $D_{i-1}$ :
\begin{align*}
\middiag{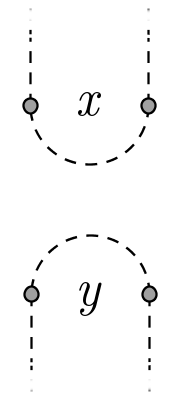} &\rightarrow \middiag{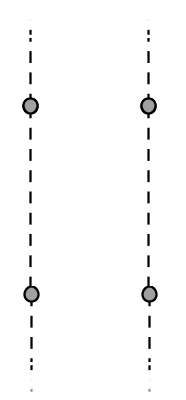}, &
\middiag{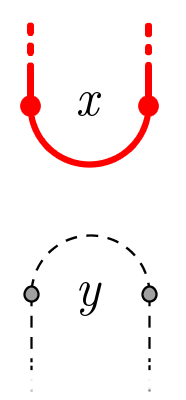} &\rightarrow \beta_x  \middiag{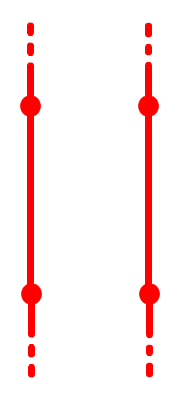}, & \\
\middiag{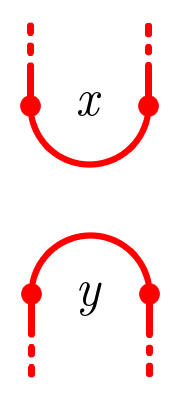} &\rightarrow 0, &
\middiag{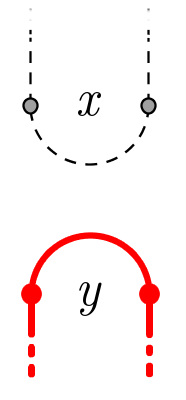} &\rightarrow \beta_y  \middiag{Images_arxiv/ResBridge_2_4.png},
\end{align*}
avec $\beta_x = (-1)^m$ (resp. $\beta_y$) pour $m$ le nombre de composantes de cercles coloriées de $D_{i-1}$ supérieures à celle contenant $y$ (resp. $x$) et inférieures à celle contenant $x$ (resp. y).
\item S'ils appartiennent à la même composante de cercle dans $D_{i-1}$ :
\begin{align*}
\middiag{Images_arxiv/ResBridge_1_1_note.png} &\rightarrow \alpha \left(\beta_i \middiag{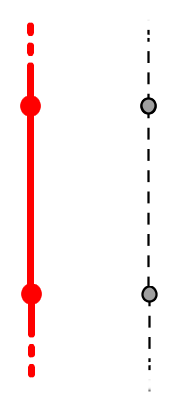} - \beta_j \middiag{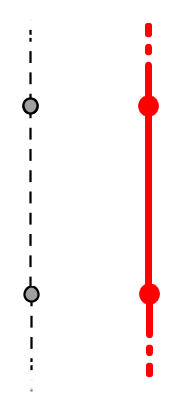}\right), &
\middiag{Images_arxiv/ResBridge_1_4_note.png} &\rightarrow \beta \middiag{Images_arxiv/ResBridge_2_4.png},
\end{align*}
avec $\alpha = +1$ si $x_i \chemarrow x_j$ dans $C$ et $-1$ sinon. On pose $\beta_i = (-1)^{m_i}$ pour $m_i$ le nombre de composantes coloriées dans $D_{i}$ strictement inférieures à celle passant par $(x_i,0)$ et idem pour $\beta_j$ et $(x_j,0)$.
On prend $\beta := \alpha(-1)^{m}$ pour $m$ le nombre de composante coloriées de $D_{i}$  inférieures ou égales à celle passant par $(x_j,0)$ ou strictement inférieures à celle passant par $(x_i,0)$.
\end{itemize}
\item On pose $i := i+1$, si $i > 2n$ on arrête, sinon on revient à l'étape 3 pour chacun des éléments de la combinaison linéaire obtenue à l'étape 3.
\end{enumerate}
On obtient au final une combinaison linéaire de $n$-diagrammes coloriés et on note $D_t \times_C D_b$ la résolution de ponts pour $C$ de $D_t$ au dessus de $D_b$. On pose aussi que $D_1 \times_C D_2 = 0$ si $D_1$ n'est pas compatible au dessus de $D_2$.
On remarque que, dans le cas de $\widetilde C$, cela revient à spécifier $\alpha = 1$ et poser $x_i = i$ et on note  $D_t \times D_b$ la résolution de ponts pour $\widetilde C$.
%

\begin{proposition}
L'isomorphisme de groupes de la Proposition \ref{prop:isodiaghn} respecte la structure multiplicative de $ {OH}_C^n$ en prenant la résolution de ponts pour $C$ qu'on étend par linéarité comme multiplication pour les $n$-diagrammes.
\end{proposition}

\begin{proof}
L'idée de la preuve est de mettre un ordre sur les composantes de $D_i$ à chaque étape de la résolution et de montrer que les signes choisis correspondent aux produits extérieurs de la multiplication dans ${OH}_C^n$ en associant les composantes coloriées aux produits extérieurs donné par l'ordre (comme dans la preuve de la Proposition \ref{prop:isodiaghn}). On définit cet ordre en parcourant les points $(1,0)$ à $(2n,0)$ puis $(1,-d)$ à $(2n,-d)$ de sorte qu'au diagramme $D_0$ soit associé le produit extérieur de celui de $\phi^{-1}(D_t) \wedge \phi^{-1}(D_b)$. On vérifie ensuite qu'à chaque étape de la résolution de ponts soit associé le produit extérieur de la fusion/scission correspondante :
\begin{itemize}
\item Si les deux demi-cercles appartiennent à des composantes différentes, alors les résoudre revient à fusionner ces composantes, expliquant la couleur des diagrammes. Pour le signe $\beta_x$ ou $\beta_y$, il vient du fait qu'en fusionnant la composante coloriée avec l'autre, l'ordre de cet élément dans le produit extérieur associé peut changer puisqu'il prend potentiellement la place de la composante non-coloriée. $\beta_x$ et $\beta_y$ calculent ces décalages.
\item S'ils appartiennent à la même composante, alors la résolution revient à séparer cette composante en deux et donc on a une scission, expliquant les couleurs. Le signe $\alpha$ calcule l'orientation de la scission pour la règle de multiplications et les signes $\beta_i$ et $\beta_j$ calculent le décalage des éléments $a_i$ et $a_j$ du facteur $(a_i - a_j)$ qu'on positionne à la bonne place dans le produit extérieur selon l'ordre défini en début de preuve. Enfin, pour $\beta$ on a deux cas à regarder. Si $x_i < x_j$, alors on a
$$\alpha(a_i - a_j) \wedge x_1 \wedge a_j \wedge x_2 = a_i \wedge x_1 \wedge a_j \wedge x_2$$
et donc $\beta = \alpha (-1)^m$ pour $m$ qui calcule le nombre de composantes inférieures ou égales à $a_j$, contenant alors $a_i$ et $a_j$. Par ailleurs on doit décaler $a_i$ pour le mettre à la place de $a_j$, donc un décalage de $|x_1|$, et on doit décaler $a_j$ pour le mettre à la bonne position dans $x_2$ (puisque $a_j$ est séparé de $a_i$ qui hérite du point de base minimal, on doit calculer le nouveau point de base de $a_j$). Au final, on a un décalage de autant de composantes que celles strictement inférieures à $a_j$ moins celle de $a_i$, mais $\beta$ est calculé en prenant $a_j$ aussi ce qui annule la contribution de $a_i$. Si par contre $a_j < a_i$ alors on a
$$\alpha(a_i - a_j) \wedge x_1 \wedge a_i \wedge x_2 = -a_j \wedge x_1 \wedge a_i \wedge x_2$$
et le calcul du décalage se fait de la même façon en inversant les rôles de $a_i$ et $a_j$ si ce n'est qu'on a pas le $a_i$ pour annuler la contribution de $a_j$, d'où le signe devant l'expression.
\end{itemize}
\end{proof}

\begin{corollaire}
Les combinaisons linéaires de $n$-diagrammes coloriés munies de la résolution de ponts sur $C$ comme multiplication forment un anneau gradué et cet anneau est isomorphe à ${OH}_C^n$.
\end{corollaire}

De par ce corollaire, on confond dans la suite de ce travail les éléments de ${OH}_C^n$ avec des diagrammes colorés.

\begin{remarque}
Le calcul diagrammatique peut être aisément adapté pour $H^n$, où on associe simplement les composantes colorées et les produits tensoriel sans se soucier de l'ordre ainsi qu'en oubliant les signes dans les résolutions de ponts et en prenant un plus dans la scission.
\end{remarque}

\begin{exemple}
On refait l'Exemple \ref{ex:OH2_1} en calculs diagrammatiques :
\begin{align*}
\phi({_a1_b}.c_1) = \smalldiag{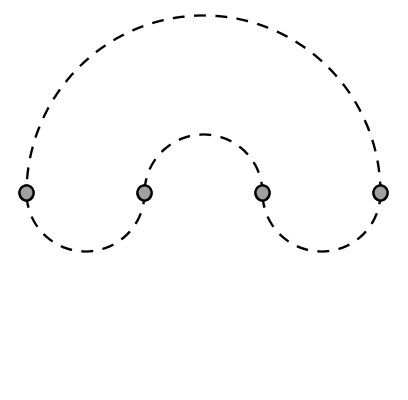} \times \smalldiag{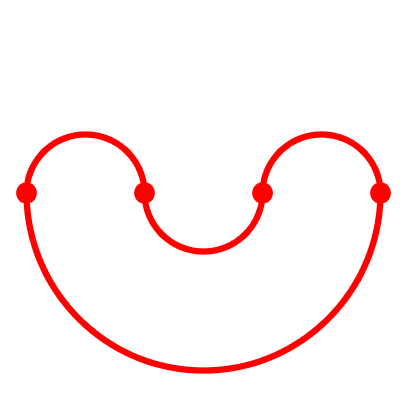} = -\smalldiag{Images_arxiv/B2_bb_11.png} = \phi(-a_1 \wedge a_2)
\end{align*}
puisqu'on a la résolution de ponts
\begin{align*}
\middiag{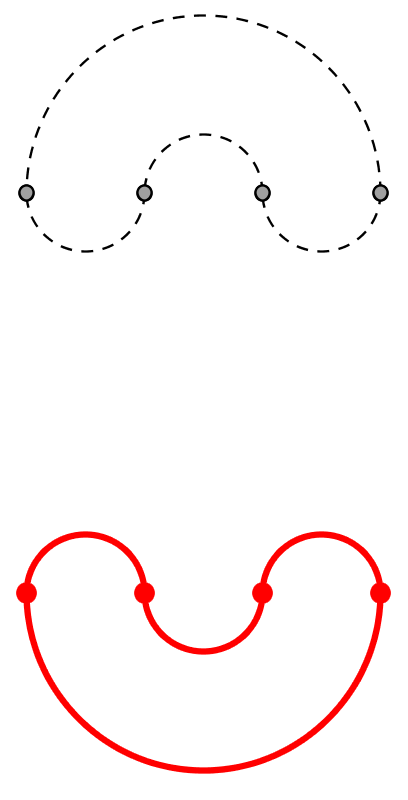} \rightarrow \middiag{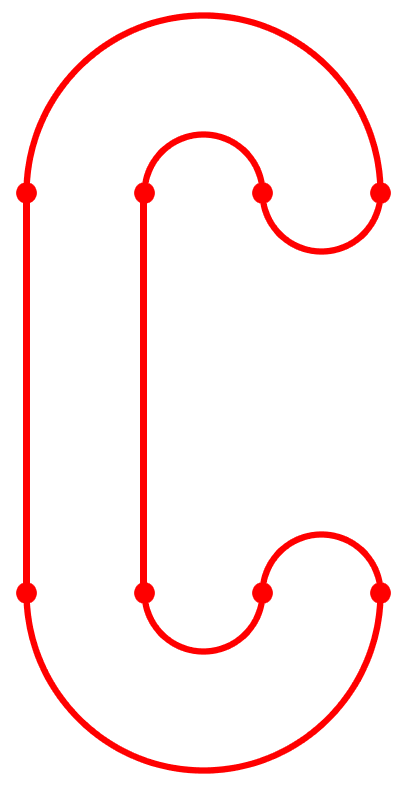} \rightarrow -\middiag{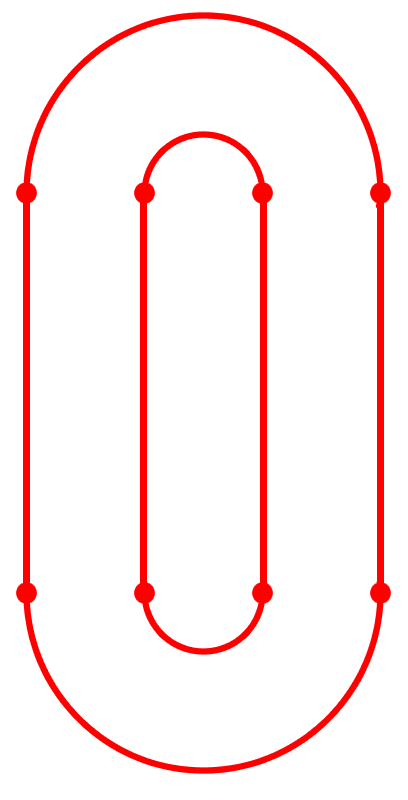}.
\end{align*}
\end{exemple}

\begin{exemple}
On refait l'Exemple  \ref{ex:OH2_2} en calculs diagrammatiques :
\begin{align*}
\phi(c_2.{_b1_a}) = \diagg{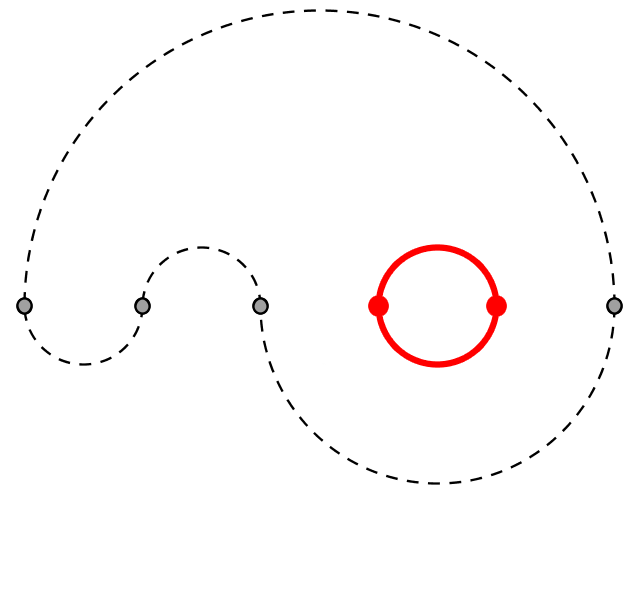} \times \diagg{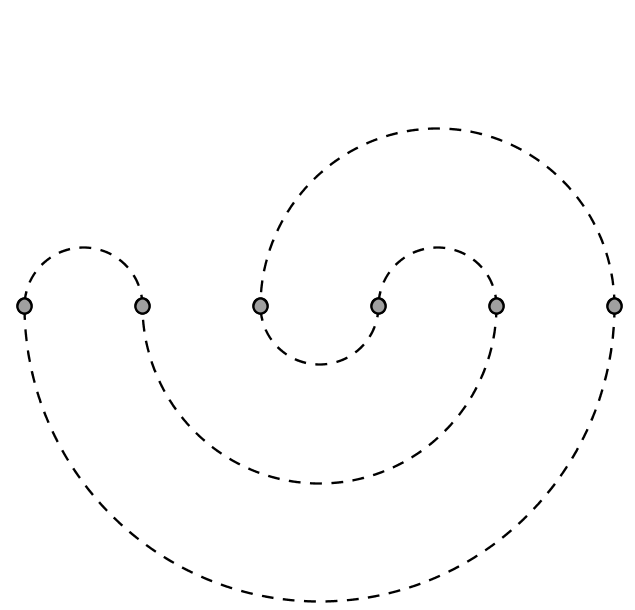} =- \diagg{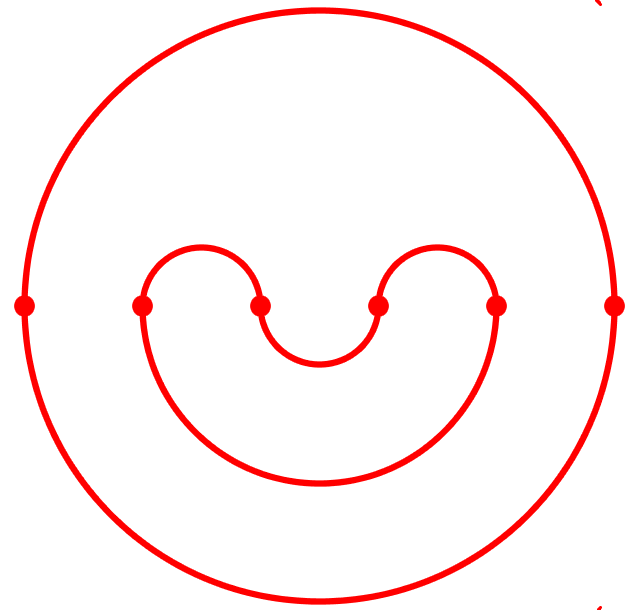} = \phi(-a_1 \wedge a_2)
\end{align*}
puisqu'on a la résolution de ponts
\begin{align*}
\bigdiag{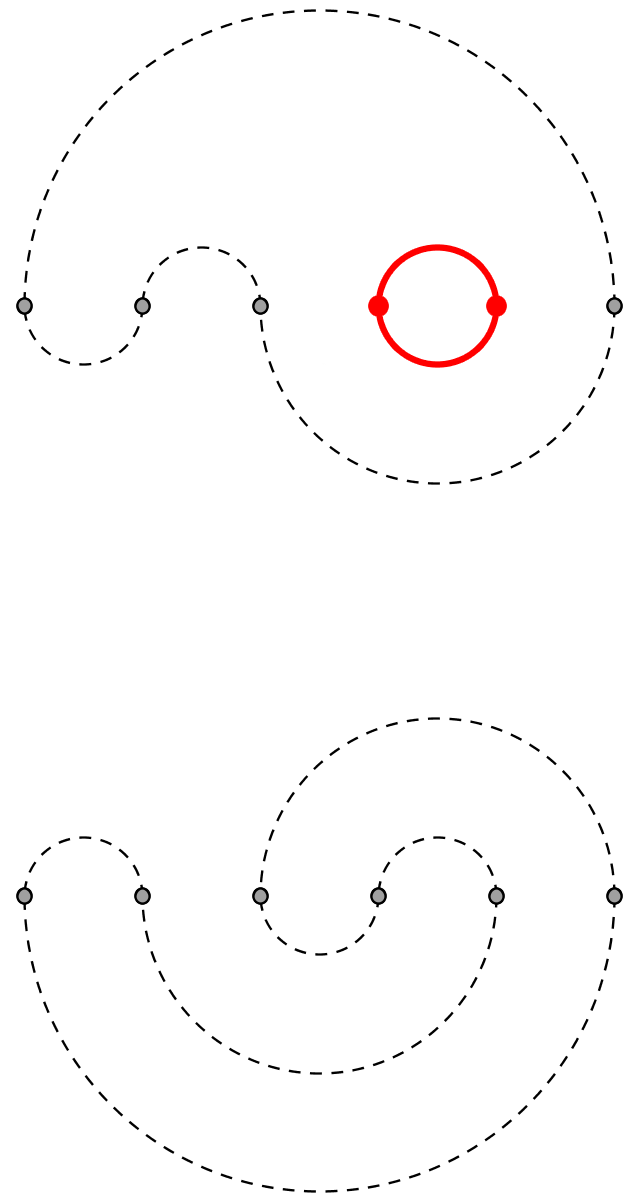} \rightarrow \bigdiag{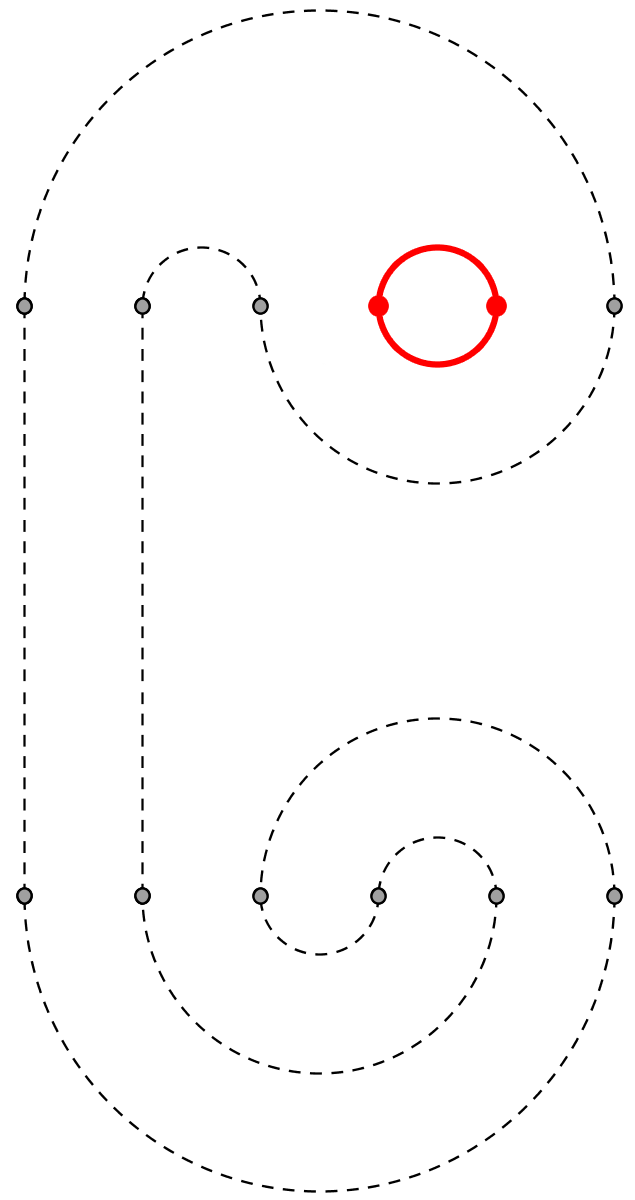} \rightarrow \left( \bigdiag{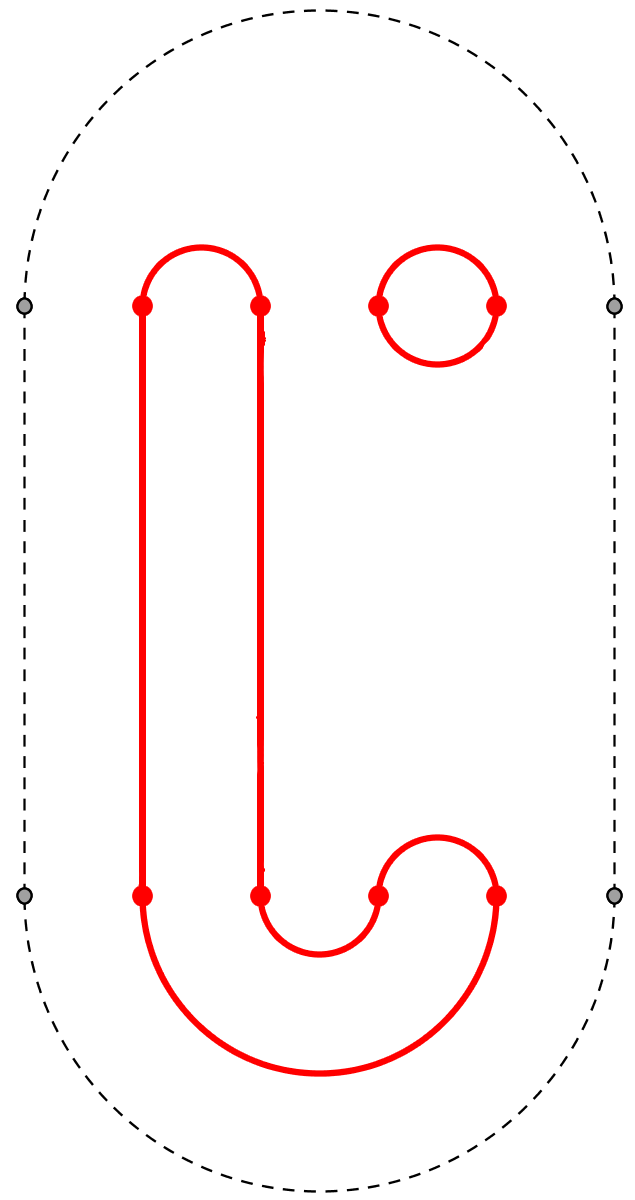} -  \bigdiag{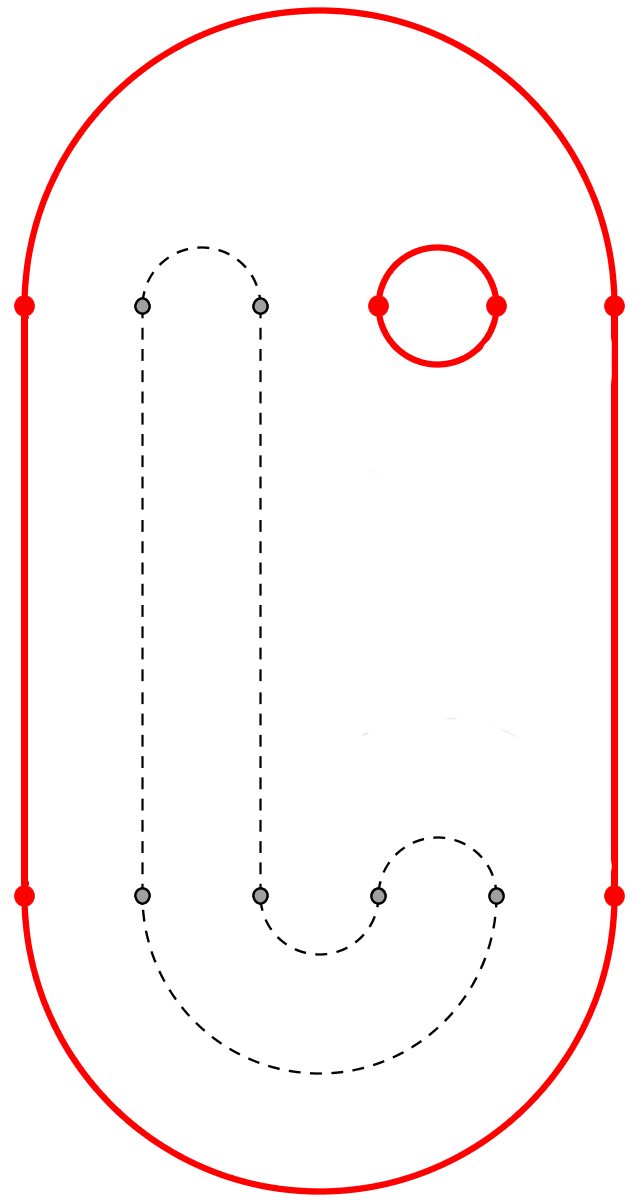} \right) \rightarrow \left( 0-\bigdiag{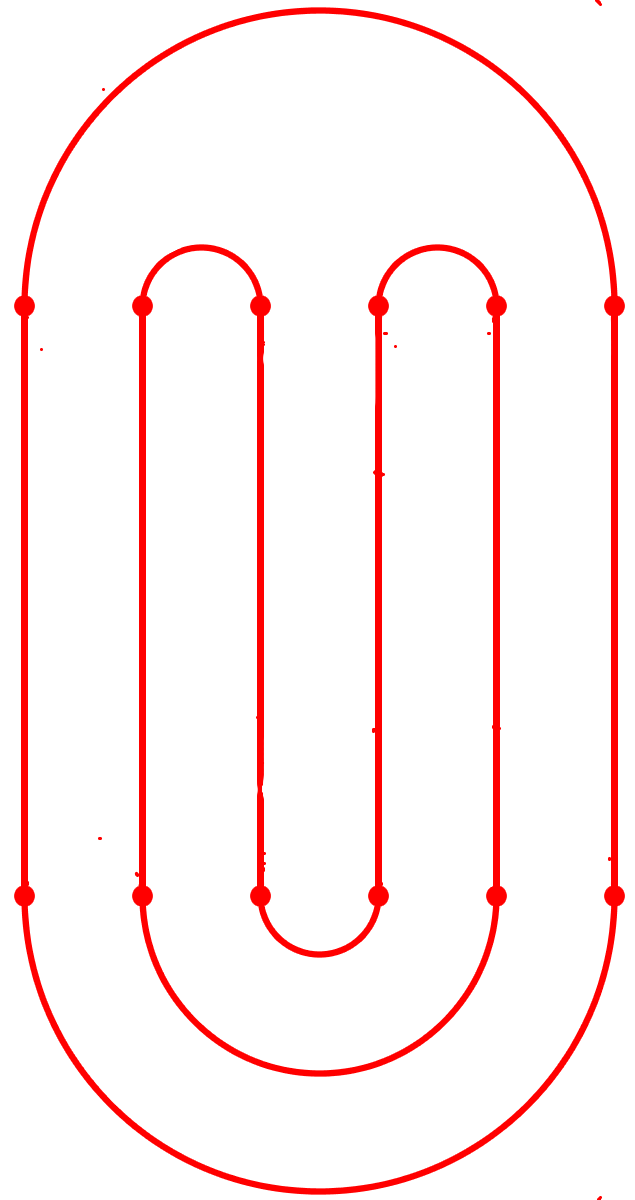} \right).
\end{align*}
\end{exemple}

\begin{exemple}\label{ex:changementordre}
On calcule dans $\widetilde {OH}^3$
\begin{align*}
\phi(x_3.y_2) &= \left(\diagg{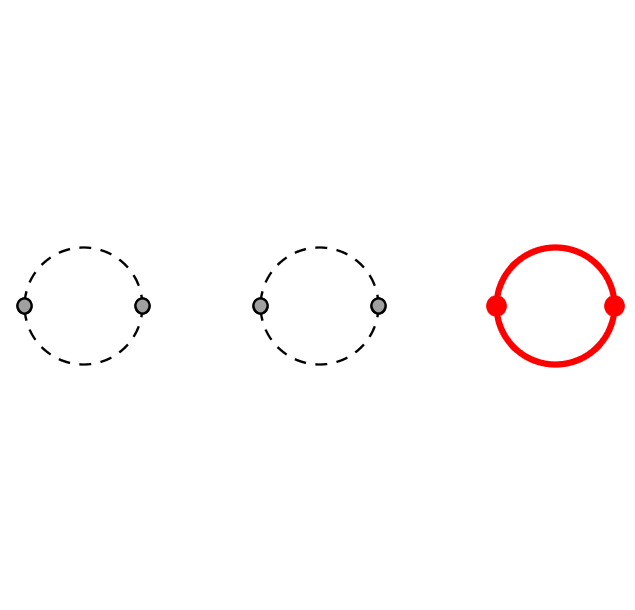} \times \diagg{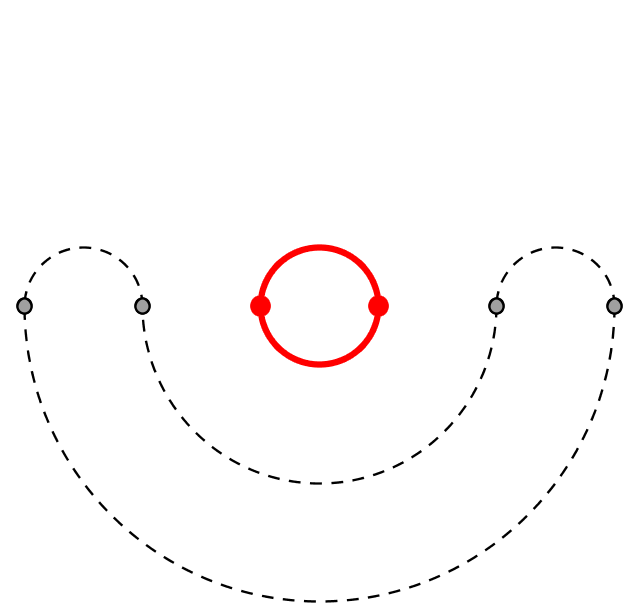}\right) = \left(\diagg{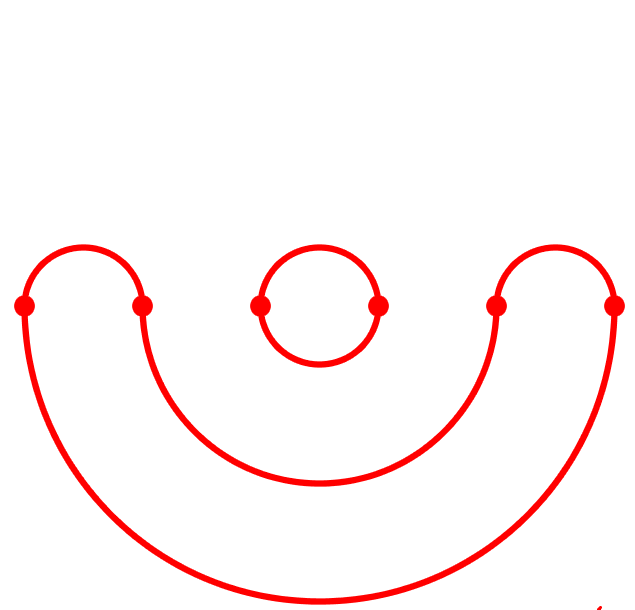}\right) = \phi(y_1 \wedge y_2), \\
\phi(x_2.y_1) &= \left(\diagg{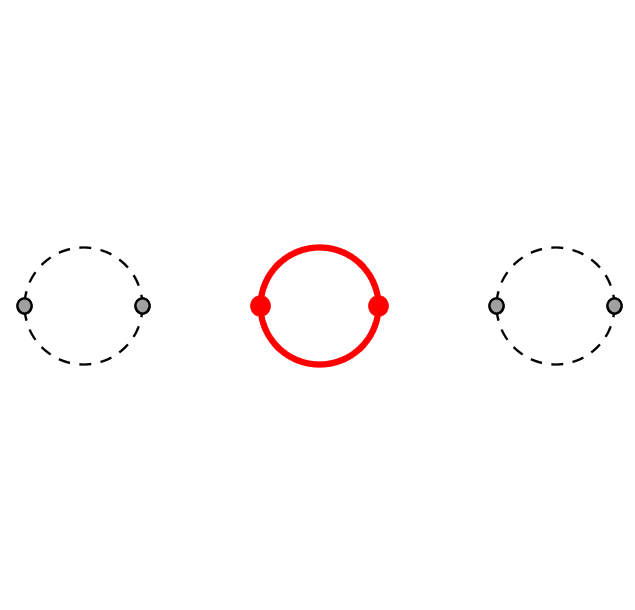} \times \diagg{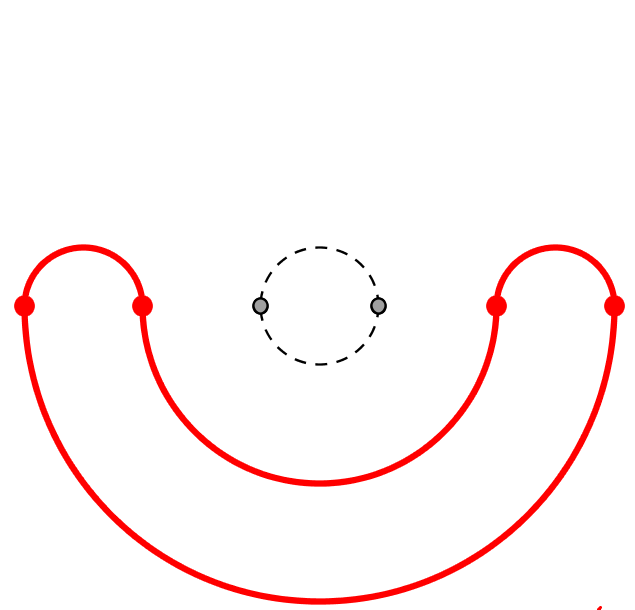}\right) = -\left(\diagg{Images_arxiv/B3_xy_col_2.png}\right) = \phi(-y_1 \wedge y_2).
\end{align*}
\end{exemple}

\section{Propriétés générales de $OH^n_C$}

Dans cette section, on étudie quelques propriétés de $OH^n_C$. On observe notamment que $OH^n_C$ est non-associatif pour tout $n \ge 2$ et on montre que $OH^n_C$ correspond avec $H^n$ quand on les considère modulo $2$.

\begin{proposition}\label{prop:sousOHN}
Pour tout $m \le n$, $OH_C^m$ est un sous-anneau de $OH_C^n$.
\end{proposition}

\begin{proof}
Tout élément $x$ de $B^m$ peut être étendu à un élément $\bar x$ de $B^n$ en le décalant vers la droite de $n-m$ et en ajoutant des paires $(1, 2n), (2,2n-1), \dots, (n-m,n+m+1)$, comme illustré en Figure \ref{fig:bninclusion}. Dès lors, pour $a,b \in B^m$, on retrouve toutes les composantes de $W(b)a$ dans $W(\bar b)\bar a$ et donc on obtient une inclusion canonique de groupes 
$$V(W(b)a) \subset V(W(\bar b)\bar a).$$
Cela induit une injection de groupes gradués $OH^m_C \hookrightarrow OH^n_C$ par
$$b(OH^m_C)a = \Ext^* V(W(b)a)\{m-|W(b)a|\} \hookrightarrow \Ext^* V(W(\bar b) \bar a)\{n-|W(\bar b)\bar a|\} = \bar b(OH^n_C)\bar a $$
puisque $|W(\bar b)\bar a| = |W(b)a| + n - m$. De plus cette injection respecte la structure de multiplication puisque le arcs qu'on ajoute pour former $\bar x$ n'engendrent que des fusions de composantes non coloriées (les signes des résolutions de ponts ne sont calculées qu'à partir des composantes coloriées).
\end{proof}

\begin{figure}[h]
    \center
    \includegraphics[width=8cm]{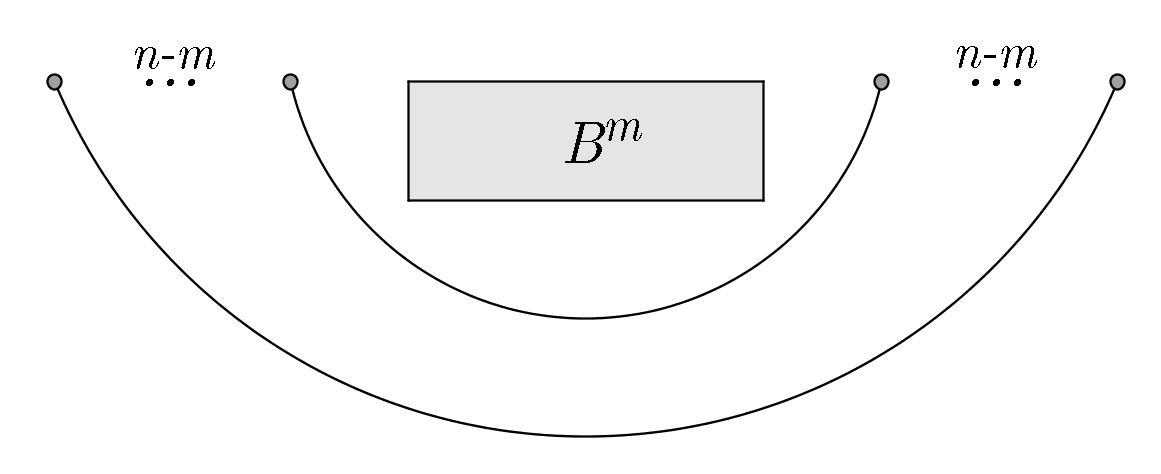}
    \caption{\label{fig:bninclusion} Inclusion de $B^m$ dans $B^n$ en ajoutant $n-m$ arcs.}
\end{figure}

\begin{remarque}
On choisit cette injection de $B^m$ dans $B^n$ car elle vient de l'injection canonique de l'algèbre de Temperley-Lieb $TL_m$ (voir Définition \ref{def:temperleyliebalg}) dans $TL_n$ qui consiste à ajouter des brins verticaux à droite. En effet on peut voir les éléments de $\widetilde B^n$ comme des combinaisons de $U_i$ qu'on courbe, déplaçant les $n$ point en bas pour les mettre à gauche. Pour la preuve de la proposition, on aurait aussi bien pu compléter un élément en ajoutant des paires de points $(2m+1, 2m+1), \dots, (2n-1,2n)$ mais cela donne une inclusion de $OH_C^m$ dans $OH_C^n$ qui n'est pas usuelle.
\end{remarque}

\begin{proposition}
Soient $C$ et $C'$ deux règles de multiplications ainsi que $a,b,c \in B^n$, alors $C_{cba}$ et $C'_{cba}$ sont équivalents vu en tant que cobordismes sans chronologie.
\end{proposition}

\begin{proof}
Si on oublie l'orientation et qu'on accepte l'associativité, la coassociativité et les relations de Frobenius, les deux cobordismes deviennent équivalents puisqu'ils contiennent tous deux le même nombre de scissions et fusions pour les même composantes.
\end{proof}

\begin{corollaire}\label{prop:multisignepres}
Soient $C$ et $C'$ deux règles de multiplication, alors les multiplications
$$\bigl( c(OH^n_C)b \times b(OH^n_C)a \rightarrow c(OH^n_C)a \bigr) = \pm\bigl(
c(OH^n_{C'})b \times b(OH^n_{C'})a \rightarrow c(OH^n_{C'})a\bigr)$$
sont identiques à signe près, signe dépendant uniquement de $a,b$ et $c$.
\end{corollaire}

\begin{proof}
On sait que $C_{cba}$ et $C'_{cba}$ sont équivalents par la proposition précédente et donc par la Proposition \ref{prop:isotopesignepres} on obtient la propriété voulue.
\end{proof}

\subsection{Non-associativité de $\Hno_C$}

On voit facilement que $OH_C^1$ est associatif par l'Exemple \ref{ex:OH1}. Par contre pour $n = 2$ on perd l'associativité si on utilise la règle de multiplication $\widetilde C$ comme le montre l'exemple suivant.

\begin{exemple} \label{ex:OH2nonassoc}
On considère les calculs diagrammatiques suivant dans $\widetilde {OH}^2$ :
\begin{align*}
\left(\smalldiag{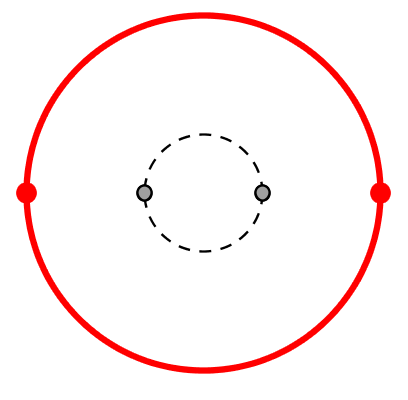} \times   \smalldiag{Images_arxiv/B2_ba_0.png}\right) \times  \smalldiag{Images_arxiv/B2_ab_0.png} &=
\smalldiag{Images_arxiv/B2_ba_1.png} \times \smalldiag{Images_arxiv/B2_ab_0.png} \\
&= -\smalldiag{Images_arxiv/B2_bb_11.png}
\end{align*}
et en réarrangeant les parenthèses on obtient
\begin{align*}
\smalldiag{Images_arxiv/B2_bb_10.png} \times \left( \smalldiag{Images_arxiv/B2_ba_0.png} \times  \smalldiag{Images_arxiv/B2_ab_0.png}\right) &=
\smalldiag{Images_arxiv/B2_bb_10.png} \times \left(\smalldiag{Images_arxiv/B2_bb_01.png} - \smalldiag{Images_arxiv/B2_bb_10.png} \right) \\
&= \smalldiag{Images_arxiv/B2_bb_11.png}.
\end{align*}
\end{exemple}

Cet exemple mène à un autre exemple qui montre que le produit de trois éléments peut parfois être nul et parfois non suivant l'ordre dans lequel on multiplie.
\begin{exemple}
On considère les multiplications suivantes dans $\widetilde {OH}^2$ :
\begin{align*}
\underbrace{\smalldiag{Images_arxiv/B2_bb_10.png}}_{x}  \times\underbrace{ \left(\smalldiag{Images_arxiv/B2_ba_0.png}+\smalldiag{Images_arxiv/B2_bb_00.png} \right)}_{y} \times \underbrace{\left(\smalldiag{Images_arxiv/B2_ab_0.png} + \smalldiag{Images_arxiv/B2_bb_01.png} \right)}_{z} .
\end{align*}
On commence par rassembler les termes à gauche pour avoir $(xy)$ puis on multiplie par le terme de droite pour obtenir $(xy)z$
\begin{align*}
\smalldiag{Images_arxiv/B2_bb_10.png} \times \left(\smalldiag{Images_arxiv/B2_ba_0.png}+\smalldiag{Images_arxiv/B2_bb_00.png} \right) &= 
\smalldiag{Images_arxiv/B2_ba_1.png}  + \smalldiag{Images_arxiv/B2_bb_10.png}, \\
\left(\smalldiag{Images_arxiv/B2_ba_1.png}  + \smalldiag{Images_arxiv/B2_bb_10.png} \right)  \times \left(\smalldiag{Images_arxiv/B2_ab_0.png} + \smalldiag{Images_arxiv/B2_bb_01.png} \right)  &=
-\smalldiag{Images_arxiv/B2_bb_11.png} + \smalldiag{Images_arxiv/B2_bb_11.png} \\
&= 0.
\end{align*}
Ensuite on rassemble les termes de droite pour avoir $(yz)$ puis on multiplie par l'élément de gauche pour obtenir $x(yz)$
\begin{align*}
\left(\smalldiag{Images_arxiv/B2_ba_0.png}+\smalldiag{Images_arxiv/B2_bb_00.png} \right) \times \left(\smalldiag{Images_arxiv/B2_ab_0.png} + \smalldiag{Images_arxiv/B2_bb_01.png} \right) &=
\smalldiag{Images_arxiv/B2_bb_01.png}  - \smalldiag{Images_arxiv/B2_bb_10.png} + \smalldiag{Images_arxiv/B2_bb_01.png}, \\
&= 2\smalldiag{Images_arxiv/B2_bb_01.png}-\smalldiag{Images_arxiv/B2_bb_10.png}\\
\left(2\smalldiag{Images_arxiv/B2_bb_01.png}-\smalldiag{Images_arxiv/B2_bb_10.png}\right) \times \smalldiag{Images_arxiv/B2_bb_10.png} &= -2\smalldiag{Images_arxiv/B2_bb_11.png}.
\end{align*}
On a donc trouvé un exemple de $x,y,z \in \widetilde{ OH}^2$ tel que $(xy)z = 0$ et $x(yz) = 2w$ pour un certain $w \in \widetilde{OH}^2$ non-nul.
\end{exemple}

\begin{proposition}
Pour $n \ge 2$ et $C$ quelconque, la multiplication dans $OH^n_C$ n'est pas associative.
\end{proposition}

\begin{proof}
On généralise l'Exemple \ref{ex:OH2nonassoc} sur tout $C$ en remarquant que si
$$ \smalldiag{Images_arxiv/B2_ba_0.png} \times_C  \smalldiag{Images_arxiv/B2_ab_0.png} = \smalldiag{Images_arxiv/B2_bb_10.png} - \smalldiag{Images_arxiv/B2_bb_01.png}$$
alors par le Corollaire \ref{prop:multisignepres} le signe de l'homomorphisme obtenu pour un $C$ est fixé et donc
$$\smalldiag{Images_arxiv/B2_ba_1.png} \times_C \smalldiag{Images_arxiv/B2_ab_0.png} = \smalldiag{Images_arxiv/B2_bb_11.png}$$
et idem pour les signes opposés, en remarquant que les autres multiplications n'impliquent que des fusions et ne sont donc pas influencées par le choix de $C$. On généralise ensuite pour tout $n \ge 2$ par la Proposition \ref{prop:sousOHN}.
\end{proof}

Cela ne constitue qu'une des deux sources de non-associativité d'un produit de $x$, $y$ et $z$ dans $OH^n_C$ venant de la permutation du facteur $x$ avec les facteurs $(a_1 - a_2)$ qui apparaissent par les scissions de la multiplication $yz$, comme illustré en Figure \ref{fig:chgtac}. Un second changement de signes peut être du au changement de chronologies dans le cobordisme $W(d)cW(c)bW(b)a \rightarrow W(d)a$. Pour $x \in d(OH^n_C)c$, $y \in c(OH^n_C)b$ et $z \in b(OH^n_C)a$ on a
\begin{align*}
 xy &= S(d,c,b) \wedge  x_{db} \wedge y_{db}, \\
(xy)z &= S(d,b,a) \wedge S(d,c,b) \wedge x_{da} \wedge y_{da} \wedge z_{da},\\
yz &= S(c,b,a) \wedge y_{ca} \wedge z_{ca}, \\
x(yz) &= S(d,c,a) \wedge x_{da} \wedge S(c,b,a) \wedge y_{da} \wedge z_{da},
\end{align*}
où $S(d,c,b)$ est le produit de $(a_1 - a_2)$ ajouté par les scissions de l'image par $OF$ du cobordisme $W(d)cW(c)b \rightarrow W(d)b$ et $x_{db}$ est l'image de $x$ par ce cobordisme et de même pour les autres éléments. On a donc un premier changement de signes pour faire permuter $x_{da}$ avec $S(c,b,a)$ et un second donné par le changement de chronologies
$$ S(d,b,a) \wedge S(d,c,b) = \pm  S(d,c,a) \wedge S(c,b,a).$$
Ce second phénomène est illustré dans le prochain exemple.

\begin{figure}[h]
    \center
    \includegraphics[width=5cm]{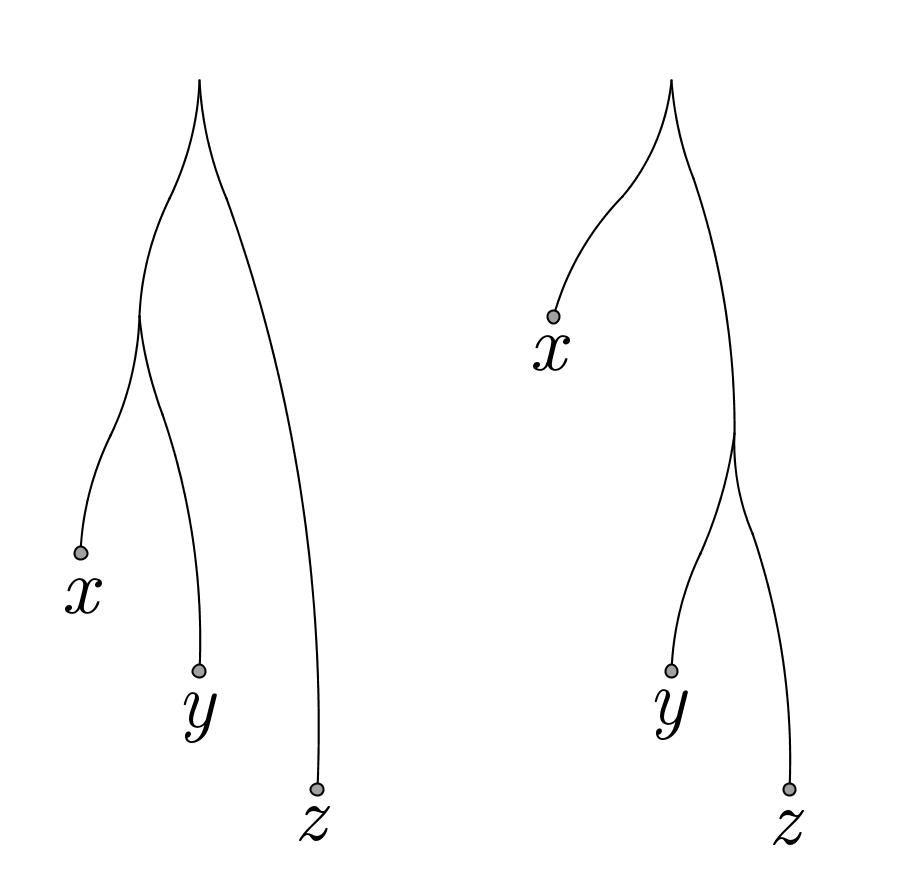}
    \caption{\label{fig:chgtac} A gauche on a $(xy)z$ et à droite $x(yz)$ : $x$ permute avec le cobordisme de la multiplication $yz$ et donc avec les éléments engendrés par ses scissions.}
\end{figure}

\begin{exemple}\label{ex:chgtch}
On considère les multiplications dans $\widetilde{OH}^3$
\begin{align*}
\left(\diagsc{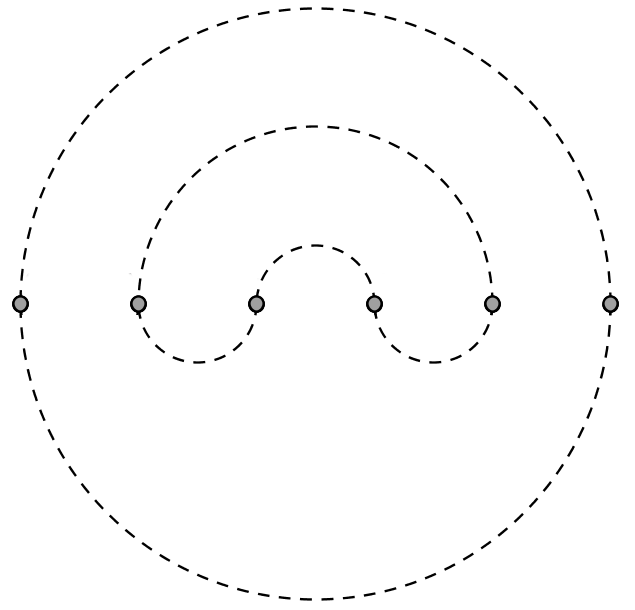} \times   \diagsc{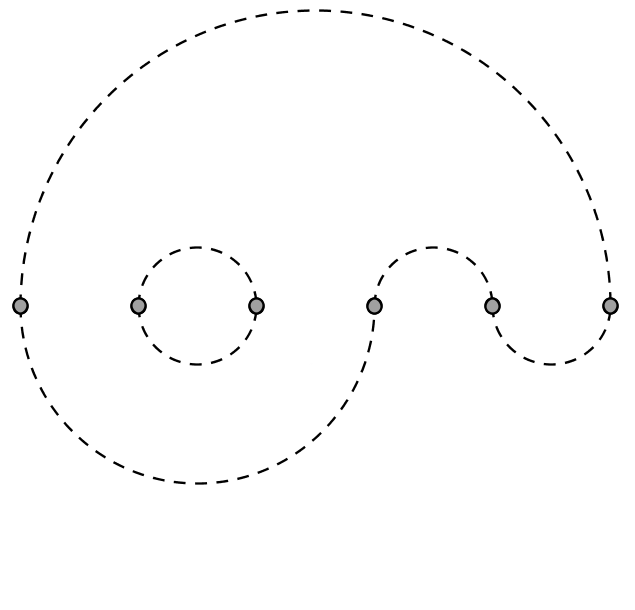}\right) \times  \diagsc{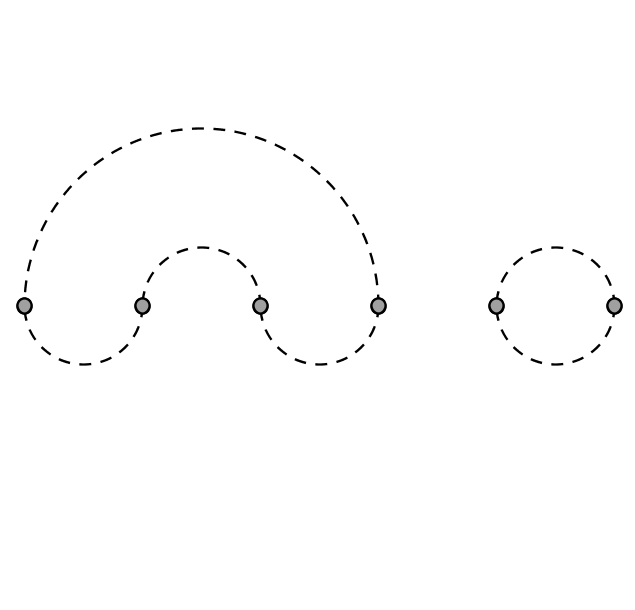}
 &= \diagsc{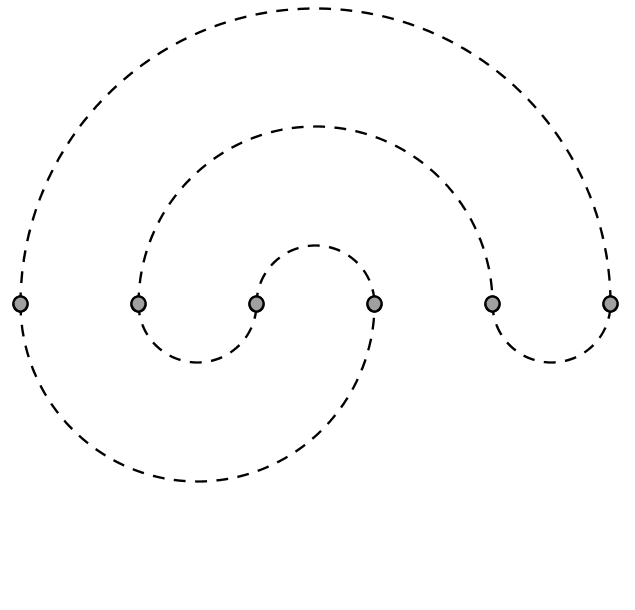} \times \diagsc{Images_arxiv/B3_bx_col_00.png} \\
&= \diagsc{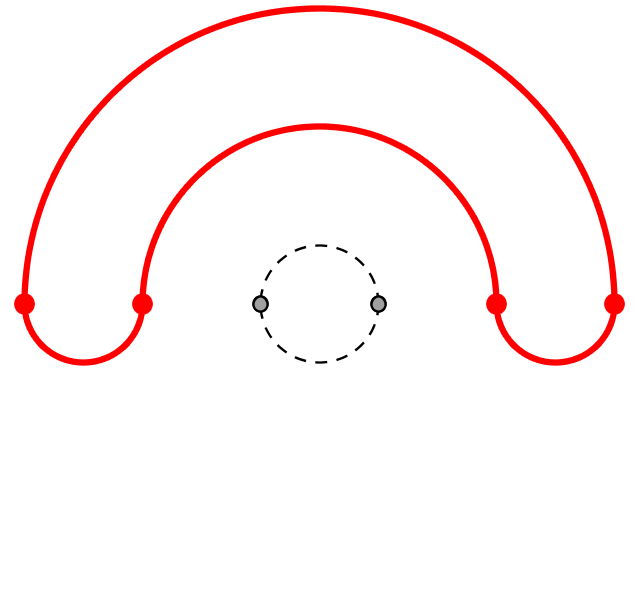} - \diagsc{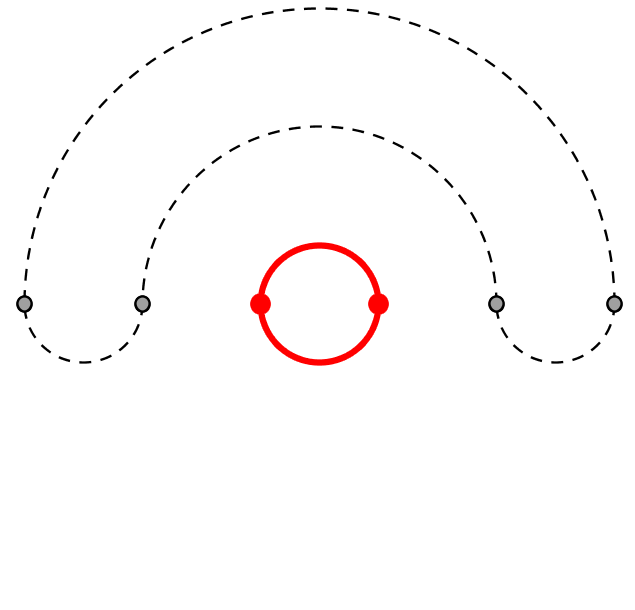}
\end{align*}
et en réarrangeant les parenthèses
\begin{align*}
\diagsc{Images_arxiv/B3_ac_col.png} \times   \left(\diagsc{Images_arxiv/B3_cb_col_rev.png} \times  \diagsc{Images_arxiv/B3_bx_col_00}\right)
 &=\diagsc{Images_arxiv/B3_ac_col.png} \times \diagsc{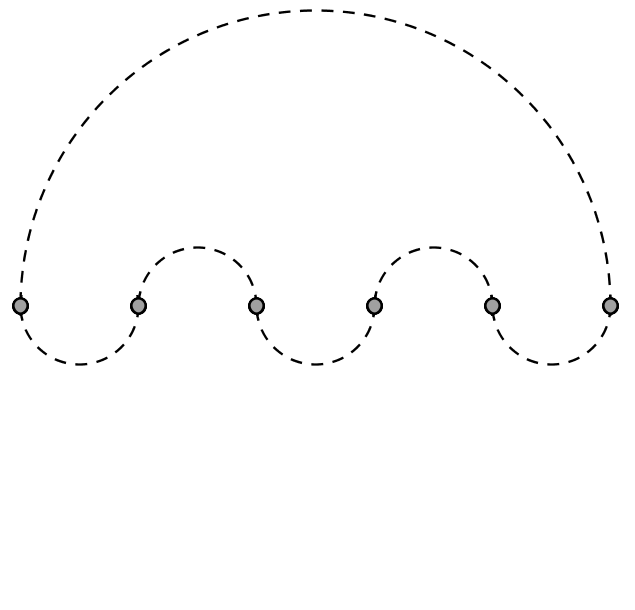}  \\
&=  \diagsc{Images_arxiv/B3_yx_col_01.png} - \diagsc{Images_arxiv/B3_yx_col_10.png}.
\end{align*}
\end{exemple}

On a donc un autre exemple de non-associativité mais qui est du à l'ordre dans lequel on fait les scissions et fusions cette fois. Il est intéressant de noter que, comme le montre l'exemple suivant, en utilisant une autre règle de multiplication on peut ne pas retrouver la non-associativité de ces éléments. En fait, on retrouve le signe du à la permutation du facteur de gauche dans toutes les règles mais le signe du au changement de chronologie dépend du choix d'orientation des cobordismes et de leurs chronologies.

\begin{exemple}\label{ex:chgtch2}
On définit la règle de multiplication $C_{ord}$ comme celle où on donne l'ordre usuel de gauche à droite pour les points de base et on oriente les scissions $x_i \chemarrow x_j$ si dans le diagramme $D_{i}$ de la procédure de construction de $M(C_{ord})$ la composante $X_i$ passant par $x_i$ est plus petite dans l'ordre induit par les points de base que la composante $X_j$ passant par $x_j$ et inversement pour $x_j \chemarrow x_j$. Autrement dit si $X_i$ passe par un point de base plus petit que tous les points de base de $X_j$ alors on prend l'orientation $x_i \chemarrow x_j$. 
%

On considère les multiplications suivante dans $OH_{C_{ord}}^3$ :
\begin{align*}
\left(\diagsc{Images_arxiv/B3_ac_col.png} \times_{C_{ord}}   \diagsc{Images_arxiv/B3_cb_col_rev.png}\right) \times_{C_{ord}}  \diagsc{Images_arxiv/B3_bx_col_00.png}
 &= \diagsc{Images_arxiv/B3_ab_col_rev.png} \times_{C_{ord}} \diagsc{Images_arxiv/B3_bx_col_00.png} \\
&= \diagsc{Images_arxiv/B3_yx_col_10.png} - \diagsc{Images_arxiv/B3_yx_col_01.png}
\end{align*}
et en réarrangeant les parenthèses
\begin{align*}
\diagsc{Images_arxiv/B3_ac_col.png} \times_{C_{ord}}   \left(\diagsc{Images_arxiv/B3_cb_col_rev.png} \times_{C_{ord}}  \diagsc{Images_arxiv/B3_bx_col_00}\right)
 &=\diagsc{Images_arxiv/B3_ac_col.png} \times_{C_{ord}} \diagsc{Images_arxiv/B3_cx_0.png}  \\
&=  \diagsc{Images_arxiv/B3_yx_col_10.png}-\diagsc{Images_arxiv/B3_yx_col_01.png}.
\end{align*}
\end{exemple}

\subsection{Comparaison de $\Hno_C$ avec $H^n$}

Les Exemples \ref{ex:H1} et \ref{ex:OH1} montrent que $OH_C^1 \simeq H^1$ puisque $\Ext^* \Z \simeq A\{1\}$. Par contre, la non-associativité montre que pour $n \ge 2$, $H^n \ne \Hno_C$ étant donné que $H^n$ est associatif. Par ailleurs, il est intéressant de noter que $H^n$ et $\Hno_C$ sont équivalents en modulo $2$, montrant que $OH^n_C$ est associatif modulo $2$.
%

\begin{proposition}\label{prop:OHvsH}
Pour tout $n \in \N$, on a l'équivalence modulo $2$ entre $H^n$ et $OH^n_C$
$$H^n \otimes_\Z \Z/2\Z \simeq \Hno_C \otimes_\Z \Z/2\Z.$$
\end{proposition}

\begin{proof}
L'idée de la preuve est très simple :  le produit extérieur modulo $2$ correspond avec le produit tensoriels des $A$ modulo $2$ et le foncteur $OF$ est équivalent à $F$ en modulo $2$, donnant directement l'isomorphisme. Plus formellement pour une collection de cercles $\{a_1, \dots, a_m\}$ dans un certain $W(a)b$ on considère l'application
\begin{align*}
a(OH^n)b \otimes_\Z \Z/2\Z &\rightarrow A^{\otimes m}\otimes_\Z \Z/2\Z \simeq a(H^n)b \otimes_\Z \Z/2\Z :\\
a_i \otimes 1 &\mapsto \bar a_i \otimes 1 :=\underbrace{1\otimes \dots \otimes 1}_{i-1} \otimes t_{a_i} \otimes \underbrace{1\otimes \dots \otimes 1}_{m-i} \otimes 1 
\end{align*}
qu'on étend linéairement et en envoyant le produit extérieur vers le produit tensoriel. Cela est bien défini puisque 
$$ x \wedge y \otimes 1 = y  \wedge x \otimes -1 = y \wedge x \otimes 1.$$
C'est clairement bijectif (on a même un inverse explicite qui consiste à envoyer le produit tensoriel vers le produit extérieur) et donc cela forme un isomorphisme de groupes gradués. Il reste à montrer qu'on respecte la multiplication pour avoir un isomorphisme d'anneaux gradués. Pour cela, il suffit de montrer que les fusions et scissions (orientées ou non) sont équivalentes modulo $2$ puisque l'ordre dans le produit extérieur n'a plus d'importance. C'est clair pour les fusions et pour les scissions on observe d'abord que l'orientation n'a plus d'importance en modulo 2 et ensuite que
$$(a_1 - a_2) \wedge x \otimes 1 = a_1 \wedge x \otimes 1 + a_2 \wedge x \otimes 1 \mapsto \bar a_1 \otimes \bar x \otimes 1 + \bar a_2 \otimes \bar x \otimes 1$$
ce qui est équivalent à la scission dans $H^n$.
\end{proof}

\chapter{Centre impair de $\Hno_C$ et cohomologie de la variété de Springer}

Dans \cite{HnCenter}, M. Khovanov montre que le centre de l'anneau $H^n$ est isomorphe à l'anneau de cohomologie de la variété de Springer pour une partition $(n,n)$ (voir Annexes, Section \ref{sec:cohomology} pour une définition générale d'anneau de cohomologie). Dans ce chapitre, on montre que de façon équivalente le centre impair de $OH^n_C$, qui est le centre généralisé pour des éléments anticommutatifs, est isomorphe à la construction impaire, "odd" en anglais, de la cohomologie de la variété de Springer $(n,n)$ introduite par A. Lauda et H. Russell dans \cite{OddSpringer}. Ce chapitre ayant pour objectif de montrer un résultat assez complexe, il est plus technique que les précédents.

\section{Centre de $H^n$ et variété de Springer}

\begin{definition}
Soit $V$ un espace vectoriel complexe de dimension finie. Un \emph{drapeau} est une suite strictement croissante de sous-espaces de $V$ :
$$\{0\} = V_0 \subset V_1 \subset \dots \subset V_k = V.$$
Un drapeau est dit \emph{complet} si pour tout $0\le i \le k$ on a 
$$\dim(V_i) = i.$$
\end{definition}

\begin{definition}
Soient $E_n$ un espace vectoriel complexe de dimension $2n$ et $z_n : E_n \rightarrow E_n$ un endomorphisme linéaire nilpotent composé de deux bloques de Jordan nilpotents de taille $n$. La variété de Springer $(n,n)$ est l'ensemble
$$\mathfrak{B}_{n,n} := \{\text{drapeaux complets dans $E_n$ stabilisés par $z_n$}\},$$
un drapeau $V_0\subset \dots \subset V_{2n}$ étant stabilisé si  $z_n V_i \subset V_i$ pour tout $0 \le i \le 2n$.
\end{definition}

De façon générale, on peut définir une variété de Springer pour toute partition $\lambda_1 + \dots + \lambda_k$ de $n$ en prenant des bloques de taille $\lambda_i$, mais dans ce travail on ne considère que le cas d'une partition de $2n$ en $n + n$ qui est le cas relié à $H^n$. Dans \cite{HnCenter}, M. Khovanov calcule une présentation de l'anneau de cohomologie de cette variété et montre ensuite qu'il est isomorphe au centre de $H^n$. Il donne en plus une construction explicite pour cet isomorphisme.
\begin{theoreme}  \label{thm:hncenter} \emph{ (M. Khovanov \cite[Théorème 1.1 et 1.2]{HnCenter})}
Le centre de $H^n$, $Z(H^n)$, est isomorphe en tant qu'anneau gradué à l'anneau de cohomologie de la variété de Springer $H(\mathfrak B_{n,n})$. De plus, ils sont tous deux isomorphes au quotient de l'anneau polynomial $\Z[X_1, \dots, X_{2n}]$, $\deg(X_i) =2$, par l'idéal engendré par
\begin{align*}
X_i^2, && i\in [1, 2n]; \\
\sum_{|I|=k} X_I, && k\in [1,2n];
\end{align*}
où $X_I = X_{i_1}\dots X_{i_k}$ pour $I = \{i_1, \dots, i_k\}$ et où la somme est prise sur tous les sous-ensembles de cardinalité $k$ de $[1,2n]$.
\end{theoreme}

Explicitement, cet isomorphisme est donné par
$$X_i \mapsto \sum_{a\in B^n} a(X_i)a$$
où  $a(X_i)a = (-1)^i 1^{\otimes (n-1)}\otimes t \in A^{(n-1)}\otimes A$ avec le $A$ séparé étant celui associé à la composante de cercle dans $W(a)a$ qui contient le $i-$ème point de base en partant de la gauche. Il montre de plus la propriété suivante qui sera utile par la suite.

\begin{proposition}\emph{(M. Khovanov \cite[preuve du Théorème 3]{HnCenter})} \label{prop:hncenter}
Un élément $z\in H^n$ est dans le centre si et seulement si
$$z = \sum_{a \in B^n} z_a$$
avec $z_a \in a(H^n)a$ et pour tout $a,b \in B^n$
$$z_a.{_a1_b} = {_a1_b}.z_b.$$
\end{proposition}

Par ailleurs, on peut montrer que le rang de l'anneau de cohomologie de la variété de Springer, vue comme un groupe abélien, est le même que la dimension de l'espace vectoriel obtenu en prenant le produit tensoriel par $\Z/2\Z$, ce qui revient à regarder les polynômes avec des coefficients modulo 2 dans la présentation. 

\begin{proposition}\label{prop:basemod2}
On a l'égalité
$$\rank H(\mathfrak B_{n,n}) = \dim_{\Z/2\Z} \bigl (H(\mathfrak B_{n,n}) \otimes_\Z \Z/2\Z \bigr).$$
En particulier, une base modulo 2 de $H(\mathfrak B_{n,n})$ est une base de $H(\mathfrak B_{n,n}) \otimes_\Z \Z/2\Z$.
\end{proposition}

\begin{proof}
$H(\mathfrak B_{n,n})$ est un groupe abélien libre (de Concini et Procesi en construisent une base dans \cite{ConciniProcesi}), donc on a une base et donc elle donne une base pour l'espace tensorisé. En effet toute relation dans l'espace tensorisé devient une relation dans le groupe quitte à multiplier par des coefficients $\pm 1$.
\end{proof}

%
%
%

\section{Centre et centre impair de $\Hno_C$}

On s'intéresse maintenant à étudier le centre de $OH^n_C$ pour le comparer à celui de $H^n$. On remarque que le centre n'a pas les propriétés voulues pour continuer l'analogie avec le cas de $H^n$ et on introduit une notion de centre impair qui correspond au centre étendu aux éléments anticommutatifs.

\subsection{Centre de $OH^n_C$}

\begin{proposition} \label{prop:centreaa}
Pour tout $z\in Z(OH^n_C)$ on a 
$$z = \sum_{a\in B^n} z_a$$
avec $z_a \in a(OH^n_C)a$, c'est-à-dire $Z(OH^n_C) \subset \bigoplus_{a\in B^n} a(OH^n_C)a$.
\end{proposition}

\begin{proof}
On considère $z \in OH^n_C$ un élément central. On peut décomposer cet élément $z = \sum_{a,b \in B^n} {_az_b}$ avec $_az_b \in a(OH^n_C)b$. Mais alors on a ${_az_b} = (1_a z) 1_b = (z 1_a)1_b = 0$ si $a \ne b$, ce qui signifie que $z = \sum_a {_az_a}$ et on note $z_a = {_az_a}$. 
\end{proof}

\begin{proposition}\label{prop:centercaract1}
Un élément $z = \sum_a {z_a} \in \bigoplus_{a\in B^n} a(OH^n_C)a$ est dans $Z(OH^n_C)$ si et seulement si 
$$z_{ab} := z_a.{_a1_b} = {_a1_b}.z_b =: z'_{ab}$$
avec $_a1_b = 1 \{n\} \in \Ext^*V(W(a)b)$ et pour tout $a,b \in B^n$ et $y \in a(OH^n_C)b$ on a
$$z_{ab}\wedge y = y \wedge z_{ab}.$$
\end{proposition}

\begin{proof}
Soit $x = \sum_{a,b \in B^n} {_ax_b} \in OH^n_C$, on a 
\begin{align*}zx &= \sum_{a,b} z_a.{_ax_b}, & &\text{et} &
xz &= \sum_{a,b} {_ax_b}.z_b.\end{align*}
Puisque les morphismes de modules obtenus par $OF$ pour ces deux multiplications sont respectivement les mêmes que ceux pour $z_a {_a1_b}$ et ${_a1_b}z_b$ et qu'ils ne sont composés que de fusions par la Proposition \ref{prop:aaabfusions}, on obtient
\begin{align*}z_a.{_ax_b} &= z_{ab} \wedge {_ax_b} \in a(OH^n_C)b, \\
{_ax_b}.z_b &= {_ax_b}\wedge z'_{ab} \in a(OH^n_C)b\end{align*}
On en conclut que
$$ z_a.{_ax_b} =  {_ax_b}.z_b$$
quand on a les hypothèses de la suffisance en prenant $y = {_ax_b}$.\\

Si $z$ commute avec tout élément, en particulier on doit avoir
$$z_a.{_a1_b} = z1_a.{_a1_b} = z  {_a1_b} =  {_a1_b} z =  {_a1_b}.z_b.$$
Par ailleurs pour tout $y \in a(OH^n_C)b$, on peut trouver des $y_a + y_b \in a(OH^n_C)a \oplus b(OH^n_C)b$ tels que $y_a.{_a1_b} = y = {_a1_b}.y_b$ et l'hypothèse $zy = yz$ donne la condition voulue.
\end{proof}

\begin{corollaire}\label{prop:centercaract2}
Un élément $z = \sum_a z_a$ est dans le centre de $OH^n_C$ si et seulement si pour tout $a,b \in B^n$
$$z_a.{_a1_b} = {_a1_b}.z_b$$
et pour tout $a \in B^n$ et $y \in a(OH^n_C)a$
$$z_a \wedge y = y \wedge z_a.$$
\end{corollaire}

\begin{proof}
Puisque $a(OH^n_C)a\times a(OH^n_C)b$ n'est composé que de fusions par la Proposition \ref{prop:aaabfusions}, un élément homogène est soit envoyé soit sur $0$ soit sur un élément homogène de la même longueur vu en tant que produit extérieur, c'est-à-dire que s'il commute dans $a(OH^n_C)a$ alors il commute dans $a(OH^n_C)b$, ce qui permet de restreindre la seconde condition de la proposition précédente.
\end{proof}

\begin{remarque}\label{rem:commute}
On peut décomposer $z_{a} = \sum_i z_{a}^i$ en une somme d'éléments homogènes de degré $i$. Si $z_{a}^i = x_1 \wedge \dots \wedge x_m$ avec $m < |W(a)a| = n$ alors la seconde condition de la proposition est équivalente à  $ \deg( z_{a}^i) = \deg({_a1_a}) \mod 4$ puisque cela signifie que $m$ est pair. Dans le cas où $m = |W(a)b|$ alors la condition est toujours trivialement obtenue. En particulier on sait que $\deg({_a1_a}) = 0$ et donc un élément du centre est une combinaison linéaire d'éléments de degrés égaux à 0 modulo 4 et de degrés $2n$.
\end{remarque}

\begin{lemme}\label{lem:associacentre}
Pour tout $x_a, x'_a \in a(OH^n_C)a$, $x_b, x'_b \in b(OH^n_C)b$ et $y \in a(OH^n_C)b$ on a
\begin{align*}
(x_ay)x_b &= x_a(yx_b), \\
(x_ax'_a)y &= x_a(x'_ay), \\
(yx_b)x'_b &= y(x_bx'_b).
\end{align*}
\end{lemme}

\begin{proof}
La preuve est directe une fois qu'on a remarqué que toutes les multiplications n'utilisent que des fusions par la Proposition \ref{prop:aaabfusions} et que le produit extérieur est associatif.
\end{proof}

\begin{proposition}
La sous-anneau du centre de $OH^n_C$ est associatif.
\end{proposition}

\begin{proof}
On applique le Lemme \ref{lem:associacentre} et la Proposition \ref{prop:centreaa}.
\end{proof}

Quelle que soit la règle de multiplication qu'on choisit, le centre de $OH^n_C$ est le même.

\begin{proposition}\label{prop:centreiso}
Pour $C$ et $C'$ des règles de multiplications quelconques, on a l'isomorphisme d'anneaux gradués
$$Z(OH^n_C) \simeq Z(OH^n_{C'}).$$
\end{proposition}

\begin{proof}
Les cobordismes pour définir les multiplications de la Proposition \ref{prop:centercaract2} ne sont composés que de fusions et donc, puisqu'ils sont équivalents, ils ont la même image par $OF$ par la Proposition \ref{prop:isofusions}. On peut alors prendre l'isomorphisme induit par l'identification de $OH^n_C$ et $OH^n_{C'}$ au groupe abélien $OH^n$.
\end{proof}

\begin{exemple}
On calcule $Z\left( {OH}^2_C\right)$. $B^2$ est composé des éléments
\begin{align*}
a = \deuxdiag{Images_arxiv/B2_2.png}, &&& b = \deuxdiag{Images_arxiv/B2_1.png},
\end{align*}
qui donnent les collections de cercles
\begin{align*}
W(a)a = \diagg{Images_arxiv/B2_bb_note.png}, &&& W(a)b = \diagg{Images_arxiv/B2_ba_note_2.png},\\
W(b)a = \diagg{Images_arxiv/B2_ab_note_2.png}, &&& W(b)b = \diagg{Images_arxiv/B2_aa_note.png},
\end{align*}
et on note $a_1$ le cercle extérieur de $W(a)a$ et $a_2$ l'intérieur.  De même, on note $b_1$ le cercle de gauche de $W(b)b$ et $b_2$ le droit, ainsi que $c_1$ et $d_1$ les cercles de $W(a)b$ et $W(b)a$. 

On considère $z \in Z\left({OH}^2_C\right)$ un élément central qui se décompose en $z = \sum_{a \in B^n} {z_a}$ avec
\begin{align*}
z_a &= x_0.1_a + x_1.a_1 + x_2.a_2 + x_3.a_1 \wedge a_2 \in a\left({OH}^2_C\right)a, \\
z_b &= y_0.1_b + y_1.b_1 + y_2.b_2 + y_3.b_1 \wedge b_2 \in b\left({OH}^2_C\right)b.
\end{align*}
Cela se réécrit en calcul diagrammatique comme
\begin{align*}
z_a &= x_0.\diagg{Images_arxiv/B2_bb_00.png} + x_1.\diagg{Images_arxiv/B2_bb_10.png} + x_2.\diagg{Images_arxiv/B2_bb_01.png} +  x_3. \diagg{Images_arxiv/B2_bb_11.png}  &\in a({OH}_C^2)a, \\
z_b &= y_0.\diagg{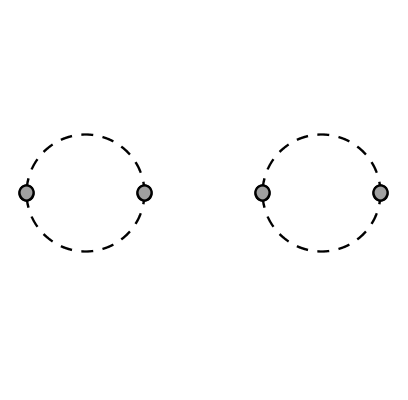} + y_1.\diagg{Images_arxiv/B2_aa_10.png} + y_2.\diagg{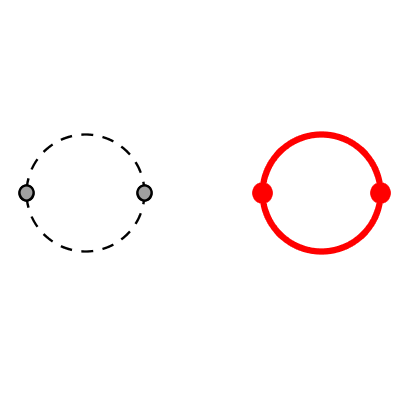} + y_3.\diagg{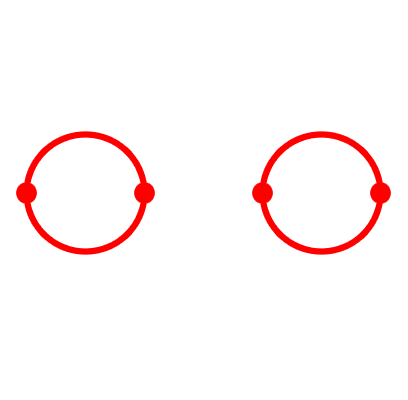}  &\in b({OH}_C^2)b.
\end{align*}
On applique la Proposition \ref{prop:centercaract2} afin d'obtenir des conditions sur les $x_i$ et $y_i$. Puisque l'anneau est gradué, on peut regarder les conditions de la proposition degré par degré. 

\begin{enumerate}
\item Pour $\deg(z) = 0$ : $z_a = x_0.1_a$ et $z_b = y_0.1_b$, on a 
$$z_a.{_a1_b} = x_0.\diagg{Images_arxiv/B2_bb_00.png} \times_C \diagg{Images_arxiv/B2_ba_0.png} = x_0. \diagg{Images_arxiv/B2_ba_0.png}, $$
$${_a1_b}.z_b = \diagg{Images_arxiv/B2_ba_0.png} \times_C y_0.\diagg{Images_arxiv/B2_aa_00.png} = y_0. \diagg{Images_arxiv/B2_ba_0.png},$$
et donc $x_0 = y_0$.

\item Pour $\deg(z) = 2$ : $z_a = x_1.a_1 + x_2.a_2$ et $z_b = y_1.b_1 + y_2.b_2$, on a
\begin{align*}
z_a.a_2 &= \left(x_1.\diagg{Images_arxiv/B2_bb_10.png} + x_2. \diagg{Images_arxiv/B2_bb_01.png}\right) \times_C \diagg{Images_arxiv/B2_bb_01.png} =  -x_1.\diagg{Images_arxiv/B2_bb_11.png},\\
a_2.z_a &=  \diagg{Images_arxiv/B2_bb_01.png} \times_C \left(x_1.\diagg{Images_arxiv/B2_bb_10.png} + x_2. \diagg{Images_arxiv/B2_bb_01.png}\right) =  x_1.\diagg{Images_arxiv/B2_bb_11.png}, \\
\end{align*}
et donc $x_1=0$. Par des calculs similaires, on obtient $x_2 = y_1=y_2=0$.

\item Pour $\deg(z) = 4$ : $z_a =  x_3.a_1 \wedge a_2 $ et $z_b =  y_3.b_1 \wedge b_2$ et la proposition ne livre aucune contrainte puisque
$$z_a.{_a1_b} = x_3.\diagg{Images_arxiv/B2_bb_11.png}  \times_C \diagg{Images_arxiv/B2_ba_0.png}  = 0$$
et de même pour les autres équations.
\end{enumerate}
Dès lors on obtient que
$$Z({OH}^2_C) = \{k.(1_a+1_b) + x.a_1\wedge a_2 + y.b_1 \wedge b_2 | k,x,y \in \Z \}.$$
%
%

\end{exemple}

Le centre de $OH^n_C$ ne remplit pas les critères souhaités pour continuer la discussion en analogie avec le cas de $H^n$. En effet, on remarque que sa dimension graduée est différente de celle du centre de $H^n$ (et de l'anneau de cohomologie de la variété de Springer) :
$$\dim_q\bigl(Z(H^2)\bigr) = 1 + 3q + 2q^2 \ne 1+2q^2 = \dim_q\left(Z\left({OH}_C^2\right)\right).$$
On introduit alors dans la prochaine section le centre impair qui possède des propriétés plus intéressantes et comparables à celles du centre de $H^n$.

\subsection{Centre impair de $\Hno_C$}

Dans le cadre des superalgèbres, c'est-à-dire des algèbres graduées sur $\Z/2\Z$ (voir\cite[Chapitre 3]{superalgebra} pour plus de détails) on peut définir la notion de supercentre en fonction du supercommutateur 
$$[x,y] = xy - (-1)^{\deg(x)\deg(y)}yx$$
pour $x,y$ dans une superalgèbre $\mathcal A$ donnant donc le supercentre
$$SZ(\mathcal A) = \{z \in \mathcal A | [z,x] = 0, \forall x\in \mathcal A\}.$$
Cette définition est aussi parfois utilisée en topologie algébrique, notamment dans \cite[Section 3.2]{Hatcher}, de sorte que l'anneau de cohomologie soit "supercommutatif", dans ce cas on l'appelle parfois tout simplement "commutatif", "skew-commutative", "anticommutative" ou encore "graded commutative".

On propose d'étendre cette définition à $OH^n_C$ et de calculer son "supercentre", qu'on appelle centre impair, cela dans le but d'obtenir les éléments qui commutent "à antisymétrie" près. Pour ce faire, on définit le \emph{degré extérieur} d'un élément homogène $z$ de $a(OH^n_C)b$ comme étant
$$p(z) := \frac{\deg(z)-\deg(_a1_b)}{2} =  \frac{\deg(z)-n + |W(b)a|}{2} $$
et on remarque aisément que $p(z)$ est le nombre de facteurs du produit extérieur qui constitue $z$, c'est-à-dire que 
$$p(a_1 \wedge \dots \wedge a_m) = m.$$ 

\begin{definition}
On définit le \emph{centre impair} de $\Hno_C$ comme
$$OZ(OH^n_C) = \{z \in OH^n_C | zx = (-1)^{p(x)p(z)}xz , \forall x \in OH^n_C\}.$$
\end{definition}

\begin{remarque}
Il faut faire attention que $OH^n_C$ n'est pas une superalgèbre pour le degré extérieur. En effet, on peut avoir
$$p(z_1z_2) \ne p(z_1) + p(z_2) \mod 2.$$
Par exemple, on prend le produit suivant dans $\widetilde{OH}^2$ avec à gauche des degrés extérieurs valant $1$ et $0$ et valant $2$ à droite :
$$\smalldiag{Images_arxiv/B2_ba_1.png} \times \smalldiag{Images_arxiv/B2_ab_0.png} =  -\smalldiag{Images_arxiv/B2_bb_11.png}$$
On ne peut pas non plus prendre simplement le degré divisé par $2$ pour obtenir le "degré impair" puisqu'on obtiendrait des degrés non-entiers pour des éléments dans $W(b)$ possédant un nombre de composantes de parité différente de $n$. Par contre, on peut montrer que $\bigoplus_{a\in B^n} a(OH^n_C)a \subset OH^n_C$ est une superalgèbre pour le degré extérieur modulo 2 car alors on a la relation
$p(x) = \deg(x)/2.$
\end{remarque}

\begin{proposition}\label{prop:inaOHa}
Pour tout $z \in OZ(OH^n_C)$ on a
$$z = \sum_{a\in B^n} z_a$$
avec $z_a \in a(OH^n_C)a$ c'est-à-dire $OZ(OH^n_C) \subset \bigoplus_{a\in B^n} a(OH^n_C)a$.
\end{proposition}

\begin{proof}
La preuve est similaire à celle de la Proposition \ref{prop:centreaa} pour le centre.
\end{proof}

\begin{proposition}
Le centre impair $OZ(OH^n_C)$ est un sous-anneau associatif.
\end{proposition}

\begin{proof}
L'associativité est directe par le Lemme \ref{lem:associacentre} et la Proposition \ref{prop:inaOHa}. Par ailleurs, il est clair que le centre impair est stable pour l'addition et donc on ne considère que la multiplication. Pour cela, on remarque d'abord que si $x_1, x_2 \in a(OH^n_C)a$ alors
$$p(x_1x_2) \equiv p(x_1) + p(x_2) \mod 2$$
puisque la multiplication est juste donnée par le produit extérieur. De là on généralise facilement par la Proposition \ref{prop:inaOHa} que pour tout $z_1, z_2 \in OZ(OH^n_C)$ on a de même
$$p(z_1z_2) \equiv p(z_1) + p(z_2) \mod 2.$$
Soient $z_1, z_2 \in OZ(OH^n_C)$, on obtient pour tout $x \in OH^n_C$ 
\begin{align*}
(z_1z_2)x =& z_1(z_2x) \\
&=  (-1)^{p(z_2)p(x)}  z_1(xz_2) \\
&=  (-1)^{p(z_2)p(x)} (z_1x)z_2 \\
&=  (-1)^{p(z_1)p(x)}  (-1)^{p(z_2)p(x)}  (xz_1)z_2 \\
&= (-1)^{(p(z_1)+p(z_2))p(x)} x(z_1z_2)
\end{align*}
et par ce qu'on a montré juste au dessus $p(z_1z_2) = p(z_1)+p(z_2) \mod 2$, donc
$$(-1)^{p(z_1z_2)p(x)} = (-1)^{(p(z_1)+p(z_2))p(x)},$$
ce qui livre la supercommutativité de $z_1z_2$.
\end{proof}

\begin{proposition}\label{prop:caractsupercentre}
Un élément $z = \sum_a z_a$ est dans $OZ(OH^n_C)$ si et seulement si
\begin{align}
z_a{_a1_b} = {_a1_b}z_b \label{eq:supercentre}
\end{align}
 pour tout $a,b \in B^n$.
\end{proposition}

\begin{proof}
La preuve est similaire à celle de la Proposition \ref{prop:centercaract1}, si ce n'est qu'on obtient l'anticommutativité par antisymétrie du produit extérieur. Il faut aussi remarquer que
$$p(z_a.{_a1_b}) = p(z_a)$$
pour les mêmes raisons que dans la preuve de la Proposition \ref{prop:centercaract2}, c'est-à-dire puisque toutes les multiplications sont définies par des fusions qui ne changent pas la longueur d'un élément du produit extérieur.
\end{proof}

\begin{remarque}
En particulier, on a $Z(OH^n_C) \subset OZ(OH^n_C)$ et, par la Remarque \ref{rem:commute}, les éléments de degré extérieur pair ou maximal du centre impair constituent le centre.
\end{remarque}

Par analogie avec la Proposition \ref{prop:centreiso}, on obtient la proposition suivante :

\begin{proposition}\label{prop:supercentreiso}
Soit $C$ et $C'$ des règles de multiplication. Il y a un isomorphisme d'anneaux gradués
$$OZ(OH^n_C) \simeq OZ(OH^n_{C'}).$$
\end{proposition}

\begin{proof}
Les cobordismes pour définir les multiplications de la Proposition \ref{prop:caractsupercentre} ne sont composés que de fusions et donc ont la même image par le foncteur $OF$.
\end{proof}

%
%
%

\begin{exemple} \label{ex:centreimpaire}
On calcule le centre impair de $OH^2_C$. On rappelle que $B^2$ est donné par
\begin{align*}
a = \deuxdiag{Images_arxiv/B2_2.png}, &&& b = \deuxdiag{Images_arxiv/B2_1.png},
\end{align*}
qui livrent les diagrammes
\begin{align*}
W(a)a = \diagg{Images_arxiv/B2_bb_note.png}, &&& W(a)b = \diagg{Images_arxiv/B2_ba_note_2.png},\\
W(b)a = \diagg{Images_arxiv/B2_ab_note_2.png}, &&& W(b)b = \diagg{Images_arxiv/B2_aa_note.png}.
\end{align*}
On prend un élément $z = \sum_{a \in B^n} {z_a} \in OZ\left({OH}^2_C\right)$  avec
\begin{align*}
z_a &= x_0.\diagg{Images_arxiv/B2_bb_00.png} + x_1.\diagg{Images_arxiv/B2_bb_10.png} + x_2.\diagg{Images_arxiv/B2_bb_01.png} + x_3. \diagg{Images_arxiv/B2_bb_11.png} &\in a({OH}_C^2)a, \\
z_b &= y_0.\diagg{Images_arxiv/B2_aa_00.png} + y_1.\diagg{Images_arxiv/B2_aa_10.png} + y_2.\diagg{Images_arxiv/B2_aa_01.png} + y_3.\diagg{Images_arxiv/B2_aa_11.png} &\in b({OH}_C^2)b.
\end{align*}
On applique la Proposition \ref{prop:caractsupercentre} pour trouver des contraintes sur les $x_i$ et $y_i$ et, puisque l'anneau est gradué, on peut regarder la supercommutativité degré par degré.
\begin{enumerate}

\item Pour $\deg(z) = 0$ : $z_a = x_0.1_a$ et $z_b = y_0.1_b$, on a 
\begin{align*}
z_a.{_a1_b} &= x_0.\diagg{Images_arxiv/B2_bb_00.png} \times_C \diagg{Images_arxiv/B2_ba_0.png} = x_0. \diagg{Images_arxiv/B2_ba_0.png}, \\
{_a1_b}.z_b &= \diagg{Images_arxiv/B2_ba_0.png} \times_C y_0.\diagg{Images_arxiv/B2_aa_00.png} = y_0. \diagg{Images_arxiv/B2_ba_0.png},
\end{align*}
donc $x_0 = y_0$.

\item Pour $\deg(z)=2$ :  on a
\begin{align*}
\left(x_1.\diagg{Images_arxiv/B2_bb_10.png} + x_2. \diagg{Images_arxiv/B2_bb_01.png}\right) \times_C \diagg{Images_arxiv/B2_ba_0.png} = (x_1+x_2).\diagg{Images_arxiv/B2_ba_1.png}, \\
\diagg{Images_arxiv/B2_ba_0.png} \times_C  \left(y_1.\diagg{Images_arxiv/B2_aa_10.png} + y_2. \diagg{Images_arxiv/B2_aa_01.png}\right) = (y_1+y_2).\diagg{Images_arxiv/B2_ba_1.png} ,
\end{align*}
et
\begin{align*}
  \left(y_1.\diagg{Images_arxiv/B2_aa_10.png} + y_2. \diagg{Images_arxiv/B2_aa_01.png}\right) \times_C \diagg{Images_arxiv/B2_ab_0.png}  = (y_1+y_2).\diagg{Images_arxiv/B2_ab_1.png} , \\
\diagg{Images_arxiv/B2_ab_0.png}  \times_C \left(x_1.\diagg{Images_arxiv/B2_bb_10.png} + x_2. \diagg{Images_arxiv/B2_bb_01.png}\right) = (x_1+x_2).\diagg{Images_arxiv/B2_ab_1.png} ,
\end{align*}
donc $x_1+x_2 = y_1+y_2$ qui se réécrit comme 
$$y_2 = x_1+x_2-y_1.$$

\item Pour $\deg(z) = 4$ : la proposition ne livre aucune contrainte puisque
$$x_3.\diagg{Images_arxiv/B2_bb_11.png}  \times_C \diagg{Images_arxiv/B2_ba_0.png}  = 0 =  \diagg{Images_arxiv/B2_ba_0.png} \times_C y_3.\diagg{Images_arxiv/B2_aa_11.png},$$
$$y_3.\diagg{Images_arxiv/B2_aa_11.png}  \times_C \diagg{Images_arxiv/B2_ab_0.png}  = 0 =  \diagg{Images_arxiv/B2_ab_0.png} \times_C x_3.\diagg{Images_arxiv/B2_bb_11.png}.$$
\end{enumerate}
Dès lors on obtient que
\begin{align*}
OZ(OH^2_C) = \{&k.(1_a+1_b) \\
 &+ x_1(a_1+b_2) + x_2(a_2+b_2) + y_1(b_1-b_2)\\
 &+ x_3.a_1\wedge a_2 + y_3.b_1 \wedge b_2 &| k,x_1,x_2,x_3,y_1,y_3 \in \Z \},
\end{align*}
avec comme prévu $Z({OH}^2_C) \subset OZ({OH}_C^2)$ ainsi que
$$\rank_q\left(OZ\left({OH}_C^2\right)\right) = 1+3q+2q^2 = \rank_q\bigl(Z(H^2)\bigr).$$

\end{exemple}

\section{Construction impaire de la cohomologie de la variété de Springer}\label{sec:oddspringer}

Dans \cite{OddSpringer}, A. Lauda et H. Russell proposent une construction impaire de l'anneau de cohomologie de la variété de Springer $\mathfrak B_{n,n}$ (et même pour toute partition) basée sur les polynômes impairs
$$OPol_{2n} := \frac{\Z\langle x_1, \dots, x_{2n}\rangle}{\langle x_ix_j+x_jx_i = 0 \text{ pour tout $i\ne j$} \rangle},$$
où $\Z\langle x_1, \dots, x_n \rangle$ sont les polynômes ordonnés avec degré $\deg(x_i) = 2$, formant donc un module gradué sur $\Z$.
L'objectif de cette section est de définir et étudier quelques propriétés de cette construction qu'on note simplement "cohomologie impaire de  $\mathfrak B_{n,n}$" pour être concis.\\


\begin{definition}On définit la \emph{cohomologie impaire de la variété de Springer $(n,n)$} comme étant le module gradué sur $OPol_{2n}$
$$OH(\mathfrak B_{n,n}) := OPol_{2n}/OI_n$$
où $OI_n$ est l'idéal à gauche engendré par l'ensemble $OC_n$

\begin{align}
OC_n := \left\{\epsilon^{I}_{r} := \sum_{1 \le {i_1} < \dots < {i_r} \le 2n} x_{i_1}^I \dots x_{i_r}^I \bigg| k \in [1,n], |I| = n+k, r \in [n-k+1, n+k]\right\}
\end{align}
pour tout $I$ sous-ensemble ordonné de $\{1, \dots, 2n\}$ de cardinalité $n+k$ et 
$$x_{i_j}^I := \begin{cases}
0 &\text{ si } i_j \notin I, \\
(-1)^{I(i_j)-1}x_{i_j} &\text{ sinon,}
\end{cases} $$
avec $I(i_j)$ la position de $i_j$ dans $I$. 
\end{definition}

\begin{exemple}
$OH(\mathfrak{B}_{1,1})$ est donné par les polynômes impairs à deux variables $x_1$ et $x_2$ quotientés par les relations
\begin{align*}
x_1 - x_2 &= 0,\\
x_1x_2 &= 0,
\end{align*}
c'est-à-dire les polynômes avec une seule variable $t$ et $t^2 = 0$ donnant donc
$$OH(\mathfrak{B}_{1,1}) \simeq \frac{\Z[t]}{t^2} \simeq OH_C^1 = OZ\left(OH_C^1\right).$$
\end{exemple}

\begin{exemple}
$OH(\mathfrak{B}_{2,2})$ est donné par les polynômes impairs à quatre variables $x_1,x_2,x_3$ et $x_4$ quotientés par les relations
\begin{align*}
x_1 - x_2 + x_3 - x_4 &= 0, \\
-x_ix_j+x_ix_k-x_jx_k &= 0 &\forall i<j<k \in [1,4] , \\
-x_1x_2+x_1x_3-x_1x_4-x_2x_3+x_2x_4-x_3x_4 &= 0, \\
x_ix_jx_k &= 0 &\forall i<j<k \in [1,4] , \\
-x_1x_2x_3+x_1x_2x_4-x_1x_3x_4+x_2x_3x_4 &= 0, \\
x_1x_2x_3x_4 &= 0,
\end{align*}
qui peuvent se simplifier un peu en
\begin{align*}
x_1 - x_2 + x_3 - x_4 &= 0, \\
-x_ix_j+x_ix_k-x_jx_k &= 0 &\forall i<j<k \in [1,4] , \\
-x_1x_2-x_3x_4 &= 0, \\
x_ix_jx_k &= 0 &\forall i<j<k \in [1,4] , \\
x_1x_2x_3x_4 &= 0.
\end{align*}
On peut donc engendrer $OH(\mathfrak{B}_{2,2})$ avec $1, x_1, x_2, x_3,x_1x_2 $ et $x_1x_3$ car
\begin{align*}
x_4 &= x_1 - x_2 + x_3, \\
x_1x_4 &= x_1( x_1 - x_2 + x_3 ) = -x_1x_2+x_1x_3, &
x_2x_3 &= -x_1x_2 + x_1x_3, \\
x_2x_4 &= -x_1x_2 + x_1x_4,&
x_3x_4 &= -x_1x_2,
\end{align*} 
et tous les polynômes de degrés supérieurs sont nuls. De plus ces éléments sont linéairement indépendants, ce qui est clair pour $1,x_1, x_2,x_3$ mais peut-être moins pour $x_1x_2$ et $x_1x_3$. Néanmoins, la Proposition \ref{prop:rankhoh} permet de le justifier. On observe alors qu'en envoyant
\begin{align*}
1 &\mapsto \diagg{Images_arxiv/B2_bb_00.png}  + \diagg{Images_arxiv/B2_aa_00.png},&
x_1 &\mapsto \diagg{Images_arxiv/B2_bb_10.png}  + \diagg{Images_arxiv/B2_aa_10.png},\\
x_2 &\mapsto \diagg{Images_arxiv/B2_bb_01.png}  + \diagg{Images_arxiv/B2_aa_10.png},&
x_3 &\mapsto \diagg{Images_arxiv/B2_bb_01.png}  + \diagg{Images_arxiv/B2_aa_01.png},\\
x_1x_2 &\mapsto  \diagg{Images_arxiv/B2_bb_11.png},&
x_1x_3 &\mapsto  \diagg{Images_arxiv/B2_aa_11.png},
\end{align*}
on obtient un isomorphisme de groupes $OH(\mathfrak{B}_{2,2}) \simeq_{ab} OZ(\widetilde{OH}^2)$ car en utilisant les notations de l'Exemple \ref{ex:centreimpaire} on a
\begin{align*}
a_1 + b_2 &= (a_1+b_1)-(a_2+b_1)+(a_2+b_2), &
b_1 - b_2 &= (a_2+b_1) - (a_2+b_2), \\
a_1 + b_1 &= (a_1+b_2) - (b_1 - b_2), &
a_2 + b_1 &= (a_2 + b_2) + (b_1 - b_2).
\end{align*}
On montre ensuite par de simples calculs dans $\widetilde{OH}^2$ que les relations de $OH(\mathfrak{B}_{2,2})$ sont dans le noyau de l'application, de sorte que l'isomorphisme de groupe soit un isomorphisme d'anneaux gradués bien défini par l'anticommutativité du produit extérieur et des polynômes impairs.
\end{exemple}

\begin{remarque}
On note que tout polynôme de degré strictement plus grand que $n$ dans $OH(\mathfrak B_{n,n})$ est nul puisqu'en prenant $I$ qui contient exactement les variables du polynôme et $r = |I|$ on obtient n'importe quel polynôme de tel degré.
\end{remarque}

En général, la cohomologie impaire de la variété de Springer pour une partition quelconque n'est pas un anneau et est donc juste un module sur $OPol_{2n}$. Mais puisque la hauteur de la partition $(n,n)$ est de deux, on a le lemme suivant qui permet de montrer que pour cette partition particulière la cohomologie impaire donne un anneau.

\begin{lemme}\emph{(\cite[Lemme 3.6]{OddSpringer})}\label{lem:xi2}
Pour tout $i \in [1,2n]$, on a
$$x_i^2 \in OI_n.$$
\end{lemme}

\begin{proposition}
$OH(\mathfrak B_{n,n})$ possède une structure d'anneau induite par la structure d'anneau des polynômes impairs.
\end{proposition}

\begin{proof}
Il suffit de calculer que l'idéal à droite engendré par $OC_n$ est égal à l'idéal à gauche $OI_n$, c'est-à-dire que
$$OC_n.OPol_{2n} = OPol_{2n}.OC_n = OI_n.$$
Soient $i \in [1,2n], k\in [1,n]$ et $r \in [n-k+1, n+k]$. On considère
$$\epsilon^{I}_{r}.x_i \in OC_n.OPol_{2n}.$$
Si $r = 0 \mod 2$, alors $x_i$ commute avec tous les éléments de $\epsilon^{I}_{r}$ qui ne contiennent pas de $x_i$, mais on peut obtenir un signe pour les autres. On élimine ces termes par le Lemme \ref{lem:xi2}. 
Dans le cas impair, on a envie de prendre $-x_i.\epsilon^{I}_{r} \in OPol_{2n}.OC_n $, mais cela ne marche pas en général puisque, à nouveau, les termes de $\epsilon^{I}_{r} $ contenant un $x_i$ posent problème car commutent sans changer de signe. Par le même raisonnement, on corrige cela en utilisant le lemme.
\end{proof}

L'anneau de cohomologie et la cohomologie impaire de la variété de Springer sont semblables en tant que groupes, comme le montre la proposition suivante :

\begin{proposition}\emph{(\cite[Théorème 3.8 et Corollaire 3.9]{OddSpringer})}\label{prop:rankhoh}
$H(\mathfrak B_{n,n})$ et $OH(\mathfrak B_{n,n})$ sont des groupes abéliens libres gradués de même rang et ce rang est donné par le coefficient binomial quantique (voir Section \ref{sec:quantumnumber} des Annexes)
$$\rank_q(H(\mathfrak B_{n,n})) = \rank_q (OH(\mathfrak B_{n,n})) = \begin{bmatrix}2n \\ n\end{bmatrix}_q.$$
\end{proposition}
En particulier on a donc
$$\rank (H(\mathfrak B_{n,n})) =  \rank (OH(\mathfrak B_{n,n})) = \begin{pmatrix}2n \\ n\end{pmatrix}.$$
On note $OB$ une certaine base monomiale de $OH(\mathfrak B_{n,n})$ et $B$ la base correspondante de $H(\mathfrak B_{n,n})$, la construction de ces bases étant expliquée dans \cite{OddSpringer} et \cite{ConciniProcesi}. Elles se correspondent dans le sens où les monômes qui constituent $OB$ sont les même que ceux de $B$.
De plus, on a un isomorphisme d'espaces vectoriels gradués
$$OH(\mathfrak B_{n,n}) \otimes_\Z \Z/2\Z \simeq H(\mathfrak B_{n,n}) \otimes_\Z \Z/2\Z$$
qui fait se correspondre les projections de $B$ et de $OB$ dans le produit tensoriel, notés $B_{\Z_2}$ et $OB_{\Z_2}$.

\begin{proposition}\label{prop:baseOBmod2}
On a l'égalité
$$\rank OH(\mathfrak B_{n,n}) = \dim_{\Z/2\Z} \bigl (OH(\mathfrak B_{n,n}) \otimes_\Z \Z/2\Z \bigr).$$
En particulier, une base modulo 2 de $OH(\mathfrak B_{n,n})$ est une base de $OH(\mathfrak B_{n,n}) \otimes_\Z \Z/2\Z$.
\end{proposition}

\begin{proof}
On utilise les mêmes arguments que pour la Proposition \ref{prop:basemod2}.
\end{proof}

%
%
%
%

\section{Isomorphisme entre $OZ\left(OH_C^n\right)$ et $OH(\mathfrak B_{n,n})$}

Cette section est dédiée à la preuve du résultat principal du chapitre qui consiste à construire un isomorphisme entre $OZ\left(OH_C^n\right)$ et $OH(\mathfrak B_{n,n})$.

\begin{theoreme}\label{thm:supercentre}
Le centre impair de $OH^n_C$ est isomorphe en tant qu'anneau gradué à la construction impaire de l'anneau de cohomologie de la variété de Springer pour une partition $(n,n)$
$$OZ(OH^n_C) \simeq OH(\mathfrak B_{n,n}).$$
\end{theoreme}

Par analogie à la construction explicite de l'isomorphisme du Théorème \ref{thm:hncenter}, on propose de définir un homomorphisme 
$$s : OH(\mathfrak B_{n,n}) \rightarrow OZ(OH_C^n)$$
similaire et adapté pour nos besoins qui induit le même homomorphisme en modulo 2. On construit dans un premier temps cet homomorphisme en le définissant sur les polynômes impairs et puis en montrant que les éléments de l'idéal par lequel on quotiente pour définir $OH(\mathfrak B_{n,n})$ sont dans le noyau, de sorte à obtenir le morphisme induit $s$. On montre dans un second temps que cet homomorphisme est injectif en montrant que si on a une relation dans l'image de la base de $OH(\mathfrak B_{n,n})$ elle remonte à une relation dans la base. Enfin on montre qu'on a l'égalité des rangs de $OZ(OH^n_C)$ et de $OH(\mathfrak B_{n,n})$ en regardant ces anneaux comme des groupes abéliens libres, donnant la surjectivité et concluant la preuve.

\subsection{Définition de l'homomorphisme $s$}

On définit d'abord le morphisme d'anneaux gradués suivant :
$$s_0 : OPol_{2n} \rightarrow \Hno_C : x_i \mapsto \sum_{a \in B^n} a_i$$
avec $a_i \in a(\Hno_C)a$ la composante de cercle de $W(a)a$ qui touche le $i$-ème point de base à partir de la gauche. On remarque que, par définition du produit dans $OH^n_C$, on a 
$$s_0(x_ix_j) = \left(\sum_{a \in B^n} a_i\right).\left(\sum_{a \in B^n} a_j\right) = \sum_{a \in B^n} a_i \wedge a_j$$
et donc en général
\begin{equation}s_0(x_{i_1}\dots x_{i_r}) = \sum_{a \in B^n} a_{i_1} \wedge \dots \wedge a_{i_r}. \label{eq:poltoprodext}\end{equation}
Ceci montre que $s_0$ est un morphisme d'anneau et qu'il est bien défini puisqu'on retrouve l'antisymétrie des polynômes impairs dans le produit extérieur.

\begin{lemme}
L'image de $s_0$ est dans le centre impair de $OH^n_C$
$$s_0(OPol_{2n}) \subset OZ(OH^n_C).$$
\end{lemme}

\begin{proof}
On utilise la Proposition \ref{prop:caractsupercentre}. Pour tout $a,b \in B^n$ on a l'égalité
$$\left(s_0(x_i)\right){_a1_b} = {_a1_b}\left(s_0(x_i)\right) $$
 car pour effectuer les produits on fusionne ensemble les composantes de $W(a)a$ et $W(b)b$ qui passent par le $i$-ème point de base. L'image de $x_i$ par $s_0$ étant la somme de ces composantes, on obtient l'égalité voulue. On étend ensuite à tout le domaine puisque le centre impair est un sous-anneau de $OH^n_C$ et que $s_0$ est un homomorphisme d'anneaux.
\end{proof}

On a donc un homomorphisme
$$s_0 : OPol_{2n} \rightarrow  OZ(OH^n_C)$$
et il nous reste à montrer que $\epsilon^{I}_{r}$ est dans son noyau pour tout $I$ et pour tout $r$ de sorte à induire un morphisme $s : OH(\mathfrak{B}_{n,n}) \rightarrow OZ(OH^n_C)$.  On remarque que l'image de $\epsilon^{I}_{r} $ est nulle dans $OZ(\Hno_C)$ si et seulement si elle est nulle dans chacun des $a(\Hno_C)a$ puisque son image est dans la somme directe de ceux-ci. Pour $a\in B^n$, on note $s_a : x_i  \mapsto a_i$ la projection de $s$ sur $a(\Hno_C)a$ de sorte que $s= \sum_a s_a$ et on remarque aisément que cela forme un morphisme d'anneaux gradués pour les même raisons que pour $s_0$. Afin de montrer le résultat voulu, on introduit les notations et lemmes suivants.

\begin{lemme}\label{lem:rpgn}
Pour tout $a \in B^n$, $I$ sous-ensemble ordonné de $\{1, \dots, 2n\}$ de cardinalité $n+k$ et $r > n$,  on a
$$s_a(\epsilon^I_{r}) = 0.$$
\end{lemme}

\begin{proof}
Puisqu'il y a $n$ composantes dans $W(a)a$, le produit extérieur est engendré par $n$ éléments et donc tout produit de plus de $n$ éléments est nul. Or $\epsilon^I_{r}$ étant une somme de polynômes à plus de $n$ termes, ils sont envoyés par $s_a$ sur des produits extérieurs de plus de $n$ éléments par (\ref{eq:poltoprodext}).
\end{proof}
Pour la suite du raisonnement, on se fixe un $a \in B^n$ arbitraire afin de ne pas devoir le mettre comme hypothèse dans chacun des lemmes. Dans le but d'alléger l'écriture, on note aussi $E_{2n} := \{1, \dots, 2n\}$ vu comme ensemble ordonné. On peut alors voir un sous-ensemble (ordonné) de $E_{2n}$ comme  un sous-ensemble des points d'extrémités des arcs de $a$.

\begin{definition}
Pour $I \subset E_{2n}$, on appelle \emph{arcs de $I$} les paires de points de $I$ qui sont identifiées aux extrémités d'un même arc de $a$ et on appelle ces points les \emph{points non-libres} (de $I$). On appelle tous les autres points des \emph{points libres} (donc les points de $I$ dont l'autre extrémité de l'arc de $a$ n'est pas dans $I$), comme illustré en  Figure \ref{fig:pointslibres}.
\end{definition}

\begin{figure}[h]
    \center
    \includegraphics[width=8cm]{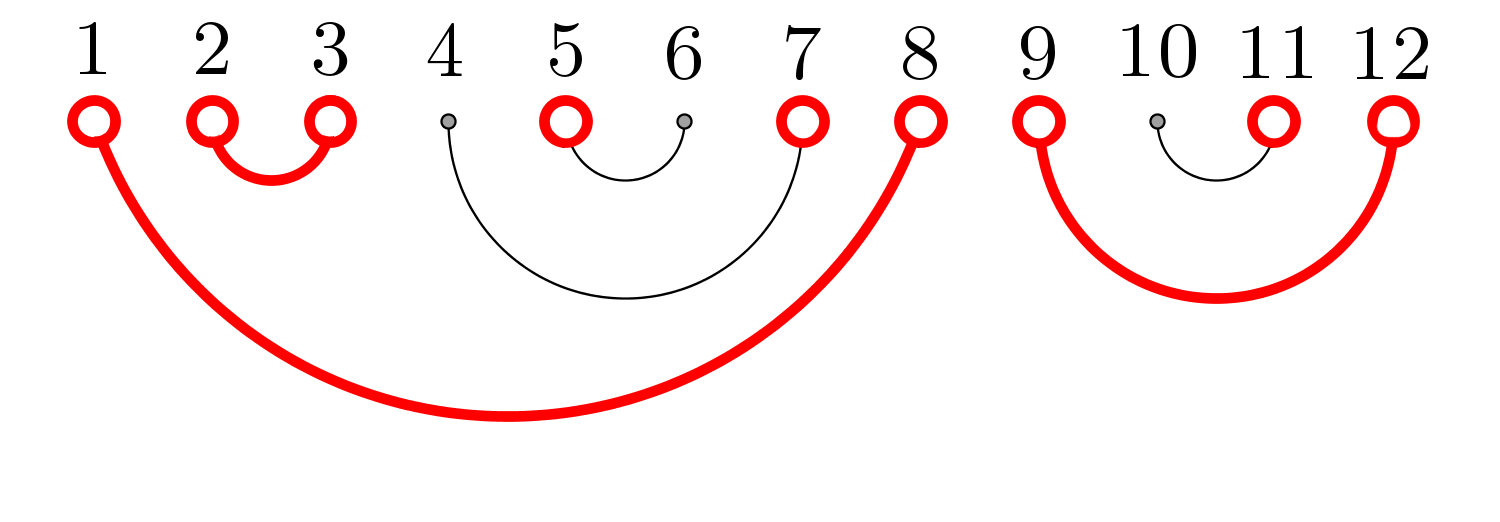}
    \caption{ \label{fig:pointslibres} On prend $n=6, k=3$ et $I = \{1,2,3,5,7,8,9,11,12\}$ (représentés par les cercles rouge) avec $\{5,7, 11\}$ les points libres de $I$ et $\{1,2,3,8,9,12\}$ les points non-libres, les arcs de $I$ étant en rouge.}
\end{figure}

\begin{lemme} \label{lem:comptage1}
Soient $I \subset E_{2n}$. Si $|I| = n+k$, alors $I$ possède au moins $k$ arcs et $2k$ points non-libres ainsi qu'au plus $n-k$ points libres.
\end{lemme}

\begin{proof}
On doit choisir $n+k$ points parmi $2n$ répartis en $n$ paires pour les $n$ arcs, donc on a au moins $k$ collisions (c'est-à-dire paires de points appartenant à un même arc). Chaque arc ayant deux extrémités et puisqu'on a au minimum $k$ arcs, on a au moins $2k$ points non-libres. Le nombre de points libres est obtenu en prenant la différence du nombre total de points de $I$ et du nombre de points non-libres, donc on obtient bien un maximum de $n+k-2k = n-k$ points libres.
\end{proof}

Pour $I \subset E_{2n}$ et $R =  \{i_1, \dots, i_r\} \subset I$ un sous-ensemble ordonné on pose
$$\epsilon^I_R := x_{i_1}^I \dots x_{i_r}^I.$$
Cette notation donne, en prenant la somme sur tous les sous-ensemble ordonnés de $I$,
\begin{equation}\epsilon^I_r = \sum_{R \subset I} \epsilon^I_R.\label{eq:epsilon} \end{equation}

\begin{lemme}\label{lem:doublonnul}
Soient $R \subset I \subset E_{2n}$. Si $R$ contient un arc de $I$, alors
$$s_a(\epsilon^I_R) = 0.$$
\end{lemme}

\begin{proof} On observe que si $i,i' \in E_{2n}$ sont reliés par un arc dans $a$, alors 
\begin{align*} s_a(x_ix_{i'}) = a_i \wedge a_{i'} &= 0 
\end{align*}
puisque $a_i = a_{i'}$. On généralise ensuite et on obtient le résultat par anticommutativité.
\end{proof}

\begin{lemme}\label{lem:pointnonlibreinR}
Pour tout $R \subset I \subset E_{2n}$, avec $|I| = n+k$ et $|R| \ge n-k+1$, il existe un point non-libre de $I$ contenu dans $R$.
\end{lemme}

\begin{proof}C'est un simple argument de comptage : par le Lemme \ref{lem:comptage1} il y au maximum $n-k$ points libres dans $I$ et par hypothèse $R$ contient au minimum $n-k+1$ points.
\end{proof}

\begin{definition}
On suppose $R \subset I \subset E_{2n}$ fixé. Pour un arc $(j,j')$ de $I$ ayant une extrémité $j$ ou $j'$ dans $R$, on considère la somme des nombres de
\begin{itemize}
\item points libres de $I$ situés entre $j$ et $j'$ non-contenus dans $R$,
\item points non-libres de $I$ situés entre $j$ et $j'$ et contenus dans $R$.
\end{itemize}
On appelle \emph{parité de l'arc $(j,j')$} la parité de cette somme.
\end{definition}

On peut calculer la parité d'un arc de $I$ pour $R$ en fonction de ses sous-arcs maximaux ayant une extrémité dans $R$, c'est-à-dire les sous-arcs contenus dans aucun autre sous-arc ayant une extrémité dans $R$.

\begin{lemme}\label{lem:sousarcparite}
Soient $R \subset I \subset E_{2n}$. Si $R$ ne contient aucun arc de $I$, alors la parité d'un arc de $I$ ayant une extrémité dans $R$ est donnée par la somme des parités inverses des sous-arcs maximaux ayant une extrémité dans $R$ et de la parité du nombre de points libres de $I$ non-contenus dans $R$ qui ne sont pas dans ces sous-arcs.
\end{lemme}

\begin{proof}Tous les points contenus dans les sous-arcs sont dans l'arc et, $R$ ne pouvant contenir deux extrémités d'un même arc, chaque sous-arc considéré ajoute un point non-libre contenu dans $R$, inversant la parité.
\end{proof}

\begin{lemme}\label{lem:parite}
Soient $R \subset I \subset E_{2n}$. Soient aussi $(j,j')$ un arc de $I$ avec $j$ (resp. $j'$) dans $R$ et $R'$ obtenu en retirant $j$ (resp. $j'$) de $R$ et ajoutant $j'$ (resp. $j$). Si $(j,j')$ est pair alors on obtient 
\begin{align*}
s_a\left(\epsilon^I_R\right) &= -s_a\left(\epsilon^I_{R'}\right). 
\end{align*}
\end{lemme}

Avant de montrer ce résultat un peu technique, on propose un exemple afin de clarifier le raisonnement.

\begin{exemple}
On prend $a \in B^6$ comme dans la Figure \ref{fig:pointslibres_R} ainsi que $n=6, k=3, r=4$, $I = \{1,2,3,5,7,8,9,11,12\}$ et $R=\{5,8,9,11\} \subset I$. On a donc $\{5,7, 11\}$ les points libres de $I$ et $\{1,2,3,8,9,12\}$ les points non-libres donnant les arcs de $I : (1,8),(2,3)$ et $(9,12)$.
\begin{figure}[h]
    \center
    \includegraphics[width=8cm]{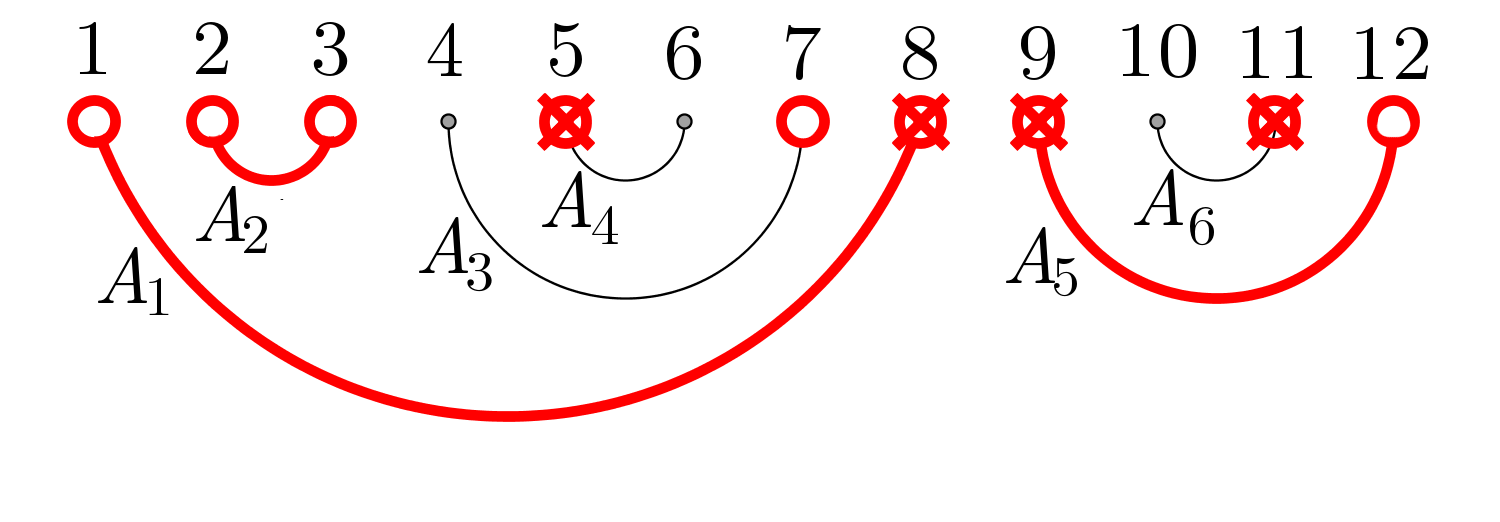}
    \caption{ \label{fig:pointslibres_R} On prend $I$ représenté par les cercles rouge. On prend en plus $R\subset I$ représenté par les croix rouge. Par ailleurs on note $A_i$ l'arc qui appartient à la $i$-ème composante de $W(a)a$ pour l'ordre induit par les points de bases.}
\end{figure}

On obtient le terme $\epsilon^I_R = (-x_5)(-x_8)x_9(-x_{11})$ de $\epsilon^I_r$, les signes étant dus à la position des points dans $I$, qui a comme image dans $a(OH^n_C)a$
 $$s_a(\epsilon^I_R) = s_a\bigl((-x_5)(-x_8)x_9(-x_{11})\bigr) =(-A_4)\wedge(-A_1)\wedge A_5 \wedge(-A_6)$$
où $A_i$ est la $i$-ème composante de cercle de $W(a)a$ pour l'ordre induit par les points de base, comme illustré en Figure \ref{fig:pointslibres_R} (on prend des majuscules dans cet exemple pour ne pas confondre avec les $a_i$ qui définissent $s_0$, dans le cas de la figure on a par exemple $a_4 = a_7 = A_3$). Puisque $R$ contient $8$ et $9$ comme points non-libres, on regarde la parité des arcs :
\begin{itemize}
\item  $(1,8)$ : impair car il contient juste 7 comme point libre de $I$ non-contenu dans $R$ et aucun non-libre contenu dans $R$,
 \item $(9,12)$ : pair car il ne contient aucun des points recherchés.
\end{itemize}
On vérifie qu'en effet l'élément $\epsilon^I_{R'}$ avec $R'$ qui est donné par $R$ où on échange $9$ avec $12$, c'est-à-dire l'élément $(-x_5)(-x_8)(-x_{11})x_{12}$, donne l'image 
$$s_a(\epsilon^I_{R'}) = (-A_4)\wedge(-A_1)\wedge (-A_6) \wedge A_5 = -\bigl((-A_4)\wedge(-A_1)\wedge A_5 \wedge(-A_6)\bigr) = -s_a(\epsilon^I_R).$$
Par contre si on échange $8$ avec $1$ on obtient comme image $A_1\wedge(-A_4)\wedge A_5 \wedge(-A_6)$ qui est la même que celle de $\epsilon^I_R$ mais qui s'annule avec l'élément où on échange $8$ avec $1$ et $9$ avec $12$. Au final, la somme de toutes les possibilités d'échanges des points non-libres de $R$ est donc nulle.
\end{exemple}

\begin{proof}\emph{(du Lemme \ref{lem:parite})}
On montre le résultat par récurrence sur la taille de $I$ et de $R$. En effet, on suppose que $I$ contient un arc $(i,i')$ avec aucun autre point de $I$ compris entre ses extrémités et qui possède donc trivialement un nombre pair de points recherchés. On a alors un changement de signes induit par la position dans $I$ ($i$ et $i'$ sont côte à côte dans $I$, donc $x_i^I$ est de signe opposé à $x_{i'}^I$), tandis que $a_i$ et $a_{i'}$ sont à la même position dans l'expression de l'image, ne changeant pas de signe. Au total, on a donc un changement de signe comme voulu. On a maintenant cinq cas à considérer : 
\begin{enumerate}
\item On ajoute un point libre compris entre $i$ et $i'$ à $I$ et $R$.
\item On ajoute un point libre compris entre $i$ et $i'$ seulement à $I$.
\item On ajoute un arc compris entre $i$ et $i'$ dans $I$ et une de ses deux extrémités dans $R$.
\item On ajoute un arc compris entre $i$ et $i'$ dans $I$ avec aucune de ses extrémités dans $R$.
\item On ajoute un arc compris entre $i$ et $i'$ dans $I$ et dans $R$.
\end{enumerate}
Le premier cas n'induit pas de changement de signe puisque $i'$ est décalé d'une position vers la droite par rapport à $i$, impliquant un premier changement de signe, et que l'élément ajouté étant dans $R$ il apparait aussi dans le produit. Commuter avec cet élément implique donc un second changement de signe donnant une égalité au final. Par contre, le deuxième cas induit un changement de signe puisqu'on conserve le changement par la position dans $I$ mais qu'on ne retrouve plus l'élément avec lequel on doit commuter dans $R$ et donc dans le produit. Le troisième cas induit aussi un changement puisqu'on ajoute $2$ éléments dans $I$, ne modifiant donc pas le signe, mais qu'on ajoute un seul élément dans $R$ impliquant une commutation dans le produit. Le quatrième cas ne change pas non plus le signe par des arguments semblables. Enfin, pour le cinquième cas, alors $R$ contient un arc de $I$ et la solution devient triviale par le Lemme \ref{lem:doublonnul}. Au final, le signe n'est donc influencé que par les deuxième et troisième cas qui sont les points considérés pour définir la parité d'un arc.
\end{proof}

\begin{lemme}\label{lem:existpair}
Soient $R \subset I \subset E_{2n}$  avec $|I|=n+k$ et $n-k+1 \le |R| \le n$. Si $R$ ne contient aucun arc de $I$, alors il existe un arc de $I$ ayant une extrémité dans $R$ et de parité paire.
\end{lemme}

\begin{proof}
Tout d'abord, par le Lemme \ref{lem:pointnonlibreinR} il existe un arc de $I$ ayant une extrémité dans $R$. On suppose par l'absurde qu'il n'existe pas d'arc rencontrant $R$ et ayant une parité pair et donc que tout arc de $I$ ayant une extrémité dans $R$ est impair. Si on pose $l \ge 1$ le nombre de points non-libres de $I$ contenus dans $R$, on observe qu'il y a au moins $n-k-l+1$ points libres dans $R$. Par le Lemme \ref{lem:sousarcparite}, la parité d'un arc ne dépend pas de ses sous-arcs (ils sont tous impairs induisant une parité paire) et il faut donc pour chaque arc au moins un point libre non-contenu qui ne se trouve dans aucun des sous-arcs pour obtenir une parité impaire. De là on déduit qu'il y a au moins autant de points libres non-contenus dans $R$ que d'arcs de $I$ ayant une extrémité dans $R$. On a donc $l$ points libres de $I$ qui ne sont pas contenus dans $R$, laissant un maximum de $n-k-l$ points libres restants par le Lemme \ref{lem:comptage1}. Or $R$ doit en contenir au moins $n-k-l+1$, ce qui est absurde.
\end{proof}

Maintenant qu'on a tous les outils nécessaires, on peut démontrer le résultat voulu donné par la proposition suivante : 

\begin{proposition}\label{prop:nk1}
Pour tout $a \in B^n$, $k \in [1,n]$ et  $r\ge n-k+1$ ainsi que pour tout $I \subset E_{2n}$ de cardinalité $n+k$, on a
$$s_a(\epsilon^I_{r}) = 0.$$
\end{proposition}

\begin{proof}
Tout d'abord, par le Lemme \ref{lem:rpgn}, on peut supposer que $r \le n$. L'idée de la preuve est de montrer que tous les termes $s_0(\epsilon^I_R)$ qui composent l'image $s_0(\epsilon^I_r)$ par (\ref{eq:epsilon}) et qui ne sont pas annulés par le Lemme \ref{lem:doublonnul} s'annulent deux à deux en utilisant le Lemme \ref{lem:parite}. On peut donc supposer que $R$ ne contient pas d'arc de $I$. Par le Lemme \ref{lem:existpair}, il existe un arc de $I$ pair pour $R$ et donc on peut appliquer le Lemme \ref{lem:parite} pour l'annuler avec $R'$. Cela permet d'annuler tous les autre termes puisque, étant donné un $R$ avec un arc pair $(j,j')$, pour tout autre $R_2$ obtenu à partir de $R$ en échangeant un point non-libre (différent de $j$ et $j'$) avec l'autre extrémité de l'arc, on peut appliquer le Lemme \ref{lem:parite} dessus aussi pour l'annuler avec $R_2'$ obtenu en échangeant $j$ et $j'$ et que tout ces ensembles donnent des $\epsilon^I_R$ différents.
\end{proof}

\begin{corollaire}\label{cor:supercentremorphisme}
Il existe un morphisme d'anneaux gradués induit par $s_0$
$$s : OH(\mathfrak B_{n,n}) \rightarrow OZ(OH^n_C) : x_i \mapsto \sum_{a \in B^n} a_i. $$
\end{corollaire}

\subsection{Injectivité de l'homomorphisme $s$}

\begin{lemme}
Le morphisme d'anneaux
$$\bar s : OH(\mathfrak B_{n,n}) \otimes_\Z \Z/2\Z \rightarrow OZ(OH^n_C)  \otimes_\Z \Z/2\Z$$
induit par $s$ sur le produit tensoriel avec $\Z/2\Z$ est un isomorphisme.
\end{lemme}

\begin{proof}
On a un isomorphisme d'anneaux
$$H^n \otimes_\Z \Z/2\Z \simeq OH^n_C \otimes_\Z \Z/2\Z$$
en envoyant le produit tensoriel d'éléments de $A$ vers le produit extérieur, comme dans la preuve de la Proposition \ref{prop:OHvsH}.
Le centre de $H^n$ et le centre impair de $OH^n_C$ étant caractérisés par la même relation ($\ref{eq:supercentre}$) une fois prise en modulo $2$, on obtient un isomorphismes d'anneaux gradués
$$t : Z(H^n) \otimes_\Z \Z/2\Z \rightarrow OZ(OH^n_C) \otimes_\Z \Z/2\Z.$$ 
On considère ensuite le diagramme commutatif de morphismes d'anneaux suivant
$$\xymatrix{OH(\mathfrak B_{n,n}) \otimes_\Z \Z/2\Z  \ar[r]^\simeq \ar[d]_{\bar s} & H(\mathfrak{B}_{n,n})  \otimes_\Z \Z/2\Z  \ar[d]^{h} \\
OZ(OH^n_C) \otimes_\Z \Z/2\Z  & \ar[l]^t Z(H^n) \otimes_\Z \Z/2\Z  }  $$
avec l'isomorphisme de droite venant du Théorème \ref{thm:hncenter} et celui du haut de la Proposition \ref{prop:rankhoh}. Le diagramme commute car $s$ est équivalent modulo 2 à l'isomorphisme de droite. On en conclut que $\bar s$ est un isomorphisme.
\end{proof}

\begin{proposition}\label{prop:supercentreinjectif}
Le morphisme $s$ du Corollaire \ref{cor:supercentremorphisme} est injectif.
\end{proposition}

\begin{proof}
On note $OB$ la base de $OH(\mathfrak B_{n,n})$ de la Proposition \ref{prop:rankhoh}. On suppose par l'absurde que $s$ ne soit pas injectif, c'est-à-dire qu'il existe une combinaison linéaire non-nulle d'éléments de $OB$, $\sum_{x \in OB} x$, envoyée sur $0$ par $s$
$$\sum_{x\in OB} s(x) = 0\in OZ(OH^n_C).$$
L'idée de la preuve est d'utiliser l'isomorphisme du lemme précédent et la conservation de la base $OB$ sous le produit produit tensoriel avec $\Z/2\Z$ pour obtenir une relation dans les éléments d'une base de $OH(\mathfrak B_{n,n}) \otimes \Z/2\Z$, c'est-à-dire une combinaisons linéaire non-nulle d'éléments de cette base donnant $0$, ce qui est absurde par définition. On obtient cela en passant par le produit tensoriel avec $\Z/2\Z$, comme illustré par les diagrammes suivants 
$$
\xymatrix{
OH(\mathfrak B_{n,n}) \ar[d]_{s} \ar[r]^-{p_{\otimes \Z/2\Z}} & OH(\mathfrak B_{n,n}) \otimes \Z/2\Z \ar[d]^{\bar s} &
 \sum_{x\in OB} x \ar[d] &   \sum_{x\in OB} x \otimes 1 = 0 \\
OZ(OH^n_C) \ar[r]_-{p_{\otimes \Z/2\Z}} & OZ(OH^n_C) \otimes \Z/2\Z &  \sum_{x\in OB} s(x) = 0 \ar[r] &  \sum_{x \in OB} \bar s(x \otimes 1) = 0 \ar[u]
}
$$
où le diagramme de gauche commute et le diagramme de droite représente le chemin de raisonnement parcouru.

On considère l'image de la combinaison linéaire en modulo 2, c'est-à-dire sa projection dans le produit tensoriel avec $\Z/2\Z$, donnant
$$\sum_{x\in OB} \bar s(x \otimes 1) =  0 \in OZ(OH^n_C) \otimes \Z/2\Z$$
puisque $0$ est projeté sur $0$ et que le diagramme de gauche commute. Par le lemme précédent $\bar s$ est un isomorphisme, impliquant que
$$\sum_{x\in OB} x \otimes 1 =  0 \in OH(\mathfrak B_{n,n}) \otimes_\Z \Z/2\Z.$$
On obtient donc une relation dans $OB_{\Z_2}$ qui est une base par la Proposition \ref{prop:baseOBmod2}.
\end{proof}

\subsection{Égalité des rangs de $OZ\left(OH^n_C\right)$ et $OH(\mathfrak B_{n,n})$}

Dans la preuve du Théorème \ref{thm:hncenter}, M. Khovanov construit une variété $\widetilde S$ à partir de sphères. Il montre que son anneau de cohomologie est isomorphe au centre de $H^n$ et qu'il possède une présentation équivalente à celle de l'anneau de cohomologie de la variété de Springer. On propose de construire une variété similaire, notée $\widetilde T$, à partir de cercles (donnant donc des algèbres extérieures comme anneau de cohomologie) et de montrer qu'il existe un épimorphisme d'anneaux non-gradués de l'anneau de cohomologie de cette variété vers le centre impair de $OH^n_C$. Ensuite on calcule le rang de cet anneau de cohomologie et on remarque qu'il est égal au rang de l'anneau de cohomologie impaire de $\mathfrak B_{n,n}$, donnant l'égalité avec le rang du centre impaire et concluant la preuve.

\begin{definition}
Pour tout $a\in B^n$, on définit $T_a \subset T^{2n} := \underset{2n}{\underbrace{S^1 \times \dots \times S^1}}$ tel que
$(x_1, \dots, x_{2n}) \in T_a$
si et seulement si tout $(i,j)$ reliés par un arc dans $a$ implique $x_i = x_j$. \\
On définit alors
$$\widetilde T := \bigcup_{a \in B^n} T_a.$$
\end{definition}

On remarque que $T_a \simeq T^n$ puisque $a$ possède $n$ arcs et donc on égalise $n$ paires de $x_i$ disjointes sur $2n$ disponibles. De même, $T_a \cap T_b$ est l'ensemble des $(x_1, \dots, x_{2n})$ tels que si $(i,j)$ est relié dans $a$ ou dans $b$ alors $x_i = x_j$ de sorte que si deux points $k,l$ appartiennent à la même composante de $W(a)b$ on a $x_k = x_l$. Géométriquement, $T_a \cap T_b$ est donc un hypertore $T^{|W(a)b|}$, avec $|W(a)b|$ le nombre de composantes de cercles, de sorte que $\widetilde T$ est une collection d'hypertores $T^n$ identifié ensemble deux par deux sur des sous-hypertores de dimensions inférieures donnés par les diagonales.

\begin{exemple}
Dans $B^2$ on considère
\begin{align*}
a &=\deuxdiag{Images_arxiv/B2_1.png} ,& 
b &=\deuxdiag{Images_arxiv/B2_2.png},&
 W(a)b = \diagg{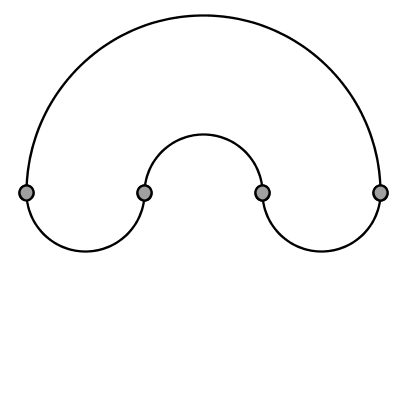}.
\end{align*}
On obtient $T_a = \{(x,x,y,y) | x,y \in S^1\} \simeq T^2$, noté par abus de notation $T_a = (x,x,y,y)$. On obtient aussi $T_b = (x,y,y,x) \simeq T^2$ et par conséquent $T_a \cap T_b = (x,x,x,x) \simeq T^1$. Visuellement, si on représente $T_a$ et $T_b$ comme des tores sous forme de rectangles où on identifie les côtés opposés, on obtient la surface représentée sur la Figure \ref{fig:toresasb}, c'est-à-dire que $T_a \cup T_b$ est l'union de deux tores dont on identifie ensemble les diagonales.
\begin{figure}[h]
    \center
    \includegraphics[width=6cm]{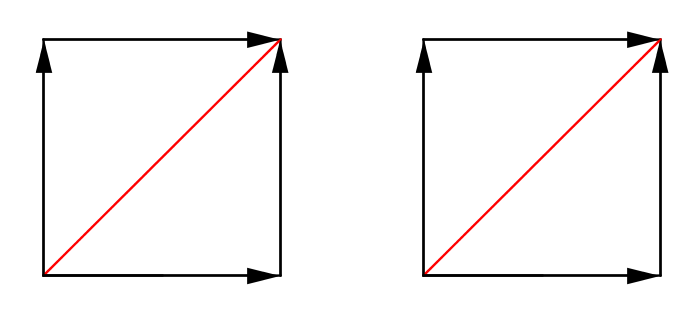}
    \caption{ \label{fig:toresasb} $T_a$ et $T_b$ sont donnés par des tores de dimension $2$ et $T_a \cap T_b$ est un sous-hypertore (en rouge) de dimension $1$, donc un cercle, défini en prenant la diagonale le long des composantes (les flèches) de $T_a$ (et $T_b$) qu'on égalise.}
\end{figure}
\end{exemple}

\begin{exemple}
Dans $B^3$ on considère les éléments
\begin{align*}
a &=\diagg{Images_arxiv/B3_4.png},& 
b &=\diagg{Images_arxiv/B3_2.png},& 
c &=\diagg{Images_arxiv/B3_5.png}.
\end{align*}
On obtient alors par exemple les diagrammes 
\begin{align*}
W(b)a &= \middiag{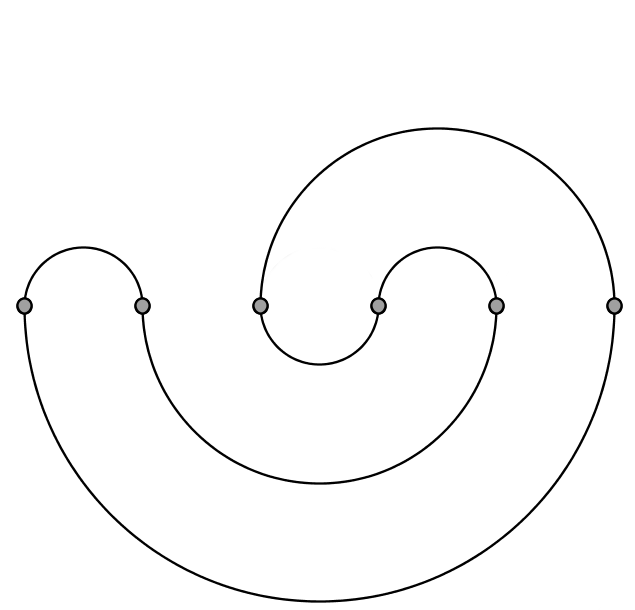},&
W(c)a &= \middiag{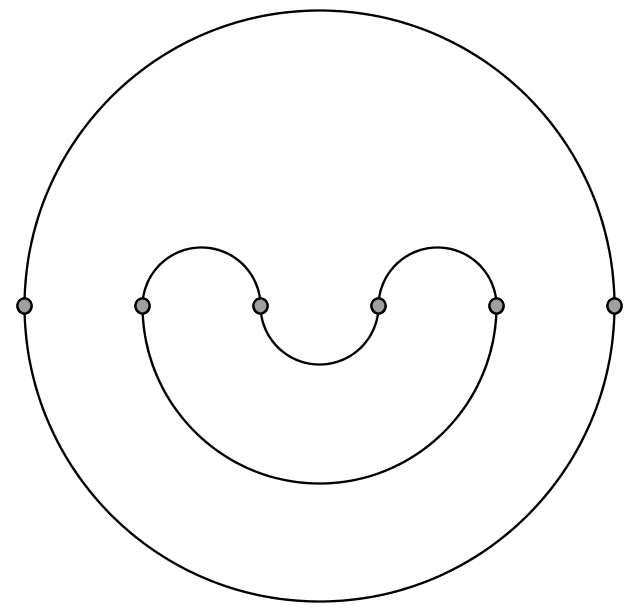}.
\end{align*}
On calcule que $T_a = (x,y,z,z,y,x) \simeq T^3$ et $T_b = (x,x,y,z,z,y) \simeq T^3$ impliquant que $T_a \cap T_b = (x,x,x,x,x,x) \simeq S^1$. Visuellement, on a pour $T_a$ représenté par la boite où on identifie deux à deux les cotés opposés la Figure \ref{fig:hypertore1} avec l'intersection $T_a \cap T_b$ en rouge.

\begin{figure}[H]
    \center
    \includegraphics[width=4.66cm]{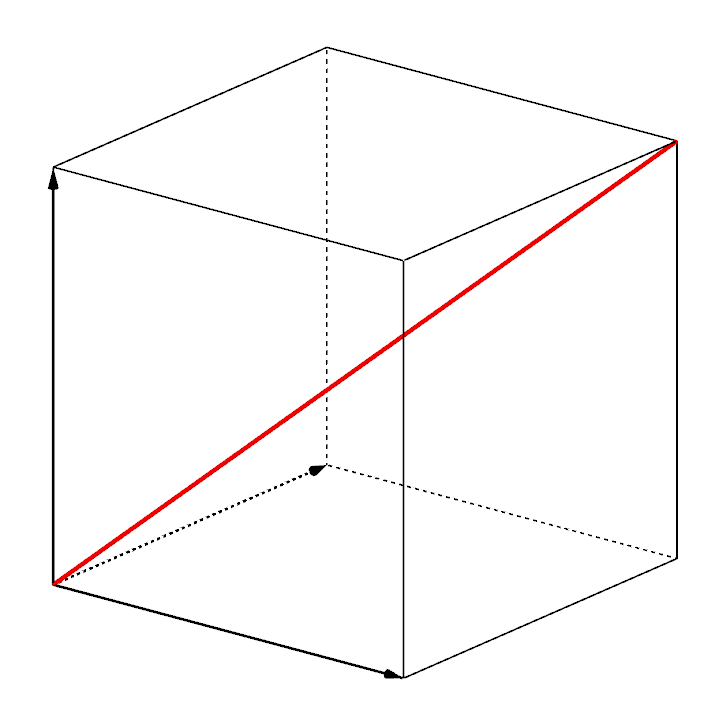}
    \caption{ \label{fig:hypertore1} $T_a$ est donné par un hypertore (en noir) de dimension $n=3$ et $T_a \cap T_b$ est un sous-hypertore (en rouge) de dimension $|W(a)b| = 1$ défini en prenant les diagonales le long des composantes (les flèches) de $T_a$ qu'on égalise.}
   
\end{figure}

De même, on calcule $T_c = (x,y,y,z,z,x)$ et donc $T_a \cap T_c = (x,y,y,y,y,x)$. On obtient alors l'hypertore représenté en Figure \ref{fig:hypertore2} avec l'intersection $T_a\cap T_c$ en rouge.

\begin{figure}[H]
    \center
    \includegraphics[width=4.66cm]{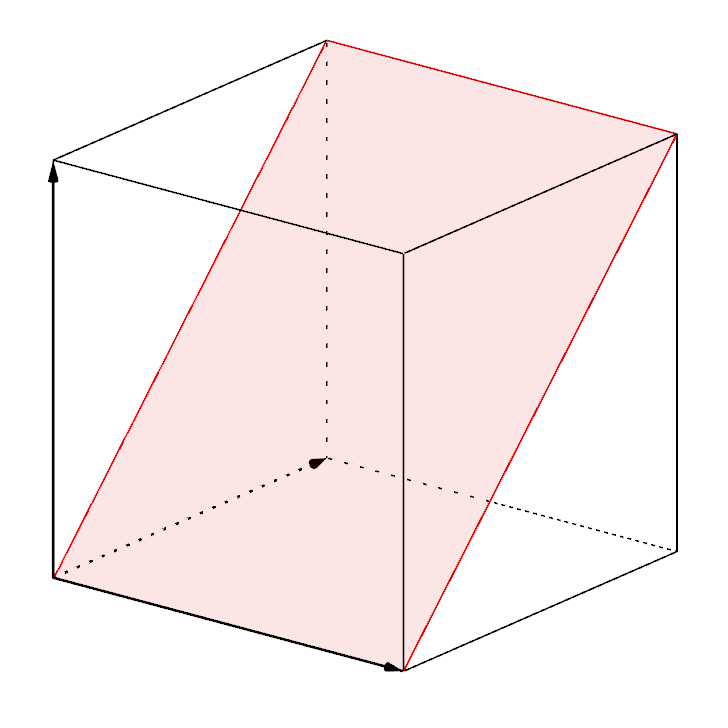}
    \caption{\label{fig:hypertore2}$T_a$ est donné par un hypertore  (en noir) de dimension $n=3$ et $T_a \cap T_c$ est un sous-hypertore (en rouge) de dimension $|W(a)c| = 2$ défini en prenant les diagonales le long des composantes (la flèche en pointillée et la flèche verticale) de $T_a$ qu'on égalise.}
    
\end{figure}

En faisant pareil pour $T_b \cap T_c = (x,x,x,y,y,x)$, on calcule finalement que $T_a \cup T_b \cup T_c$ est donné par la Figure \ref{fig:hypertore3} où les éléments colorés sont identifiés ensemble deux à deux.

\begin{figure}[H]
    \center
    \includegraphics[width=14cm]{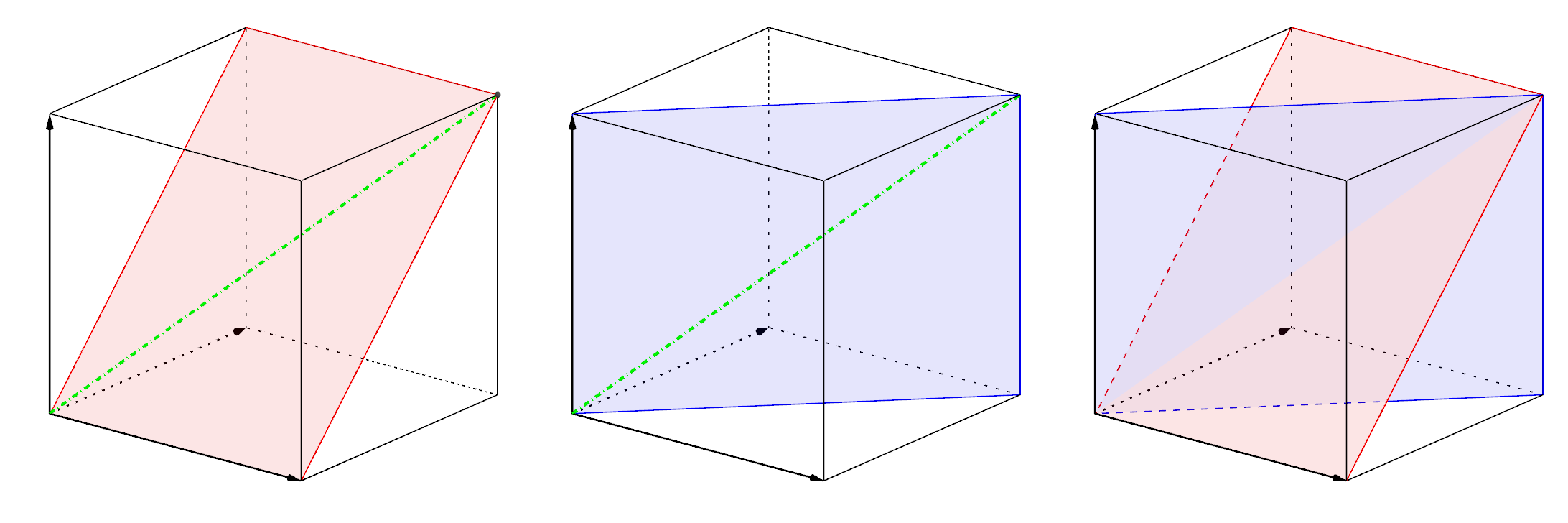}
    \caption{\`A gauche $T_a$, au milieu $T_b$ et à droite $T_c$ avec les éléments en couleurs identifiés deux à deux afin d'obtenir $T_a \cup T_b \cup T_c$.}
    \label{fig:hypertore3}
\end{figure}

\end{exemple}

\begin{lemme}
On a un isomorphisme d'anneaux (non-gradués)
$$a(OH^n_C)a \simeq H(T_a)$$
et un isomorphisme de groupes (non-gradués)
$$a(OH^n_C)b \simeq_{ab} H(T_a \cap T_b).$$
De plus, en munissant $a(OH^n_C)b$ d'une structure d'anneau induite par l'isomorphisme de groupe et en définissant
\begin{align*}
\psi_{a;a,b} &: H(T_a) \rightarrow H(T_a \cap T_b), \\
 \psi_{b;a,b} &: H(T_b) \rightarrow H(T_a \cap T_b), \\
\gamma_{a;a,b} &: a(OH^n_C)a \rightarrow a(OH^n_C)b : x \mapsto x{_a1_b}, \\
\gamma_{b;a,b} &: b(OH^n_C)b \rightarrow a(OH^n_C)b : x \mapsto {_a1_b}x,
\end{align*}
où les $\psi$ sont induit par inclusion de $T_a \cap T_b$ dans $T_a$ et $T_b$, on obtient le diagramme commutatif de morphismes d'anneaux suivant :
$$
\xymatrix{
H(T_a) \ar[r]^-{\psi_{a;a,b}} \ar[d]^{\simeq} & H(T_a \cap T_b) \ar[d]^{\simeq} & H(T_b) \ar[l]_-{\psi_{b;a,b}} \ar[d]^{\simeq} \\
a(OH^n_C)a \ar[r]_{\gamma_{a;a,b}} & a(OH^n_C)b & b(OH^n_C)b \ar[l]^{\gamma_{b;a,b}}.
}
$$
\end{lemme}

\begin{proof}
On sait par l'Exemple \ref{ex:torecohomology} des Annexes que
$$H(T^k) \simeq \Ext^* \Z^k$$
avec chaque composante $S^1$ de $T^ {k} = S^1 \times \dots \times S^1$ générant un élément de base du produit extérieur par la formule de Kunneth (voir \cite[Théorème 3.15]{Hatcher}) et donc on a
$$H(T_a) \simeq H(T^n) \simeq \Ext^* \Z^n \simeq a(OH^n_C)a.$$
De même on a des isomorphismes de groupes 
$$H(T_a \cap T_b) \simeq  \Ext^* \Z^{|W(b)a|} \simeq_{ab} a(OH^n_C)b.$$
On calcule que $\psi_{a;a,b}$ envoie un produit extérieur vers un autre en renommant les éléments suivant les composantes de $T_a$ égalisées dans $T_a \cap T_b$ puisque c'est le morphisme induit par l'inclusion
$$T^{|W(b)a|}  \hookrightarrow T^{n}$$
où on envoie chaque coordonnée de $T_a \cap T_b$ vers les coordonnées de $T_a$ égalisées pour former $T_a \cap T_b$. En effet, chaque composante $S^1$ de $T^n$ engendrant un élément de base de $H(T_a)$, deux éléments de base sont égalisés si les composantes associées sont égalisées. Par ailleurs, on sait que $\gamma_{a;a,b}$ agit exactement de la même façon puisqu'on égalise les composantes de cercles fusionnées.
\end{proof}

\begin{definition}
Soient des ensembles finis $I$ et $J$ ainsi que des anneaux $A_i$ pour tout $i \in I$ et $B_j$ pour tout $j \in J$ et des morphismes $\beta_{i,j} : A_i \rightarrow B_j$ pour certaines paires $(i,j)$. On note
\begin{align*}
\beta &:= \sum \beta_{i,j}, & \beta&: \prod_{i \in I} A_i \rightarrow \prod_{j \in J} B_j.
\end{align*}
On définit l'égalisateur $\Eq(\beta)$ de $\beta$ comme le sous-anneau de $\prod_{i\in I} A_i$ qui consiste en les $(a_i)_{i \in I}$ tels que
$$\beta_{i,j}(a_i) = \beta_{k,j}(a_k)$$
dès que $\beta_{i,j}$ et $\beta_{k,j}$ sont définis.
\end{definition}

Par la Proposition \ref{prop:inaOHa}, on remarque que, si on pose  $\gamma := \sum_{a \ne b \in B^n} \gamma_{a;a,b} + \gamma_{b;a,b}$, avec $I$ donné par l'ensemble des $a \in B^n$ et $J$ par l'ensemble des paires $(a,b), a \ne b \in B^n$ tels que $\gamma_{c, (a,b)}$ est défini si $c = a$ ou $c = b$, on obtient l'égalité $OZ(OH^n_C) = \Eq(\gamma)$. Dès lors, en posant de la même manière $\psi := \sum_{a \ne b \in B^n} \psi_{a;a,b} + \psi_{b;a,b}$, on obtient le diagramme commutatif de morphismes d'anneaux suivant :
$$\xymatrix{
\Eq(\psi) \ar[r]  \ar[d]^{\simeq} & \bigoplus\limits_{a \in B^n} H(T_a) \ar[r]^-{\psi} \ar[d]^{\simeq} & \bigoplus\limits_{a \ne b \in B^n} H(T_a \cap T_b) \ar[d]^{\simeq} \\
\Eq(\gamma) \ar[r] &   \bigoplus\limits_{a \in B^n} a(OH^n_C)a \ar[r]_-{\gamma} &  \bigoplus\limits_{a \ne b \in B^n} a(OH^n_C)b 
}$$

%

De plus, on a un morphisme $H(\widetilde T) \rightarrow \bigoplus_{a\in B^n} H(T_a)$ induit par l'inclusion $T_a \hookrightarrow \widetilde T$ pour tout $a$. On considère alors $\tau$ la composition de ce morphisme avec la projection de $ \bigoplus_{a \in B^n} H(T_a) $ sur le sous-anneau $\Eq(\gamma)$ et on obtient le diagramme commutatif suivant :
\begin{align}
\xymatrix{
H(\widetilde T) \ar[r]^{\tau} & \Eq(\psi) \ar[r]  \ar[d]^{\simeq} & \bigoplus\limits_{a \in B^n} H(T_a) \ar[r]^-{\psi} \ar[d]^-{\simeq} & \bigoplus\limits_{a \ne b \in B^n} H(T_a \cap T_b) \ar[d]^{\simeq} \\
OZ(OH^n_C) \ar[r]_{\simeq} & \Eq(\gamma) \ar[r] &   \bigoplus\limits_{a \in B^n} a(OH^n_C)a \ar[r]_-{\gamma} &  \bigoplus\limits_{a \ne b \in B^n} a(OH^n_C)b. 
}
\label{eq:tau}
\end{align}

Jusque là, la preuve est équivalente à celle de M. Khovanov en remplaçant les sphères par des cercles, sauf que maintenant on veut montrer que $\tau$ est, non pas un isomorphisme, mais un épimorphisme.\\

\begin{definition} \label{def:arrowBn}
On dit qu'il y a une flèche $a\rightarrow b$ pour $a,b \in B^n$ s'il existe un quadruplet $1 \le i<j<k<l \le 2n$ tels que $i$ est relié à $j$ et $k$ à $l$ dans $a$ et $i$ est relié à $l$ et $j$ à $k$ dans $b$, ainsi que si tous les autres arcs sont les mêmes dans $a$ et $b$. Visuellement, on a
$$\deuxdiag{Images_arxiv/B2_1.png} \longrightarrow \deuxdiag{Images_arxiv/B2_2.png}.$$
\end{definition}

\begin{exemple}
Pour $n = 3$, on obtient toutes les flèches suivantes :
$$
\xymatrix@R=.5pc{
& \smalldiag{Images_arxiv/B3_3.png} \ar[dr] & & \\
\smalldiag{Images_arxiv/B3_1.png} \ar[ur] \ar[dr] \ar@/_8pc/[rrr] & & \smalldiag{Images_arxiv/B3_5.png} \ar[r] & \smalldiag{Images_arxiv/B3_4.png}\\
& \smalldiag{Images_arxiv/B3_2.png}  \ar[ur] & & \\
\\
\\
}
$$
\end{exemple}

\begin{definition}
On définit un ordre partiel sur $B^n$ en disant que $a \prec b$ s'il existe une suite $a = a_0 \rightarrow a_1 \rightarrow \dots \rightarrow a_m = b$. On se fixe un ordre total arbitraire $<$ à partir de cet ordre partiel.\\
On définit aussi une notion de distance donnée par la distance sur le graphe non orienté induit par les flèches, c'est-à-dire que $d(a,b) = m$ pour $m$ minimal tel qu'on a une suite $a = a_0, a_1, \dots, a_m = b$ avec soit $a_i \rightarrow a_{i+1}$ soit $a_{i+1} \rightarrow a_i$ pour tout $0 \le i < m$.
\end{definition}

\begin{proposition}
Le diagramme $W(b)a$ contient exactement $n-d(a,b)$ composantes de cercles
$$|W(b)a| = n - d(a,b).$$
\end{proposition}

\begin{proof}
Tout d'abord on note que si $a \rightarrow b$ ou $b \rightarrow a$ alors il est clair que $W(b)a$ contient $n-1$ composantes puisque les 4 arcs de la définition de flèche entre $a$ et $b$ appartiennent à la même composante de $W(b)a$.

On doit avoir 
$$|W(b)a| \ge n-d(a,b)$$
car pour toute suite $a = a_0, \dots, a_{m'} = b$ avec $a_i \rightarrow a_{i+1}$ ou $a_{i+1} \rightarrow a_i$ on a le cobordisme
$$W(a_0)a_1W(a_1)a_2 \dots W(a_{m'-1})a_{m'} \rightarrow W(a_0)a_{m'}$$
et à chaque étape
$$W(a_0)a_iW(a_i)a_{i+1} \rightarrow W(a_0)a_{i+1}$$
on change le nombre de composantes de $\pm 1$ donc dans le pire des cas on fusionne $m'$ composantes.

On considère les $k$ arcs de $a$ qui appartiennent à une même composante de $W(b)a$ et on construit une suite de $a$ vers $a'$ avec tous ces arcs fusionnés, ce qui prend $k-1$ étapes. On peut faire ainsi pour chaque composante et on obtient alors une suite de $a$ à $b$ constitué de $n-|W(b)a|$ étapes et donc $d(a,b) \le n-|W(b)a|$.
\end{proof}

On pose $T_{<a} := \bigcup_{b<a} T_b$ et $T_{\le a} := \bigcup_{b \le a} T_b$. On remarque que si $c$ est le prochain élément après $a$ dans l'ordre total sur $B^n$ alors $T_{<c} = T_{\le a}$. Cela est utile par la suite pour faire des preuves par induction.

M. Khovanov utilise les trois lemmes suivants (\cite[Lemmes 3.2-3.4]{HnCenter}) pour sa preuve mais qu'il laisse non démontrés. La preuve du premier étant fort technique et trouvable dans les annexes de \cite[preuve de la Proposition 3.15]{preuveLemme}, on n'en donne que l'idée principale tandis qu'on démontre les deux autres pour le cas utilisant des tores.

\begin{lemme} \label{lem:dab=daccb}
Pour tout $a,b \in B^n$ il existe $c \in B^n$ tel que $d(a,b) = d(a,c) + d(c,b)$ et $a \succ c \prec b$.
\end{lemme}

\begin{proof}\emph{(idée)}
Tout d'abord, il faut observer que si on n'a pas de séquence minimale $a = a_0, a_1, \dots, a_m = b$ avec $a_1 \rightarrow a$, alors on a une séquence $a=a_0 \rightarrow a_1 \rightarrow \dots \rightarrow a_m = b$. Toute la technicité de la preuve est contenue dans ce résultat qu'on accepte.

La preuve se fait par récurrence sur $d(a,b)$. Si $d(a,b) = 1$, alors soit $a \rightarrow b$, soit $b \rightarrow a$ et on prend $c = a$ ou $c = b$. 
On suppose maintenant que le lemme est vrai pour des distances inférieures. Si on un chemin minimal $a = a_0, a_1, \dots, a_m = b$ tel que $a_1 \rightarrow a$ alors on prend le $c$ donné par le lemme appliqué à $a_1$ et $b$ et on remarque qu'il est correspond au $c$ recherché pour $a$ et $b$. Sinon, par le résultat technique au dessus, on peut prendre $c = a_0$.
\end{proof}

\begin{lemme} \label{lem:sac=sabc}
Pour tout $a,b,c \in B^n$ tels que $d(a,c) = d(a,b) + d(b,c)$ on a 
$$T_a \cap T_c = T_a\cap T_b \cap T_c.$$
\end{lemme}

\begin{proof}
On remarque que le corbordisme 
$$W(a)bW(b)c \rightarrow W(a)c$$
envoie $n-d(a,b) + n-d(b,c)$ composantes de cercle vers $n-d(a,c)$ composantes. Or, puisque $d(a,b)+d(b,c) = d(a,c)$, on doit éliminer et donc fusionner au moins $n$ composantes. Puisqu'on a $n$ ponts, on a exactement $n$ fusions et donc deux points reliés par un arc de $b$ sont dans la même composante de $W(a)c$ ce qui est équivalent à l'assertion du lemme.
\end{proof}

\begin{lemme} \label{lem:sasbsa}
Pour tout $a \in B^n$, on a
$$T_{<a}\cap T_a = \bigcup_{b \rightarrow a}(T_b \cap T_a).$$
\end{lemme}

\begin{proof}
Tout d'abord on observe que si $c \prec a$ alors il existe $b\rightarrow a$  tel que $c \prec b$ et par le Lemme \ref{lem:sac=sabc} on obtient 
$$T_a \cap T_c = T_a \cap T_b \cap T_c \subset T_a \cap T_b.$$
Ensuite, on prend $c < a$ et par le Lemme \ref{lem:dab=daccb} il existe $d$ tel que $d \prec a$ et $d \prec c$ ainsi que $d(a,c) = d(a,d) + d(d,c)$. Par le Lemme \ref{lem:sac=sabc} on obtient
$$T_a \cap T_c = T_a \cap T_d \cap T_c \subset T_a \cap T_d  \subset T_a \cap T_b$$
pour un certain $b \rightarrow a$.
\end{proof}

Le lemme suivant est adapté de \cite[Lemme 3.5]{HnCenter} avec comme différence que la décomposition utilise des cellules de toutes dimensions au lieu de cellules de dimensions paires. On renvoie à la Section \ref{def:decomp} des Annexes pour une définition d'une décomposition cellulaire.

\begin{lemme}\label{lem:decompcellulaire}
Il existe une décomposition cellulaire de $T_a$ qui se restreint en une décomposition cellulaire de $T_{<a}\cap T_a$.
\end{lemme}

\begin{proof}
On note $I$ l'ensemble des arcs de $a$ de sorte qu'on ait un homéomorphisme $T_a \simeq T^{|I|}$. On définit $\Gamma$ le graphe non-orienté obtenu en prenant comme sommets l'ensemble $I$ et en posant que deux arcs $y$ et $z$ sont reliés s'il existe $b \rightarrow a$ tel que $b$ est obtenu à partir de $a$ en effaçant $y$ et $z$ et en reconnectant les 4 points extrémités de ces arcs d'une façon différente. Une façon similaire de voir que deux arcs sont reliés est si l'un est contenu dans l'autre et qu'il n'y a pas d'arc intermédiaire. $\Gamma$ est alors une union disjointe d'arbres.  Un exemple de graphe $\Gamma$ pour un élément de $B^6$ est donné en Figure \ref{fig:gamma}.

\begin{figure}[H]
    \center
    \includegraphics[width=14cm]{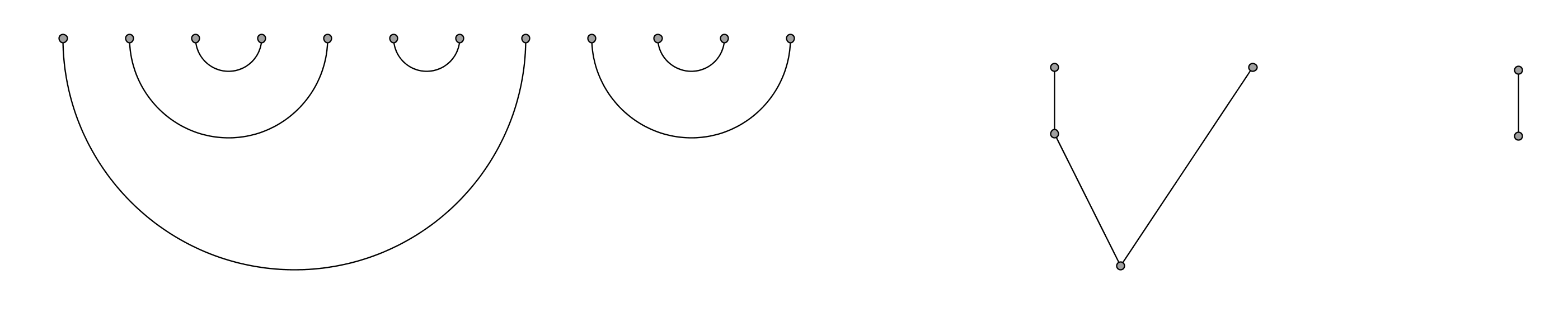}
    \caption{\label{fig:gamma} Un élément de $B^6$ et son graphe $\Gamma$.}
\end{figure}

On pose $E$ l'ensemble des arêtes de $\Gamma$. Dans chaque arbre de $\Gamma$, on choisit un sommet qu'on marque et on note $M$ l'ensemble des sommets marqués (qui contient donc autant d'éléments que d'arbres dans $\Gamma$). 
On remarque alors que $|E| + |M| = n$ puisqu'on a autant d'arêtes que le nombre d'arcs moins le nombre d'arbres (chaque arbre contenant autant d'arêtes que d'arcs moins un). On se fixe un point $p \in S^1$ et pour tout $J \subset (E \sqcup M)$ on pose $c(J)$ le sous-ensemble de $T^{|I|}$ constitué des points $(x_i)_{i\in I}$ tels que
\begin{align*}
x_i &= x_j & \text{ si } &(i,j) \in J,\\
x_i &\ne x_j & \text{ si } &(i,j) \notin J,\\
x_i &= p & \text{ si } &i\in M\cap J, \\ 
x_i &\ne p & \text{ si } &i\in M, i \notin J,
\end{align*}
avec $(i,j)$ qui est donc l'arête qui relie l'arc $i$ à l'arc $j$.
Il est évident qu'on a $T^{|I|} = \sqcup_J c(J)$, l'union étant disjointe puisqu'on a des égalités ou inégalités suivant si l'élément est présent ou non dans $J$. De là on a $T_a \simeq \sqcup_J c(J)$ où l'homéomorphisme est donné par : la composante d'arc marqué est envoyé sur la composante de cet arc et on parcourt le graphe en partant de ces points marqué, envoyant les composantes de chaque arête vers la composante de l'arc extrémité auquel on a pas encore attribué de composante, de sorte que si $(i,j) \in J$ alors les composantes engendrées par les extrémités de l'arc sont égalisées dans l'image de $c(J)$. De plus $c(J)$ est homéomorphe à $\R^{n-|J|}$ puisque chaque élément dans $J$ diminue le nombre de degrés de libertés par un, un élément de $E$ égalisant deux composantes et un élément de $M$ fixant une composante sur $p$.

Cela donne donc une décomposition cellulaire de $T_a$ dont les cellules sont des hypertores de dimension $n-|J|$ et le bord de $c(J)$ est donné par $\bigsqcup_{J' \supsetneq J} c(J')$. Cette décomposition se restreint à une décomposition cellulaire de $T_{<a} \cap T_a$. En effet, on observe que pour $b \rightarrow a$, on a
$$T_b \cap T_a = \bigsqcup_{J'\supset \{(y, z)\} } c(J')$$
avec $J'$ qui contient au moins l'arête $(y,z)$ donnée par le changement d'arcs de $b$ vers $a$. On a cela car $T_a \cap T_b$ égalise les composantes de $T_a$ données par les arcs $y$ et $z$ et est donc donné par un sous-hypertore de codimension $1$ dans $T_a$ pris sur la diagonale qui égalise les deux même composantes. Or toute cellule $c(J')$ pour $J'$ contenant $(y,z)$ a ses composantes $x_y$ et $x_z$ égalisées.
\end{proof}

\begin{exemple}\label{ex:partition}
On considère la décomposition de $T_b$ pour 
\begin{align*}
a &=\deuxdiag{Images_arxiv/B2_1.png} ,& 
b &=\deuxdiag{Images_arxiv/B2_2.png},
\end{align*}
en notant $b_1$ le grand arc de $b$ et $b_2$ l'autre. On obtient alors un graphe $\Gamma$ avec $b_1$ et $b_2$ comme sommets reliés ensemble puisque $a \rightarrow b$. On choisit $b_1$ comme sommet marqué et donc on a $E \sqcup M = \{(b_1, b_2), b_1\}$, on obtient alors le tore représentés en Figure \ref{fig:toredecomp} où la flèche horizontale est la composante associée à $b_1$ et la verticale à $(b_1, b_2)$ de sorte que l'homéomorphisme $\sqcup_J c(J) \simeq T_b$ envoie la composante associée à $b_1$ vers celle associée à $b_1$ dans $T_b$ et celle de $(b_1,b_2)$ vers celle de $b_2$ dans $T_b$. On obtient alors une décomposition comme dans la Figure \ref{fig:toredecomp} où le point bleu est donné par $c(\{(b_1, b_2), b_1\}) = \{(p,p)\}$, le cercle vert troué d'un point par $c(\{b_1\}) = (p,-) \setminus \{(p,p)\}$, le cercle rouge troué d'un point par $c(\{(b_1, b_2)\}) = (x,x) \setminus \{(p,p)\}$ et enfin le tore auquel on retire deux cercles donné par $c(\emptyset)$. On remarque de plus qu'on a $T_a \cap T_b = (x,x,x,x)$ et donc $T_a \cap T_b$ est donné dans $\sqcup_J c(J)$ par la diagonale qui est donc bien un sous-complexe.
\begin{figure}[H]
    \center
    \includegraphics[width=3cm]{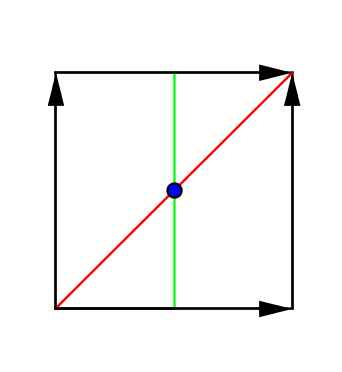}
    \caption{\label{fig:toredecomp} On décompose $T_b \simeq T^2$ en cellules : un point central en bleu $(p,p)$, un segment de droite $(p,-) \setminus \{(p,p)\}$, un segment de droite $(x,x) \setminus \{(p,p)\}$ et le restant de la surface.}
\end{figure}
Par contre, $T_a$ se décompose comme représenté en Figure \ref{fig:toredecomp2} et on remarque que $T_a \cap T_b$ n'est pas un sous-complexe puisque c'est aussi la diagonale.
\begin{figure}[H]
    \center
    \includegraphics[width=3cm]{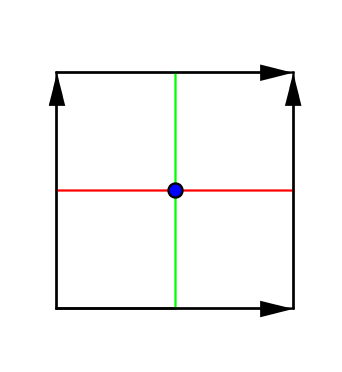}
    \caption{\label{fig:toredecomp2} On décompose $T_a \simeq T^2$ en cellules : un point central en bleu $(p,p)$, un segment de droite $(p,-) \setminus \{(p,p)\}$, un segment de droite $(-,p) \setminus \{(p,p)\}$ et le restant de la surface.}
\end{figure}

\end{exemple}

La première différence avec le travail de M. Khovanov est que dans sa décomposition cellulaire, puisqu'il utilise des sphères à la place de cercles, il obtient des cellules de dimensions paires. Dès lors, il obtient trivialement que l'anneau de cohomologie de $S_{\le a} \cap S_a$ est composé d'éléments de degrés pairs uniquement qui lui permettent de montrer facilement que le morphisme induit par l'inclusion $S_{<a} \cap S_a \hookrightarrow S_a$ est surjectif (les opérateurs de bords étant trivialement nuls) et par la suite que $\tau$ est injectif. On montre maintenant le lemme suivant adapté de \cite[Lemme 3.6]{HnCenter}.

\begin{lemme}\label{lem:SbSainS<aSainj}
Le morphisme
$$H(T_{<a} \cap T_a) \rightarrow \bigoplus_{b<a} H(T_b \cap T_a)$$
induit par l'inclusion $(T_b \cap T_a) \subset (T_{<a} \cap T_a)$ est injectif.
\end{lemme}

\begin{proof}
Il suffit de vérifier que 
$$H(T_{<a} \cap T_a) \rightarrow \bigoplus_{b \rightarrow a} H(T_b \cap T_a)$$
est injectif puisque c'est une composition par une projection. Par la preuve du Lemme \ref{lem:decompcellulaire}, on sait que la décomposition cellulaire de $T_{<a} \cap T_b$ se restreint à une décomposition cellulaire de $T_b  \cap T_a$ pour tout $b \rightarrow a$. Mais par le Lemme \ref{lem:sasbsa} on a $T_{<a} \cap T_a = \bigcup_{b\rightarrow a} (T_b \cap T_a)$ et donc c'est injectif. 

En effet, en général pour des CW-complexes $Y \subset X$, on a que l'application induite par l'injection est la projection
$$H^n(X) = \frac{\ker d_{n}}{\img d_{n+1}} \rightarrow H^n(Y) = \frac{\ker d'_{n}}{\img d'_{n+1}} \simeq  \frac{\ker d_{n}}{\img d_{n+1}}\Bigg /( X^n\setminus Y^n)$$
avec $X^n\setminus Y^n$ qui sont les cellules de $X^n$ non-présentes dans $Y^n$.
Mais donc ici, si l'application n'était pas injective, ça signifierait qu'on pourrait trouver au moins une cellule de $T_{<a} \cap T_a$ présente dans aucun des $T_b \cap T_a$ ce qui est absurde.
\end{proof}

On remarque que $T_{\le a} = T_{<a} \cup T_a$ et donc, par la proposition suivante, on a une suite de Mayer-Vietoris :
\begin{equation}
\xymatrix{
& \dots \ar[r] &  H^{m-1}(T_{a} \cap T_{<a}) \ar[dll]_{\delta}   \\
H^m(T_{a} \cup T_{<a})\ar[r]  &  H^m(T_{a}) \oplus H^m(T_{<a})\ar[r]  &  H^m(T_{a} \cap T_{<a})  \ar[dll]_\delta \\
  \ar[r] H^{m+1}(T_{a} \cup T_{<a})  & \dots
}\label{eq:mv}\end{equation}
En effet, par la Proposition \ref{prop:retracte} des Annexes, les groupes de cohomologie d'un espace et d'une déformation rétracte de celui-ci sont isomorphes et donc la suite de Mayer-Vietoris pour les voisinages de la proposition suivante donne la suite voulue.

\begin{proposition}\label{prop:voisretract}
Pour tout $a \in B^n$ il existe un voisinage ouvert $T^\epsilon_{a}$ de $T_{a}$ dans $T_{\le a}$ tel que $T_{a}$ est une déformation rétracte de ce voisinage, on l'appelle alors voisinage rétracte de $T_{a}$. De même on a des voisinages rétractes $T_{<a}^\epsilon$ et $T_{<a,a}^\epsilon$ pour respectivement $T_{<a}$ et $T_{<a} \cap T_a$ tels que
$$T_{a}^\epsilon \cap T_{<a}^\epsilon = T_{a,<a}^\epsilon.$$
\end{proposition}

\begin{proof}
Pour tout $b \in B^n$, on sait que $T_a \cup T_b$ est l'union de deux hypertores de dimension $n$ identifiés sur un sous-hypertore donné par un hyperplan diagonal de dimension $|W(a)b|$. On construit alors $T_ {a,b}^\epsilon$ comme un épaississement ouvert de cet hyperplan de largeur $2\epsilon$ dans $T_a$ et $T_b$, comme illustré en Figure \ref{fig:toreepaissi}, donnant
$$T_{a,b}^\epsilon \simeq \frac{\left( T^{|W(a)b|} \times ]-\epsilon, \epsilon[^{(n-|W(a)b|)} \right) \sqcup \left(T^{|W(a)b|} \times ]-\epsilon, \epsilon[ ^{(n-|W(a)b|)}\right)}{T^{|W(a)b|}}.$$
Il est clair que cela donne un ouvert dans $T_a \cup T_b$ et que $T_a \cap T_b$ est une déformation rétracte de $T_{a,b}^\epsilon$ en envoyant les $ ]-\epsilon, \epsilon[$ sur $\{0\}$. On construit alors 
$$T_{<a,a}^\epsilon := \bigcup_{b < a} T_{a,b}^\epsilon$$
qui est un ouvert de $T_{<a} \cup T_{a}$ et un voisinage rétracte de $T_{<a} \cap T_{a}$ pour la même raison. On définit
\begin{align*}
T_{a}^\epsilon &:= T_{a} \cup T_{<a,a}^\epsilon, & T_{<a}^\epsilon &:= T_{<a} \cup T_{<a,a}^\epsilon,
\end{align*}
et on observe que ce sont des voisinages rétractes de $T_a$ et $T_{<a}$ ayant la propriété voulue par définition.
\end{proof}

\begin{figure}[h]
    \center
    \includegraphics[width=6cm]{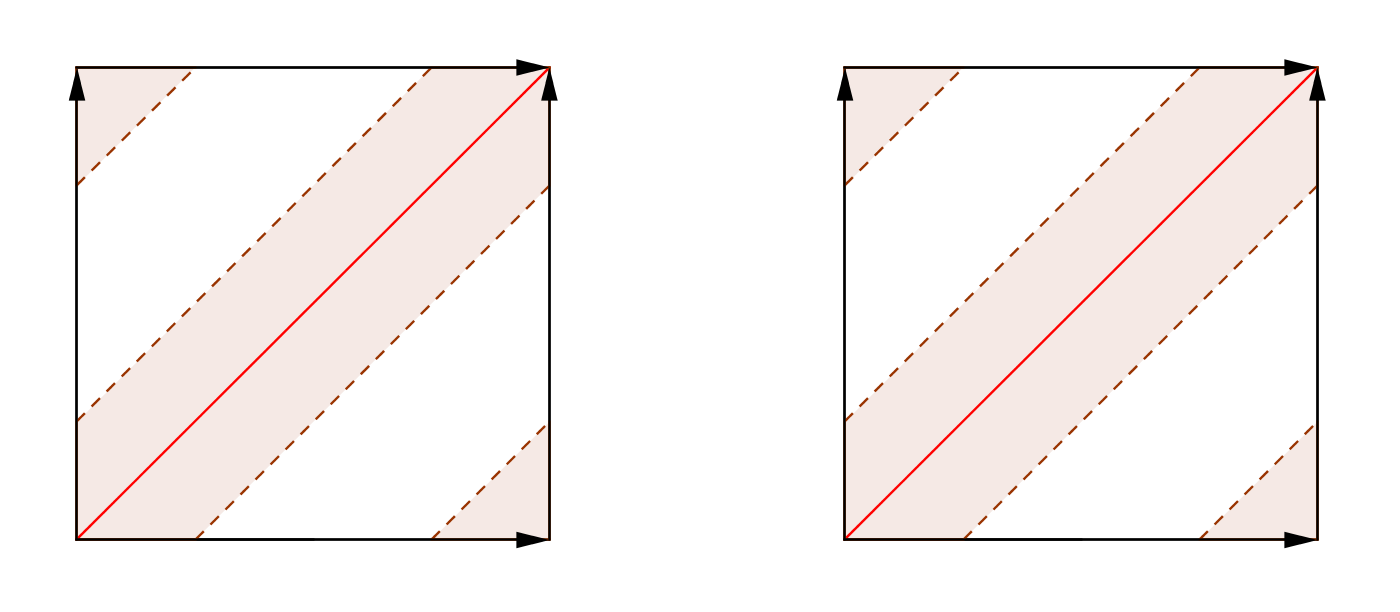}
    \caption{\label{fig:toreepaissi} $T_a \sqcup T_b$ identifiés ensemble sur la diagonale $T_a \cap T_b$ en rouge et $T_{a,b}^\epsilon$, l'intersection épaissie, en brun. }
\end{figure}

M. Khovanov obtient par surjectivité de $H(S_a) \rightarrow H(S_{<a} \cap S_a)$ une version un peu plus forte de la proposition suivante dans \cite[Proposition 3.8]{HnCenter} où il a en plus un "$0 \rightarrow$" à gauche de la suite exacte.

\begin{proposition}\label{prop:suiteexacteHtilde}
La suite suivante est exacte
$$\xymatrix{
H(T_{\le a}) \ar[r]^-{\phi}  & \bigoplus_{b \le a} H(T_b) \ar[r]^-{\psi^-} & \bigoplus_{b < c \le a} H(T_b \cap T_c)
}$$ 
où $\phi$ est le morphisme induit par les inclusions $T_b \hookrightarrow T_{\le a}$ et on définit
$$\psi^- = \sum_{b < c \le a} (\psi_{b,c} - \psi_{c,b})$$
avec
$$\psi_{b,c} = \psi_{b;b,c} : H(T_b) \rightarrow H(T_b \cap T_c)$$
induit par l'inclusion $(T_b \cap T_c) \hookrightarrow T_b$.
\end{proposition}

\begin{proof}
La preuve se fait par induction sur $a$ par rapport à l'ordre total sur $B^n$. Pour $a_0 \in B^n$ minimal, on obtient la suite
$$\xymatrix{
H(T_{a_0}) \ar[r]^-{\simeq} &  H(T_{a_0}) \ar[r]^-{\psi^-} & 0
}$$ 
qui est bien évidemment exacte. On suppose maintenant que $a$ est le prochain élément après $e$ pour l'ordre total et que la proposition soit vrai pour $e$. Par le Lemme \ref{lem:SbSainS<aSainj}, puisque le morphisme est injectif et donc que la composition préserve le noyau, on peut remplacer $H^m(T_{<a} \cap T_a)$ par $\bigoplus_{a<b} H^m(T_b \cap T_a)$ dans la suite exacte de Mayer-Vietoris  (\ref{eq:mv}) en conservant l'exactitude partout sauf pour les flèches $H^m(T_{<a} \cap T_a) \rightarrow H^{m+1}(T_{\le a})$, donnant une suite exacte
\begin{align}H(T_{\le a}) \rightarrow H(T_{< a}) \oplus H(T_a) \rightarrow \bigoplus_{b<a} H(T_b \cap T_a).  \label{eq:exact1}\end{align}
De plus, par hypothèse d'induction, la suite 
\begin{align}H(T_{\le e}) \rightarrow \bigoplus_{f \le e} H(T_f) \rightarrow \bigoplus_{f<g\le e} H(T_f \cap T_e) \label{eq:exact2}\end{align}
est exacte. Puisque $T_{<a} = T_ {\le e}$, on obtient par (\ref{eq:exact2}) et (\ref{eq:exact1}) le diagramme
$$\xymatrix{
H(T_{\le a}) \ar[r] & H(T_{\le e}) \oplus H(T_a) \ar[r] \ar[d] & \bigoplus_{b<a} H(T_b \cap T_a) \\
&\bigoplus_{f \le e} H(T_f) \oplus H(T_a) \ar[r] & \bigoplus_{f < g \le e} H(T_f \cap T_e) \oplus \bigoplus_{b < a} H(T_b \cap T_a)
}$$
qui donne donc la suite exacte voulue
$$\xymatrix{
H(T_{\le a}) \ar[r]^-{\phi}  & \bigoplus_{b \le a} H(T_b) \ar[r] & \bigoplus_{b < c \le a} H(T_b \cap T_c).
}$$ 
Plus précisément, on suppose que $x = \sum_{b \le a} x_b = \left(\sum_{f\le e} x_f\right) + x_a$ soit dans l'image de $\phi$ et on a
$$\psi^{-}(x) = \sum_{f<g\le e} (\psi_{f,g}(x) - \psi_{g,f}(x)) + \sum_{b<a}  (\psi_{b,a}(x) - \psi_{a,b}(x)) $$
mais $\psi_{f,g}(x_a) = \psi_{g,f}(x_a) = 0$ et donc par exactitude de (\ref{eq:exact2})
$$\sum_{f<g\le e} (\psi_{f,g}(x) - \psi_{g,f}(x)) = 0.$$
De plus $\psi_{b,a}(x_c) = 0$ si $c \ne b$ et $ \psi_{a,b}(x_c) = 0$ si $c \ne a$, donc on a
$$ \sum_{b<a}  (\psi_{b,a}(x) - \psi_{a,b}(x)) = \sum_{b<a}  (\psi_{b,a}(x_b) - \psi_{a,b}(x_a)) = 0$$
par le fait que la suite de Mayer-Vietoris fait apparaitre un signe similaire à $\psi^-$ et par l'exactitude de (\ref{eq:exact1}).
\end{proof}

\begin{proposition}\label{prop:supercentresurjectif}
Il existe un épimorphisme d'anneaux (non-gradués) 
$$k : H(\widetilde T) \rightarrow OZ(OH^n_C).$$
\end{proposition}

\begin{proof}
On prend $a$ maximal dans la Proposition \ref{prop:suiteexacteHtilde} donnant la suite exacte
$$\xymatrix{
H(\widetilde T) \ar[r]^-{\phi}  & \bigoplus_{b} H(T_b) \ar[r]^-{\psi^-} & \bigoplus_{b < c} H(T_b \cap T_c)
}$$ 
et on observe que
$$\Eq(\psi) = \ker(\psi^-) = \img(\phi)$$
montrant que le morphisme d'anneaux (non-gradué) $\tau : \widetilde T \rightarrow \Eq(\psi)$ défini dans le Diagramme (\ref{eq:tau}) est surjectif puisqu'il correspond avec $\phi$ et donc en composant avec les isomorphismes du même diagramme on obtient l'épimorphisme voulu.
\end{proof}

Dans le but de calculer le rang de $H(\widetilde T)$ on montre le lemme suivant, cité par M. Khovanov dans sa preuve de \cite[Lemme 4.1]{HnCenter} mais laissé non-démontré.

\begin{lemme}\label{lem:ta2n}
Pour tout $a \in B^n$, on note $t(a)$ le nombre d'arcs inférieurs de $a$, c'est-à-dire le nombre d'arcs qui ne sont pas contenus dans un autre arc. On obtient alors
$$\sum_{a \in B^n} 2^{t(a)} = \begin{pmatrix}2n \\ n \end{pmatrix}.$$
\end{lemme}

\begin{proof}
%
%

On montre l'identité équivalente
$$K_n := \frac{1}{n+1}\sum_{a \in B^n} 2^{t(a)} = \frac{1}{n+1}\begin{pmatrix}2n \\ n \end{pmatrix} = C_n$$
avec $C_n$ le $n-$ème nombre de Catalan. On a à gauche pour $n=1$
$$K_1 = \frac{1}{2} 2^1 = 1 = C_1.$$
On affirme que l'élément de gauche respecte la relation de récurrence de Segner 
$$K_{n+1} = \sum_{i=0}^n K_iK_{n-i} = \sum_{i=1}^{n+1} K_{i-1}K_{n+1-i}$$
qui est une des façon de définir les nombres de Catalan, voir \cite[Chapitre 5]{catalan}. Pour prouver cela, on suppose par induction qu'elle la respecte pour tout $k \le n$ et donc que $K_k = C_k$ et $\sum_{a \in B^k} 2^{t(a)}= \begin{pmatrix} 2k \\ k \end{pmatrix}$.

Tout d'abord on observe qu'on peut construire $B^{n+1}$ à partir des $B^i$ pour $i \le n$ par une construction comme en Figure \ref{fig:reccurrenceBn} :  
on peut écrire $B^{n+1}$ comme l'union  $\bigcup_{1 \le k \le n+1}(B^{n+1})_k$ où les éléments de $(B^{n+1})_k$ sont donnés par un arc $(1,2k)$ avec dans cet arc un élément de $B^{k-1}$ et à droite un élément de $B^{n+1-k}$.

\begin{figure}[h]
    \center
    \includegraphics[width=10cm]{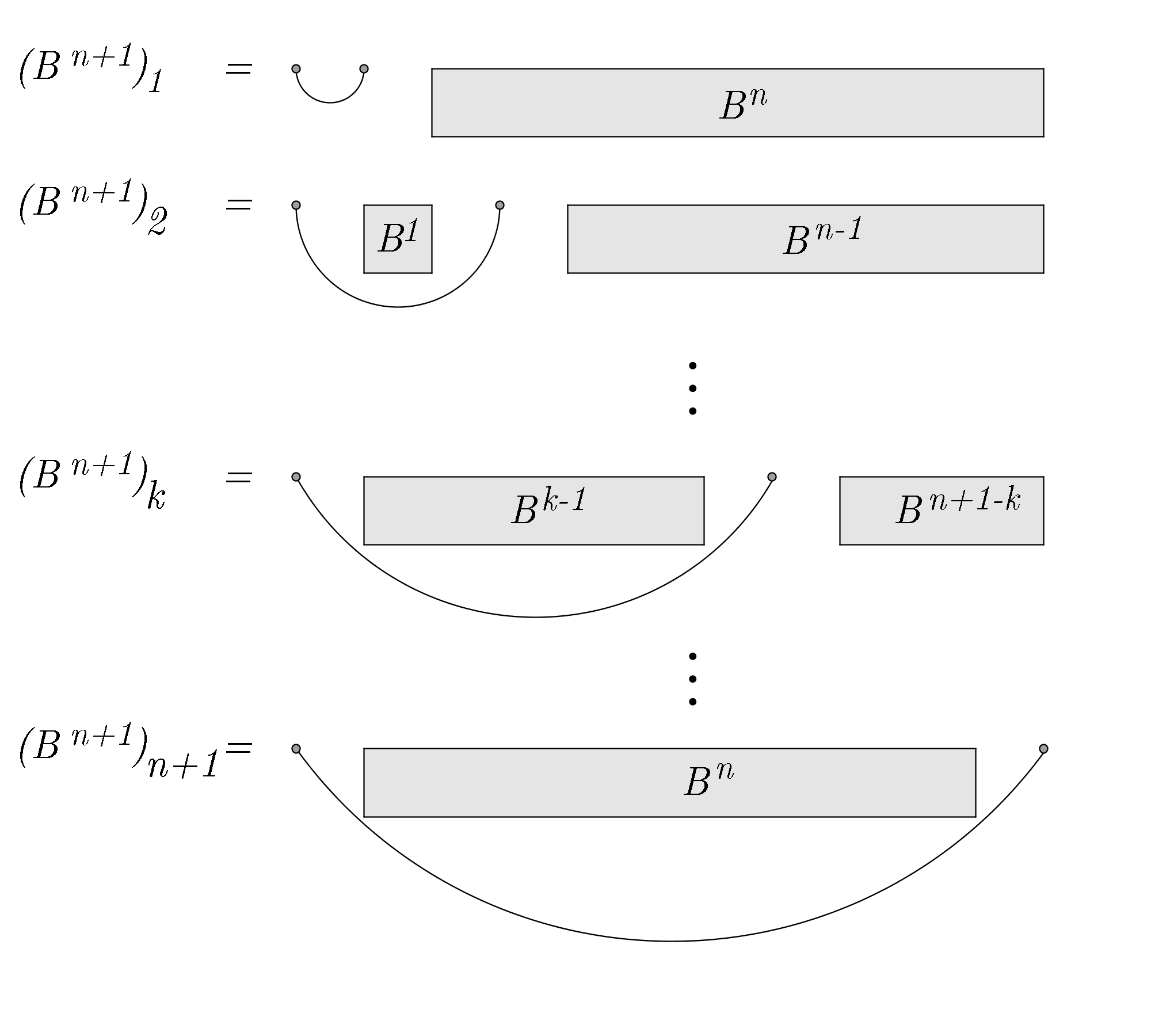}
    \caption{\label{fig:reccurrenceBn} Construction par récurrence de $B^{n+1} = \bigcup_{1 \le k \le n+1} (B^{n+1})_k$ à partir de $B^i$ pour tout $i \le n$.}
\end{figure}

 On obtient alors qu'un élément de $(B^{n+1})_k$ possède autant d'arcs inférieurs que l'élément de droite de $B^{n+1-k}$ plus un. Dès lors, si on somme sur toutes les possibilités d'éléments à droite, on obtient par hypothèse d'induction $2\begin{pmatrix} 2(n+1-k) \\ n+1-k \end{pmatrix}$ et on multiplie par $\left|B^{k-1}\right| = C_{k-1}$ pour le nombre de possibilités d'éléments de $B^{k-1}$ dans l'arc $(1,2k)$, livrant l'équation
$$\sum_{a \in B^{n+1}} 2^{t(a)} = \sum_{k=1}^{n+1} \left( \left|B^{k-1}\right| 2\sum_{b \in B^{n+1-k}} 2^{t(b)}\right)  = \sum_{k=1}^{n+1} C_{k-1} 2\begin{pmatrix} 2(n+1-k) \\ n+1-k \end{pmatrix}.$$
Cela se traduit en termes de $K_{n}$, en utilisant l'hypothèse d'induction que $K_{k-1} = C_{k-1}$, par
\begin{align*}
K_{n+1} &=  \frac{1}{n+2}\sum_{k = 1}^{n+1} K_{k-1}.2(n-k+2).K_{n-k+1}.
\end{align*}
On a donc
$$K_{n+1} = 2 \sum_{k = 1}^{n+1} K_{k-1}.K_{n-k+1} - \sum_{k = 1}^{n+1}\frac{2k}{n+2} K_{k-1}K_{n-k+1}.$$
On montre que pour tout $k \le \lfloor (n+1)/2 \rfloor$ on a en sommant le $k$-ème terme de la somme et le $(n+2-k)$-ème
$$\frac{2k}{n+2} K_{k-1}.K_{n-k+1} + \frac{2(n+2-k)}{n+2} K_{n+1-k}.K_{k-1} = K_{k-1}.K_{n-k+1} + K_{n+1-k}.K_{k-1} $$
et si $n+1$ est impair alors pour $k = \lceil (n+1)/2 \rceil$, on a $2k = n+2$ et donc
$$ \frac{2k}{n+2} K_{k-1}.K_{n-k+1} = K_{k-1}.K_{n-k+1}$$
ce qui nous donne
$$\sum_{k = 1}^{n+1}\frac{2k}{n+2} K_{k-1}.K_{n-k+1} = \sum_{k = 1}^{n+1} K_{k-1}.K_{n-k+1}$$
et cela conclut la preuve.
\end{proof}

\begin{proposition}\label{prop:rangHS}
Vu en tant que groupe abélien, on a
$$\rank\bigl(H(\widetilde T)\bigr) =  \begin{pmatrix}2n \\ n\end{pmatrix}.$$
De plus, $H(\widetilde T)$ est un groupe abélien libre.
\end{proposition}

\begin{proof}
Par le Lemme \ref{lem:decompcellulaire}, on a une décomposition cellulaire de $S_a$ qui se restreint à une décomposition de $T_a \cap T_{<a}$. Dès lors, on peut obtenir une partition cellulaire de $\widetilde T$ à partir de cette décomposition, en commençant par prendre la décomposition de $T_{a_0}$ pour $a_0$ l'élément minimal de $B^n$ puis en ajoutant les cellules de $T_{a_m} \setminus T_{<a_m}$ pour tout $a_m$ en suivant l'ordre total défini sur $B^n$. Il est important de noter que cela forme une partition cellulaire et non pas une décomposition puisque la fermeture d'une cellule n'est pas en général une union de cellules, comme le montre l'Exemple \ref{ex:partition}. Mais cela suffit pour la preuve car on affirme que le rang de $H(\widetilde T)$ est donné par le nombre de cellules de la partition. Puisque le rang des groupes d'homologie et de cohomologie sont égaux, on montre que la somme des rangs des groupes d'homologie est donnée par le nombre de cellules. En effet, on a la suite de Mayer-Vietoris, donnée par la Proposition \ref{prop:voisretract},
$$\xymatrix{
& \dots & \ar[l] H_{k-1}(T_{a_m} \cap T_{<a_m})   \\
\ar[urr]^{\partial} H_k(T_{a_m} \cup T_{<a_m})  & \ar[l] H_k(T_{a_m}) \oplus H_k(T_{<a_m})  & \ar[l] H_k(T_{a_m} \cap T_{<a_m}) \\
  \ar[urr]^\partial H_{k+1}(T_{a_m} \cup T_{<a_m})  & \ar[l] \dots
}$$
et les $\partial$ sont nuls car l'application $H_k(T_{a_m} \cap T_{<a_m}) \rightarrow H_k(T_{a_m}) $ est injective. Intuitivement, cela vient du fait que l'intersection est une union d'hypertores de codimension un donné par des diagonales dont les intersections sont des hypertores de dimension inférieure et donc un cycle non nul dans un de ces hypertores reste non nul dans l'hypertore de dimension $n$.
Plus formellement, on sait par l'Exemple \ref{ex:torehomology} des Annexes que $H_k(T_{a_m})$ est donné par le groupe abélien libre à $\begin{pmatrix}n \\ k \end{pmatrix}$ éléments et on a une décomposition cellulaire de $T_{a_m}$ en le même nombre de cellules de dimension $k$ (les cellules de dimension $k$ étant celles avec $|J| = n-k$ et donc on doit piocher $n-k$ éléments dans $n$ disponibles, ce qui égal à en piocher $k$ dans $n$). Cela signifie que chacune des cellules de la décomposition engendre un unique générateur du groupe d'homologie de $T_{a_m}$. Puisque $T_{a_m} \cap T_{<a_m}$ est un sous-complexe, son homologie est engendrée par ses cellules. Dès lors, on ne peut pas engendrer de nouvelle relation par l'inclusion puisqu'on sait qu'elles sont linéairement indépendantes dans $T_{a_m}$. Ceci donne l'information supplémentaire que $H_k(T_{a_m} \cap T_{<a_m})$ est un groupe abélien libre de rang égal au nombre de cellules du sous-complexe. De là on obtient 
$$H_k(T_{a_m}  \cup T_{<a_m}) \simeq \frac{H_k(T_{a_m}) \oplus H_k(T_{<a_m}) }{H_k(T_{a_m} \cap T_{<a_m})}.$$
qui se traduit en termes de rang par
$$\rank(H_k(T_{a_m} \cup T_{<a_m}) = \rank(H_k(T_{a_m})) + \rank(H_k(T_{<a_m})) - \rank(H_k(T_{a_m} \cap T_{<a_m}))$$
et donc on obtient par récurrence que le rang de $H_k(\widetilde T)$ est égal au nombre de cellules de dimension $k$ puisque $ \rank(H_k(T_{a_m})) - \rank(H_k(T_{a_m} \cap T_{<a_m}))$ est le nombre de cellules qu'on ajoute.

Pour $a \in B^n$, on pose $t(a)$ le nombre d'arcs inférieurs de $a$, c'est-à-dire le nombre d'arcs qui ne sont pas contenu dans un autre arc. On remarque que $t(a)$ est aussi le nombre de composantes connexes du graphe $\Gamma$ de la preuve du Lemme \ref{lem:decompcellulaire}. En effet, chaque composante possède au moins un arc inférieur sinon le graphe serait infini et on ne peut pas avoir plusieurs arcs inférieurs dans une même composante par le fait que deux arcs sont reliés si l'un contient l'autre sans arc intermédiaire.
La décomposition de $T_a \setminus T_{<a}$ possède $2^{t(a)}$ cellules puisque les cellules qu'il reste à attacher sont celles qui ne possèdent aucune des arêtes données par les flèches $b \rightarrow a$, donc il ne reste que les cellules définies par les $J$ ne possédant que des arcs marqués, donc  une par composante connexe du graphe. Dès lors, on a $2^{t(a)}$ possibilités suivant si on prend où non l'arc marqué dans $J$. De là on obtient que la partition cellulaire de $\widetilde T$ possède $\sum_a 2^{t(a)}$ cellules puisque l'élément minimal est composé de $n$ arcs inférieurs et que la décomposition de $T_{a_0}$ possède $2^n$ cellules. Par le Lemme \ref{lem:ta2n} cette somme est égale à $\begin{pmatrix}2n \\ n\end{pmatrix}$.
\end{proof}
%

\begin{proof}\emph{du Théorème \ref{thm:supercentre}. }
Par la Proposition \ref{prop:supercentreinjectif} on obtient que 
$$\rank(OZ(OH^n_C)) \ge \rank (OH(\mathfrak B_{n,n})).$$
Par les Propositions \ref{prop:supercentresurjectif}, \ref{prop:rangHS} et \ref{prop:rankhoh} on obtient que
$$\rank(OZ(OH^n_C)) \le \rank(\HSt) = \begin{pmatrix}2n \\ n\end{pmatrix} = \rank (OH(\mathfrak B_{n,n}))$$
et donc on a
$$\rank(OZ(OH^n_C)) = \rank (OH(\mathfrak B_{n,n}))$$
qui permet de conclure que le monomorphisme $j$ de la Proposition \ref{prop:supercentreinjectif} est un isomorphisme.
\end{proof}

\begin{remarque}
On peut remarquer qu'on a en plus montré que le rang de $OZ(OH^n_C)$ est le même que $H(\widetilde T)$ et donc $\tau$ est en fait un isomorphisme (non-gradué). En composant avec l'isomorphisme vers la présentation de la cohomologie impaire de $\mathfrak{B}_{n,n}$, on obtient une présentation de $H(\widetilde T)$ en changeant le degré des générateurs $x_1, \dots, x_{2n}$ des polynôme impairs en $1$. On pouvait aussi l'observer directement en voyant que les opérateurs de bord de la suite de Mayer-Vietoris (\ref{eq:mv}) sont nuls, ajoutant un $0 \rightarrow$ à gauche dans la suite exacte de la Proposition \ref{prop:suiteexacteHtilde} et donnant l'injectivité de $\tau$. Sans doute est-il possible de montrer directement que $H(\widetilde T)$ admet une présentation semblable à l'anneau de cohomologie impaire de $\mathfrak{B}_{n,n}$ en adaptant la preuve de la Proposition \ref{prop:nk1} et/ou en utilisant un raisonnement similaire à M. Khovanov dans la preuve de \cite[Théorème 1.4]{HnCenter}, donnant un isomorphisme non-gradué entre les deux, et il suffirait alors de montrer que la composition de $\tau$ avec cet isomorphisme donne un morphisme gradué, livrant une preuve alternative. 

Par ailleurs, une preuve similaire serait certainement possible en construisant $\widetilde T$ à partir d'hypersphère $S^3$ (ou toute autre dimension impaire $m$ plus grande que $1$), permettant d'utiliser des arguments similaires à ceux de M. Khovanov pour montrer que les opérateurs de bords sont nuls (puisqu'on a des groupe de cohomologie triviaux pour les entiers qui ne sont pas multiple de cette dimension). De plus, cela donne une présentation de ces anneaux de cohomologie toujours en modifiant le degré des générateurs $x_1, \dots, x_{2n}$ des polynôme impairs en fonction de la dimension choisie pour $S^m$.
\end{remarque}


\chapter{Pour aller plus loin}\label{chap:allerplusloin}

Ce chapitre propose une série d'idées pour aller plus loin dans l'étude et la compréhension de l'anneau impair des arcs des Khovanov, notamment en proposant un substitut de module qui pourrait permettre d'imiter une bonne partie des résultats connus sur $H^n$ et pourrait mener à un nouvel invariant d'enchevêtrements. On montre aussi qu'il existe une façon de tourner $OH^n_C$ en un anneau associatif et on étudie les règles de multiplications qui donnent des anneaux isomorphes.


\section{Rendre $OH^n_C$ associatif}

On peut se demander s'il existe des coefficients $\lambda(x,y) \in \{-1,1\}$ tels que la multiplication
$$(x,y) \rightarrow \lambda(x,y)xy$$
soit associative pour tout $x,y \in OH^n_C$. La réponse est affirmative mais demande des résultats qui dépassent le cadre de ce travail. On en propose quand-même l'idée en renvoyant le lecteur vers des références pour les démonstrations et détails des résultats utilisés.

\begin{figure}[h]
    \center
    \includegraphics[width=5cm]{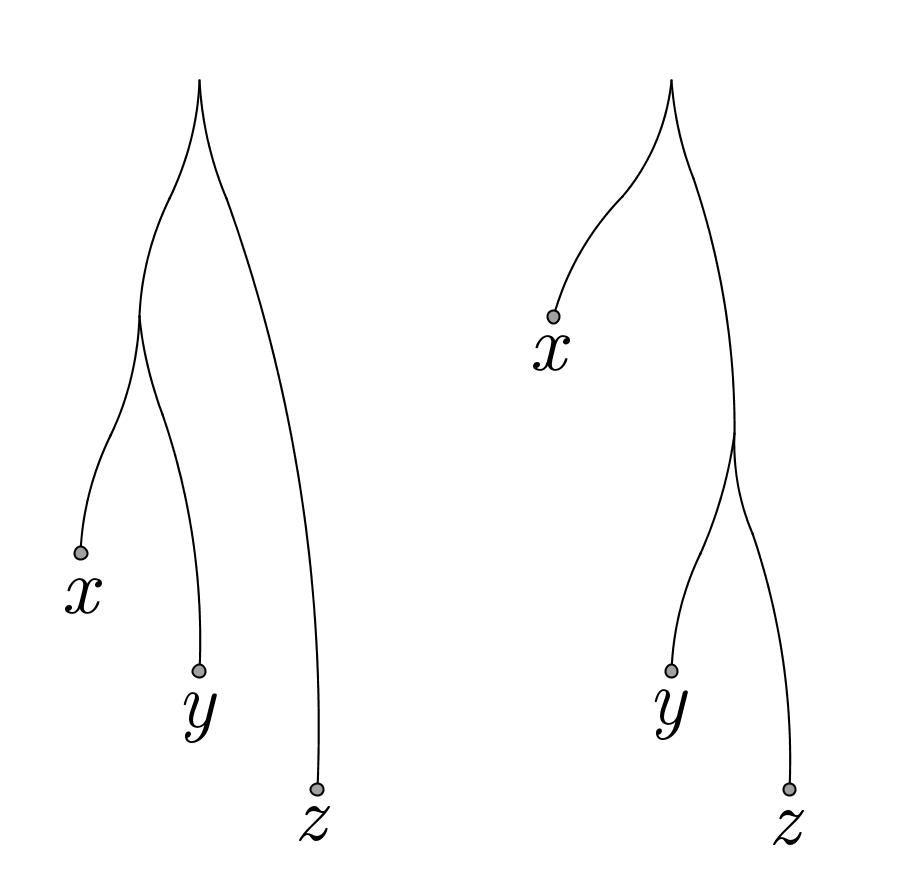}
    \caption{\label{fig:chgtac2} \`A gauche on a $(xy)z$ et à droite $x(yz)$ : $x$ commute avec les scissions du cobordisme de la multiplication $yz$.}
\end{figure}

Tout d'abord il faut observer plus attentivement d'où vient la non-associativité. En effet, elle est due à deux phénomènes différents : d'une part les changements de chronologie dans les cobordismes, dépendant juste des diagrammes, comme dans l'Exemple \ref{ex:chgtch}, et d'autre part l'anticommutativité avec l'élément de gauche du double produit et des scissions de la multiplication de droite, dépendant du degré et des diagrammes, représenté en Figure \ref{fig:chgtac2} et illustré dans l'Exemple \ref{ex:OH2nonassoc}. \'Etant donné la nature du second, on observe le résultat suivant qui permet de calculer le nombre de scissions d'un cobordisme $W(c)bW(b)a \rightarrow W(c)b$ et donc le nombre de scissions avec lesquelles l'élément de gauche doit permuter.


\begin{proposition}\label{prop:nbrscissions}
La multiplication 
$$c(OH^n_C)b \times b(OH^n_C)a \rightarrow c(OH^n_C)a$$
est composée de $S(c,b,a) := \frac{d(c,b)+d(b,a) - d(c,a)}{2}$ scissions.
\end{proposition} 

\begin{proof}
Il suffit d'observer qu'on a autant de scissions que $n$ moins le nombre de fusions. Le cobordisme
$$W(c)bW(b)a \rightarrow W(c)a$$
doit envoyer $|W(c)b| + |W(b)a|$ composantes sur $|W(c)a|$ c'est-à-dire qu'il faut en fusionner au moins $|W(c)b| + |W(b)a| - |W(c)a|$. Pour toute fusion supplémentaire, on doit avoir une scission correspondante afin d'équilibrer le nombre de composantes, ce qui donne
$$|W(c)b| + |W(b)a| - |W(c)a| + F_{supp} + S_{supp} = n$$
avec $F_{supp} = S_{supp}$, pour $F_{supp}$ et $S_{supp}$ le nombre de fusions et scissions supplémentaires, la somme devant faire $n$ car on a $n$ ponts. Ceci donne
\begin{align*}
 2S_{supp} &= n - |W(c)b| - |W(b)a| + |W(c)a| \\
&=  n - |W(c)b| + n - |W(b)a| - n + |W(c)a| \\
&=  d(c,b)+d(b,a) - d(c,a),
\end{align*}
d'où on obtient la formule de la proposition.
\end{proof}

Cela montre bien que peut importe le choix de $C$, on a toujours autant de scissions dans le produit. On considère maintenant le groupoïde $\ddot B^n$ (voir Section \ref{sec:groupoide} des Annexes pour une définition de groupoïde) dont les objets sont les éléments de $B^n$ et les flèches les diagrammes $W(b)a : a \rightarrow b$ qu'on peut composer $W(c)b \circ W(b)a := W(c)a : a \rightarrow c$. Chaque flèche possède un inverse puisque $W(a)b \circ W(b)a = W(a)a = \Id_a$. Par ailleurs, on peut définir une graduation sur les éléments de $OH^n$ (vu en tant que groupe abélien gradué) donnée par les flèches de $\ddot B^n$ et la décomposition en somme directe 
$$OH^n := \bigoplus_{a,b \in B^n} b(OH^n)a =  \bigoplus_{W(b)a \in \Hom\left( \ddot B^n\right)} b(OH^n)a .$$
Cela forme une graduation pour $OH^n_C$ car la multiplication de deux éléments de $OH^n_C$ a comme degré la composition des degrés. On note le degré pour cette graduation $|-|_2$ et on l'appelle \emph{degré diagrammatique}. On note celui de la graduation usuelle par $|-|_1$ et le double degré dans le groupoïde $\Z \times \ddot B^n$ est noté $|-|$. Pour $|x|_2 = W(c)b$ et $|y|_2 = W(b)a$, on note $S(|x|_2, |y|_2) = S(c,b,a)$ le nombre de scissions donné par la Proposition \ref{prop:nbrscissions} et on note $d(|x|_2) = d(c,b)$.
On considère l'associateur
$$\alpha : (OH^n_C \otimes OH^n_C) \otimes OH^n_C \rightarrow OH^n_C \otimes (OH^n_C \otimes OH^n_C),$$
$$(x \otimes y) \otimes z \mapsto \phi(|x|,|y|,|z|)(x \otimes (y \otimes z))$$
avec $\phi(|x|,|y|,|z|) \in \{-1, 1\}$ tel que
$$(xy)z = \phi(|x|,|y|,|z|) x(yz).$$
Puisque la non-associativité est due à deux phénomènes on peut séparer cet associateur en deux éléments
\begin{align*}
\phi(|x|,|y|,|z|) &=(-1)^{p(x).S(|y|_2, |z|_2)}.(-1)^ {\phi_0(|x|_2, |y|_2, |z|_2)}\\
&= (-1)^{\frac{(|x|_1-d(|x|_2))}{2}.S(|y|_2, |z|_2)}.(-1)^{\phi_0(|x|_2, |y|_2, |z|_2)}\\
&= (-1)^{\phi_1(|x|,|y|,|z|)}.(-1)^{\phi_0(|x|_2, |y|_2, |z|_2)}
\end{align*}
où $\phi_0 : \Hom(\ddot B^n)^3 \rightarrow \Z/2\Z$  calcule le signe du au changement de chronologie en fonction des diagrammes et $\phi_1 : \Hom(\Z \times \ddot B^n)^3 \rightarrow \Z/2\Z$ calcule le signe de la permutation de $x$ avec la multiplication $yz$ en fonction des diagrammes et du degré de $x$.

L'idée de la preuve est d'observer que le groupoïde $\ddot B^n$ a une forme de $m$-simplexe puisque pour toute paire d'objets on a une unique flèche de l'un vers l'autre et une unique flèche dans l'autre sens, avec $m = 1, 4$ ou plus (puisque $|B^2| = 2$ et $|B^3| = 5$). Dès lors, par l'Exemple \ref{ex:groupoidsimplexe} des Annexe, sa cohomologie de degré $3$ dans $\Z/2\Z$ est triviale (on renvoie vers la Section \ref{sec:groupoide} pour une définition de cohomologie de groupoïdes). Par ailleurs, puisque $\phi$ est un associateur, le diagramme suivant, où $O := OH^n_C$, commute :
\begin{align} \label{diag:assoc}
\xymatrix{
((O \otimes O ) \otimes O ) \otimes O  \ar[rr]^{\alpha \otimes \Id} \ar[d]_{\alpha} && (O\otimes (O \otimes O))\otimes O \ar[d]^{\alpha} \\
(O\otimes O) \otimes (O\otimes O) \ar[dr]_{\alpha} && O\otimes((O\otimes O)\otimes O) \ar[dl]^{\Id\otimes\alpha} \\
&O\otimes(O\otimes(O\otimes O)) &
}
\end{align}
Cela se traduit en termes de $\phi_0$ et en prenant les éléments ${_e1_d},{_d1_c},{_c1_b}$ et ${_b1_a}$ par
$$\phi_0(h,k,l)+\phi_0(g,hk,l)+\phi_0(g,h,k)-\phi_0(gh,k,l)-\phi_0(g,h,kl) = 0 = d^3 \phi_0(g,h,k,l)$$
pour toute suite
$$\xymatrix{e & \ar[l]_{g} d & \ar[l]_{h} c & \ar[l]_{k} b & \ar[l]_{l} a}$$
dans $\ddot B^n$, ce qui signifie que $\phi_0$ est un cocycle. Puisque la cohomologie de degré $3$ de $\ddot B$ est triviale, $\phi_0$ est un cobord et donc il existe $\lambda_0 : \Hom(\ddot B^n)^2 \rightarrow \Z/2\Z$ tel que $d^2 \lambda_0 = \phi_0$, c'est-à-dire
$$(d^2 \lambda_0)(|x|_2, |y|_2, |z|_2) = \lambda_0(y,z)-\lambda_0(xy,z)+\lambda_0(x,yz)-\lambda_0(x,y) = \phi_0(|x|_2, |y|_2, |z|_2).$$
De même, on obtient que $\phi_1$ est un cocycle dans le groupoïde $\Z \times \ddot B^n$. On calcule que la réalisation géométrique de $\Z \times \ddot B^n$ est
$$|\Z \times \ddot B^n| \simeq S^1 \times \Delta^m$$
et donc on peut les décomposer en cellules. Par la formule de Künneth (\cite[Théorème 3.15]{Hatcher}) on obtient alors
$$H^3(\Z \times \ddot B^n, \Z/2\Z) \simeq H^3(S^1 \times \Delta^m, \Z/2\Z) \simeq H^3(S^1, \Z/2\Z) \otimes_{\Z/2\Z} H^3(\Delta^m,  \Z/2\Z) \simeq \{0\}.$$
Dès lors $\phi_1$ est un cobord et on a un $\lambda_1 : \Hom(\Z \times \ddot B^n)^2 \rightarrow \Z/2\Z$ tel que $d^2 \lambda_1 = \phi_1$. On pose
$$\lambda(|x|,|y|) = (-1)^{\lambda_1(|x|,|y|)}.(-1)^{\lambda_0(|x|_2, |y|_2)}$$
et on obtient 
\begin{align*} 
(x,\lambda(|y|,|z|) yz) &\mapsto \lambda(|x|,|yz|).\lambda(|y|,|z|).x(yz), \\
(\lambda(|x|,|y|)xy, z) &\mapsto \lambda(|x|,|y|).\lambda(|xy|,|z|).(xy)z,
\end{align*}
avec
\begin{align*}
\lambda(|x|,|yz|).\lambda(|y|,|z|)
&= (-1)^{\lambda_1(|x|,|yz|)+\lambda_1(|y|,|z|)}.(-1)^{\lambda_0(|x|_2,|yz|_2)+\lambda_0(|y|_2,|z|_2)},\\
\lambda(|x|,|y|).\lambda(|xy|,|z|)
 &=(-1)^{\lambda_1(|x|,|y|) + \lambda_1(|xy|, |z|)}.(-1)^{\lambda_0(|x|_2,|y|_2) + \lambda_0(|xy|_2, |z|_2)}.
\end{align*}
On en conclut que
\begin{align*}
\frac{ \lambda(|x|,|yz|).\lambda(|y|,|z|)}{ \lambda(|x|,|y|).\lambda(|xy|,|z|)} &=  (-1)^{\phi_1(|x|, |y|, |z|)}.(-1)^{\phi_0(|x|_2, |y|_2, |z|_2)} \\
&= \phi(|x|,|y|,|z|).
\end{align*}
et cela donne l'associativité pour cette nouvelle multiplication. Hélas $\lambda(|x|,|y|)$ est très difficilement calculable et donc la version associative de $OH^n_C$ n'est pas intéressante pour ce qu'on en fait. De plus, il faudrait encore prouver que l'anneau qu'on obtient n'est pas isomorphe à $H^n$.

\section{Classes d'isomorphismes de $OH^n_C$}

On peut étudier les collections de règles de multiplications qui donnent des $OH^n_C$ isomorphes. En fait, on peut montrer que tous les $OH^n_C$ qui ont des règles de multiplications avec les même associateurs sont isomorphes. Cela se montre par des arguments similaires à ce qu'on a fait pour trouver une multiplication associative. \\

On se fixe $C$ et $C'$ deux règles de multiplications et afin d'éviter toute confusion on note $*_C$ la multiplication dans $OH^n_C$ et $*_{C'}$ dans $OH^n_{C'}$. Puisque les cobordismes sont équivalents si on les regarde sans chronologie ni orientation, on sait qu'on a
$$OF(C_{cba}) = \gamma(c,b,a) OF(C'_{cba})$$
avec $\gamma(c,b,a) \in \{-1, 1\}$ pour tout $c,b,a \in B^n$. On peut alors définir l'application
$$\eta : \Hom\left(\ddot B^n\right)^2\rightarrow \Z/2\Z $$
définie par $(-1)^{\eta(|x|_2, |y|_2)} := \gamma(c,b,a)$ si $|x|_2 = W(c)b$ et $|y|_2 = W(b)a$ et $0$ sinon. Cela signifie que si $x *_C y = w \in OH^n_C$ alors $x *_{C'} y =  (-1)^{\eta(|x|_2,|y|_2)} w \in OH^n_{C'}$. On souhaite avoir
$$\eta(|x|_2,|y|_2) = \epsilon(|x|_2)+\epsilon(|y|_2) - \epsilon(|xy|_2)$$
pour un certain $\epsilon(|x|_2) : \Hom\left(\ddot B^n\right) \rightarrow \Z/2\Z$ de sorte que
$$\theta : OH^n_C \rightarrow OH^n_{C'} : x \mapsto (-1)^{\epsilon(|x|_2)}x$$
soit un isomorphisme d'anneaux gradués. Puisque $\ddot B$ a une cohomologie de degré $2$ triviale, tout cocycle est un cobord et donc on voudrait avoir que $\eta$ soit un cocycle. On observe alors que
\begin{align*}
(\delta \eta)(|x|_2,|y|_2,|z|_2) &= \eta(|y|_2,|z|_2) - \eta(|xy|_2,|z|_2) + \eta(|x|_2,|yz|_2) - \eta(|x|_2,|y|_2)
\end{align*}
et on peut prendre $x = {_d1_c}, y = {_c1_b}$ et $z = {_b1_a}$  pour avoir
\begin{align*}
(x *_C y) *_C z &= (-1)^{\eta(|x|_2,|y|_2)+\eta(|xy|_2,|z|_2)} (x *_{C'} y ) *_{C'} z, \\
x *_C (y *_C z) &= (-1)^{\eta(|x|_2, |yz|_2)+ \eta(|y|_2, |z|_2)} x *_{C'} (y *_{C'} z), 
\end{align*}
avec le produit, s'il n'est pas spécifié, dans $OH^n_C$. Cela signifie que si
\begin{align*}
\frac{x *_C (y *_C z)}{(x *_C y) *_C z} &= \frac{x *_{C'} (y *_{C'} z)}{(x *_{C'} y) *_{C'} z}
\end{align*}
 la fraction étant un abus de notation pour le signe induit par la non-associativité, donc si $OH^n_C$ et $OH^n_{C'}$ ont le même associateur, alors $\eta$ est un cocycle. On note donc 
$$\phi_0^C : \Hom\left(\ddot B^n\right)^3 \rightarrow \Z/2\Z$$
l'associateur tel que pour  $x = {_d1_c}, y = {_c1_b}, z = {_b1_a}$ on a
$$(x *_C y) *_C z = (-1)^{\phi_0^C(|x_2, |y|_2, |z|_2)} (x *_C (y *_C z)),$$
c'est-à-dire que c'est la fonction qui calcule le signe du au changement de chronologies, comme dans la section précédente. On ajoute l'hypothèse que $\phi_0^C = \phi_0^{C'}$  et 
on en conclut que $\eta$ est un cobord. Cela signifie qu'il existe $\epsilon : \Hom\left(\ddot B^n\right) \rightarrow \Z/2\Z$ tel que $\delta \epsilon = \eta$. On a
$$(\delta \epsilon)(|x|_2,|y|_2) = \epsilon(|x|_2)-\epsilon(|xy|_2)+\epsilon(|y|_2) = \eta(|x|_2, |y|_2)$$
et donc
\begin{align*}
\theta(x) *_{C'} \theta(y) &= (-1)^{\epsilon(|x|_2)+\epsilon(|y|_2)}(x *_{C'}y), \\
\theta(x *_C y) &= (-1)^{\epsilon(|xy|_2)}(x *_{C}y)
\end{align*}
avec comme prévu
$$\frac{(-1)^{\epsilon(|x|_2)+\epsilon(|y|_2)}}{ (-1)^{\epsilon(|xy|_2)}} = (-1)^{\eta(|x|_2,|y|_2)}.$$
donnant l'égalité $\theta(x) *_{C'} \theta(y) = \theta(x *_C y) $. On en conclut le théorème suivant :

\begin{theoreme}
Pour toutes règles de multiplication $C$ et $C'$ telles que $OH^n_C$ et $OH^n_{C'}$ ont le même associateur, donc si $\phi^C = \phi^{C'}$, on a un isomorphisme d'anneaux gradués
$$OH^n_C \simeq OH^n_{C'}.$$
\end{theoreme}

\begin{proof}
On utilise le fait que l'associateur $\phi^C$ est décomposé en deux éléments : un dépendant du degré du facteur de gauche et du nombre de scissions de la multiplication de droite, donc indépendant de $C$, et l'autre dépendant du changement de chronologie et noté $\phi_0$. Cela signifie que $\phi^C = \phi^{C'}$ si et seulement si $\phi_0^{C} = \phi_0^{C'}$. Il suffit ensuite d'appliquer le raisonnement fait au-dessus pour conclure.
\end{proof}

On peut alors noter $OH^n_\phi$ la classe d'équivalence de ces anneaux, avec $\phi$ l'associateur. Comme le montrent les Exemples \ref{ex:chgtch} et  \ref{ex:chgtch2} on n'a pas toujours l'hypothèse d'égalité des associateurs. La question qu'on se pose maintenant est de savoir si cette condition est suffisante ou s'il existe d'autres isomorphismes pour entre ces anneaux ? En tout cas, cela donne une borne supérieure sur le nombre d'anneaux $OH^n_\phi$ pour chaque $n$ puisque l'associateur est définir par un choix d'autant de $1$ ou $-1$ que de triplets de diagrammes $W(d)c$, $W(c)b$ et $W(b)a$, c'est-à-dire quadruplets d'éléments de $B^n$. On a donc au maximum 
$$2^{|B^n|^4} = 2^{C_n^4}$$
différents $OH^n_\phi$, où $C_n$ est le $n$-ème nombre de Catalan.

%
%

\section{Modules sur $OH^n_C$}
Puisque $OH^n_C$ n'est pas associatif, il ne peut pas former un module sur lui-même et donc il n'y a pas vraiment d'intérêt de parler de modules sur $OH^n_C$ au sens usuel du terme. On se demande alors s'il existe un moyen de définir un substitut de module. Une réponse possible est de voir $OH^n_C$ comme un quasi-anneau et d'étudier ses quasi-modules et quasi-bimodules, un quasi-anneau étant un anneau gradué non-associatif mais avec un associateur (devant donc faire commuter le diagramme (\ref{diag:assoc})) prenant comme arguments les degrés des éléments. Dès lors, un quasi-module $M$ est un groupe abélien gradué sur le même ensemble que le quasi-anneau et avec une action devant respecter l'associateur aussi, c'est-à-dire que pour tout $m \in M$ et $x,y \in OH^n_C$ on a
$$x \bullet (y \bullet m )) = \phi(|x|,|y|,|m|) (xy)\bullet m.$$
On peut alors définir le produit tensoriel de deux bimodules $M$ et $N$ comme
$$M \otimes_{OH^n_C} N := \frac{M \otimes_\Z N}{\langle mh \otimes n - \phi(|m|,|h|,|n|) m \otimes hn \rangle}.$$
Comme expliqué dans \cite{Khovanov}, les bimodules sur $H^n$ forment une catégorification de l'algèbre de Templerley-Lieb qui est définie comme suit :
\begin{definition}\label{def:temperleyliebalg}
L'\emph{algèbre de Temperley-Lieb}, notée $TL^n$, est l'algèbre sur $\Z[q,q^{-1}]$  avec générateurs $U_1, \dots, U_{n-1}$ et les relations
\begin{align*}
U_i^2 &= (q+q^{-1})U_i, \\
U_iU_{i\pm1}U_i &= U_i, \\
U_iU_j &= U_jU_i,& |i-j| > 1. 
\end{align*}
\end{definition}
On peut montrer que cette algèbre est isomorphe à l'algèbre diagrammatique sur $\Z[q,q^{-1}]$ engendrée par
$$\middiag{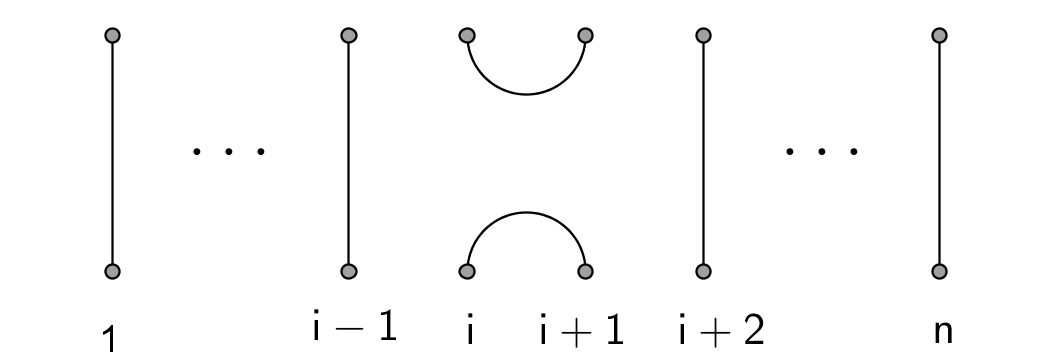}$$
où on regarde donc les éléments à isotopie près et un cercle disjoint revient à multiplier par $(q+q^{-1})$. Plus précisément, cette catégorification consiste à associer à un diagramme $U \in TL^n$ le bimodule sur $H^n$ donné par 
\begin{align*}
C(T) &:= \bigoplus_{a,b \in B^n} b(U)a, & &\text{où}& b(U)a &:= F(W(b)Ua)\{n\}
\end{align*}
avec l'action à gauche 
$$c(H^n)b \times b(U)a \rightarrow c(U)a$$
donnée par le cobordisme $W(c)bW(b)Ua \rightarrow W(c)Ua$ où on construit les même ponts que pour la multiplication dans $H^n$ et idem pour l'action à droite. De plus, la catégorification du coefficient $q+q^{-1}$ est donnée par le groupe abélien $A$ (le rang gradué de $A$ est $q+q^{-1}$). On veut alors définir de la même façon des quasi-bimodules sur $OH^n_C$, en utilisant la règle de multiplication $C$ pour choisir l'ordre de construction des ponts et leurs orientations donnés par $(c,b,a)$ (on pourrait aussi étendre la règle à un choix dépendant de l'élément $U$). Pour que ceci forme une catégorification, il faut que la composition de deux bimodules donne le bimodule donné par la concaténation des diagrammes associés, c'est-à-dire que pour tout $T,T' \in TL^n$ on ait
$$C(TT') = C(T) \otimes_{OH^n_C} C(T').$$
Si on montre cela, on peut alors définir des complexes de quasi-bimodules et espérer obtenir un invariant d'enchevêtrement, comme M. Khovanov le fait avec les bimodules sur $H^n$ (voir \cite{Khovanov} pour plus de détails). La question suivante serait alors de savoir si $H^n$ et $OH^n_C$ livrent des invariants différents (ce qui est fort probablement le cas puisque l'homologie de Khovanov et sa version impaire sont différents) et même de savoir si pour des règles de multiplications différentes on peut obtenir des invariants différents.

\section{Généralisation de $H^n$ et $OH^n_C$}

K. Putyra a défini dans \cite{Covering} une famille de foncteurs avec des paramètres tels que, une fois spécialisés, on obtient $F$ de $H^n$ ou $OF$ de $OH^n$. Pour cela on pose 
\begin{align*}
R &:= \frac{ \Z[X,Y,Z^{\pm 1}]}{X^2=Y^2=1},&& \text{et}&
A' &:= R\langle v_+, v_- \rangle,
\end{align*}
avec $\deg(v_+) = (1,0)$ et $\deg(v_-) = (0,-1)$ et les applications qui servent de permutation, fusion, scission, naissance de cercle et mort de cercle :
\begin{align*}
\tau : A' \otimes A'  \rightarrow A' \otimes A' : &\begin{cases}
(v_+, v_+) &\mapsto X(v_+, v_+), \\
(v_+, v_-) &\mapsto Z^{-1}(v_-, v_+), \\
(v_-, v_+) &\mapsto Z(v_+, v_-), \\
(v_-, v_-) &\mapsto Y(v_-, v_-), 
\end{cases} \\
\mu : A' \otimes A'  \rightarrow A' : &\begin{cases}
(v_+, v_+) &\mapsto  v_+, \\
(v_+, v_-) &\mapsto v_-, \\
(v_-, v_+) &\mapsto XZv_-, \\
(v_-, v_-) &\mapsto 0, 
\end{cases} \\
\Delta : A' \rightarrow A' \otimes A' : &\begin{cases}
v_+ \mapsto (v_-,v_+) + YZ(v_+, v_-), \\
v_- \mapsto (v_-, v_-),
\end{cases} \\
\eta : R \rightarrow A' : &1 \mapsto v_+, \\
\epsilon : A' \rightarrow R : &\begin{cases}
v_+ \mapsto 0, \\
v_- \mapsto 1.
\end{cases}
\end{align*}
On note $RH^n_C$ l'anneau obtenu par une construction similaire à $OH^n_C$ où on remplace les produit extérieurs par des produits tensoriels de $A'$ et on associe aux cobordismes élémentaires les applications associées. On remarque qu'en spécialisant les paramètres $(X,Y,Z) = (1,1,1)$ on obtient $H^n$ et en les spécialisant à $(1,-1,1)$ on obtient ${OH}^n_C$. On pourrait alors définir un centre pour $RH^n_C$ donné par
$$Z(RH^n_C) := \{z \in RH^n_C | zx = X^{|z|_1|x|_1}Y^{|z|_2|x|_2}Z^{|z|_1|x|_2-|z|_2|x|_1}xz, \forall x \in RH^n_C\}$$
où $|v_+|_1 = 1, |v_+|_2 = 0$ et $|v_-|_1 = 0, |v_-|_2 = -1$. Il serait intéressant de vérifier si, en spécialisant les paramètres, on obtient que ce centre donne $OZ(OH^n_C)$ et $Z(H^n)$. Par ailleurs, cela mène à la question de savoir s'il existe une généralisation de l'anneau de cohomologie de la variété de Springer et qui serait isomorphe au centre de $RH^n_C$ pour une partition $(n,n)$.

\section{Action de $\mathcal{H}_{-1}$ sur le centre impair}

On dit qu'un groupe $G$ agit faiblement sur une catégorie $C$ s'il existe des foncteurs $F_g$ pour tout $g \in G$ tels que $F_1 \simeq \Id$ et $F_gF_h \simeq F_{gh}$. Comme expliqué dans \cite[Section 6.5]{Khovanov} et dans \cite[Section 5]{HnCenter}, le groupe des tresses à $2n$ brins agit faiblement sur la catégorie triangulée $\mathcal{K}$ des complexes bornés de bimodules sur $H^n$ en tensorisant par le complexe de bimodule obtenu par résolutions de la tresse vue comme enchevêtrement. Cette action faible descend à une action du groupe symétrique sur le centre de $H^n$ par permutation des $X_i$ dans la présentation du Théorème \ref{thm:hncenter}. 

Par ailleurs, dans \cite[Corollaire 3.4]{OddSpringer}, A. Lauda et H. Russell montrent que l'algèbre de Hecke $\mathcal{H}_{-1}$ agit à gauche sur la construction impaire de la cohomologie de la variété de Springer et donc sur le centre impair de $OH^n_C$. On se demande alors si l'action du groupe des tresses sur la catégorie des quasi-bimodules sur $OH^n_C$ descend à l'action de $\mathcal{H}_{-1}$ sur le centre impair.

\section{Algèbres de Stroppel-Ehrig}

C. Stroppel et M. Ehrig ont développé une famille d'algèbres diagrammatiques de type D dans \cite{Stroppel} qui généralisent les anneaux des arcs de Khovanov $H^n$ qui sont de type A (puisque reliés à $\mathfrak{sl}(2)$). On peut alors se demander quels liens pourraient exister entre la version impaire des anneaux des arcs et les algèbres de Stroppel-Ehrig ? 

\appendix

\renewcommand\chaptername{Annexes}
\chapter*{Annexes}
\markboth{\textbf{Annexes}}{}
\addcontentsline{toc}{chapter}{Annexes} 
\renewcommand\thesection{A.\arabic{section}}
\counterwithin{theoreme}{section}
\counterwithin{figure}{section}

\section{Notations}

Afin d'éviter toute confusion possible, on donne une liste des notations utilisées dans ce travail, sauf celles qui sont explicitement définies dedans.

\paragraph{Ensembles}
On note les ensembles des naturels $\N$, des entiers $\Z$, des réels $\R$ et des complexes $\C$. Par ailleurs, pour $S$ un ensemble, on note $\Z[S]$ le groupe abélien libre engendré par les éléments de $S$.

\paragraph{Rangs de groupes}
Pour un groupe $G$ on note son rang $\rank(G)$ (c'est-à-dire le nombre minimal d'éléments nécessaires pour générer le groupe) et pour un groupe abélien $G'$ on note son rang $\rank(G')$ (c'est-à-dire la taille du plus grand sous-groupe abélien libre).

\paragraph{Coefficients binomiaux}
On note pour deux naturels $n$ et $k$ le coefficient binomial
$$\begin{pmatrix} n \\ k \end{pmatrix} := \frac{n!}{k!(n-k)!}.$$

\paragraph{Caractéristique d'Euler}
On note $\chi(X)$ pour un espace topologique $X$ sa caractéristique d'Euler-Poincaré.

\paragraph{Espaces vectoriels}
Soit $E$ un espace vectoriel sur un corps $k$. On note $\dim(E)$ ou $\dim_k(E)$ sa dimension.

\paragraph{Fonctions} On note $A \hookrightarrow B$ l'application induite par l'inclusion de l'ensemble $A$ dans l'ensemble $B$. On note $A \simeq B$ s'il existe un isomorphisme respectant la structure algébrique de $A$ et de $B$ entre les deux. On note $A \simeq_{ab} B$ si $A$ et $B$ possèdent une structure de groupe abélien et qu'ils sont isomorphes en tant que groupes abéliens.

\paragraph{Produit tensoriel} On note $M \otimes_A N$ le produit tensoriel sur $A$ de deux modules $N$ et $M$ sur un anneau $A$. On note $M \otimes N := M \otimes_A N$ si le choix de $A$ est clair du contexte. On note aussi $M^{\otimes n} := \underset{n}{\underbrace{M \otimes \dots \otimes M}}$.

\section{Nombres quantiques et groupe quantique $U_q(\mathfrak{sl}_2)$}\label{sec:quantumnumber}

Tout ce qui est présenté dans cette section se trouve en détails dans \cite{quantumj} à quelques changements de notation près (on prend par exemple un $q$ tel qu'il soit la racine de $q$ utilisé dans la référence).\\

 L'algèbre de Lie $\mathfrak{sl}_2$ admet une enveloppe $U(\mathfrak{sl}_2)$ avec générateurs $E,F,H$ tels que $EF-FE=H$, $HE-EH=2E$ et $HF-FH=-2F$ (le $H$ ici est $2H$ dans la référence).

 L'enveloppe $U(\mathfrak{sl}_2)$  admet une unique représentation irréductible de dimension $2n$, à isomorphisme près, qu'on note $V^n$. Pour la dimension $2$, on note simplement $V := V^1$. Cette représentation est donnée par l'espace vectoriel des polynômes homogènes de degré $2n-1$ en les variables $x,y$ avec $U(\mathfrak{sl}_2)$ agissant sur le vecteur $e_{n,m} := x^{2n-1-m}y^{m}$ de poids $2n-1-2m$, par
\begin{align*}
Ee_{n,m} &= (m+1)e_{n,m+1}, \\
Fe_{n,m} &= (2n-m)e_{n,m-1}, \\
He_{n,m} &= (2n-1-2m)e_{n,m},
\end{align*}
avec la convention $e_{n,-1} = e_{n,2n} = 0$.

Par ailleurs, on peut définir une version quantique des entiers, c'est-à-dire qu'on déforme un nombre par un paramètre complexe $q$. Plus formellement, on a la définition suivante :
\begin{definition}
Pour $n \ge 1$, on définit le $n$-ème nombre quantique (pour $q \in \C$) comme
$$[n]_q := \frac{q^{n}-q^{-n}}{q-q^{-1}} \in \Z[q,q^{-1}]$$
et on pose $[0]_q = 0$ ainsi que $[-n]_q = -[n]_q$.
\end{definition}
De là on peut définir la factorielle quantique et par extension le coefficient binomial quantique
\begin{align*}
[n]_q! = \prod_{i=1}^n [i]_q,&&& \begin{bmatrix}n \\ k\end{bmatrix}_q = \frac{[n]_q!}{[a]_q![n-a]_q!}.
\end{align*}
On observe que
$$[n]_q = q^{n-1} + q^{n-3} + \dots + q^{-n+3} + q^{-n+1}$$
et donc la somme des coefficients de $[n]_q$ est $n$. De là on obtient que la somme des coefficients de $[n]_q.[m]_q$ est $nm$ et on peut montrer que la somme des coefficients du coefficient binomial quantique donne le coefficient binomial usuel.

Tout comme on a une version quantique des entiers, on peut définir une version quantique de $U(\mathfrak{sl}_2)$ comme une déformation de cette algèbre par un paramètre complexe $q \in \C$. Intuitivement, cela signifie qu'au lieu de prendre $H$, on prend la déformation quantique de $H$ par un paramètre $q$, c'est-à-dire $\frac{q^H - q^{-H}}{q - q^{-1}}$, et on pose $K ``=" q^{H}$ (cela peut se formaliser en termes de séries de puissances en prenant $q = e^h$). On obtient alors la définition suivante :
\begin{definition}
Le groupe quantique $U_q(\mathfrak{sl}_2)$ est l'algèbre sur $\C[q]$ engendrée par les éléments $E,F,K$ et $K^{-1}$ tels que 
\begin{align*}
EF-FE &= \frac{K-K^{-1}}{q-q^{-1}}, \\
KE &= q^2EK, \\
KF &= q^{-2}FK, \\
KK^{-1} &= 1 = K^{-1}K.
\end{align*}
\end{definition}

On peut alors définir une version quantique de $V^n$, qu'on note aussi $V^n$, qui est donnée par l'espace vectoriel de dimension $2n$ ayant comme base 
$$\{e_{n,m} | m \in \{0,1 , \dots, 2n-1\},$$
et avec l'action définie par les mêmes relations que celles de $U(\mathfrak{sl}_2)$ en remplaçant les coefficients entiers par leurs analogues quantiques. On obtient donc 
\begin{align*}
Ee_{n,m} &= [m+1]_q.e_{n,m+1}, \\
Fe_{n,m} &= [2n-m]_q.e_{n,m-1}, \\
Ke_{n,m} &= q^{2n-2m-1}.e_{n,m}.
\end{align*}

%
%

\section{Groupes, anneaux et modules gradués}\label{sec:gradue}

\subsection{Groupes gradués}

Un groupe abélien $G$ gradué (sur $\Z$) est un groupe abélien avec neutre $0$ qui se décompose en somme directe
$$G = \bigoplus_{i \in \Z} G_i.$$
On définit alors l'application degré pour tout $x\in G_i\setminus \{0\}$ comme 
$$\deg(x) = i$$
et on appelle de tels $x$ les éléments \emph{homogènes} de $G$.
On définit le rang gradué $\rank_q (G)$ comme le polynôme de Laurent obtenu en prenant
$$\rank_q (G) := \sum_{i \in \Z} q^i \rank(G_i).$$

\subsection{Anneaux gradués}

Un anneau gradué $A$ est un anneau qui, vu comme groupe abélien, est gradué
$$A = \bigoplus_{i\in \Z} A_i$$
et tel que pour tout $i$ et $j$ dans $\Z$ on a
$$A_i.A_j \subset A_{i+j}.$$
Autrement dit pour des éléments homogènes $x,y \in A$, on a $\deg(xy) = \deg(x) + \deg(y)$.

\begin{exemple}
L'anneau des polynômes en une variable $X$ est gradué par 
$$\deg(k_0 + k_1 X + \dots + k_n X^n) = n.$$
\end{exemple}

\subsection{Modules gradués}

Un module gradué $M$ est un module sur un anneau gradué $A$ qui, vu comme groupe abélien, est gradué
$$M = \bigoplus_{i\in \Z} M_i$$
et tel que pour tout $i$ et $j$ dans $\Z$ on a
$$A_i.M_j \subset M_{i+j}.$$
On dit qu'un homomorphisme $f : M_1 \rightarrow M_2$ entre deux modules gradués est de degré $n$ si pour tout $x \in M_1$ homogène on a $f(x) \in M_2$ homogène et 
$$\deg(f(x)) = \deg(x) + n.$$
On note alors $\deg f$ le degré du morphisme.

\begin{definition}
On définit la notion de décalage par $n\in \Z$ d'un module gradué $M$, noté $M\{n\}$, comme
\begin{align*}
M\{n\} = \bigoplus_{k\in \Z} M\{n\}_k &,& &\text{où}  & M\{n\}_k = M_{k-n}.
\end{align*}
\end{definition}
Autrement dit, on a $M\{n\} \simeq M$ vu en tant que modules non-gradués et pour tout $x \in M$ homogène on obtient
$$\deg_{M\{n\}} (x) = \deg_M(x) + n,$$
c'est-à-dire qu'on décale les degrés des éléments de $M$ par $n$.

Par ailleurs, on peut définir une graduation sur le produit tensoriel de deux modules gradués $M_1$ et $M_2$ sur un anneau $A$, ce qui forme un module $M_1 \otimes_A M_2$ en posant pour $x_1 \in M_1$ et $x_2 \in M_2$
$$\deg_{M_1 \otimes_A M_2} (x_1 \otimes x_2) = \deg_{M_1} (x_1) + \deg_{M_2} (x_2).$$

\begin{remarque}
Ici on ne considère que des graduations sur $\Z$ mais on peut définir une graduations sur n'importe quel groupe en demandant que le degré du produit soit donné par la composition des degrés. On peut même étendre cette définition sur des groupoïdes (voir Section \ref{sec:groupoide}) en imposant que la multiplication soit non-nulle si et seulement si les degrés sont composables dans le groupoïde.
\end{remarque}

\section{Algèbres extérieures de modules}\label{sec:prodext}

En général, la notion d'algèbre extérieure est définie sur un espace vectoriel. On définit ici une généralisation de cette notion pour un module.

\begin{definition}Soit $M$ un module sur un anneau commutatif unitaire $A$. On définit \emph{l'algèbre extérieure} sur le module $M$ comme
$$\Ext^*M = \frac{A \oplus \bigoplus_{k=1}^\infty M^{\otimes k}}{I}$$
avec le produit tensoriel pris sur $A$ et
$$I = \langle \{x \otimes x | x\in M\} \rangle$$
qui est l'idéal bilatère engendré par les éléments de l'ensemble.\\
On définit la multiplication sur l'algèbre extérieure
$$\wedge : \Ext^*M \times \Ext^*M \rightarrow \Ext^*M$$
par
$$(x+I) \wedge (y+I) = (x \otimes y) + I.$$
\end{definition}
On remarque que cette définition de multiplication donne alors pour $x \in \Ext^*M$ et $k \in A$
$$x \wedge k = k \wedge x =  kx.$$
On remarque aussi que pour $x,y \in M$ on a 
$$x \wedge y = - y\wedge x$$
puisque
$$0 = (x+y)\wedge(x+y) = x\wedge x + x \wedge y + y \wedge x + y \wedge y =  x \wedge y + y \wedge x.$$

\begin{definition}\label{def:contractdual}
Pour un élément $a \in M$ on définit l'application de \emph{contraction par le dual} de $a$ comme 
$$a^* : \Ext^*M \rightarrow \Ext^*M$$
et qui vaut pour tout $x_1,\dots,x_n,y_1,\dots,y_m \in M$, $k \in A$, avec $x_i \ne a$ et $y_j \ne a$,
$$a^*(x_1 \wedge \dots \wedge x_n \wedge ka \wedge y_1 \wedge \dots \wedge y_m) = k(-1)^n (x_1 \wedge \dots \wedge x_n \wedge y_1 \wedge \dots \wedge y_m)$$
et 
$$a^*(x_1\wedge \dots \wedge x_n) = 0.$$
\end{definition}
%
%
%

Si le module est gradué, on a une graduation induite sur cette algèbre en donnant à un produit extérieur la somme des degrés
$$\deg(a_1 \wedge \dots \wedge a_n) = \deg(a_1) + \dots + \deg(a_n).$$
Sinon, on a une graduation naturelle sur cette algèbre en donnant un degré $0$ aux éléments de $A$ et un degré $k$ pour les éléments de $M^{\otimes k} + I$, c'est-à-dire la graduation induite en donnant à tout élément de $M$ un degré $1$.


\begin{exemple}
On considère $G$ le groupe abélien libre engendré par l'ensemble $\{a,b\}$ vu comme module sur $\Z$. On a alors
$$\Ext^*G = \langle 1, a, b, a\wedge b \rangle$$
avec $\deg(1) = 0, \deg(a) = \deg(b) = 1$ et $\deg(a\wedge b) = 2$.
\end{exemple}

%
%

\section{Anneau de cohomologie}\label{sec:cohomology}

En premier lieu, on fait quelques rappels sur les théories d'homologie et de cohomologie singulières d'un espace topologique, tous les détails se trouvant dans \cite{Hatcher} et ensuite on définit l'anneau de cohomologie à proprement parlé.

\subsection{Homologie et cohomologie singulières}

\begin{definition}
Un \emph{$n$-simplexe singulier} $\sigma$ dans un espace topologique $X$ est une application continue
$$\sigma : \Delta^n \rightarrow X$$
avec $\Delta^n$ le $n$-ème simplexe topologique
$$\Delta^n := \left\{ (t_0, \dots, t_n) \in \R^{n+1} \big| \sum_i t_i = 1, t_i \ge 0\right\}$$
où on ordonne $v_i \rightarrow v_j$ si $i < j$ pour $[v_0, \dots, v_n]$ les sommets du simplexe.\\
On pose $C_n(X)$ le groupe abélien libre engendré par les $n$-simplexes singuliers et on appelle ses éléments des \emph{$n$-chaines}.\\
On définit l'\emph{opérateur de bord} $\partial_n : C_n(X) \rightarrow C_{n-1}(X)$
$$\partial_n(\sigma) := \sum_i (-1)^i \sigma|[v_0, \dots, \widehat v_i, \dots, v_n]$$
où $[v_0, \dots, \widehat v_i, \dots, v_n]$ signifie qu'on prend le $n$-simplexe engendré par tous les $v_j$ sauf $v_i$.
\end{definition}

On peut montrer que $\partial_n \circ \partial_{n+1} = 0$ pour tout $n$ de sorte que $\img \partial_{n+1} \subset \ker \partial_n$ et on définit les \emph{groupes d'homologie singulière} de $X$ comme
$$H_n(X) := \frac{\ker \partial_n}{\img \partial_{n+1}}.$$
On peut montrer que ces groupes sont isomorphes pour deux espaces topologiques homotopes.

\begin{proposition}
Soit $f : X \rightarrow Y$ une fonction continue entre des espaces topologiques $X$ et $Y$, alors il y a un morphisme de groupes induit
$$f_* : H_n(X) \rightarrow H_n(Y)$$
qui est l'application induite par le composition d'un $n$-simplexe singulier et $f$
$$f_*(\Delta^n \rightarrow X) = \Delta^n \rightarrow X  \overset{f}{\rightarrow} Y.$$  
\end{proposition}

\begin{definition}
Pour un espace topologique $X$, on définit le groupe des \emph{$n$-cochaines} comme le dual du groupe des $n$-chaines
$$C^n(X) := \Hom(C_n(X), \Z)$$
c'est-à-dire qu'une $n$-cochaine associe à chaque $n$-simplexe singulier un entier.\\
On définit l'\emph{opérateur de cobord } $\delta_n : C^{n-1}(X) \rightarrow  C^n(X)$ comme étant le dual de $\partial_{n}$, donc
$$\delta_{n} \phi(\sigma) = \sum_i (-1)^i \phi(\sigma|[v_0, \dots, \widehat {v_i}, \dots, v_{n}])$$
pour toute $(n-1)$-cochaine $\phi$ et tout $n$-simplexe singulier $\sigma$.
\end{definition}
La définition se généralise en prenant le dual sur un groupe abélien quelconque au lieu de $\Z$, mais on n'en a pas besoin dans ce travail. Comme $\partial_{n} \circ \partial_{n+1} = 0$, on obtient que $\delta_{n+1} \circ \delta_{n} = 0$ et on définit les \emph{groupes de cohomologie singulière} de $X$ comme
$$H^n(X) := \frac{\ker \delta_{n+1}}{\img \delta_{n}}.$$
On appelle les éléments de $\ker \delta_{n+1}$ des \emph{cocycles} et les éléments de $\img \delta_{n}$ des \emph{cobords}.

\begin{proposition}
Soit $f : X \rightarrow Y$ une fonction continue entre des espaces topologiques $X$ et $Y$, alors il y a un morphisme de groupes induit
$$f^* : H^n(Y) \rightarrow H^n(X)$$
qui est l'application induite par le composition de $f$ et d'une $n$-cochaine
$$f^*(\phi) : (\Delta^n \rightarrow X) \mapsto \phi(\Delta^n \rightarrow X \overset{f}{\rightarrow} Y).$$
\end{proposition}

\begin{proposition}
Pour tout $n$ on a 
$$\rank(H^n(X)) = \rank(H_n(X)).$$
\end{proposition}

\begin{proof}
Cela se prouve en utilisant le théorème des coefficients universels et on renvoie vers \cite[Section 3.1]{Hatcher} pour plus de détails sur celui-ci, notamment pour la construction de $\ext(-,\Z)$. Ce théorème donne une suite exacte scindée
$$0 \rightarrow \ext(H_{n-1}(X),\Z) \rightarrow H^n(X) \rightarrow \Hom(H_n(X), \Z) \rightarrow 0$$
et donc
$$\rank(H^n(X)) = \rank(\Hom(H_n(X), \Z)) + \rank(\ext(H_{n-1}(X),\Z)).$$
Puisque $H_n(X)$ est un groupe abélien de type fini on a 
$$\rank(\Hom(H_n(X), \Z)) = \rank(H_n(X)).$$
Par ailleurs
$$\ext(H_{n-1}(X),\Z) = 0$$
puisque si on décompose $H_{n-1}(X) = F_{n-1}+T_{n-1}$ avec $F_{n-1}$ sa partie sans torsion et $T_{n-1}$ sa partie avec torsion on obtient
$$\ext(H_{n-1}(X),\Z) = \ext(F_{n-1},\Z) \oplus \ext(T_{n-1}(X),\Z) = \ext(T_{n-1}(X),\Z) \simeq T_{n-1}$$
et, $T_{n-1}$ ne contenant pas de sous-groupe libre non-trivial, son rang est nul. On en conclut le résultat voulu.
\end{proof}

\begin{definition}
Un sous-espace $A$ d'un espace topologique $X$ est une \emph{déformation rétracte} de celui-ci s'il existe une application continue 
$$r : X \rightarrow I$$
telle que $r \circ i = Id_A$ pour $i : A \hookrightarrow X$ l'inclusion. 
\end{definition}

\begin{proposition}\label{prop:retracte}
Si $A$ est une déformation rétracte de $X$ alors 
\begin{align*}
H_k(X) &\simeq H_k(A), & \text{et}& &  H^k(X) &\simeq H^k(A).
\end{align*}
\end{proposition}

\subsection{Anneau de cohomologie}

Tout comme pour la sous-section précédente, tous les détails et toutes les preuves des propositions se trouvent dans \cite{Hatcher}. Tout d'abord, il faut introduire un produit sur les groupes de cohomologie d'un espace topologique $X$
$$H^i(X) \times H^j(X) \rightarrow H^{i+j}(X).$$
Ce produit est induit par un produit de cochaines donné par, pour $\phi$ et $\psi$ de dimensions respectives $i$ et $j$, la cochaine $\phi \smile \psi$ de dimension $i+j$ telle que
$$(\phi \smile \psi)(\sigma) : \phi(\sigma | _{[v_0, \dots, v_i]}) \psi(\sigma |_{[v_{i+1}, \dots, v_{i+k}]})$$
pour tout simplexe singulier $\sigma : \Delta^{i+j} \rightarrow X$.

\begin{lemme}
Pour toutes cochaines $\phi$ et $\psi$ de dimensions respectives $i$ et $j$ on a
$$\delta_{i+j+1}(\phi \smile \psi) = \delta_{i+1} (\phi) \smile \psi + (-1)^i\phi \smile \delta_{j+1} (\psi). $$
\end{lemme}

Le produit de deux cocycles est alors un cocycle (les deux cobords étant nuls et le produit par une cochaine nulle donnant une cochaine nulle) et le produit d'un cocycle et d'un cobord est un cobord puisque pour $\delta(\psi) = 0$ on a
$$ \delta(\phi) \smile \psi = \pm \delta(\phi \smile \psi)$$
et de même pour $\delta(\phi) = 0$ on a
$$ \phi \smile \delta(\psi) = \delta(\phi \smile \psi).$$
Cela montre que le produit sur les cochaines induit un produit sur les groupes de cohomologie. Ce produit est la base pour définir l'anneau de cohomologie de $X$.

\begin{definition}
Pour un espace topologique $X$ on définit l'anneau de cohomologie $H(X)$ comme le groupe abélien gradué
$$H(X) := \bigoplus_k H^k(X)$$
muni de la multiplication 
$$\left(\sum_i \alpha_i \right)\left(\sum_j \beta_j\right) = \sum_{i,j} \alpha_i \smile \beta_i.$$
\end{definition}
On montre aisément que cela forme bien un anneau associatif unitaire.

\begin{exemple}\label{ex:torecohomology}
On peut montrer (\cite[Exemple 3.16]{Hatcher}) que l'anneau de cohomologie du tore de dimension $n$ est donné par le produit extérieur sur $n$ éléments
$$H(T^n) \simeq \Ext^* \Z^{n}$$
avec les éléments générateurs du groupe libre ayant un degré $1$.
\end{exemple}

\begin{exemple}
L'anneau de cohomologie de l'espace projectif complexe de dimension $n$ est donné par les polynômes
$$H(\C P^n) \simeq \frac{\Z[x]}{x^{n+1}}$$
avec $\deg(x) = 2$.
\end{exemple}


\section{CW-complexe} \label{def:decomp}

On fait quelques rappels sur les $CW$-complexes, en commençant par leur définition et en expliquant l'homologie cellulaire ensuite. Toute comme pour la section précédente, tous les détails se trouvent dans \cite{Hatcher}.

\subsection{Définition} 
 
Un \emph{$CW$-complexe} est une construction inductive définie comme suit :

\begin{enumerate}
\item On commence par un ensemble discret de points $X^0$ qu'on appelle des cellules de dimension $0$ et notées $e^0_\alpha$ pour $\alpha \in X^0$.
\item On construit $X^n$ à partir de $X^{n-1}$ en attachant des cellules $e^n_\alpha$ de dimension $n$ par des applications continues
$$\phi_\alpha : S^{n-1} \rightarrow X^{n-1}$$
c'est-à-dire que $X^n$ est l'espace quotient de l'union disjointe de $X^{n-1}$ et d'une collection de disques de dimension $n$, $\sqcup_\alpha D_\alpha^n$, sous l'identification $x \sim \phi_a(x)$ pour tout $x \in \partial D_\alpha^n$.
\item On peut soit arrêter ce processus pour un $n$ fini, donnant $X^n$ pour un certain $n$, soit prendre l'union de tous les $X^n$ munie de la topologie faible, c'est-à-dire qu'un ensemble est ouvert si son intersection avec tout  $X^n$ est ouverte dans $X^n$.
\end{enumerate}
On appelle \emph{décomposition cellulaire} d'un espace $X$ un $CW$-complexe $X^n$ tel que $X^n \simeq X$.

\begin{exemple}
On construit une décomposition cellulaire du tore $T^2$.
\begin{enumerate}
\item On commence par fixer un point $p$ donnant $X^0$.
\item On attache deux cellules $e^1_1, e^1_2 \simeq [0,1]$ de dimension 1 dont les bords sont attachés à $p$, donnant $X^1$ (qui est donc un bouquet de deux cercles).
\item On attache une cellule $e^2_1 \simeq [0,1] \times [0,1]$ de dimension 2 dont le bord est attaché par
\begin{align*}
\{0\}\times [0,1] \rightarrow e^1_1 &: (0,y) \mapsto y, \\
\{1\}\times [0,1] \rightarrow e^1_1 &: (1,y) \mapsto y, \\
[0,1]\times\{0\} \rightarrow e^1_2 &: (x,0) \mapsto x, \\
[0,1] \times \{1\} \rightarrow e^1_2 &: (x,1) \mapsto x. 
\end{align*}
\end{enumerate}
Si on considère  $T^2$ comme un pavé $[0,1] \times [0,1]$ où on identifie les côtés opposés deux à deux, en envoyant $p$ vers le point bleu, $e^1_1$ vers le segment vert et $e^1_2$ vers le segment rouge et $e^2_1$ vers le pavé entier, on a visuellement que la décomposition cellulaire donne la Figure \ref{fig:tore2decomp}.

\begin{figure}[h]
    \center
    \includegraphics[width=3cm]{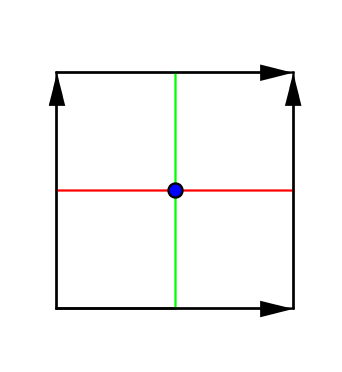}
    \caption{\label{fig:tore2decomp} On décompose $T^2$ en cellules : un point central en bleu $(p,p)$, un segment de droite $(p,-) \setminus \{(p,p)\}$, un segment de droite $(-,p) \setminus \{(p,p)\}$ et le restant de la surface.}
\end{figure}

\end{exemple}

\subsection{Homologie cellulaire}

Tout comme on définit l'homologie singulière à partir des simplexes singuliers, il est possible de définir une homologie cellulaire à partir des cellules d'un $CW$-complexe. Formellement, on la définit à partir des groupes d'homologie singulière relatifs $H_n(X^n, X^{n-1})$, mais ici on la définit directement à partir de chaines de cellules et d'un opérateur de bord défini sur celles-ci.

\begin{definition}
On définit pour un $CW-$complexe $X$ le groupe des chaines cellulaires de dimension $n$, $C_n^{CW}(X)$, comme le groupe abélien libre engendré par les cellules de dimension $n$ de $X$.\\
On définit l'opérateur de bord $d_n : C_n^{CW}(X) \rightarrow C_{n-1}^{CW}(X)$ comme
$$d_n(e_\alpha^n) = \sum_\beta d_{\alpha\beta} e^{n-1}_\beta$$
où $d_{\alpha\beta}$ est le degré de l'application 
$$S^{n-1}_\alpha \overset{\phi_a}{\rightarrow} X^{n-1} \rightarrow S^{n-1}_\beta$$
où la seconde application est celle qui identifie $X^{n-1} \setminus \{e^{n-1}_\beta\}$ en un point.
\end{definition}
On peut montrer que $d_{n} \circ d_{n+1} = 0$ et donc on peut définir des groupes d'homologie cellulaire
$$H_n^{CW} = \frac{\ker d_n}{\img d_{n+1}}.$$
On peut montrer que ces groupes sont isomorphes à ceux de l'homologie singulière :
\begin{theoreme}\emph{(\cite[Théorème 2.35, p. 139]{Hatcher})}
Soit $X$ un $CW$-complexe. Alors pour tout $n$ on a
$$H_n^{CW}(X) \simeq H_n(X).$$
\end{theoreme}
Cela donne des informations intéressantes sur l'homologie singulière d'un espace topologique $X$, notamment le fait que le rang du $n$-ème groupe d'homologie est borné par le nombre de cellules de dimension $n$ de toute décomposition cellulaire de $X$.

\begin{exemple}\label{ex:torehomology}
Il est connu que le tore $T^n$ peut être vu comme une hyperboite de dimension $n$ où on identifie deux à deux les faces opposées. On considère une décomposition cellulaire donnée par un point de base qui est un sommet de l'hyperboite pour $X^0$. Ensuite on attache $n$ cellules de dimension $1$, une pour chaque arête de l'hyperboite (à cause des identifications des faces, il y a des arêtes identifiées ensemble et donc on a $n$ différentes). Ensuite, chaque cellule de dimension $k$ est donnée par un choix de $k$ arêtes parmi les $n$ (par exemple $2$ arêtes donnent une face de dimension $2$ avec comme bord la première arête suivi de la seconde, moins la première et moins la seconde). On a donc une décomposition en $\begin{pmatrix} n \\ k \end{pmatrix}$ cellules de dimension $k$. On remarque que l'opérateur de bord est nul (le bord est donné par des paires d'hyperfaces qui s'annulent deux à deux) et donc on obtient que le groupe d'homologie $H_k(T^n)$ est le groupe abélien libre engendré par les cellules de dimension $k$, donc de rang $\begin{pmatrix} n \\ k \end{pmatrix}$.
\end{exemple}

\section{Groupoïdes et cohomologie de catégories} \label{sec:groupoide} 

Le but de cette section est de donner les idées derrière la notion de groupoïde et cohomologie de catégories. On renvoie vers \cite{simphomtheory}, \cite{introtohom} et \cite{may} pour plus de détails, on se limite ici à une approche un peu simple qui donne les définitions et propriétés nécessaires pour ce qu'on en fait dans ce travail. Un groupoïde est une sorte de groupe où l'opération de composition n'est définie que partiellement, c'est-à-dire qu'on ne peut composer que certains éléments du groupoïde. Tout comme un groupe peut s'exprimer comme une catégorie avec un seul élément (les flèches étant les éléments du groupe qu'on compose sur un unique objet abstrait), on a la définition suivante pour un groupoïde :
\begin{definition}
Un \emph{groupoïde} $G$ est une petite catégorie dans laquelle chaque flèche est un isomorphisme (donc possède un inverse).
\end{definition}
Il existe aussi une définition algébrique où on définit le groupoïde comme un ensemble munit d'une opération d'inverse et d'une composition partielle mais on n'en a pas besoin.
Afin de définir une notion de cohomologie sur les groupoïdes, il nous faut définir une notion de simplexe, ce qu'on appelle le nerf d'une catégorie. L'idée est simple : si on a deux morphismes composable $A_0 \xrightarrow{f} A_1$ et $A_1 \xrightarrow{g} A_2$ alors on a un morphisme $A_0 \xrightarrow{gf} A_2$, formant donc un triangle
$$\xymatrix{
A_1 \ar[dr]^{g} &\\
A_0 \ar[u]^{f} \ar[r]_{gf} & A_2
}$$
ce qui est semblable à un $2-$simplexe.

\begin{definition}
Le \emph{nerf} $N(D)$ d'une petite catégorie $D$ est donné par l'ensemble des $N_n(D)$ où on définit $N_0(D)$ comme l'ensemble des objets de $D$ et $N_n(D)$ comme l'ensemble des compositions de $n$ morphismes
$$A_0 \rightarrow A_1 \rightarrow \dots \rightarrow A_n$$
munis des opérateurs de faces (donc qui prennent chacune une face du simplexe)
$$\delta_i^n : N_n(D) \rightarrow N_{n-1}(D)$$
envoyant pour $1 \le i \le n-1$ 
$$A_0 \rightarrow \dots  \rightarrow A_{i-1} \rightarrow A_i \rightarrow A_{i+1} \rightarrow \dots \rightarrow A_n$$
sur
$$A_0 \rightarrow \dots  \rightarrow A_{i-1} \rightarrow A_{i+1} \rightarrow \dots \rightarrow A_n$$
en composant les morphismes $ A_{i-1} \rightarrow A_i $ et $A_i \rightarrow A_{i+1}$ et on définit $\delta_i^n$ pour $i \in \{0,n\}$ comme la composition où on retire l'élément $A_i$. On définit aussi les opérateurs de dégénérescence $s_i^n : N_n(D) \rightarrow N_{n+1}(D)$ où on ajoute l'identité sur $A_i$.
\end{definition}

On peut alors maintenant définir une théorie de cohomologie simpliciale pour une catégorie $D$ (et donc pour un groupoïde) sur un groupe abélien $A$. On définit d'abord les \emph{complexes de cochaines} comme 
$$C^n := \Hom(\Z[N_n(D)],A)$$
avec l'\emph{opérateur de cobord} $d^n :  C^{n-1} \rightarrow C^{n}$ donné pour un certain $\phi \in C^{n-1}$ par
$$d^n \phi := \sum_{i=0}^{n} (-1)^i (\phi \circ \delta^n_i).$$
On peut alors facilement calculer que $d_{n+1} \circ d_n = 0$ et on définit le $n$-ème \emph{groupe de cohomologie simpliciale} par
$$H^n(D, A) := \frac{\ker d_{n+1}}{\img d_n}.$$
Afin de calculer les groupes de cohomologie d'une catégorie, il existe un outil utile donné par la réalisation géométrique et le théorème qui suit. Tout cela existe de façon générale pour des ensembles simpliciaux, mais on en a pas besoin pour cette discussion et on renvoie vers \cite[Chapitre 3]{may} pour de plus amples détails.

\begin{definition}
La \emph{réalisation géométrique} $|N(D)|$ du nerf $N(D)$ d'une catégorie $D$ est défini par
$$|N(D)| := \frac{\coprod_{n \in \N} N_n(D) \times \Delta_n}{\sim}$$
où $\Delta_n$ est donné par l'union disjointe d'autant de $n-$simplexes que d'éléments dans $N_n(D)$ et $\sim$ est défini par
\begin{align*}
(\delta_i^n(s), x) &\sim (s, \delta_i^n(x)),  & (s_i^n(s), x) &\sim (s, s_i^n(x)),
\end{align*}
avec $\delta_i^n(x)$ la $i$-ème face du $n$-simplexe et $s_i^n(x)$ le $(n+1)$-simplexe écrasé sur sa $i$-ème face.
\end{definition}
Autrement dit, on identifie chaque simplexe du nerf à un simplexe topologique avec les faces attachées comme il faut.

\begin{theoreme}\emph{(\cite[Proposition 16.2, p. 63]{may})}
Pour une catégorie petite $D$ et un groupe abélien $A$ il y a un isomorphisme
$$H^n\bigl(|N(D)|, A\bigr) \simeq H^n\bigl(D,A\bigr).$$
\end{theoreme}
Cela signifie donc que pour calculer la cohomologie simpliciale d'une catégorie, il suffit de connaitre la cohomologie simpliciale (ou singulière) de la réalisation géométrique de son nerf.

\begin{exemple}
On considère la catégorie $\Z$ possédant un seul objet $\ast$ et dont les flèches sont indexées par les entiers avec la composition donnée par la somme
$$\ast \xrightarrow{a} \ast \xrightarrow{b} \ast = \ast \xrightarrow{a+b} \ast$$
pour tout $a,b \in \Z$. Puisque $\Z$ est un groupe, on appelle sa réalisation géométrique son espace de classification, noté $B\Z$. On peut montrer que $B\Z \simeq S^1$ puisque 
$$\ast \xrightarrow{n} \ast = \ast \xrightarrow{1} \ast \xrightarrow{1} \dots \xrightarrow{1} \ast$$
signifiant qu'on a un $1$-cycle qui génère tous les $n$-cycles. On obtient alors que pour tout groupe abélien $A$
$$H^n(\Z, A) \simeq H^n(B\Z,A) \simeq H^n(S^1,A).$$
\end{exemple}

\begin{exemple}\label{ex:groupoidsimplexe}
On considère un groupoïde $G$ possédant $m+1$ objets et tel que pour toute paire d'objets $A, B \in G$ il existe une unique flèche $A \rightarrow B$ (qui est donc l'inverse de l'unique flèche $B \rightarrow A$). Autrement dit, $G$ a une forme de $m$-simplexe, ce qui donne la réalisation géométrique
$$|N(G)| \simeq \Delta^{m}$$
et pour tout groupe abélien $A$ on obtient
\begin{align*}
H^0(G,A) &\simeq H^0(\Delta^{m},A) \simeq \Z,\\
H^k(G, A) &\simeq H^k(\Delta^{m}, A) \simeq \{0\}, & 1 \le k \le m-1, \\
H^{m}(G, A) &\simeq H^{m}(\Delta^{m}, A) \simeq \Z, \\
H^k(G, A) &\simeq H^k(\Delta^{m}, A) \simeq \{0\}, & k \ge m+1. \\
\end{align*}
\end{exemple}


\end{document}